\documentclass[9pt]{article}
\usepackage{amssymb}
\usepackage{tikz}
\usetikzlibrary{arrows, calc, decorations.pathmorphing, backgrounds, positioning, fit, petri, automata}
\definecolor{grey1}{rgb}{0.5,0.5,0.5}
\usepackage[top=2.54cm, bottom=2.54cm, left=2.2cm, right=2.2cm]{geometry}
\usepackage{algorithmicx}
\usepackage{algcompatible}
\usepackage{algorithm}

\usepackage{changebar}

\usepackage{hyperref}
\hypersetup{ 
  colorlinks,
  citecolor=blue,
  linkcolor=red,
  urlcolor=magenta}

\usepackage{mathtools}
\usepackage{graphicx}
\usepackage{subfigure}
\usepackage{amsmath,amsthm,bm,bbm}
\usepackage{rotating}
\usepackage{mathrsfs}
\usepackage{chngcntr}
\usepackage{apptools}
\usepackage[titletoc,title]{appendix}
\usepackage{comment}

\usepackage{xcolor}         
\definecolor{grau}{rgb}{0.8,0.8,0.8}

\newcommand{\chen}[1]{\color{orange}}
\usepackage{setspace,lscape,longtable}
\usepackage{amsmath,amsthm,bm}
\usepackage{fullpage}  

\usepackage{titlesec}
\usepackage[english]{babel}
\usepackage[shortlabels]{enumitem}

\numberwithin{equation}{section}
\theoremstyle{plain}

\newtheorem{assumption}{Assumption}

\newtheorem{theorem}{Theorem}[section]
\newtheorem{corollary}{Corollary}[section]
\newtheorem{lemma}[theorem]{Lemma}
\theoremstyle{remark}
\newtheorem{definition}[theorem]{Definition}

\newtheorem{result}{Result}[section]
\newtheorem{remark}{Remark}

\DeclareMathOperator*{\argmax}{arg\,max}

\DeclareMathOperator*{\Log}{Log}

\newcommand{\prob}{{\mathbb{P}}}
\newcommand{\var}{{\mathrm{var}}}
\newcommand{\expect}{\mathbb{E}}
\newcommand{\vect}{\mathrm{vec}}

\newcommand{\transpose}{^{\mathrm{T}}}

\newcommand{\calA}{{\mathcal{A}}}

\newcommand{\calE}{{\mathcal{E}}}

\newcommand{\calI}{{\mathcal{I}}}
\newcommand{\calJ}{{\mathcal{J}}}

\newcommand{\calL}{{\mathcal{L}}}

\newcommand{\calO}{{\mathcal{O}}}

\newcommand{\calS}{{\mathcal{S}}}

\newcommand{\calX}{{\mathcal{X}}}

\newcommand{\bT}{{\mathbf{T}}}
\newcommand{\bU}{{\mathbf{U}}}
\newcommand{\bP}{{\mathbf{P}}}

\newcommand{\bv}{{\mathbf{v}}}
\newcommand{\bx}{{\mathbf{x}}}
\newcommand{\bX}{{\mathbf{X}}}
\newcommand{\bA}{{\mathbf{A}}}
\newcommand{\bB}{{\mathbf{B}}}

\newcommand{\bD}{{\mathbf{D}}}
\newcommand{\bE}{{\mathbf{E}}}

\newcommand{\bG}{{\mathbf{G}}}
\newcommand{\bH}{{\mathbf{H}}}
\newcommand{\bM}{{\mathbf{M}}}

\newcommand{\bR}{{\mathbf{R}}}
\newcommand{\bQ}{{\mathbf{Q}}}
\newcommand{\bS}{{\mathbf{S}}}
\newcommand{\bW}{{\mathbf{W}}}
\newcommand{\bZ}{{\mathbf{Z}}}
\newcommand{\bw}{{\mathbf{w}}}

\newcommand{\bg}{{\mathbf{g}}}
\newcommand{\bh}{{\mathbf{h}}}

\newcommand{\br}{{\mathbf{r}}}

\newcommand{\bt}{{\mathbf{t}}}
\newcommand{\bu}{{\mathbf{u}}}

\newcommand{\by}{{\mathbf{y}}}
\newcommand{\bz}{{\mathbf{z}}}

\newcommand{\be}{{\mathbf{e}}}

\newcommand{\bV}{{\mathbf{V}}}

\newcommand{\bGamma}{{\bm{\Gamma}}}

\newcommand{\bDelta}{{\bm{\Delta}}}
\newcommand{\bOmega}{{\bm{\Omega}}}
\newcommand{\bomega}{{\bm{\omega}}}

\newcommand{\bSigma}{{\bm{\Sigma}}}
\newcommand{\bxi}{{\boldsymbol{\xi}}}

\newcommand{\eye}{{\mathbf{I}}}
\newcommand{\one}{{\mathbf{1}}}
\newcommand{\bTheta}{{\bm{\Theta}}}
\newcommand{\btheta}{{\bm{\theta}}}

\newcommand{\bmu}{{\bm{\mu}}}

\newcommand{\bPi}{{\bm{\Pi}}}

\newcommand{\bphi}{{\bm{\phi}}}

\newcommand{\zero}{{\bm{0}}}
\newcommand{\eps}{\epsilon}

\newcommand{\keywords}[1]{\par\addvspace\baselineskip\noindent\enspace\ignorespaces \textbf{Keywords: }#1}

\author{Fangzheng Xie}
\author{Fangzheng Xie\thanks{Department of Statistics, Indiana University}
}
\title{\bf Entrywise limit theorems of eigenvectors for signal-plus-noise matrix models with weak signals}
\begin{document}
\allowdisplaybreaks

\maketitle

\begin{abstract}
We establish a finite-sample Berry-Esseen theorem for the entrywise limits of the eigenvectors 
for a broad collection of signal-plus-noise random matrix models under challenging weak signal regimes. The signal strength is characterized by a scaling factor $\rho_n$ through $n\rho_n$, where $n$ is the dimension of the random matrix, and we allow $n\rho_n$ to grow at the rate of $\log n$.
The key technical contribution is a sharp finite-sample entrywise eigenvector perturbation bound. The existing error bounds on the two-to-infinity norms of the higher-order remainders are not sufficient when $n\rho_n$ is proportional to $\log n$. We apply the general entrywise eigenvector analysis results to the symmetric noisy matrix completion problem, random dot product graphs, and two subsequent inference tasks for random graphs: the estimation of pure nodes in mixed membership stochastic block models and the hypothesis testing of the equality of latent positions in random graphs. 
\end{abstract}


\tableofcontents

\section{Introduction} 
\label{sec:introduction}


In the contemporary world of data science, many statistical problems involve random matrix models with low-rank structures. Random matrices with low expected rank, also referred to as the signal-plus-noise matrix models, are pervasive in many applications, including social networks \cite{HOLLAND1983109,nickel2008random,young2007random}, compressed sensing \cite{1614066,eldar2012compressed}, and recommendation systems \cite{bennett2007netflix,goldberg1992using}. A broad range of statistical models also fall into the category of signal-plus-noise matrix models, such as the low-rank matrix denoising model \cite{chatterjee2015,10.1214/14-AOS1257,SHABALIN201367}, matrix completion problems \cite{tight_oracle_inequalities,candes2009exact,keshavan2010matrix}, principal component analysis \cite{Anderson2003multivariate,doi:10.1198/jasa.2009.0121}, and stochastic block models \cite{abbe2017community,HOLLAND1983109}. 

In
signal-plus-noise matrix models, spectral estimators and eigenvectors of random matrices have been extensively explored. These estimators can either be applied to obtain the desired inference results \cite{868688,rohe2011,sussman2012consistent} or serve as ideal initial guesses of certain iterative algorithms \cite{gao2017achieving,keshavan2010matrix,xie2019efficient}. 
The theoretical support of spectral estimators
is fundamentally backboned by the matrix perturbation theory \cite{10.1214/17-AOS1541,doi:10.1137/0707001,Stewart90,wedin1972perturbation} and the recent progress in random matrix theory \cite{bai2010spectral,BENAYCHGEORGES2011494,PAUL20141,yao2015sample}. From the practical perspective, the implementation of these spectral-based estimators typically only requires the truncated spectral/singular value decomposition of the data matrix, which is computationally cheap. In contrast, the maximum likelihood estimators for low-rank matrix models are less preferred because they are intractable to compute in general due to the nonconvex optimization problems involved \cite{10.1214/19-AOS1854}.


\subsection{Overview} 
\label{sub:overview}

This paper investigates the entrywise behavior of the leading eigenvectors of a symmetric random  matrix $\bA$ whose expected value $\bP = \expect\bA$ has a low rank. This class of random matrix models is referred to as the signal-plus-noise matrix models (see Section \ref{sub:setup} for the formal description). We establish a generic finite-sample Berry-Esseen theorem for the rows of the leading eigenvectors under challenging weak signal regimes.
The resulting Berry-Esseen bound is quite general and allows for a possibly increasing $\mathrm{rank}(\bP)$. 

As a special case of the entrywise eigenvector limit theorem for the signal-plus-noise matrix models, we obtain the Berry-Esseen bounds for the rows of the eigenvectors of a random matrix generated from the symmetric noisy matrix completion model (see Section \ref{sub:SNMC} for the formal definition). 
Our analysis is sharper than the  two-to-infinity norm error bounds for the eigenvectors
obtained by \cite{10.1214/19-AOS1854}.

Our generic entrywise Berry-Esseen theorem leads to the limit results of the rows of the adjacency spectral embedding of the random dot product graph model (see Section \ref{sub:RDPG} for the formal definition) under the sparse regime that the graph average expected degree is at the order of $\Omega(\log n)$, where $n$ is the number of vertices. The sparsity assumption is minimal because the graph adjacency matrix $\bA$ no longer concentrates around its expected value $\bP$ when the average expected degree is $o(\log n)$. Our result also relaxes the sparsity assumptions posited in \cite{cape2019signal,tang2018,xie2019efficient}. 

Leveraging the generic entrywise eigenvector concentration bound for the signal-plus-noise matrix models, we further study the entrywise limit theorem of the one-step refinement of the eigenvectors for random dot product graphs proposed in \cite{xie2019efficient}. 
The corresponding covariance matrix of the rows of the one-step estimator is no greater than that of the rows of the eigenvectors. 
We then investigate the impact of the one-step estimator for two subsequent inference tasks. Specifically, the one-step estimator has smaller asymptotic variances than the eigenvectors for estimating the pure nodes in mixed membership stochastic block models; It also leads to a more powerful test than the eigenvectors for testing the equality of latent positions in random dot product graphs.


\subsection{A motivating example} 
\label{sub:a_motivating_example}

Let us take a glimpse into a simple yet popular random graph model that has attracted much attention in the recent decade: the stochastic block model. Consider a graph with $n$ vertices that are labeled as $[n]:=\{1,2,\ldots,n\}$. These vertices are partitioned into two communities by a community assignment rule $\tau:[n]\to\{1,2\}$, where $\tau(i) = 1$ indicates that vertex $i$ lies in the first community, and $\tau(i) = 2$ otherwise. Let $\bA = [A_{ij}]_{n\times n}$ be the adjacency matrix of the stochastic block model, $\rho_n\in (0, 1]$ be the sparsity factor, and $a,b\in (0, 1)$ be constants. For each vertex pair $(i, j)$ with $i \leq j$, $(A_{ij})_{i\leq j}$ are independent, $A_{ij}\sim \mathrm{Bernoulli}(\rho_n a)$ if $\tau(i) = \tau(j)$, $A_{ij}\sim \mathrm{Bernoulli}(\rho_n b)$ if $\tau(i) \neq \tau(j)$,
and $A_{ij} = A_{ji}$ for all $i > j$. Here, $\rho_na$ and $\rho_nb$ represent the within-community probability and between-community probability, respectively, and $n\rho_n$ governs the growing rate of the graph average expected degree as a function of $n$.  

The stochastic block models were first introduced in \cite{HOLLAND1983109} and have motivated the development of network science and analysis substantially in recent years. There have also been countless papers addressing statistical analyses of stochastic block models and their fundamental limits. The readers are referred to \cite{abbe2017community} for a survey. 

A fundamental inference task for stochastic block models is the community detection, namely, the recovery of the cluster assignment rule $\tau$. In the context of the aforementioned two-block stochastic block model, we are particularly interested in the case where $n\rho_n = \alpha\log n$ for some constant $\alpha > 0$. 
There are, however, other fundamental aspects of the behavior of the leading eigenvectors of the adjacency matrix $\bA$ beyond the community detection. In this work, we focus on the asymptotic distribution of the rows of the leading eigenvector matrix of $\bA$. We begin the analysis with the population eigenvectors. For simplicity, we assume that $n$ is an even integer, $\tau(i) = 1$ if $i = 1,\ldots,n/2$, and $\tau(i) = 2$ if $i = n/2 + 1,\ldots,n$. Namely, the first $n/2$ vertices are in the first community, and the rest of the $n/2$ vertices fall into the second community. 
The non-zero eigenvalues of $\expect\bA$ are $\lambda_1 = n\rho_n(a + b)/2$ and $\lambda_2 = n\rho_n(a - b)/2$, and the associated eigenvectors are 
$\bu_1 = n^{-1/2}[1,\ldots,1]\transpose$ and $\bu_2 = n^{-1/2}[1,\ldots,1,-1,\ldots,-1]\transpose$. 
We also consider the scaled eigenvectors $\bv_1 = \lambda_1^{1/2}\bu_1$ and $\bv_2 = \lambda_2^{1/2}\bu_2$. 
Because $\bv_1$ and $\bu_1$ are non-informative for the community structure whereas the signs of $\bv_2$ and $\bu_2$ encode the community assignment, we focus on $\bv_2$ and $\bu_2$. Let $\widehat{\bv}_2 = [\widehat{v}_{12},\ldots,\widehat{v}_{n2}]\transpose{}$ be the eigenvector of $\bA$ associated with the second largest eigenvalue $\widehat{\lambda}_2$ of $\bA$ and $\widehat{\bu}_2 = \widehat{\bv}_2/\|\widehat{\bv}_2\|_2$. We scale $\widehat{\bv}_2$ such that $\|\widehat{\bv}_2\|_2 = |\widehat{\lambda}_2|^{1/2}$ to keep the scaling consistent. 

To explore the entrywise asymptotic distributions of $\widehat{\bv}_2$ and $\widehat{\bu}_2$, 
we consider
the following decompositions motivated by \cite{10.1214/19-AOS1854} and \cite{cape2019signal} for each fixed $i\in[n]$:
\begin{align}
\label{eqn:motivating_example_scaled_eigenvector_decomposition}
\sqrt{n}(\widehat{v}_{i2} - v_{i2}) &= \sqrt{n}\sum_{j = 1}^n\frac{(A_{ij} - \expect A_{ij})v_{2j}}{\lambda_2} + \sqrt{n}\mathrel{\Big(}\widehat{v}_{i2} - \sum_{j = 1}^n\frac{A_{ij}v_{j2}}{\lambda_2}\mathrel{\Big)},\\
\label{eqn:motivating_example_unscaled_eigenvector_decomposition}
n\rho_n^{1/2}(\widehat{u}_{i2} - u_{i2}) &= n\rho_n^{1/2}\sum_{j = 1}^n\frac{(A_{ij} - \expect A_{ij})u_{2j}}{\lambda_2} + n\rho_n^{1/2}\mathrel{\Big(}\widehat{u}_{i2} - \sum_{j = 1}^n\frac{A_{ij}u_{j2}}{\lambda_2}\mathrel{\Big)}.
\end{align}
The key observation is that the first terms on the right-hand sides of \eqref{eqn:motivating_example_scaled_eigenvector_decomposition} and \eqref{eqn:motivating_example_unscaled_eigenvector_decomposition} are two sums of independent mean-zero random variables. These two terms converge to $\mathrm{N}(0, (a + b)/(a - b))$ and $\mathrm{N}(0, 2(a + b)/(a - b)^2)$ in distribution, respectively, by Lyapunov's central limit theorem (see, e.g., Theorem 7.1.2. in \cite{chung2001course}). The technical challenge lies in sharp controls of the second terms arising in these equations.

We pause the theoretical discussion for a moment and turn to a simulation study. The parameters for the simulation are set as follows: $n = 5000$, $\alpha = 5$, $a = 0.9$, $b = 0.05$, and $n\rho_n = \alpha\log n$. We then generate $3000$ independent Monte Carlo replicates of $\bA$ and compute the corresponding eigenvectors $\widehat{\bv}_2$ and $\widehat{\bu}_2$.
Below, the left panels of Figures \ref{fig:motivating_example_scaled_eigenvector} and \ref{fig:motivating_example_unscaled_eigenvector} visualize the histograms of $\sqrt{n}(\widehat{v}_{12} - v_{12})$ and $n\rho_n^{1/2}(\widehat{u}_{12} - u_{12})$ (for the vertex $i = 1$), respectively. The shapes of the two histograms are closely aligned with the corresponding asymptotic normal densities.
This observation leads to the conjecture that $\sqrt{n}(\widehat{v}_{i2} - v_{i2})$ and $n\rho_n^{1/2}(\widehat{u}_{i2} - u_{i2})$ are asymptotically normal.
\begin{figure}[t]
\includegraphics[width = 15cm]{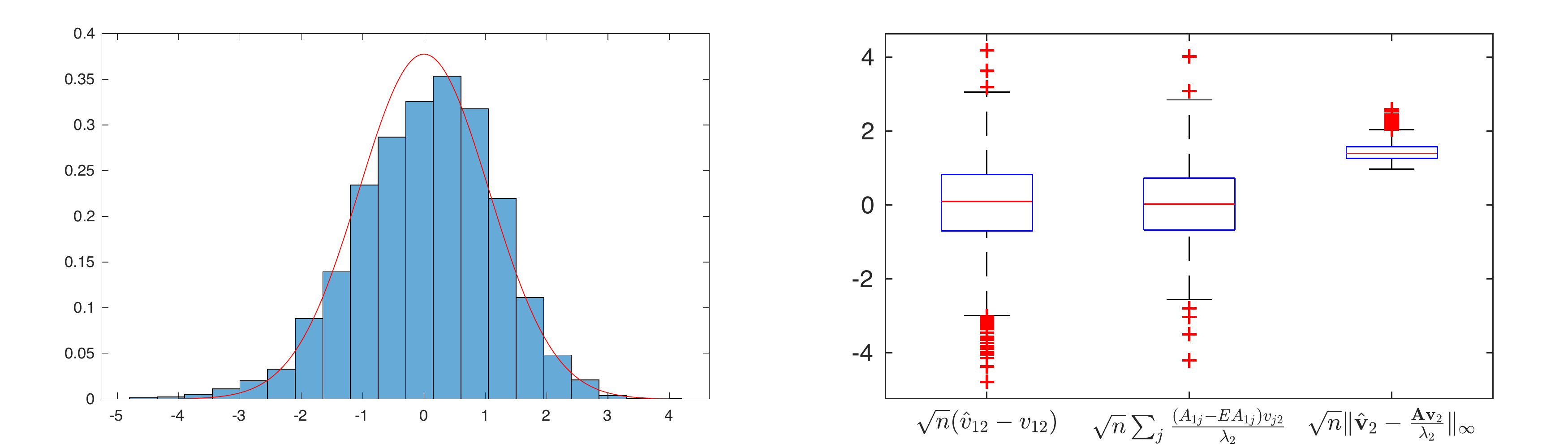}
\caption{Left panel: The histogram of $\sqrt{n}(\widehat{v}_{12} - v_{12})$ over the $3000$ Monte Carlo replicates with the density of $\mathrm{N}(0, (a + b)/(a - b))$ highlighted in the red curve. Right panel: The boxplots of $\sqrt{n}(\widehat{v}_{12} - v_{12})$, its linear approximation $\sqrt{n}\sum_j(A_{1j} - \expect A_{1j})v_{j2}/\lambda_2$, and the infinity norm of the remainder $\sqrt{n}\|\widehat{\bv}_2 - \bA\bv_2/\lambda_2\|_\infty$ across the $3000$ Monte Carlo replicates. }
\label{fig:motivating_example_scaled_eigenvector}
\end{figure}

\begin{figure}[t]
\includegraphics[width = 15cm]{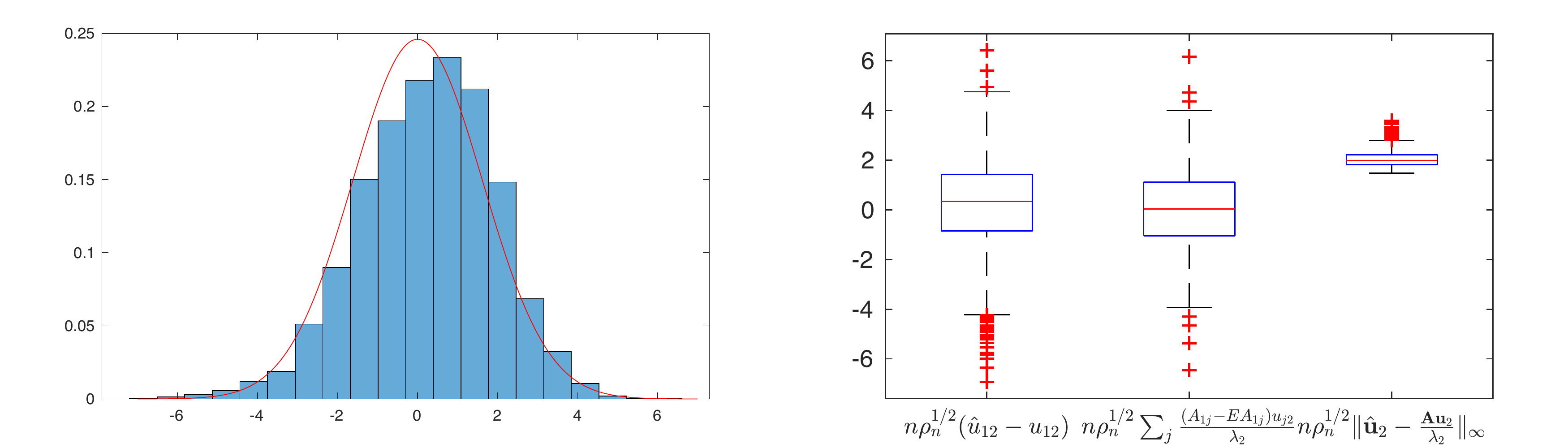}
\caption{Left panel: The histogram of $n\rho_n^{1/2}(\widehat{u}_{12} - u_{12})$ over the $3000$ Monte Carlo replicates with the density of $\mathrm{N}(0, 2(a + b)/(a - b)^2)$ highlighted in the red curve. Right panel: The boxplots of $n\rho_n^{1/2}(\widehat{u}_{12} - u_{12})$, its linear approximation $n\rho_n^{1/2}\sum_j(A_{1j} - \expect A_{1j})u_{j2}/\lambda_2$, and the infinity norm of the remainder $n\rho_n^{1/2}\|\widehat{\bu}_2 - \bA\bu_2/\lambda_2\|_\infty$ across the $3000$ Monte Carlo replicates. }
\label{fig:motivating_example_unscaled_eigenvector}
\end{figure}

Continuing the theoretical investigation of $\sqrt{n}(\widehat{v}_{i2} - v_{i2})$ and $n\rho_n^{1/2}(\widehat{u}_{i2} - u_{i2})$,
we can write \eqref{eqn:motivating_example_scaled_eigenvector_decomposition} and \eqref{eqn:motivating_example_unscaled_eigenvector_decomposition} alternatively as
\begin{align*}
\sqrt{n}(\widehat\bv_2 - \bv_2) & = \sqrt{n}\frac{(\bA - \mathbb{E}\bA)\bv_2}{\lambda_2} + \sqrt{n}\mathrel{\Big(}\widehat{\bv}_{2} - \frac{\bA\bv_2}{\lambda_2}\mathrel{\Big)},\\
n\rho_n^{1/2}(\widehat\bu_2 - \bu_2) & = n\rho_n^{1/2}\frac{(\bA - \mathbb{E}\bA)\bu_2}{\lambda_2} + n\rho_n^{1/2}\mathrel{\Big(}\widehat{\bu}_{2} - \frac{\bA\bu_2}{\lambda_2}\mathrel{\Big)}.
\end{align*}
One seemingly plausible approach is to show that $\sqrt{n}\|\widehat\bv_2 - \bA\bv_2/\lambda_2\|_\infty$ and $n\rho_n^{1/2}\|\widehat\bu_2 - \bA\bu_2/\lambda_2\|_\infty$ are $o_{\prob}(1)$
using the recently developed tools in 
\cite{10.1214/19-AOS1854,cape2017two,cape2019signal,pmlr-v83-eldridge18a,fan2018eigenvector,lei2019unified,doi:10.1080/01621459.2020.1751645}. However, the right panels of Figures \ref{fig:motivating_example_scaled_eigenvector} and \ref{fig:motivating_example_unscaled_eigenvector} suggest that this strategy may fail. Taking the unscaled eigenvectors for example, we present the boxplots of $n\rho_n^{1/2}(\widehat{u}_{12} - u_{12})$, $n\rho_n^{1/2}\sum_{j = 1}^n(A_{1j} - \expect A_{1j})u_{j2}/\lambda_2$, and $n\rho_n^{1/2}\|\widehat\bu_2 - \bA\bu_2/\lambda_2\|_\infty$ over the aforementioned $3000$ Monte Carlo replicates in the right panel of Figure \ref{fig:motivating_example_unscaled_eigenvector}. The boxplots suggest that $n\rho_n^{1/2}\|\widehat\bu_2 - \bA\bu_2/\lambda_2\|_\infty\neq o_{\prob}(1)$. 
A similar phenomenon for the scaled eigenvectors can also be observed from the right panel of Figure \ref{fig:motivating_example_scaled_eigenvector}. These numerical results motivate us to explore the entrywise limits of the eigenvectors for signal-plus-noise matrices beyond the two-to-infinity error bounds.


\subsection{Related work} 
\label{sub:related_work}

Entrywise limit theorems for the eigenvectors of random matrices first appeared in the context of network models. Based on the random dot product graph model \cite{nickel2008random,young2007random}, the authors of \cite{athreya2016limit} explored the asymptotic distributions of the rows of the eigenvectors of the random adjacency matrix for dense graphs. Generalizations of \cite{athreya2016limit} to sparse graphs were later explored in \cite{tang2018} and \cite{xie2019efficient} under a weaker condition that the average expected degree scales at $\omega((\log n)^4)$. The authors of \cite{cape2019signal} established a general entrywise limit theorem for the eigenvectors of random matrices with low expected rank by exploiting the von-Neumann matrix series expansion of the solution to a matrix Sylvester equation \cite{bhatia2013matrix,pmlr-v83-eldridge18a}. 
Recently, a general framework for studying the asymptotic theory of eigenvectors for generalized spiked Wigner models has been developed in \cite{doi:10.1080/01621459.2020.1840990}.

Another line of the related research is on the two-to-infinity norm error bounds for eigenvectors of random matrices. 
Previously, the authors of \cite{lyzinski2014} have explored the exact community detection of stochastic block models by studying the eigenvector error bound with respect to the two-to-infinity norm. 
Recently, the authors of \cite{cape2017two} established a general framework for studying the two-to-infinity norm eigenvector perturbation bounds.
However, the deterministic nature of their approach may lead to sub-optimal results in challenging low signal-to-noise ratio regimes \cite{10.1214/19-AOS1854}. 
Since then, several related papers have emerged to address the entrywise eigenvector estimation problems under various contexts \cite{10.1214/19-AOS1854,agterberg2021entrywise,cape2019signal,lei2019unified,xia2019sup}.

The literature on the specific applications considered in this paper is quite rich. The symmetric noisy matrix completion models can be viewed as a special case of the general noisy matrix completion models for rectangular random matrices, which have seen enormous progress in the past decades. For an incomplete list of reference, see \cite{bennett2007netflix,tight_oracle_inequalities,candes2009exact,5452187,chatterjee2015,5714248,jain2013low,keshavan2010matrix,10.1214/11-AOS894}. The random dot product graphs, which were originally developed for social networks \cite{nickel2008random,young2007random}, have been extensively studied in recent years, including the theoretical properties \cite{athreya2016limit,sussman2012consistent,tang2017,tang2018,tang2013,xie2019efficient,xie2019optimal} and the involved applications \cite{priebe2017semiparametric,8570772}. We refer to the survey paper \cite{JMLR:v18:17-448} for a review of random dot product graphs.



\subsection{Organization} 
\label{sub:organization}

The rest of the paper is structured as follows. Section \ref{sec:preliminaries} sets the stage for the generic signal-plus-noise matrix models and introduces the corresponding entrywise eigenvector analysis framework. Section \ref{sec:entrywise_limit_theorem_for_the_eigenvectors}, which is the main technical contribution of this paper, elaborates on the Berry-Esseen theorem for the rows of the eigenvectors of the signal-plus-noise matrix models.
We apply the main results to the symmetric noisy matrix completion models and random dot product graphs in Section \ref{sec:applications}.  
Section \ref{sec:simulation_study} provides illustrative numerical examples, and we conclude the paper with some discussions concerning future extensions in Section \ref{sec:discussion}. 


\subsection{Notations} 
\label{sub:notations}

The symbol $:=$ is used to assign mathematical definitions. For any positive integer $n$, let $[n]:= \{1,2,\ldots,n\}$. The set of all positive integers is denoted by $\mathbb{N}_+$. For any $a,b\in\mathbb{R}$, we denote $a\wedge b := \min(a,b)$ and $a\vee b := \max(a,b)$.  For any two non-negative sequences $(a_n)_{n\in\mathbb{N}_+}$, $(b_n)_{n\in\mathbb{N}_+}$, we write $a_n\lesssim b_n$ ($a_n\gtrsim b_n$, resp.), if there exists some absolute constant $C > 0$, such that $a_n\leq Cb_n$ ($a_n\geq Cb_n$, resp.) for all $n\in\mathbb{N}_+$. If the constant $C$ also depends on another parameter $c$ that is independent of $n\in\mathbb{N}_+$, then we write $a_n\lesssim_cb_n$ ($a_n\gtrsim_c b_n$, resp.). We use $K_c,N_c,\ldots$ to denote constants that may depend on another parameter $c$ but is independent of the varying index $n\in\mathbb{N}_+$. Absolute constants are usually hidden using notations $\lesssim$ and $\gtrsim$, and, when necessary, we use $C_0$ and $c_0$ to denote generic absolute constants that may vary from line to line. 
We use the notation $a_n\asymp b_n$ to indicate that $a_n\lesssim b_n$ and $a_n\gtrsim b_n$. 
If $a_n/b_n$ stays bounded away from $+\infty$, we write $a_n = O(b_n)$ and $b_n = \Omega(a_n)$, and if $a_n/b_n\to 0$, we denote $a_n = o(b_n)$ and $b_n = \omega(a_n)$. 
For any symmetric positive semidefinite matrices $\bSigma$ and $\bGamma$, we denote $\bSigma\succeq \bGamma$ ($\bSigma\preceq \bGamma$, resp.), if $\bSigma - \bGamma$ ($\bGamma - \bSigma$, resp.) is positive semidefinite. When $\bSigma - \bGamma$ ($\bGamma - \bSigma$, resp.) is strictly positive definite, we use the notation $\bSigma\succ \bGamma$ ($\bSigma\prec \bGamma$, resp.). For any $d\in\mathbb{N}_+$, we use $\eye_d$ to denote the $d\times d$ identity matrix and $\zero_d$ to denote the zero vector in $\mathbb{R}^d$. For $n, d\in\mathbb{N}_+$, $n\geq d$, let $\mathbb{O}(n, d) := \{\bU\in\mathbb{R}^{n\times d}:\bU\transpose\bU = \eye_d\}$ denote the set of all orthonormal $d$-frames in $\mathbb{R}^n$. When $n = d$, we simply write $\mathbb{O}(d) = \mathbb{O}(d, d)$. 
For an $n\times d$ matrix $\bM$, we denote $\sigma_k(\bM)$ the $k$th largest singular value of $\bM$, $k \in [\min(n, d)]$. For any $j\in [n]$ and $k\in [d]$, we use $[\bM]_{j*}$ to denote its $j$th row, $[\bM]_{*k}$ to denote its $k$th column, and $[\bM]_{jk}$ to denote its $(j, k)$th entry. When $\bM\in\mathbb{R}^{n\times n}$ is a square symmetric matrix, we use $\lambda_k(\bM)$ to denote the $k$th largest eigenvalue of $\bM$, namely, $\lambda_1(\bM)\geq\lambda_2(\bM)\geq\ldots\geq \lambda_n(\bM)$, and $\mathrm{tr}(\bM)$ the trace of $\bM$ (the sum of its diagonal elements). 
If $\bM\in\mathbb{R}^{d\times d}$ is positive definite, then we let $\kappa(\bM):=\lambda_1(\bM)/\lambda_d(\bM)$ denote the condition number of $\bM$. 
The spectral norm of a rectangular matrix $\bM$, denoted by $\|\bM\|_2$, is defined as the largest singular value of $\bM$. The Frobenius norm of a rectangular matrix $\bM$, denoted by $\|\bM\|_{\mathrm{F}}$, is defined as $\|\bM\|_{\mathrm{F}} = \sqrt{\mathrm{tr}(\bM\transpose\bM)}$. We use $\|\bM\|_{2\to\infty}$ to denote the two-to-infinity norm of a matrix $\bM = [M_{jk}]_{n\times d}$, defined as $\|\bM\|_{2\to\infty} = \max_{j\in [n]}\sqrt{\sum_{k = 1}^dM_{jk}^2}$, and $\|\bM\|_\infty$ to denote the matrix infinity norm $\|\bM\|_\infty = \max_{j \in [n]}\sum_{k = 1}^d|M_{jk}|$. 
Given $d$ real numbers $a_1,\ldots,a_d\in\mathbb{R}$, we let $\mathrm{diag}(a_1,\ldots,a_d)$ to denote the $d\times d$ diagonal matrix whose $(k, k)$th element is $a_k$ for $k \in [d]$. 
For a Euclidean vector $\bx = [x_1,\ldots,x_d]\transpose\in\mathbb{R}^d$, $\|\bx\|_2$ denotes the Euclidean norm of $\bx$ given by $\|\bx\|_2 = \sqrt{\sum_{k = 1}^dx_k^2}$ and $\|\bx\|_\infty$ denotes the infinity norm of $\bx$ defined as $\|\bx\|_\infty = \max_{k\in [d]}|x_k|$. When the dimension of the underlying Euclidean space is clear, with use $\be_i$ to denote the unit basis vector whose $i$th coordinate is one and the rest of the coordinates are zeros. 



\section{Preliminaries} 
\label{sec:preliminaries}

\subsection{Setup} 
\label{sub:setup}

Consider an $n\times n$ symmetric observable data matrix $\bA$ that can be viewed as a noisy version of an unobserved low-rank signal matrix $\bP$ through the following signal-plus-noise matrix model: 
\begin{align}\label{eqn:signal_plus_noise}
\bA = \bP + \bE,
\end{align}
where $\bE$ is an $n\times n$ symmetric noise matrix that is unobserved. Suppose $\mathrm{rank}(\bP) = d$ and $d\ll n$. Let $p\in\{1,2,\ldots,d\}$ be the number of positive eigenvalues of $\bP$ and $q \overset{\Delta}{ = }d - p$ be the number of negative eigenvalues of $\bP$. Namely, 
$\lambda_1(\bP)\geq\ldots\geq\lambda_p(\bP) > 0 > \lambda_{n - q + 1}(\bP) \geq\ldots\geq\lambda_n(\bP)$.
Let $\bU_{\bP_+}\in\mathbb{O}(n, p)$ be the eigenvector matrix of $\bP$ corresponding to the positive eigenvalues $\lambda_1(\bP),\ldots,\lambda_p(\bP)$, $\bU_{\bP_-}\in\mathbb{O}(n, q)$ be the eigenvector matrix of $\bP$ corresponding to the negative eigenvalues $\lambda_{n - q + 1}(\bP),\ldots,\lambda_n(\bP)$, $\bU_\bP:=[\bU_{\bP_+}, \bU_{\bP_-}]$, $\bS_{\bP_+} := \mathrm{diag}\{\lambda_1(\bP),\ldots,\lambda_p(\bP)\}$, $\bS_{\bP_-} := \mathrm{diag}\{\lambda_{n - q + 1}(\bP),\ldots,\lambda_n(\bP)\}$,
and $\bS_\bP := \mathrm{diag}(\bS_{\bP_+},\bS_{\bP_-})$. 

The signal matrix $\bP$ is associated with a scaling factor $\rho_n\in (0, 1]$ that governs the signal strength of the model \eqref{eqn:signal_plus_noise} through $n\rho_n$. For example, in the context of network models, $n\rho_n$ controls the average expected degree of the resulting random graphs. Note that the spectral decomposition of $\bP$ can be written as $\bP = \bU_{\bP_+}|\bS_{\bP_+}|\bU_{\bP_+}\transpose - \bU_{\bP_-}|\bS_{\bP_-}|\bU_{\bP_-}\transpose$, where the absolute value $|\cdot|$ is applied entrywise on the eigenvalues. 
We define
$\bX_{\pm} := \rho_n^{-1/2}\bU_{\bP_\pm}|\bS_{\bP_\pm}|^{1/2}\bW_{\bX_{\pm}}$, where $\bW_{\bX_+}\in\mathbb{O}(p)$ and $\bW_{\bX_-}\in\mathbb{O}(q)$ are deterministic orthogonal matrices. This allows us to write $\bP$ alternatively as $\bP = \rho_n\bX_+\bX_+\transpose - \rho_n\bX_-\bX_-\transpose$. Denote $\bX := [\bX_+, \bX_-]$, $\bDelta_{n\pm} := (1/n)\bX_{\pm}\transpose\bX_{\pm}$, $\bDelta_n := (1/n)\bX\transpose\bX$, and $\bW_\bX:=\mathrm{diag}(\bW_{\bX_+}, \bX_{\bX_-})$. Clearly, $\bDelta_n = \mathrm{diag}(\bDelta_{n+},\bDelta_{n-})$ because of the orthogonality between $\bX_+$ and $\bX_-$.

The focus of this work is to characterize the entrywise limit behavior of the eigenvector matrices of the data matrix $\bA$ as the sample versions of their population counterparts $\bU_{\bP_\pm}$ and $\bX_{\pm}$. To this end, we let $\bU_{\bA_+}\in\mathbb{O}(n, p)$ be the eigenvector matrix of $\bA$ corresponding to the positive sample eigenvalues $\lambda_1(\bA),\ldots,\lambda_p(\bA)$, $\bU_{\bA_-}\in\mathbb{O}(n, q)$ be the eigenvector matrix of $\bA$ corresponding to the negative sample eigenvalues $\lambda_{n - q + 1}(\bA),\ldots,\lambda_n(\bA)$, $\bU_\bA := [\bU_{\bA_+}, \bU_{\bA_-}]$, $\bS_{\bA_+} = \mathrm{diag}\{\lambda_1(\bA),\ldots,\lambda_p(\bA)\}$,\\ $\bS_{\bA_-} = \mathrm{diag}\{\lambda_{n - q + 1}(\bA),\ldots,\lambda_n(\bA)\}$,
and $\bS_\bA := \mathrm{diag}(\bS_{\bA_+},\bS_{\bA_-})$. Let $\widetilde\bX_\pm := \bU_{\bA_\pm}|\bS_{\bA_\pm}|^{1/2}$ and $\widetilde{\bX} := [\widetilde{\bX}_+, \widetilde{\bX}_-]$. To reiterate, $\bU_{\bA_\pm}$ and $\widetilde{\bX}_\pm$ play the role of the population counterparts of $\bU_{\bP_\pm}$ and $\rho_n^{1/2}\bX_{\pm}$, respectively.

\subsection{Entrywise eigenvector analysis framework} 
\label{sub:spectral_analysis_framework_for_random_graphs}

We now briefly discuss the entrywise eigenvector analysis framework for the signal-plus-noise matrix model \eqref{eqn:signal_plus_noise}. Unlike the case in Section \ref{sub:a_motivating_example}, $\bX_\pm$ and $\bU_{\bP_\pm}$ are only identifiable up to an orthogonal matrix due to the potential multiplicity of the non-zero eigenvalues of $\bP$. To find the suitable orthogonal alignment matrix, we follow the Procrustes analysis idea in \cite{7298436,cape2020orthogonal,5714248,schonemann1966generalized}. Let $\bU_{\bP_\pm}\transpose{}\bU_{\bA_\pm}$ yield the singular value decomposition 
$\bU_{\bP_\pm}\transpose{}\bU_{\bA_\pm} = \bW_{1\pm}\mathrm{diag}\{\sigma_1(\bU_{\bP_\pm}\transpose{}\bU_{\bA_\pm}),\ldots,\sigma_d(\bU_{\bP_\pm}\transpose{}\bU_{\bA_\pm})\}\bW_{2\pm}\transpose$,
where $\bW_{1+},\bW_{2+}\in\mathbb{O}(p)$ and $\bW_{1-},\bW_{2-}\in\mathbb{O}(q)$. Denote $\bW^*_{\pm} = \mathrm{sgn}(\bU_{\bP_\pm}\transpose\bU_{\bA_\pm})$ the matrix sign of $\bU_{\bP_\pm}\transpose\bU_{\bA_\pm}$ defined as $\bW^*_{\pm} = \bW_{1\pm}\bW_{2\pm}\transpose$ \cite{10.1214/19-AOS1854,5714248} and let $\bW^* := \mathrm{diag}(\bW_+^*, \bW_-^*)$. Then the orthogonal alignment matrix between $\widetilde\bX_\pm$ and $\bX_\pm$ is selected as $\bW_\pm = (\bW^*_\pm)\transpose{}\bW_{\bX_\pm}$.
It is believable that $\widetilde\bX_\pm\bW_\pm$ and $\bU_{\bA_\pm}$ are reasonable approximations to $\rho_n^{1/2}\bX_\pm$ and $\bU_{\bP_\pm}\bW_\pm^*$, respectively.
For convenience, we denote $\bW:=\mathrm{diag}(\bW_+, \bW_-)$. 

The keystone observation of the framework lies in the following two decompositions:
\begin{align}
\label{eqn:keystone_decomposition}
\widetilde\bX_\pm\bW_\pm - \rho_n^{1/2}\bX_\pm & = \pm\rho_n^{-1/2}{\bE\bX_\pm(\bX_\pm\transpose{}\bX_\pm)^{-1}} + \left\{\widetilde\bX_\pm\bW_\pm \mp \rho_n^{-1/2}\bA\bX_\pm(\bX_\pm\transpose{}\bX_\pm)^{-1}\right\},\\
\label{eqn:keystone_decomposition_unscaled_eigenvector}
\bU_{\bA_\pm} - \bU_{\bP_\pm}\bW_\pm^* & = \bE\bU_{\bP_\pm}\bS_{\bP_\pm}^{-1}\bW^*_\pm + (\bU_{\bA_\pm} - \bA\bU_{\bP_\pm}\bS_{\bP_\pm}^{-1}\bW_\pm^*)
.
\end{align}
To see why equation \eqref{eqn:keystone_decomposition} holds, we first observe that $\bP = \rho_n\bX_+\bX_+\transpose - \rho_n\bX_-\bX_-\transpose$, so that $\bP\bX_\pm = \pm\rho_n\bX_\pm(\bX_\pm\transpose{}\bX_\pm)$,
implying that
$\pm\rho_n^{-1/2}{\bE\bX_\pm(\bX_\pm\transpose{}\bX_\pm)^{-1}} = \pm\rho_n^{-1/2}{\bA\bX_\pm(\bX_\pm\transpose{}\bX_\pm)^{-1}}  - \rho_n^{1/2}\bX_\pm$
since we assume that $\lambda_p(\bP) > 0$ and $\lambda_{n - q + 1}(\bP) < 0$.
Substituting $\rho_n^{1/2}\bX_\pm$ above to the left-hand side of \eqref{eqn:keystone_decomposition} leads to the right-hand side of \eqref{eqn:keystone_decomposition}. The argument for \eqref{eqn:keystone_decomposition_unscaled_eigenvector} is similar.
As observed in \cite{10.1214/19-AOS1854} and \cite{cape2019signal}, viewing $(\widetilde\bX_\pm,\bU_{\bA_\pm})$ and $(\rho_n^{1/2}\bX_\pm,\bU_{\bP_\pm})$ as functionals of $\bA$ and $\bP$, we see that the first terms on the right-hand sides of \eqref{eqn:keystone_decomposition} and \eqref{eqn:keystone_decomposition_unscaled_eigenvector} are linear approximations to $\widetilde\bX_\pm\bW_\pm - \rho_n^{1/2}\bX_\pm$ and $\bU_{\bA_\pm} - \bU_{\bP_\pm}\bW^*_\pm$, respectively, whereas the second terms are the higher-order remainders. 

To shed some light on the entrywise limits of $\widetilde\bX_\pm$, we fix the vertex $i\in [n]$ and re-write \eqref{eqn:keystone_decomposition} as
\begin{align*}
\bW\transpose{}(\widetilde\bx_i)_\pm - \rho_n^{1/2}(\bx_i)_\pm
& = \pm\sum_{j = 1}^n\frac{[\bE]_{ij}(\bX_\pm\transpose{}\bX_\pm)^{-1}(\bx_j)_\pm}{\rho_n^{1/2}}
 + \left\{\widetilde\bX_\pm\bW_\pm \mp \frac{\bA\bX_\pm(\bX_\pm\transpose{}\bX_\pm)^{-1}}{\rho_n^{1/2}}\right\}\transpose{}\be_i,
\end{align*}
where $(\widetilde\bx_i)_\pm$ and $(\bx_i)_\pm$ denote the $i$th row of $\widetilde{\bX}_\pm$ and $\bX_\pm$, respectively.
An immediate observation is that the first term above is a sum of independent mean-zero random variables, which is quite accessible for the analysis. The non-trivial part is a sharp control of the second term above. Using the fact that $\rho_n\bX_\pm\transpose{}\bX_\pm = \bW_{\bX_\pm}\transpose{}|\bS_{\bP_\pm}|\bW_{\bX_\pm}$, we further write
\begin{align}
\label{eqn:keystone_remainder_term}
\begin{aligned}
\widetilde\bX_\pm\bW_\pm - \frac{\pm\bA\bX_\pm(\bX_\pm\transpose{}\bX_\pm)^{-1}}{\rho_n^{1/2}}
& = \bU_{\bA_\pm}(\bW^*_\pm|\bS_{\bA_\pm}|^{1/2} - |\bS_{\bP_\pm}|^{1/2}\bW_\pm^*)\transpose\bW_{\bX_\pm}\\
&\quad + (\bU_{\bA_\pm} - \bA\bU_{\bP_\pm}\bS_{\bP_\pm}^{-1}\bW^*_\pm)(\bW^*_\pm)\transpose{}|\bS_{\bP_\pm}|^{1/2}\bW_{\bX_\pm}.
\end{aligned}
\end{align}
Because the analysis of the first line in \eqref{eqn:keystone_remainder_term} is relatively easy (see, for example, Lemma 49 in \cite{JMLR:v18:17-448}), we focus on the entrywise control of $\bU_{\bA_\pm} - \bA\bU_{\bP_\pm}\bS_{\bP_\pm}^{-1}\bW_\pm^*$ in the second line in \eqref{eqn:keystone_remainder_term}, which is also related to the entrywise limit of $\bU_{\bA_\pm}$ through \eqref{eqn:keystone_decomposition_unscaled_eigenvector}. Although there has been some recent progress on the uniform control $\|\bU_{\bA_\pm} - \bA\bU_{\bP_\pm}\bS_{\bP_\pm}^{-1}\bW_\pm^*\|_{2\to\infty}$ (see \cite{10.1214/19-AOS1854,cape2019signal,lei2019unified}), the numerical experiment in Section \ref{sub:a_motivating_example} suggests that the uniform error bound may not be sufficient for studying the entrywise limits of $\widetilde\bX_\pm$ and $\bU_{\bA_\pm}$. This motivates us to develop a sharp control of $\|\be_i\transpose(\bU_{\bA_\pm} - \bA\bU_{\bP_\pm}\bS_{\bP_\pm}^{-1}\bW_\pm^*)\|_2$ for each fixed $i\in [n]$. 



\section{Entrywise limit theorem for the eigenvectors} 
\label{sec:entrywise_limit_theorem_for_the_eigenvectors}

\subsection{Main results}
\label{sub:main_results}

This section establishes the entrywise limit results for the eigenvectors $\bU_{\bA_\pm}$ and $\bX_\pm$. We first present several necessary assumptions for the signal-plus-noise matrix model \eqref{eqn:signal_plus_noise}. 

\begin{assumption}\label{assumption:incoherence}
$\|\bX\|_{2\to\infty}$ is upper bounded by a constant.
\end{assumption}

\begin{assumption}\label{assumption:sparsity}
$\rho_n\in (0, 1]$, $\rho:=\lim_{n\to\infty}\rho_n$ exists, and $n\rho_n = \Omega(\log n)$. 
\end{assumption}

\begin{assumption}\label{assumption:distribution}
The upper diagonal entries of $\bE$, $([\bE]_{ij}:1\leq i\leq j\leq n)$, are independent mean-zero random variables; There exists mean-zero random variables $([\bE_1]_{ij}$, $[\bE_2]_{ij}:1\leq i\leq j\leq n)$, such that $[\bE]_{ij} = [\bE_1]_{ij} + [\bE_2]_{ij}$, and they satisfy the following conditions:
\begin{enumerate}[(i)]
  \item There exists constants $B,\sigma^2 > 0$ independent of $n$ such that $\max_{i,j\in [n]}|[\bE_1]_{ij}|\leq B$ with probability one and $\max_{i,j\in [n]}\var([\bE_1]_{ij})\leq \sigma^2\rho_n$. 

  \item The random variables $([\bE_2]_{ij}:1\leq i\leq j\leq n)$ are uniformly sub-Gaussian in the following sense: $\max_{i,j\in [n]}\|[\bE_2]_{ij}\|_{\psi_2}\leq \sigma\rho_n^{1/2}$ for some constant $\sigma > 0$ independent of $n$, where $\|\cdot\|_{\psi_2}$ is the sub-Gaussian norm of a random variable (see, for example, \cite{kosorok2008introduction,vershynin2010introduction}).
\end{enumerate}
\end{assumption}

\begin{assumption}\label{assumption:rowwise_concentration}
There exist absolute constants $c_0 > 0$, $\xi \geq 1$, and a non-decreasing function $\varphi(\cdot):\mathbb{R}_+\to\mathbb{R}_+$ with $\varphi(0) = 0$, $\varphi(x)/x$ non-increasing in $\mathbb{R}_+$, such that for all $i,m\in [n]$ and any deterministic $n\times d$ matrix $\bV$, with probability at least $1 - c_0n^{-(1 + \xi)}$,
\begin{align*}
\max\{\|\be_i\transpose\bE^{(m)}\bV\|_2,\|\be_i\transpose\bE\bV\|_2\}\leq n\rho_n\lambda_d(\bDelta_n)\|\bV\|_{2\to\infty}\varphi\left(\frac{\|\bV\|_{\mathrm{F}}}{\sqrt{n}\|\bV\|_{2\to\infty}}\right),
\end{align*}
where $\bE^{(m)}$ is obtained by replacing the $m$th row and $m$th column of $\bE$ by zeros. 
\end{assumption}

\begin{assumption}\label{assumption:spectral_norm_concentration}
There exist absolute constants $K, c_0 > 0$, $\zeta\geq 1$, such that $\|\bE\|_2\leq K(n\rho_n)^{1/2}$ with probability at least $1 - c_0n^{-\zeta}$ and
$32\kappa(\bDelta_n)\max\{\gamma, \varphi(\gamma)\}\leq 1$, where $\varphi(\cdot)$ is the function in Assumption \ref{assumption:rowwise_concentration} and
$\gamma := {\max\{3K, \|\bX\|_{2\to\infty}^2\}}/\{(n\rho_n)^{1/2}\lambda_d(\bDelta_{n})\}\to 0$
.
\end{assumption}

Several remarks regarding Assumptions \ref{assumption:incoherence}-\ref{assumption:spectral_norm_concentration} are in order. 
Assumption \ref{assumption:incoherence} is related to the notion of bounded coherence in random matrix theory and matrix recovery \cite{tight_oracle_inequalities,candes2009exact}. Indeed, observe that
$\|\bU_\bP\|_{2\to\infty} 
\leq {\|\bX\|_{2\to\infty}}/\sqrt{n\lambda_d(\bDelta_n)}$.
Therefore, Assumption \ref{assumption:incoherence} implies the bounded coherence of $\bU_\bP$ (i.e., $\|\bU_\bP\|_{2\to\infty}\leq C_\mu\sqrt{d/n}$ for some constant $C_\mu \geq 1$) as long as $d\lambda_d(\bDelta_n) = \Omega(1)$, which is a mild condition. 
Assumption \ref{assumption:sparsity} requires that $n\rho_n = \Omega(\log n)$. In the context of the two-block stochastic block model illustrated in Section \ref{sub:a_motivating_example}, this amounts to requiring that the average graph expected degree is $\Omega(\log n)$. 
Assumption \ref{assumption:distribution} is a general requirement for the tail of the distributions of the noise $\bE$. 
It includes a variety of popular random matrix models such as random dot product graphs, the low-rank matrix denoising model, and the matrix completion model.
Assumption \ref{assumption:rowwise_concentration} is motivated by the row-wise concentration assumption in \cite{10.1214/19-AOS1854}. The row-wise concentration behavior of $\bE$ is characterized by a function $\varphi(\cdot)$ that depends on the distributions of $\bE$ fundamentally. 
Assumption \ref{assumption:spectral_norm_concentration} is a standard assumption on the spectral concentration of the noise matrix $\bE$ and is satisfied under the binary random graph model by \cite{lei2015} and the matrix completion model by \cite{keshavan2010matrix}.

Theorem \ref{thm:ASE_Berry_Esseen} below is the main result of this section. It asserts that when $n\rho_n = \Omega(\log n)$, the distributions of the rows of $\widetilde{\bX}_\pm\bW_\pm - \rho_n^{1/2}\bX_\pm$ and $\bU_{\bA_\pm} - \bU_{\bP_\pm}\bW^*_\pm$ are approximately Gaussians. 
\begin{theorem}
\label{thm:ASE_Berry_Esseen}
Suppose Assumptions \ref{assumption:incoherence}-\ref{assumption:spectral_norm_concentration} hold. 
For each $i\in [n]$, let
\[
\bSigma_{ni\pm} = \bDelta_{n\pm}^{-1}\mathrel{\Bigg\{}\frac{1}{n\rho_n}\sum_{j = 1}^n\expect([\bE]_{ij}^2)
(\bx_{j})_\pm(\bx_{j})_\pm\transpose\mathrel{\Bigg\}}\bDelta_{n\pm}^{-1},\quad\bGamma_{ni\pm} = \bDelta_{n\pm}^{-1/2}\bSigma_{ni\pm}\bDelta_{n\pm}^{-1/2}
\]
be invertible. Define $\chi = \varphi(1) + (\|\bX\|_{2\to\infty}^2\vee 1)/\lambda_d(\bDelta_n)$.
Then for each fixed index $i\in[n]$ and for any sufficiently large $n$,
\begin{align}
\label{eqn:ASE_normality}
\begin{aligned}
&\sup_{A\in\calA}\left|\prob\left\{\sqrt{n}\bSigma_{ni\pm}^{-1/2}(\bW\transpose(\widetilde\bx_i)_\pm - \rho_n^{1/2}(\bx_{i})_\pm)\in A\right\} - \prob\left(\bz\in A\right)\right|\\
&\quad\lesssim_\sigma \frac{d^{1/2}\chi\|\bSigma_{ni\pm}^{-1/2}\|_2\|\bX\|_{2\to\infty}^2}{(n\rho_n)^{1/2}\lambda_d(\bDelta_{n})^{3/2}}
\max\left\{\frac{(\|\bX\|_{2\to\infty}^2\vee1)(\log n\rho_n)^{1/2}}{\lambda_d(\bDelta_n)^2}, \frac{\kappa(\bDelta_n)}{\lambda_d(\bDelta_n)^2}, {\log n\rho_n} \right\}\\
&\quad\quad + \frac{d^{1/2}\|\bSigma_{n\pm}^{-1/2}\|_2\|\bX\|_{2\to\infty}}{(n\rho_n)^{3/2}\lambda_d(\bDelta_n)}\sum_{j = 1}^n\expect|[\bE]_{ij}|^3(\bx_{j})_\pm\transpose{}\bDelta_{n\pm}^{-1}\bSigma_{ni\pm}^{-1}\bDelta_{n\pm}^{-1}(\bx_{j})_\pm,
\end{aligned}
\end{align}
and
\begin{align}
\label{eqn:Eigenvector_normality}
\begin{aligned}
&\sup_{A\in\calA}\left|\prob\left\{n\rho_n^{1/2}\bGamma_{ni\pm}^{-1/2}\bW_{\bX_\pm}\transpose(\bW_\pm^*[\bU_{\bA_\pm}]_{i*} - [\bU_{\bP_\pm}]_{i*})\in A\right\} - \prob\left(\bz\in A\right)\right|\\
&\quad\lesssim_\sigma \frac{d^{1/2}\chi\|\bGamma_{ni\pm}^{-1/2}\|_2\|\bX\|_{2\to\infty}}{(n\rho_n)^{1/2}\lambda_d(\bDelta_n)^{3/2}}\max\left\{\frac{(\log n\rho_n)^{1/2}}{\lambda_d(\bDelta_n)}, \frac{1}{\lambda_d(\bDelta_n)^2}, {\log n\rho_n} \right\}\\
&\quad\quad + \frac{d^{1/2}\|\bGamma_{ni\pm}^{-1/2}\|_2\|\bX\|_{2\to\infty}}{(n\rho_n)^{3/2}\lambda_d(\bDelta_{n})^{3/2}}\sum_{j = 1}^n\expect|[\bE]_{ij}|^3(\bx_{j})_\pm\transpose{}\bDelta_{n\pm}^{-3/2}\bGamma_{ni\pm}^{-1}\bDelta_{n\pm}^{-3/2}(\bx_{j})_\pm,
\end{aligned}
\end{align}
where $\calA$ is the collection of all convex measurable sets in $\mathbb{R}^d$ and $\bz\sim\mathrm{N}_d(\zero_d, \eye_d)$. 
\end{theorem}
\begin{remark}[Generality of Theorem \ref{thm:ASE_Berry_Esseen}]
Theorem \ref{thm:ASE_Berry_Esseen} is stated in terms of Berry-Esseen type bounds for $\sqrt{n}\bSigma_{ni\pm}^{-1/2}\{\bW_\pm\transpose(\widetilde{\bx}_i)_\pm - \rho_n^{1/2}(\bx_i)_\pm\}$ and $\sqrt{n}\bGamma_{ni\pm}^{-1/2}\bW_{\bX_\pm}\transpose(\bW_\pm^*[\bU_{\bA_\pm}]_{i*} - [\bU_{\bP_\pm}]_{i*})$. The upper bounds only depend on $n\rho_n$, the rank of $\bP$, the eigenvalues of $\bDelta_n$, a constant depending on $\sigma$, $\|\bX\|_{2\to\infty}$, and the third absolute moments of $[\bE]_{ij}$'s. Compared to the limit theorems in \cite{cape2019signal,tang2018,xie2019efficient}, Theorem \ref{thm:ASE_Berry_Esseen} allows the rank $d$ and the eigenvalues of $\bP$ to vary with the number of vertices $n$. Consequently, as long as the right-hand sides of \eqref{eqn:ASE_normality} and \eqref{eqn:Eigenvector_normality} converge to $0$ as $n\to\infty$, the asymptotic shapes of the distributions of $\bW_\pm\transpose(\widetilde{\bx}_i)_\pm - \rho_n^{1/2}(\bx_i)_\pm$ and $\bW_{\bX_\pm}\transpose(\bW_\pm^*[\bU_{\bA_\pm}]_{i*} - [\bU_{\bP_\pm}]_{i*})$ can be approximated by multivariate Gaussians. 
\end{remark}

The key to the proof of Theorem \ref{thm:ASE_Berry_Esseen} is Theorem \ref{thm:eigenvector_deviation} below. It provides the entrywise perturbation bounds for the eigenvectors $\bU_{\bA_\pm}$ and $\widetilde\bX_{\pm}$. 
\begin{theorem}\label{thm:eigenvector_deviation}
Suppose the conditions of Theorem \ref{thm:ASE_Berry_Esseen} hold. 
Then there exists an absolute constant $c_0 > 0$, such that for each fixed $m\in [n]$, for all $t\geq 1$, $t\lesssim n\rho_n$, for sufficiently large $n$, with probability at least $1 - c_0n^{-\zeta\wedge\xi} - c_0de^{-t}$,
\begin{align*}
&\|\be_m\transpose{}(\bU_{\bA_\pm} - \bA\bU_{\bP_\pm}\bS_{\bP_\pm}^{-1}\bW_\pm^*)\|_2
\lesssim_\sigma
\frac{\chi\|\bU_\bP\|_{2\to\infty}}{n\rho_n \lambda_d(\bDelta_n)}\max\left\{
  \frac{t^{1/2}}{\lambda_d(\bDelta_n)}, \frac{1}{\lambda_d(\bDelta_n)^2}, t\right\},\\
  &\|\be_i\transpose{}\{\widetilde\bX_\pm\bW_\pm - (\pm)\rho_n^{-1/2}\bA\bX_\pm(\bX_\pm\transpose{}\bX_\pm)^{-1}\}\|_2\\
&\quad\lesssim_\sigma \frac{\chi\|\bX\|_{2\to\infty}\|\bU_\bP\|_{2\to\infty}}{(n\rho_n)^{1/2}\lambda_d(\bDelta_n)}\max\left\{\frac{(\|\bX\|_{2\to\infty}^2\vee1)t^{1/2}}{\lambda_d(\bDelta_n)^2}, \frac{\kappa(\bDelta_n)}{\lambda_d(\bDelta_n)^2}, t\right\}.
\end{align*}
Furthermore, if $\zeta\wedge\xi > 1$, then for sufficiently large $n$, with probability at least $1 - c_0n^{-\zeta\wedge\xi}$, 
\begin{align*}
&\|\bU_{\bA_\pm} - \bA\bU_{\bP_\pm}\bS_{\bP_\pm}^{-1}\bW_\pm^*\|_{2\to\infty}
\lesssim_\sigma
\frac{\chi\|\bU_\bP\|_{2\to\infty}}{n\rho_n \lambda_d(\bDelta_n)}\max\left\{
  \frac{(\log n)^{1/2}}{\lambda_d(\bDelta_n)}, \frac{1}{\lambda_d(\bDelta_n)^2}, \log n\right\},\\
  &\|\widetilde\bX_\pm\bW_\pm - (\pm)\rho_n^{-1/2}\bA\bX_\pm(\bX_\pm\transpose{}\bX_\pm)^{-1}\|_{2\to\infty}\\
&\quad\lesssim_\sigma \frac{\chi\|\bX\|_{2\to\infty}\|\bU_\bP\|_{2\to\infty}}{(n\rho_n)^{1/2}\lambda_d(\bDelta_n)}\max\left\{\frac{(\|\bX\|_{2\to\infty}^2\vee1)(\log n)^{1/2}}{\lambda_d(\bDelta_n)^2}, \frac{\kappa(\bDelta_n)}{\lambda_d(\bDelta_n)^2}, \log n\right\}.
\end{align*}
\end{theorem}

\subsection{Comparison with existing results}
\label{sub:comparison_with_existed_results}

We first briefly compare Theorem \ref{thm:ASE_Berry_Esseen} with some existing entrywise limit theorems for the eigenvectors of signal-plus-noise matrix models. For simplicity, we assume that the non-zero eigenvalues of $\bP$ are positive. In \cite{cape2019signal}, the authors established the asymptotic normality of $\bW_\bX\transpose(\bW^*[\bU_\bA]_{i*} - [\bU_\bP]_{i*})$ when $n\rho_n = \omega((\log n)^{4\xi})$ for some constant $\xi > 1$, provided that $d$ is fixed across all $n$ and $\bDelta_n$ converges to some positive definite $\bDelta$. Later, the requirement for $n\rho_n$ is relaxed to $n\rho_n = \omega((\log n)^4)$ in \cite{xie2019efficient} for the rows of the scaled eigenvectors in the context of random dot product graphs (see Section \ref{sub:RDPG} for the formal definition). The same sparsity requirement for $n\rho_n$ was required in \cite{tang2018} when the rows of $\bX$ are i.i.d. latent random vectors, and their limit result is stated as multivariate normal mixtures.
In contrast, Theorem \ref{thm:ASE_Berry_Esseen} only requires that $n\rho_n = \Omega(\log n)$ when $d$ is fixed and $\lambda_d(\bDelta_n)$ is bounded away from $0$. 

We next provide several remarks regarding Theorem \ref{thm:eigenvector_deviation} and compare it with some results in the literature. Again, for simplicity, we assume that the non-zero eigenvalues of $\bP$ are positive and $\lambda_d(\bDelta_n)$ stays bounded away from $0$. 
Then the asymptotic normality of \eqref{eqn:Eigenvector_normality} holds only if 
\begin{align}\label{eqn:eigenvector_remainder}
&\|\be_m\transpose(\bU_\bA - \bA\bU_\bP\bS_\bP^{-1}\bW^*)\|_2 = o_{\prob}\left(\frac{1}{n\rho_n^{1/2}}\right),
\end{align}
which can be obtained from Theorem \ref{thm:eigenvector_deviation} with $t = \log n\rho_n$. 
We argue that the concentration bound \eqref{eqn:eigenvector_remainder} is sharper than the recently developed  concentration bounds for $\|\bU_\bA - \bA\bU_\bP\bS_\bP^{-1}\bW^*\|_{2\to\infty}$ in \cite{10.1214/19-AOS1854,cape2019signal,lei2019unified}.
In \cite{cape2019signal}, the authors assumed that $n\rho_n = \omega((\log n)^{2\xi})$ for some $\xi > 1$ and showed that 
\begin{align}\label{eqn:Cape_two_to_infinity_norm_bound}
\|\bU_\bA - \bA\bU_\bP\bS_\bP^{-1}\bW^*\|_{2\to\infty} = O_{\prob}\left\{\frac{1}{n\rho_n^{1/2}}\times\frac{(\log n)^{2\xi}}{(n\rho_n)^{1/2}}\right\}.
\end{align}
The bound \eqref{eqn:Cape_two_to_infinity_norm_bound} is not sufficient for \eqref{eqn:eigenvector_remainder} to occur unless $n\rho_n = \omega((\log n)^{4\xi})$. Under the most challenging regime that $n\rho_n = \Omega(\log n)$, in the context of random graph models, the authors of \cite{10.1214/19-AOS1854} and \cite{lei2019unified} have established that 
\begin{align}\label{eqn:AFWZ_two_to_infinity_norm_bound}
\|\bU_\bA - \bA\bU_\bP\bS_\bP^{-1}\bW^*\|_{2\to\infty} = O_{\prob}\left(\frac{1}{\sqrt{n}\log\log n}\right).
\end{align}
This bound leads to a sharp analysis of the community detection using the signs of the second leading eigenvector of $\bA$ for a two-block stochastic block model but does not imply \eqref{eqn:eigenvector_remainder} either. The underlying reason is that these two-to-infinity norm error bounds are obtained using a union bound, leading to sub-optimal entrywise concentration bounds for $\bU_\bA - \bA\bU_\bP\bS_\bP^{-1}\bW^*$. 



\section{Applications}
\label{sec:applications}

\subsection{Symmetric noisy matrix completion}
\label{sub:SNMC}

The matrix completion problem has been extensively explored in recent decades, and the literature review included here is by no means complete and exhaustive. It refers to a large class of random matrix problems where the observed data matrix contains partial observations, and the task of interest is to predict the missing entries. A canonical real-world application is the ``Netflix problem'' \cite{bennett2007netflix}, where the data matrix consists of multiple users' ratings of multiple movies. 
The missingness is intrinsic to the nature of the problem because it is unlikely to have the users watch all movies available in the database. 
Predicting the missing entries is worthwhile because accurate predictions allow the system to make appropriate individual-wise recommendations to the users. Theoretical properties of the matrix completion model have also been well studied. For example, the theory of noiseless matrix completion has been explored in \cite{candes2009exact,5452187,5714248}, whereas the extensions for more general noisy matrix completion problems have been developed in \cite{tight_oracle_inequalities,chatterjee2015,jain2013low,keshavan2010matrix,10.1214/11-AOS894}. 

This subsection considers a special case of the noisy matrix completion problem where the data matrix is a symmetric random matrix with missing observations, also referred to as the symmetric noisy matrix completion (SNMC) model. It also appears in the context of network cross-validation by edge sampling \cite{10.1093/biomet/asaa006}.
We follow the definition from \cite{10.1214/19-AOS1854} and assume that the non-zero eigenvalues of $\bP$ are  positive for the ease of exposition.
\begin{definition}
Let $\bX = [\bx_1,\ldots,\bx_n]\transpose\in\mathbb{R}^{n\times d}$. The symmetric noisy matrix completion model, denoted by $\mathrm{SNMC}(\rho_n\bX\bX\transpose, \rho_n, \sigma)$, is the distribution of a symmetric random matrix $\bA = [A_{ij}]_{n\times n}$ given by $A_{ij} = (\rho_n\bx_i\transpose\bx_j + \eps_{ij})I_{ij}/\rho_n$, where $(I_{ij}, \eps_{ij}:1\leq i\leq j\leq n)$ are jointly independent, $I_{ij}\sim\mathrm{Bernoulli}(\rho_n)$, $\eps_{ij}\sim\mathrm{N}(0, \sigma^2)$, and $A_{ij} = A_{ji}$ for all $i > j$. 
\end{definition}

Below, Theorem \ref{thm:SNMC} establishes the entrywise Berry-Esseen bounds for the eigenvectors of $\bA$ generated from $\mathrm{SNMC}(\rho_n\bX\bX\transpose, \rho_n, \tau\rho_n^2)$ under the conditions that $n\rho_n\lambda_d(\bDelta_n)^2 = \omega(\kappa(\bDelta_n)^2\log n)$ and $n\rho_n\geq 6\log n$. We follow the same notations and definitions in Sections \ref{sec:preliminaries} and \ref{sec:entrywise_limit_theorem_for_the_eigenvectors}. 

\begin{theorem}\label{thm:SNMC}
Suppose $\bA\sim\mathrm{SNMC}(\rho_n\bX\bX\transpose, \rho_n, \sigma)$. Assume that $n\rho_n\geq 6\log n$, $\sigma = \tau\rho_n^2$ for some constant $\tau > 0$, $\|\bX\|_{2\to\infty}$ is upper bounded by a constant, and $n\rho_n\lambda_d(\bDelta_n)^2 = \omega(\kappa(\bDelta_n)^2\log n)$. Let
\[
\bSigma_{ni} = \bDelta_n^{-1}\left\{\frac{1}{n}\sum_{j = 1}^n\{(1 - \rho_n)(\bx_i\transpose\bx_j)^2 + \tau^2\rho_n^3\}\bx_j\bx_j\transpose\right\}\bDelta_n^{-1},\quad
\bGamma_{ni} = \bDelta_n^{-1/2}\bSigma_{ni}\bDelta_n^{-1/2}.
\]
If $\bSigma_{ni}$ and $\bGamma_{ni}$ are invertible, then for each fixed index $i\in[n]$ and for any sufficiently large $n$,
\begin{align*}
&\sup_{A\in\calA}\left|\prob\left\{\sqrt{n}\bSigma_{ni}^{-1/2}(\bW\transpose\widetilde\bx_i - \rho_n^{1/2}\bx_{i})\in A\right\} - \prob\left(\bz\in A\right)\right|\\
&\quad\lesssim_\tau 
\frac{d^{1/2}\|\bSigma_{ni}^{-1/2}\|_2(\|\bX\|_{2\to\infty}^7\vee 1)}{(n\rho_n)^{1/2}\lambda_d(\bDelta_{n})^{5/2}}
\max\left\{\frac{(\|\bX\|_{2\to\infty}^2\vee1)(\log n\rho_n)^{1/2}}{\lambda_d(\bDelta_n)^2}, \frac{\kappa(\bDelta_n)}{\lambda_d(\bDelta_n)^2}, {\log n\rho_n} \right\},
\end{align*}
and
\begin{align*}
&\sup_{A\in\calA}\left|\prob\left\{n\rho_n^{1/2}\bGamma_{ni}^{-1/2}\bW_{\bX}\transpose(\bW^*[\bU_{\bA }]_{i*} - [\bU_{\bP}]_{i*})\in A\right\} - \prob\left(\bz\in A\right)\right|\\
&\quad\lesssim_\tau 
\frac{d^{1/2}\|\bGamma_{ni}^{-1/2}\|_2(\|\bX\|_{2\to\infty}^5\vee 1)}{(n\rho_n)^{1/2}\lambda_d(\bDelta_n)^{5/2}}
\max\left\{\frac{(\log n\rho_n)^{1/2}}{\lambda_d(\bDelta_n)}, \frac{1}{\lambda_d(\bDelta_n)^2}, {\log n\rho_n} \right\},
\end{align*}
where $\calA$ is the collection of all convex measurable sets in $\mathbb{R}^d$ and $\bz\sim\mathrm{N}_d(\zero_d, \eye_d)$. 
\end{theorem}

We now argue that Theorem \ref{thm:SNMC} is sharper than the two-to-infinity norm error bounds obtained in \cite{10.1214/19-AOS1854}. Again, we assume that $d$ is fixed and $\lambda_d(\bDelta_n),\lambda_1(\bDelta_n)$ are bounded away from $0$ and $\infty$ for simplicity. Under the condition that $n\rho_n = \omega(\log n)$, the asymptotic normality of the rows of $\bU_\bA$ in Theorem \ref{thm:SNMC} implies the entrywise error bound \eqref{eqn:eigenvector_remainder}:
$\|\be_m\transpose(\bU_\bA - \bA\bU_\bP\bS_\bP^{-1}\bW^*)\|_2 = o_{\prob}\{(n\rho_n^{1/2})^{-1}\}$.
In contrast, Lemma 13 in \cite{10.1214/19-AOS1854} implies that
$\|\bU_\bA - \bA\bU_\bP\bS_\bP^{-1}\bW^*\|_{2\to\infty} = O_{\prob}\{{(\log n)^{1/2}}/{(n\rho_n^{1/2})}\}$ under the same conditions.
Similar to the reasoning in Section \ref{sub:comparison_with_existed_results}, the above two-to-infinity norm error bound does not imply the error bound \eqref{eqn:eigenvector_remainder}. Therefore, Theorem \ref{thm:SNMC} provides a sharper entrywise eigenvector analysis compared to \cite{10.1214/19-AOS1854} for the symmetric noisy matrix completion model. 



\subsection{Eigenvectors of random dot product graphs}
\label{sub:RDPG}

In recent years, statistical network analysis has attracted much attention and has gained substantial progress in theoretical foundations and methodological development. Network data are also pervasive in numerous application domains, including social networks \cite{Girvan7821,wasserman1994social,young2007random}, neuroscience \cite{priebe2017semiparametric,8570772}, and computer networks \cite{6623779,7745482}. 
In the statistical analyses of network data, spectral methods and eigenvector analysis of random adjacency matrices are of fundamental interest because the eigenvectors not only contain the underlying network latent structure but also provide gateways to various subsequent inference tasks, such as community detection \cite{rohe2011,sussman2012consistent}, vertex classification \cite{6565321,tang2013}, and nonparametric graph testing \cite{tang2017}. 

In this subsection, we focus on the random dot product graph model \cite{young2007random} and study the behavior of its eigenvectors. It is a class of random graphs in which each vertex is assigned a latent position vector encoding the vertex-wise information. 
The random dot product graph model is easy to interpret (especially in social networks) and rich enough to include a variety of popular network models, including stochastic block models \cite{HOLLAND1983109} and their offspring \cite{airoldi2008mixed,PhysRevE.83.016107,7769223}. Below, we first provide the formal definition of the random dot product graph model. 


\begin{definition}
Consider a graph with $n$ vertices that are labeled as $[n] = \{1,2,\ldots,n\}$. Let $\calX$ be a subset of $\mathbb{R}^d$ such that $\bx_1\transpose{}\bx_2\in [0, 1]$ for all $\bx_1,\bx_2\in\calX$, where $d\leq n$, and let $\rho_n\in (0, 1]$ be a sparsity factor. Each vertex $i\in [n]$ is associated with a vector $\bx_i\in\calX$, referred to as the latent position for vetex $i$. We say that a symmetric random matrix $\bA = [A_{ij}]_{n\times n}\in\{0, 1\}^{n\times n}$ is the adjacency matrix of a random dot product graph with latent position matrix $\bX = [\bx_1,\ldots,\bx_n]\transpose{}$ and sparsity factor $\rho_n$, denoted by $\bA\sim\mathrm{RDPG}(\rho_n^{1/2}\bX)$, if the random variables $A_{ij}\sim\mathrm{Bernoulli}(\rho_n\bx_{i}\transpose{}\bx_{j})$ independently for all $i,j\in [n]$, $i\leq j$, and $A_{ij} = A_{ji}$ for all $i > j$. 
\end{definition}
The sparsity factor $\rho_n$ in a random dot product graph model $\mathrm{RDPG}(\rho_n^{1/2}\bX)$ fundamentally controls the graph average expected degree through $n\rho_n$ as a function of the number of vertices, provided that $\sum_{i,j}\bx_i\transpose\bx_j = \Omega(n^2)$. When $\rho_n\equiv 1$, the resulting graph is dense, and the average expected degree scales as $\Omega(n)$. The more interesting scenario happens when $\rho_n\to 0$ as $n\to\infty$, which gives rise to a sparse random graph whose average expected degree is a vanishing proportion of the number of vertices. A fast decaying $\rho_n$ corresponds to a challenging weak signal regime, which is one of the focuses of this subsection.

We now present the Berry-Esseen theorem for the rows of the leading eigenvectors for random dot product graphs. The scaled eigenvector matrix $\widetilde{\bX}$ is also referred to as the adjacency spectral embedding of $\bA$ into $\mathbb{R}^d$ \cite{sussman2012consistent}. 

\begin{theorem}
\label{thm:ASE_Berry_Esseen_RDPG}
Let $\bA\sim\mathrm{RDPG}(\rho_n^{1/2}\bX)$ with $n\rho_n\gtrsim \log n$. Denote $\bDelta_n = (1/n)\bX\transpose{}\bX$ and suppose there exists a constant $\delta > 0$ such that $\min_{i\in [n]}(1/n)\sum_{j = 1}^n\bx_i\transpose\bx_j\geq\delta$. For each $i\in [n]$, let
\[
\bSigma_n(\bx_{i}) = \bDelta_n^{-1}\left\{\frac{1}{n}\sum_{j = 1}^n\bx_{i}\transpose{}\bx_{j}(1 - \rho_n\bx_{i}\transpose\bx_{j})\bx_{j}\bx_{j}\transpose\right\}\bDelta^{-1},\quad\bGamma_n(\bx_i) = \bDelta_n^{-1/2}\bSigma_n(\bx_i)\bDelta_n^{-1/2}.
\]
If $\bSigma_n(\bx_i)$ and $\bGamma_n(\bx_i)$ are invertible and
${\kappa(\bDelta_n)}/{\lambda_d(\bDelta_n)}\{{(n\rho_n)^{-1/2}}\vee
{\log( {n\rho_n}\lambda_d(\bDelta_n)^2)}^{-1}
\}\to 0$,
then for each fixed index $i\in[n]$ and for any sufficiently large $n$,
\begin{align*}
&\begin{aligned}
&\sup_{A\in\calA}\left|\prob\left\{\sqrt{n}\bSigma_n(\bx_{i})^{-1/2}(\bW\transpose\widetilde\bx_i  - \rho_n^{1/2}\bx_{i})\in A\right\} - \prob\left(\bz\in A\right)\right|\\
&\quad\lesssim 
\frac{d^{1/2}\|\bSigma_{n}(\bx_i)^{-1/2}\|_2}{(n\rho_n)^{1/2}\lambda_d(\bDelta_n)^{5/2}}\max\left\{\frac{(\log n\rho_n)^{1/2}}{\lambda_d(\bDelta_n)^2}, \frac{\kappa(\bDelta_n)}{\lambda_d(\bDelta_n)^2}, {\log n\rho_n} \right\},
\end{aligned}\\
&\begin{aligned}
&\sup_{A\in\calA}\left|\prob\left\{n\rho_n^{1/2}\bGamma_n(\bx_{i})^{-1/2}\bW_\bX\transpose(\bW^*[\bU_\bA]_{i*} - [\bU_\bP]_{i*})\in A\right\} - \prob\left(\bz\in A\right)\right|\\
&\quad\lesssim \frac{d^{1/2}\|\bGamma_n(\bx_i)^{-1/2}\|_2}{(n\rho_n)^{1/2}\lambda_d(\bDelta_n)^{5/2}}\max\left\{\frac{(\log n\rho_n)^{1/2}}{\lambda_d(\bDelta_n)}, \frac{1}{\lambda_d(\bDelta_n)^2}, {\log n\rho_n}\right\},
\end{aligned}
\end{align*}
where $\calA$ is the collection of all convex measurable sets in $\mathbb{R}^d$ and $\bz\sim\mathrm{N}_d(\zero_d, \eye_d)$. 
\end{theorem}

Compared to the eigenvector limit theorems for random dot product graphs in \cite{athreya2016limit,cape2019signal,tang2018,xie2019efficient}, Theorem \ref{thm:ASE_Berry_Esseen_RDPG} requires a much weaker sparsity condition on $\rho_n$. Specifically, the authors \cite{athreya2016limit} explored the entrywise eigenvector limits by assuming that $\rho_n\equiv 1$ and the minimal sparsity condition in \cite{cape2019signal,tang2018,xie2019efficient} is $n\rho_n = \omega((\log n)^4)$. In contrast, in Theorem \ref{thm:ASE_Berry_Esseen_RDPG}, we only require that $n\rho_n = \Omega(\log n)$ if the eigenvalues of $\bDelta_n$ are bounded away from $0$ and $\infty$. As mentioned in Section \ref{sub:overview}, our sparsity assumption $n\rho_n = \Omega(\log n)$ is minimal because $\bA$ no longer concentrates around $\bP$ in spectral norm when $n\rho_n = o(\log n)$ \cite{tang2018}. 

Next, we establish the two-to-infinity norm perturbation bounds for the eigenvectors of random dot product graphs in Corollary \ref{corr:Two_to_infinity_norm_eigenvector_bound} below. 
\begin{corollary}\label{corr:Two_to_infinity_norm_eigenvector_bound}
Suppose $\bA\sim\mathrm{RDPG}(\rho_n^{1/2}\bX)$ and the conditions of Theorem \ref{thm:ASE_Berry_Esseen} hold. Denote $\bDelta_n = (1/n)\bX\transpose{}\bX$. 
Then there exists an absolute constant $c_0 > 0$, such that given any fixed $c > 0$, 
\begin{align*}
\|\bU_\bA - \bA\bU_\bP\bS_\bP^{-1}\bW^*\|_{2\to\infty}&\lesssim_c \frac{\|\bU_\bP\|_{2\to\infty}}{n \rho_n \lambda_d(\bDelta_n)^2}\max\left\{\frac{(\log n)^{1/2}}{\lambda_d(\bDelta_n)}, \frac{1}{ \lambda_d(\bDelta_n)^2},{\log n} \right\},\\
\|\bU_\bA -  \bU_\bP\bW^*\|_{2\to\infty}&\lesssim_c 
\|\bU_\bA - \bA\bU_\bP\bS_\bP^{-1}\bW^*\|_{2\to\infty} + \frac{(\log n)^{1/2}\|\bU_\bP\|_{2\to\infty}}{(n\rho_n)^{1/2}\lambda_d(\bDelta_n)},
\\
\left\|\widetilde{\bX}\bW - \frac{\bA\bX(\bX\transpose\bX)^{-1}}{\rho_n^{1/2}}\right\|_{2\to\infty}
&\lesssim_c
\frac{\|\bU_\bP\|_{2\to\infty}}{(n\rho_n)^{1/2}\lambda_d(\bDelta_n)^{2}}\max\left\{
\frac{(\log n)^{1/2}}{\lambda_d(\bDelta_n)^2}, \frac{\kappa(\bDelta_n)}{\lambda_d(\bDelta_n)^2}, \log n
\right\},\\
\|\widetilde\bX\bW - \rho_n^{1/2}\bX\|_{2\to\infty}
&\lesssim_c \left\|\widetilde{\bX}\bW - \frac{\bA\bX(\bX\transpose\bX)^{-1}}{\rho_n^{1/2}}\right\|_{2\to\infty} + \frac{(\log n)^{1/2}\|\bU_\bP\|_{2\to\infty}}{\lambda_d(\bDelta_n)^{1/2}}
\end{align*}
with probability at least $1 - c_0n^{-c}$ for sufficiently large $n$. 
\end{corollary}

Corollary \ref{corr:Two_to_infinity_norm_eigenvector_bound} provides a sharp concentration bound for $\|\bU_\bA - \bU_\bP\bW^*\|_{2\to\infty}$ compared to some recently obtained results. Assuming that $\lambda_d(\bDelta_n)$ is bounded away from $0$ for simplicity, we see that Corollary \ref{corr:Two_to_infinity_norm_eigenvector_bound} leads to 
$\|\bU_\bA - \bU_\bP\bW^*\|_{2\to\infty}\lesssim_{\lambda_d(\bDelta_n)} \sqrt{{(\log n)}/{(n\rho_n)}}\|\bU_\bP\|_{2\to\infty}$
with high probability. This also coincides with the concentration bound obtained in \cite{lei2019unified}. 
In \cite{cape2019signal} and \cite{doi:10.1080/01621459.2020.1751645}, it has been shown that 
$\|\bU_\bA - \bU_\bP\bW^*\|_{2\to\infty}\lesssim_{\lambda_d(\bDelta_n)} \sqrt{{(\log n)^{2\xi}}/{(n\rho_n)}}\|\bU_\bP\|_{2\to\infty}$
with high probability under a stronger assumption that $n\rho_n = \omega((\log n)^{2\xi})$ for some $\xi > 1$. Our result is tighter than the above large probability bound by a $(\log n)^{\xi - 1/2}$ factor. In \cite{10.1214/19-AOS1854}, the authors proved that 
$\|\bU_\bA - \bU_\bP\bW^*\|_{2\to\infty}\lesssim_{\lambda_d(\bDelta_n)} \|\bU_\bP\|_{2\to\infty}$
with high probability, which coincides with Corollary \ref{corr:Two_to_infinity_norm_eigenvector_bound} when $n\rho_n\asymp \log n$ but deteriorates when $n\rho_n = \omega(\log n)$. 

We also remark that the concentration bound on $\|\widetilde\bX\bW - \rho_n^{1/2}\bX\|_{2\to\infty}$ plays a fundamental role in establishing the entrywise limit theorem for the one-step estimator in Section \ref{sub:entrywise_limit_theorem_for_the_one_step_estimator} next. 

\subsection{One-step estimator for random dot product graphs} 
\label{sub:entrywise_limit_theorem_for_the_one_step_estimator}

We continue the investigation of the entrywise estimation of the eigenvectors of random dot product graphs. As observed in \cite{xie2019efficient}, the adjacency spectral embedding (the scaled eigenvector matrix $\widetilde{\bX}$) can be further refined by a one-step procedure implemented in the following vertex-wise fashion. 
\begin{definition}
Let $\bA\sim\mathrm{RDPG}(\rho_n^{1/2}\bX)$ and $\widetilde\bX = [\widetilde\bx_1,\ldots,\widetilde\bx_n]\transpose\in\mathbb{R}^{n\times d}$ be the adjacency spectral embedding of $\bA$ into $\mathbb{R}^d$. Then the one-step refinement of $\widetilde\bX$ is the $n\times d$ matrix $\widehat\bX = [\widehat\bx_1,\ldots,\widehat\bx_n]\transpose{}$, whose $i$th row $\widehat\bx_i$ is given by
\begin{align}\label{eqn:one_step_estimator}
\widehat\bx_i = \widetilde\bx_i + \left\{ \sum_{j = 1}^n\frac{\widetilde\bx_{j}\widetilde\bx_{j}\transpose}{\widetilde\bx_{i}\transpose\widetilde\bx_{j}(1 - \widetilde\bx_i\transpose\widetilde\bx_j)}\right\}^{-1}\left\{ \sum_{j = 1}^n\frac{(A_{ij} - \widetilde\bx_i\transpose\widetilde\bx_j)\widetilde\bx_j}{\widetilde\bx_{i}\transpose\widetilde\bx_{j}(1 - \widetilde\bx_i\transpose\widetilde\bx_j)}\right\},\quad i = 1,2,\ldots,n.
\end{align}
\end{definition}
The one-step refinement above is motivated by the one-step estimator in the classical M-estimation theory for parametric models (see, for example, Section 5.7 in \cite{van2000asymptotic}). In short, under mild conditions, given a root-n consistent initial estimator, the one-step refinement achieves the information lower bound in a parametric model asymptotically. The same idea also applies to the random dot product graph model. Denote $\ell_\bA(\bX)$ the log-likelihood function of $\mathrm{RDPG}(\rho_n^{1/2}\bX)$. Then a straightforward computation shows that the score function and the Fisher information matrix with regard to $\bx_i$ are
\[
\nabla_{\bx_i}\ell_\bA(\bX) = \sum_{j = 1}^n\frac{(A_{ij} - \rho_n\bx_i\transpose{}\bx_j)\bx_j}{\bx_i\transpose{}\bx_j(1 - \rho_n\bx_i\transpose{}\bx_j)}\quad\text{and}\quad
{\calI}_i(\bX) = \rho_n\sum_{j = 1}^n\frac{\bx_j\bx_j\transpose{}}{\bx_i\transpose{}\bx_j(1 - \rho_n\bx_i\transpose{}\bx_j)}.
\]
Given the adjacency spectral embedding $\widetilde\bX$ as an initial guess, the right-hand side of \eqref{eqn:one_step_estimator} is precisely the updating rule of the Newton-Raphson algorithm for $\rho_n^{1/2}\bx_i$ initialized at $\widetilde\bx_i$, with the Hessian replaced by the negative Fisher information matrix. 

Below, Theorem \ref{thm:Berry_Esseen_OSE_multivariate} presents
the Berry-Esseen bound for the rows of the one-step refinement $\widehat{\bX}$ of the adjacency spectral embedding (the scaled eigenvectors $\widetilde\bX$). 
\begin{theorem}\label{thm:Berry_Esseen_OSE_multivariate}
Let $\bA\sim\mathrm{RDPG}(\rho_n^{1/2}\bX)$ and suppose the conditions of Theorem \ref{thm:ASE_Berry_Esseen} hold. Further assume that there exists a constant $\delta > 0$ such that $\min_{i,j\in [n]}\{\bx_i\transpose\bx_j\wedge(1 - \bx_i\transpose\bx_j)\}\geq \delta$. 
Denote $\bDelta_n = (1/n)\bX\transpose\bX$ and 
$\bG_n(\bx_{i}) = (1/n)\sum_{j = 1}^n{\bx_{j}\bx_{j}\transpose}\{\bx_{i}\transpose\bx_{j}(1 - \rho_n\bx_{i}\transpose\bx_{j})\}^{-1}$
for each $i\in [n]$. 
If 
\[
\frac{1}{(n\rho_n)\lambda_d(\bDelta_n)^{5/2}}\max\left\{
\frac{(\log n)^{1/2}}{\lambda_d(\bDelta_n)^2}, \frac{\kappa(\bDelta_n)}{\lambda_d(\bDelta_n)^2}, \log n
\right\} \to 0,
\]
then 
for each fixed index $i\in [n]$ and for all sufficiently large $n$,
\begin{align}\label{eqn:OSE_normality}
\begin{aligned}
&\sup_{A\in\calA}\left|
\prob\left\{\sqrt{n}\bG_n(\bx_{i})^{1/2}(\bW\transpose{}\widehat{\bx}_i - \rho_n^{1/2}\bx_{i})\in A\right\}
 - \prob\left(\bz\in A\right)
\right|\\
&\quad\lesssim \frac{d^{1/2}}{n\rho_n^{1/2}\delta^8\lambda_d(\bDelta_n)^{9/2}}\max\left\{
\frac{\log n\rho_n}{\lambda_d(\bDelta_n)^4}, \frac{\kappa(\bDelta_n)^2}{\lambda_d(\bDelta_n)^4}, (\log n\rho_n)^2
\right\},
\end{aligned}
\end{align}
where $\calA$ is the set of all convex measurable sets in $\mathbb{R}^d$ and $\bz\sim\mathrm{N}_d(\zero_d, \eye_d)$. 
\end{theorem}
\begin{remark}
Theorem \ref{thm:Berry_Esseen_OSE_multivariate} generalizes Theorem 5 in \cite{xie2019efficient} in the following aspects: First, we allow $n\rho_n$ to grow at $\omega(\log n)$ when $\lambda_d(\bDelta_n) = \Omega(1)$, which is significantly weaker than the assumption $n\rho_n^5 = \omega((\log n)^2)$ in \cite{xie2019efficient}; Secondly, we have the least requirement on the embedding dimension $d$ and the latent position matrix $\bX$, whereas the authors of \cite{xie2019efficient} assumed that $d$ is fixed and $\bX$ satisfies a Glivenko-Cantelli type condition. In addition, Theorem \ref{thm:Berry_Esseen_OSE_multivariate} is also stated in terms of a Berry-Esseen type bound that only depends on $n\rho_n$, the embedding dimension $d$, the eigenvalues of $\bDelta_n$, and a constant $\delta$ governing the entries of $\bX\bX\transpose$. Hence, the rows of $\widehat{\bX}\bW - \rho_n^{1/2}\bX$ can be approximated by a multivariate Gaussian as long as the right-hand side of \eqref{eqn:OSE_normality} converges to $0$. 
\end{remark}

The authors of \cite{xie2019efficient} have shown that the covariance matrix $\bG_n(\bx_i)^{-1}$ for the rows of $\widehat{\bX}$ satisfies $\bG_n(\bx_i)^{-1}\preceq \bSigma_n(\bx_i)$. Consequently, the one-step refinement of $\widetilde\bX$ reduces the asymptotic variance of the rows of the scaled eigenvectors $\widetilde\bX$ in spectra. This result is particularly useful in stochastic block models whose block probability matrix is rank-deficient (see Section \ref{sec:simulation_study} below for a numerical example). 

Theorem \ref{thm:asymptotic_normality_OS} below provides a row-wise concentration bound for the one-step refinement $\widehat{\bX}$ and is instrumental towards establishing Theorem \ref{thm:asymptotic_normality_OS}. It also generalizes Theorem 4 in \cite{xie2019efficient}.
\begin{theorem}\label{thm:asymptotic_normality_OS}
Let $\bA\sim\mathrm{RDPG}(\rho_n^{1/2}\bX)$ and assume the conditions of Theorem \ref{thm:Berry_Esseen_OSE_multivariate} hold. Then 
\begin{align}
\label{eqn:OSE_main_decomposition}
\bG_n(\bx_{i})^{1/2}(\bW\transpose{}\widehat{\bx}_i - \rho_n^{1/2}\bx_{i})
& = \frac{1}{n\rho_n^{1/2}}\sum_{j = 1}^n\frac{(A_{ij} - \rho_n\bx_{i}\transpose{}\bx_{j})}{\bx_{i}\transpose{}\bx_{j}(1 - \rho_n\bx_i\transpose{}\bx_{j})}\bG_n(\bx_{i})^{-1/2}\bx_{j} + \widehat{\br}_i,
\end{align}
where, given any fixed $c > 0$, 
for all 
$t\geq 1$, $t\lesssim \log n$, and sufficiently large $n$,
the remainder $\widehat{\br}_i$ satisfies
\[
\|\widehat{\br}_i\|_2\lesssim_c \frac{1}{n\rho_n^{1/2}\delta^8\lambda_d(\bDelta_n)^{9/2}}\max\left\{
\frac{t}{\lambda_d(\bDelta_n)^4}, \frac{\kappa(\bDelta_n)^2}{\lambda_d(\bDelta_n)^4}, t^2
\right\}
\]
with probability at least $1 - c_0n^{-c} - c_0e^{-t}$ for some absolute constant $c_0 > 0$. 
\end{theorem}

\subsection{Eigenvector-based subsequent inference for random graphs}
\label{sub:subsequent_graph_inference}

In this subsection, we apply the theory in Sections \ref{sub:RDPG} and \ref{sub:entrywise_limit_theorem_for_the_one_step_estimator} to two subsequent random graph inference problems: the estimation of pure nodes in mixed membership stochastic block models and the hypothesis testing of the equality of latent positions in random dot product graphs. 

\subsubsection*{Pure node estimation in mixed membership stochastic block models}

The mixed membership stochastic block model \cite{airoldi2008mixed} generalizes the stochastic block model \cite{HOLLAND1983109} in which the community memberships are continuously relaxed. Each vertex can have multiple community memberships governed by a probability vector called the community membership profile. There have been several works that explore the computation algorithms for mixed membership stochastic block models \cite{airoldi2008mixed,Gopalan14534}. 
There have also been several recent attempts in exploring the theoretical aspects of mixed membership stochastic block models (see, for example, \cite{anandkumar2014tensor,hopkins2017efficient,jin2017estimating,pmlr-v70-mao17a,doi:10.1080/01621459.2020.1751645,doi:10.1137/19M1272238}). 

We first introduce the formal definition of the mixed membership stochastic block models. 

\begin{definition}
Let $\bTheta = [\theta_{jk}]_{n\times d}\in[0, 1]^{n\times d}$ be the membership profile matrix with $\sum_{k = 1}^d\theta_{jk} = 1$ for all $j\in [n]$, $\bB\in (0, 1)^{d\times d}$ be the block probability matrix, and $\rho_n\in (0, 1]$ be the sparsity factor. We say that a symmetric random matrix $\bA\in\{0, 1\}^{n\times n}$ is the adjacency matrix of a mixed membership stochastic block model $\mathrm{MMSBM}(\bTheta, \bB, \rho_n)$, if $A_{ij}\sim\mathrm{Bernoulli}(\rho_n\btheta_i\transpose\bB\btheta_j)$ independently for all $i,j\in [n]$, $i\leq j$, and $A_{ij} = A_{ji}$ for all $i > j$, where $\btheta_i = [\theta_{i1},\ldots,\theta_{id}]\transpose$. 
\end{definition}

For simplicity, we assume that the block probability matrix $\bB$ is positive definite and there exist $\bX^* = [\bx_1^*,\ldots,\bx_d^*]\transpose\in\mathbb{R}^{d\times d}$ such that $\bB = (\bX^*)(\bX^*)\transpose$. 
Namely,
$\bA\sim\mathrm{MMSBM}(\bTheta, (\bX^*)(\bX^*)\transpose, \rho_n)$ implies that $\bA\sim\mathrm{RDPG}(\rho_n^{1/2}\bTheta\bX^*)$. 
Geometrically, the latent positions (i.e., the rows of $\bTheta\bX^*$) can be viewed as scatter points taken from a simplex whose corners are the rows of $\bX^*$, and the rows of $\bX^*$ are referred to as the pure nodes \cite{doi:10.1080/01621459.2020.1751645}. 
A standard condition for estimating the membership profile matrix $\bTheta$ is the existence of a pure node for each community \cite{doi:10.1080/01621459.2020.1751645}. Formally, we say that each of the $d$ communities contains at least one pure node, if the vertex set $\{i\in [n]:\btheta_i = \be_k\}$ is non-empty for each $k\in [d]$. Then there exists $d$ distinct row indices $i_1,\ldots,i_d\in[n]$ such that $i_k = \min\{i\in [n]:\btheta_i = \be_k\}$, where $\btheta_i$ is the $i$th row of $\bTheta$, $i\in [n]$. Namely, $\{i_1,\ldots,i_d\}$ are the vertices in the graph whose latent positions are exactly given by one of the pure nodes. 

Given $\bA\sim\mathrm{MMSBM}(\bTheta, (\bX^*)(\bX^*)\transpose, \rho_n)$, an important inference task is to detect and estimate the pure nodes $\bx_1^*,\ldots,\bx_d^*$. There are several earlier attempts in detecting the row indices corresponding to the pure nodes \cite{6656801,jin2017estimating,pmlr-v70-mao17a,doi:10.1080/01621459.2020.1751645}. These algorithms are based on the finding that the corners of a simplex have the highest norm (see Lemma 2.1 in \cite{doi:10.1080/01621459.2020.1751645}). Here, we adopt the successive projection algorithm proposed in \cite{6656801}. The detailed algorithm is provided in the Supplementary Material for completeness. 

We now construct two estimators for the pure nodes in a mixed membership stochastic block model based on the adjacency spectral embedding $\widetilde{\bX} = [\widetilde{\bx}_1,\ldots,\widetilde{\bx}_n]\transpose$ (i.e., the scaled eigenvectors) and its one-step refinement $\widehat{\bX} = [\widehat{\bx}_1,\ldots,\widehat{\bx}_n]\transpose$. Let $J := \texttt{SPA}(\bA, d)$ be the output row indices of the successive projection algorithm (see Algorithm 1 in the Supplementary Material) and $\bV_\bA\in\mathbb{R}^{d\times d}$ the sub-matrix of $\bU_\bA$ corresponding to the row indices in $J$. We then estimate the membership profile matrix $\bTheta$ by $\widehat{\bTheta} := \bU_\bA\bV_\bA^{-1}$. Define 
\begin{align}\label{eqn:pure_node_index}
\iota_k := \min\left\{i\in [n]:\|\be_i\transpose\widehat{\bTheta} - \be_k\transpose\|_2\leq \eta\right\},\quad k\in[d],
\end{align}
where $\eta > 0$ is a tuning parameter taken to be sufficiently small. 
Note that the membership profile matrix is only identifiable up to a permutation. The two estimators for $\bx_k^*$ (modulus a permutation) are then given by $\widetilde{\bx}_{\iota_k}$ and $\widehat{\bx}_{\iota_k}$, which are based on $\widetilde{\bX}$ and $\widehat{\bX}$, respectively. Leveraging Corollary \ref{corr:Two_to_infinity_norm_eigenvector_bound}, Theorem \ref{thm:ASE_Berry_Esseen_RDPG}, and Theorem \ref{thm:asymptotic_normality_OS}, we establish the two-to-infinity norm error bound for $\widehat{\bTheta}$ and the asymptotic normality of $\widetilde{\bx}_{\iota_k}$ and $\widehat{\bx}_{\iota_k}$ in Theorem \ref{thm:MMSBM} below.


\begin{theorem}
\label{thm:MMSBM}
Suppose $\bA\sim\mathrm{MMSBM}(\bTheta, (\bX^*)(\bX^*)\transpose, \rho_n)$ and the following conditions hold:
\begin{enumerate}[(i)]
  \item There exists at least one pure node for each of the $d$ communities.

  \item $n\rho_n = \omega(\log n)$ and $d$ is fixed.

  \item There exists a positive constant $c_1 > 0$ such that $\min\{n^{-1/2}\sigma_d(\bTheta), \sigma_d(\bX^*)\}\geq c_1$. 

  \item There exists a positive constant $\delta > 0$ such that $\min_{k,l\in [d]}[(\bx_k^*)\transpose(\bx_l^*)\wedge \{1 - (\bx_k^*)\transpose(\bx_l^*)\}] \geq \delta$. 

\end{enumerate}
Then for each sufficiently large $n$, there exists a permutation matrix $\bPi_n\in\{0, 1\}^{d\times d}$, such that with probability at least $1 - c_0n^{-2}$, $\|\widehat{\bTheta} - \bTheta\bPi_n\|_{2\to\infty}\leq K\sqrt{{(\log n)}/{(n\rho_n)}}$ 
for some constants $K, c_0 > 0$. Furthermore, if $\min_{i\in [n],\btheta_i\neq \be_k}\|\btheta_i - \be_k\|_2 \geq c_2$ for a constant $c_2 > 0$ for all $k\in [d]$ and $\eta\leq c_2/2$, then there exists a sequence of permutations $(\pi_n)_n$ over $[d]$, such that for each $k\in [d]$,
\begin{align*}
\sqrt{n}\bSigma_n(\bx_{\pi_n(k)}^*)^{-1/2}(\bW\transpose\widetilde{\bx}_{\iota_k} - \rho_n^{1/2}\bx_{\pi_n(k)}^*)\overset{\calL}{\to}\mathrm{N}_d(\zero_d, \eye_d),\\
\sqrt{n}\bG_n(\bx_{\pi_n(k)}^*)^{1/2}(\bW\transpose\widehat{\bx}_{\iota_k} - \rho_n^{1/2}\bx_{\pi_n(k)}^*)\overset{\calL}{\to}\mathrm{N}_d(\zero_d, \eye_d),
\end{align*}
where $\bSigma_n(\cdot)$ and $\bG_n(\cdot)$ are defined in Theorem \ref{thm:ASE_Berry_Esseen_RDPG} and Theorem \ref{thm:Berry_Esseen_OSE_multivariate}, respectively, with $\bX:=\bTheta\bX^*$. 
\end{theorem}

The implication of Theorem \ref{thm:MMSBM} is two-fold. Firstly, we establish the following uniform error bound for the membership profile estimator $\widehat{\bTheta}$: 
$\|\widehat{\bTheta} - \bTheta\bPi_n\|_{2\to\infty} = O\{\sqrt{{(\log n)}{(n\rho_n)}}\}$ with probability at least $1 - O(n^{-2})$,
where $\bPi_n$ is a $d\times d$ permutation matrix. 
This concentration bound is sharper than that in \cite{doi:10.1080/01621459.2020.1751645} by a poly-$\log n$ factor and our sparsity assumption is weaker: we only assume that $n\rho_n = \omega(\log n)$, whereas the authors of \cite{doi:10.1080/01621459.2020.1751645} required that $n\rho_n = \Omega((\log n)^{2\xi})$ for some constant $\xi > 1$. Secondly, we show the asymptotic normality for the pure node estimators $\widetilde{\bx}_{\iota_k}$ based on the adjacency spectral embedding and $\widehat{\bx}_{\iota_k}$ based on the one-step estimator, with the asymptotic covariance matrices being $\bSigma_n(\bx_{\pi_n(k)}^*)$ and $\bG_n(\bx_{\pi_n(k)}^*)^{-1}$, respectively. By Theorem 2 in \cite{xie2019efficient}, we have $\bG_n(\bx_{\pi_n(k)}^*)^{-1}\preceq \bSigma_n(\bx_{\pi_n(k)}^*)$. Therefore, the estimator $\widehat{\bx}_{\iota_k}$, which is derived from the one-step estimator $\widehat{\bX}$, improves upon the eigenvector-based estimator $\widetilde{\bx}_{\iota_k}$ with a smaller asymptotic covariance matrix in spectra. 

\subsubsection*{Hypothesis testing for equality of latent positions}

The second subsequent inference problem is to test whether the latent positions of two given vertices are the same or not in a random dot product graph. This subsequent network inference task is inspired by the hypothesis testing of the membership profiles in degree-corrected mixed membership stochastic block models proposed in \cite{fan2019simple}. The testing procedure could be useful in, e.g., diversifying the portfolios in the stock market investment and maximizing the expected returns \cite{fan2019simple}. Formally, given $\bA\sim\mathrm{RDPG}(\rho_n^{1/2}\bX)$ and fixed vertex indices $i,j\in [n]$, $i\neq j$, we consider testing the null hypothesis $H_0:\bx_i = \bx_j$ against the alternative hypothesis $H_A:\bx_i\neq \bx_j$. Motivated by the asymptotic normality in Theorems \ref{thm:ASE_Berry_Esseen_RDPG} and \ref{thm:Berry_Esseen_OSE_multivariate}, we consider the following two test statistics associated with the adjacency spectral embedding $\widetilde\bX$ and its one-step refinement $\widehat{\bX}$, respectively:
$T_{ij}^{(\mathrm{ASE})} = n(\widetilde{\bx}_i - \widetilde{\bx}_j)\transpose\widetilde{\bSigma}_{ij}^{-1}(\widetilde{\bx}_i - \widetilde{\bx}_j)$ and $T_{ij}^{(\mathrm{OSE})} = n(\widehat{\bx}_i - \widehat{\bx}_j)\transpose\widetilde{\bG}_{ij}^{-1}(\widehat{\bx}_i - \widehat{\bx}_j)$,
where $\widetilde{\bSigma}_{ij} = \widetilde{\bSigma}_n(\widetilde{\bx}_i) + \widetilde{\bSigma}_n(\widetilde{\bx}_j)$, $\widetilde{\bG}_{ij} = \widetilde{\bG}_n(\widetilde{\bx}_i)^{-1} + \widetilde{\bG}_n(\widetilde{\bx}_j)^{-1}$, 
\begin{align*}
\widetilde{\bSigma}_n(\bx) = \widetilde{\bDelta}_n^{-1}\left\{\frac{1}{n}\sum_{j = 1}^n\bx_i\transpose\widetilde{\bx}_j(1 - \bx_i\transpose\widetilde{\bx}_j)\widetilde{\bx}_j\widetilde{\bx}_j\transpose\right\}\widetilde{\bDelta}_n,\quad
\widetilde{\bG}_n(\bx) = \frac{1}{n}\sum_{j = 1}^n\frac{\widetilde\bx_j\widetilde\bx_j\transpose}{\bx_i\transpose\widetilde{\bx}_j(1 - \bx_i\transpose\widetilde{\bx}_j)},
\end{align*}
and $\widetilde{\bDelta}_n = (1/n)\widetilde{\bX}\transpose\widetilde{\bX}$. In what follows, we establish the asymptotic distributions of the test statistics $T_{ij}^{(\mathrm{ASE})}$ and $T_{ij}^{(\mathrm{OSE})}$ under the null and alternative hypotheses. 

\begin{theorem}\label{thm:ASE_OSE_testing}
Let $\bA\sim\mathrm{RDPG}(\rho_n^{1/2}\bX)$ and assume the conditions of Theorem \ref{thm:Berry_Esseen_OSE_multivariate} hold. Further assume that $d$ is fixed and $\lambda_d(\bDelta_n)$ is bounded away from $0$. 
\begin{enumerate}[(i)]
  \item Under the null hypothesis $H_0:\bx_i = \bx_j$, we have
  $T_{ij}^{(\mathrm{ASE})}\overset{\calL}{\to}\chi^2_d$
  and
  $T_{ij}^{(\mathrm{OSE})}\overset{\calL}{\to}\chi^2_d$.
  \item Under the alternative hypothesis $H_A:\bx_i\neq \bx_j$, if $(n\rho_n)^{1/2}(\bx_i - \bx_j)\to \bmu$ for some non-zero vector $\bmu\in\mathbb{R}^d$, $\bSigma_n(\bx_i)\to \bSigma_i$, and $\bG_n(\bx_i) \to\bG_i$ for some fixed positive definite $\bSigma_i$ and $\bG_i$ as $n\to\infty$, then
  $T_{ij}^{(\mathrm{ASE})}\overset{\calL}{\to}\chi^2_d(\bmu\transpose(\bSigma_i + \bSigma_j)^{-1}\bmu)$
  and
  $T_{ij}^{(\mathrm{OSE})}\overset{\calL}{\to}\chi^2_d(\bmu\transpose(\bG_i^{-1} + \bG_j^{-1})^{-1}\bmu)$,
  where, for any $a > 0$, $\chi_d^2(a)$ is the noncentral chi-squared distribution with noncentral parameter $a$ and degree of freedom $d$.
\end{enumerate}
\end{theorem}

An important consequence of Theorem \ref{thm:ASE_OSE_testing} is the power comparison between the two test statistics. It turns out that the test based on $T_{ij}^{(\mathrm{OSE})}$ is more powerful than the test based on $T_{ij}^{(\mathrm{ASE})}$ under the conditions of Theorem \ref{thm:ASE_OSE_testing} (ii). Given a significance level $\alpha\in (0, 1)$, we can construct the following test functions:
\[
\phi_{ij}^{(\mathrm{ASE})} = \mathbbm{1}\left\{T_{ij}^{(\mathrm{ASE})} > q_{\chi^2_d}(1 - \alpha)\right\}\quad\mbox{and}\quad
\phi_{ij}^{(\mathrm{OSE})} = \mathbbm{1}\left\{T_{ij}^{(\mathrm{OSE})} > q_{\chi^2_d}(1 - \alpha)\right\},
\]
where $ q_{\chi^2_d}(1 - \alpha)$ is the $(1 - \alpha)$ quantile of the chi-squared distribution with degree of freedom $d$. 
Then under the conditions of Theorem \ref{thm:ASE_OSE_testing} (i), we see that the two tests are asymptotically valid level-$\alpha$ tests, i.e., $\expect_{H_0}\phi_{ij}^{(\mathrm{ASE})} \to \alpha$ and $\expect_{H_0}\phi_{ij}^{(\mathrm{OSE})} \to \alpha$ as $n\to\infty$. To compare the power of the two tests under the alternative $H_A:\bx_i\neq \bx_j$ under the conditions of Theorem \ref{thm:ASE_OSE_testing} (ii), we first observe that the non-central chi-squared distribution is stochastic increasing in its non-central parameter \cite{5429099}. By Theorem 2 in \cite{xie2019efficient}, the non-central parameters for $T_{ij}^{(\mathrm{ASE})}$ and $T_{ij}^{(\mathrm{OSE})}$ satisfy the inequality $\bmu\transpose(\bSigma_i + \bSigma_j)^{-1}\bmu\leq \bmu\transpose(\bG_i^{-1} + \bG_j^{-1})^{-1}\bmu$. Therefore, under the alternative hypothesis $H_A:\bx_i\neq \bx_j$ and the conditions of Theorem \ref{thm:ASE_OSE_testing} (ii), we conclude that 
$\lim_{n\to\infty}\expect_{H_A}\phi_{ij}^{(\mathrm{ASE})}\leq\lim_{n\to\infty}\expect_{H_A}\phi_{ij}^{(\mathrm{OSE})}$.
Namely, the test based on $T_{ij}^{(\mathrm{OSE})}$ is asymptotically more powerful than the test based on $T_{ij}^{(\mathrm{ASE})}$.

\section{Simulation study} 
\label{sec:simulation_study}


In this section, we present a simulated example of random dot product graphs. 
Consider a stochastic block model on $n$ vertices with a cluster assignment rule $\tau:[n]\to\{1,2\}$ and a block probability matrix
\[
\bB = \rho_n\begin{bmatrix*}
p ^ 2 & p q\\
p q   & q ^ 2 
\end{bmatrix*},
\]
where $\rho_n\in (0, 1)$ is a sparsity factor and $p, q\in (0, 1)$. The adjacency matrix $\bA = [A_{ij}]_{n\times n}$ is generated as follows: For all $i\leq j$, $i,j\in [n]$, let $A_{ij}\sim\mathrm{Bernoulli}([\bB]_{\tau(i)\tau(j)})$ independently for $i\leq j$ and let $A_{ij} = A_{ji}$ for all $i > j$. 
We take $\tau(i) = 1$ if $i = 1,\ldots,n/2$, and $\tau(i) = 2$ if $i = n/2 + 1,\ldots,n$ for simplicity. The number of vertices $n$ is set to $5000$ and we take $n\rho_n = 5(\log n)^{3/2}$ such that the conditions of Theorem \ref{thm:ASE_Berry_Esseen} and Theorem \ref{thm:Berry_Esseen_OSE_multivariate} are both satisfied. The values of $p$ and $q$ are selected to be $p = 0.95$ and $q = 0.3$.

We generate $3000$ independent copies of the adjacency matrix $\bA$ from the aforementioned stochastic block model. For each realization of $\bA$, we compute the adjacency spectral embedding $\widetilde\bx$ of $\bA$ into $\mathbb{R}$, the unscaled top eigenvector $\bu_\bA$ of $\bA$, and the one-step refinement $\widehat\bx$ of $\widetilde\bx$. The population scaled eigenvector and the unscaled eigenvector are denoted by $\rho_n^{1/2}\bx$ and $\bu_\bP$, respectively. For this specific model, it is straightforward to obtain $\bx = \begin{bmatrix*}
p & \ldots & p & q & \ldots & q
\end{bmatrix*}$ and $\bu_\bP =  (np^2/2 + nq^2/2)^{-1/2}\begin{bmatrix*}
p & \ldots & p & q & \ldots & q
\end{bmatrix*}$.
The only non-zero eigenvalue of $\bP$ is $\lambda = n\rho_n(p ^ 2 / 2 + q ^ 2 / 2)$. 
For each $i\in [n]$, we denote $\widetilde{x}_i$, $[\bu_\bA]_i$, $\widehat{x}_i$, $x_i$, and $[\bu_\bP]_i$ the $i$th coordinates of $\widetilde{\bx}$, $\bu_\bA$, $\widehat{\bx}$, $\bx$, and $\bu_\bP$, respectively. 
Then by Theorem \ref{thm:ASE_Berry_Esseen} and Theorem \ref{thm:Berry_Esseen_OSE_multivariate}, for each $i\in [n]$, the random variables $\sqrt{n}(\widetilde{x}_i - \rho_n^{1/2}x_i)$, $\sqrt{n}(\widehat{x}_i - \rho_n^{1/2}x_i)$, and $n\rho_n^{1/2}([\bu_\bA]_i - [\bu_\bP]_i)$ converge to mean-zero Gaussians in distribution with the variances depending on $p$,$q$, and the community membership $\tau(i)$.

We take $i = 1$ as an illustrative vertex and visualize the numerical performance of $\widetilde\bx$ and $\bu_\bA$ in Figures \ref{fig:Simulation_ASE} and \ref{fig:Simulation_unscaled_eigenvectors}. The left panels of Figures \ref{fig:Simulation_ASE}, \ref{fig:Simulation_unscaled_eigenvectors} are the histograms of $\sqrt{n}(\widetilde{x}_1 - \rho_n^{1/2}x_1)$ and $n\rho_n^{1/2}([\bu_\bA]_1 - [\bu_\bP]_1)$ with the corresponding asymptotic normal densities highlighted in the red curves. We see that the shapes of the two histograms are closely aligned with the limit densities, verifying the conclusion of Theorem \ref{thm:ASE_Berry_Esseen} empirically. 
\begin{figure}[htbp]
\includegraphics[width = 15cm]{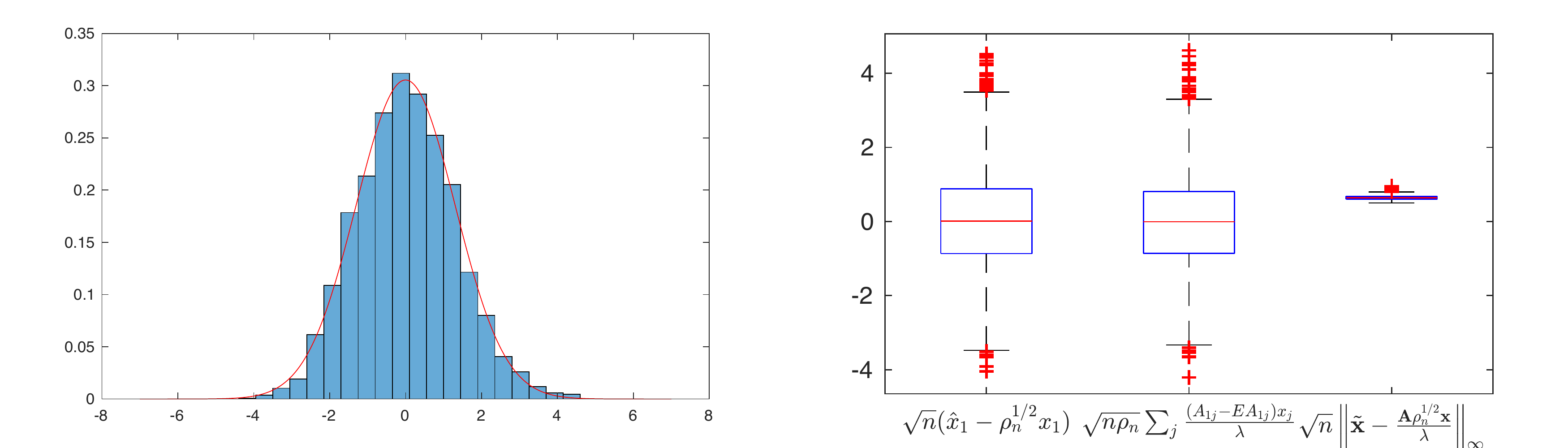}
\caption{Numerical results for Section \ref{sec:simulation_study}. Left panel: The histogram of $\sqrt{n}(\widetilde{x}_{1} - \rho_n^{1/2}x_{1})$ over the $3000$ Monte Carlo replicates with the asymptotic normal density highlighted in the red curve. Right panel: The boxplots of $\sqrt{n}(\widehat{x}_{1} - \rho_n^{1/2}x_{1})$, its linear approximation $\sqrt{n\rho_n}\sum_j(A_{1j} - \expect A_{1j})x_{j}/\lambda_2$, and the infinity norm of the higher-order remainder $\sqrt{n}\|\widetilde{\bx} - \bA\rho_n^{1/2}\bx/\lambda\|_\infty$ across the $3000$ Monte Carlo replicates. }
\label{fig:Simulation_ASE}
\end{figure}
The right panels of Figures \ref{fig:Simulation_ASE} and \ref{fig:Simulation_unscaled_eigenvectors} present the boxplots of $\sqrt{n}(\widehat{x}_{1} - \rho_n^{1/2}x_{1})$ and $n\rho_n^{1/2}([\bu_\bA]_1 - [\bu_\bP]_1)$, their linear approximations, and the infinity norms of the corresponding higher-order remainders. From the right panel of Figure \ref{fig:Simulation_ASE}, we can see that the dominating term for  $\sqrt{n}(\widehat{x}_{1} - \rho_n^{1/2}x_{1})$ is $\sqrt{n\rho_n}\sum_j(A_{1j} - \expect A_{1j})x_j/\lambda$. However, the infinity norm of the higher-order remainder $\sqrt{n}\|\widetilde\bx - \bA\rho_n^{1/2}\bx/\lambda\|_\infty$ is not necessarily negligible. This agrees with the observation in Section \ref{sub:a_motivating_example}. A similar observation regarding the unscaled eigenvector can be found in the right panel of Figure \ref{fig:Simulation_unscaled_eigenvectors} as well. 
\begin{figure}[htbp]
\includegraphics[width = 15cm]{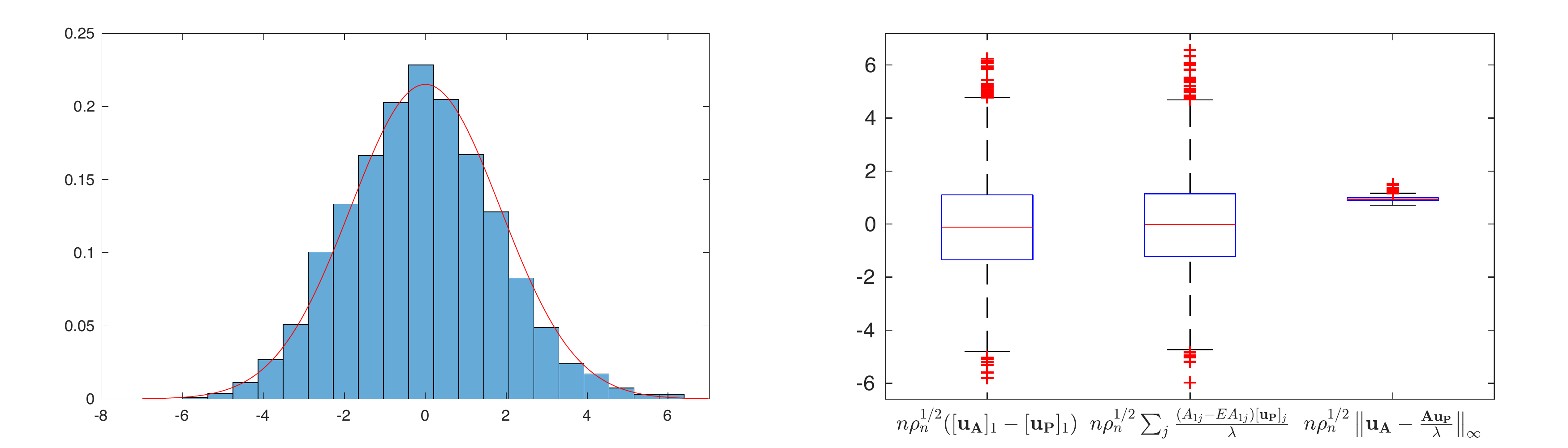}
\caption{Numerical results for Section \ref{sec:simulation_study}. Left panel: The histogram of $n\rho_n^{1/2}([\bu_\bA]_1 - [\bu_\bP]_1)$ over the $3000$ Monte Carlo replicates with the asymptotic normal density highlighted in the red curve. Right panel: The boxplots of $n\rho_n^{1/2}([\bu_\bA]_1 - [\bu_\bP]_1)$, its linear approximation $n\rho_n^{1/2}\sum_j(A_{1j} - \expect A_{1j})[\bu_\bP]_j/\lambda$, and the infinity norm of the higher-order remainder $n\rho_n^{1/2}\|\bu_\bA - \bA\bu_\bP/\lambda\|_\infty$ across the $3000$ Monte Carlo replicates. }
\label{fig:Simulation_unscaled_eigenvectors}
\end{figure}

We also compare the performance between the adjacency spectral embedding $\widetilde\bx$ and its one-step refinement $\widehat{\bx}$ in Figure \ref{fig:Simulation_ASE_OSE_comparison} below. Taking $i = 1$ as an illustrative vertex, we visualize the histogram of $\sqrt{n}(\widehat{x}_1 - \rho_n^{1/2}x_1)$ in the left panel of Figure \ref{fig:Simulation_ASE_OSE_comparison}, overlaid with the corresponding asymptotic normal density in the red curve. The limit normal density is almost perfectly aligned with the histogram, verifying Theorem \ref{thm:Berry_Esseen_OSE_multivariate} numerically. The right panel compares the boxplot of $\|\widetilde{\bx}w - \rho_n^{1/2}x\|_2^2$ and that of $\|\widehat{\bx}w - \rho_n^{1/2}x\|_2^2$ across the $3000$ Monte Carlo replicates, where $w$ is the sign of $\bu_\bA\transpose\bu_\bP$. It is clear that the errors of the one-step refinement $\widehat{\bx}$ are smaller than those of the adjacency spectral embedding, which also agrees with the observation in \cite{xie2019efficient} but under a much sparser regime that $n\rho_n\propto (\log n)^{3/2}$. 
\begin{figure}[htbp]
\includegraphics[width = 15cm]{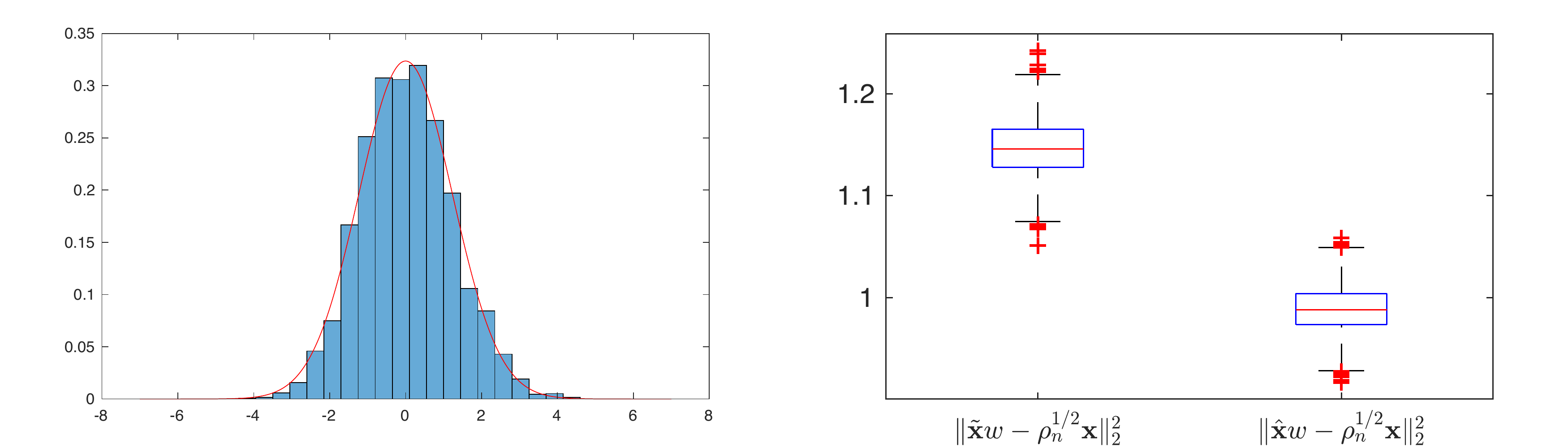}
\caption{Numerical results for Section \ref{sec:simulation_study}. Left panel: The histogram of $\sqrt{n}(\widehat{x}_{1} - \rho_n^{1/2}x_{1})$ over the $3000$ Monte Carlo replicates with the asymptotic normal density highlighted in the red curve. Right panel: The boxplots of $\|\widetilde{\bx}w - \rho_n^{1/2}\bx\|_2^2$ for the adjacency spectral embedding and $\|\widehat{\bx}w - \rho_n^{1/2}\bx\|_2^2$ for its one-step refinement across the $3000$ Monte Carlo replicates, where $w$ is the sign of $\bu_\bA\transpose\bu_\bP$. }
\label{fig:Simulation_ASE_OSE_comparison}
\end{figure}


\section{Discussion} 
\label{sec:discussion}

In this paper, we establish the Berry-Esseen theorems for the entrywise limits of the eigenvectors for a broad class of random matrix models with low expected rank, referred to as the signal-plus-noise matrix model. Our generic entrywise eigenvector limit theorem leads to new and sharp results for several concrete statistical applications: the symmetric noisy matrix completion model, the eigenvectors and their one-step refinement of random dot product graphs, the estimation of pure nodes in mixed membership stochastic block models, and the hypothesis testing of the equality of latent positions in random graphs.

Several potential future research directions are worth exploring. In terms of the general signal-plus-noise matrix model framework, we restrict ourselves within the class of symmetric random matrices whose upper diagonal entries are independent random variables. Extensions to singular vectors of rectangular random matrices may be interesting for rectangular noisy matrix completion problems, bipartite network analysis, and high-dimensional principal component analysis \cite{agterberg2021entrywise}. 

For the symmetric matrix completion problem, we require that the variance of the mean-zero normal errors scales at the rate $\rho_n^4$. It is possible to relax this requirement and assume that $\mathrm{var}(\eps_{ij})$ scales at the rate $\rho_n^2$ by modifying the proof technique in \cite{10.1214/19-AOS1854}. This relaxation may require additional work because Assumption \ref{assumption:distribution} no longer holds when $\mathrm{var}(\eps_{ij}) \propto \rho_n^2$. 

For random dot product graphs, we have focused on the eigenvector analysis of the graph adjacency matrix. It has also been observed that the eigenvectors of the normalized Laplacian matrix have decent performance when the graph becomes sparse \cite{sarkar2015,tang2018}. The entrywise limit theorems for the eigenvectors of the normalized Laplacian have been established in \cite{tang2018} under the sparsity assumption that $n\rho_n = \omega((\log n)^4)$. An interesting future research direction is to explore the entrywise limit theorems for the eigenvectors of the normalized Laplacian when $n\rho_n = \Omega(\log n)$. 
In addition, there has also been a growing interest in developing limit theorems for spectral analysis of multiple graphs \cite{arroyo2021inference,levin2017central}. We believe that the results and the approach developed in the present work may shed some light on the entrywise estimation of the eigenvectors for multiple random graph models. 



\vspace*{4ex}
\begin{appendix}

\begin{center}
  \begin{Large}
    \textbf{Supplement: Proofs and Additional Implementation Details}
  \end{Large}
\end{center}

\section{Technical preparations} 
\label{sec:technical_preparations}


The supplementary material begins with several auxiliary results that have already been established in the literature. We first present a theorem due to \cite{10.1214/19-AOS1854}. It is quite useful to obtain sharp concentration bounds for $\|\bU_{\bA_\pm}\|_{2\to\infty}$, $\|\bU_{\bA_\pm}^{(m)}\|_{2\to\infty}$, and $\|\bU_{\bA_\pm}^{(m)}\mathrm{sgn}(\bH_\pm^{(m)}) - \bU_{\bP_\pm}\|_{2\to\infty}$. Although it can also lead to a sharp error bound for $\|\bU_{\bA_\pm} - \bA\bU_{\bP_\pm}\bS_{\bP_\pm}^{-1}\bW_\pm^*\|_{2\to\infty}$ when $\log n\asymp n\rho_n$, it does not provide an enough control of the entrywise error $\|\be_m\transpose(\bU_{\bA_\pm} - \bA\bU_{\bP_\pm}\bS_{\bP_\pm}^{-1}\bW_\pm^*)\|_2$ for each individual $m\in [n]$. 
\begin{theorem}[Theorem 2.1 in \cite{10.1214/19-AOS1854}]
\label{thm:AFWZ2020AoS_corr2.1}
Let $\bM$ be an $n\times n$ symmetric random matrices with $\expect\bM = \bP$. 
Suppose $r,s$ are integers with $1\leq r\leq n$, $0\leq s\leq n - r$. Let $\bU,\bU_\bP\in\mathbb{O}(n, r)$ be the eigenvector matrices of $\bM$ and $\bP$, respectively, such that $\bM\bU = \bU\bS$ and $\bP\bU_\bP = \bU_\bP\bS_\bP$, where $\bS = \mathrm{diag}\{\lambda_{s + 1}(\bM),\ldots,\lambda_{s + r}(\bM)\}$, and $\bS_\bP = \mathrm{diag}\{\lambda_{s + 1}(\bP),\ldots,\lambda_{s + r}(\bP)\}$. 
We adopt the convention that $\lambda_{0}(\bP) = +\infty$ and $\lambda_n(\bP) = -\infty$. 
Define the eigengap
\[
\Delta = \min\{\lambda_s(\bP) - \lambda_{s + 1}(\bP),\lambda_{s + r}(\bP) - \lambda_{s + r + 1}(\bP)\}\wedge\min_{k\in[r]}|\lambda_{s + k}(\bP)|
\]
and $\kappa = \max_{k\in [r]}\lambda_{s + k}(\bP)/\Delta$. Suppose there exists some $\bar{\gamma} > 0$ and a function $\omega(\cdot):\mathbb{R}_+\to\mathbb{R}_+$, such that the following conditions hold:
\begin{itemize}
  \item[(A1)] (Incoherence) $\|\bP\|_{2\to\infty}\leq \bar{\gamma}\Delta$.
  \item[(A2)] (Row and columnwise independence) For any $i\in [n]$, the entries in the $i$th row and column of $\bM$ are independent of others. 
  \item[(A3)] (Spectral norm concentration) $32\kappa\max\{\bar{\gamma},\omega(\bar{\gamma})\}\leq 1$ and $\prob(\|\bM - \bP\|_2 > \bar{\gamma}\Delta) \leq \delta_0$ for some $\delta_0\in (0, 1)$. 
  \item[(A4)] (Row concentration) Suppose $\omega(x)$ is non-decreasing in $\mathbb{R}_+$ with $\omega(0) = 0$, $\omega(x)/x$ is non-increasing in $\mathbb{R}_+$. There exists some $\delta_1 \in (0, 1)$, such that for all $i\in [n]$ and any $n\times r$ matrix $\bV$, 
  \[
  \prob\left\{\|\be_i\transpose{}(\bM - \bP)\bV\|_2\leq \Delta\|\bV\|_{2\to\infty}\omega\left(\frac{\|\bV\|_{\mathrm{F}}}{\sqrt{n}\|\bV\|_{2\to\infty}}\right)\right\}\geq 1 - \frac{\delta_1}{n}.
  \]
\end{itemize}
Then with probability at least $1 - \delta_0 - 2\delta_1$, we have
\begin{align*}
\|\bU\|_{2\to\infty}&\lesssim \{\kappa + \omega(1)\}\|\bU_\bP\|_{2\to\infty} + \frac{\bar{\gamma}\|\bP\|_{2\to\infty}}{\Delta},\\
\|\bU\mathrm{sgn}(\bU\transpose{}\bU_\bP) - \bU_\bP\|_{2\to\infty}&\lesssim \kappa\{\kappa + \omega(1)\}\{\bar{\gamma} + \omega(\bar{\gamma})\}\|\bU_\bP\|_{2\to\infty} + \frac{\bar{\gamma}\|\bP\|_{2\to\infty}}{\Delta} + \omega(1)\|\bU_\bP\|_{2\to\infty}.
\end{align*}
\end{theorem}
We next state a vector version of the Bernstein's inequality due to \cite{MINSKER2017111}. The advantage of this concentration inequality is that it is dimension free. 
\begin{lemma}[Corollary 4.1 in \cite{MINSKER2017111}]
\label{lemma:Bernstein_inequality_vector}
Let $\by_1,\ldots,\by_n\in\mathbb{R}^d$ be a sequence of independent random vectors such that $\expect(\by_i) = \zero_d$ and $\|\by_i\|_2\leq U$ almost surely for all $1\leq i\leq n$ and some $U > 0$. Denote $\tau^2 := \sum_{i = 1}^n\expect\|\by_i\|_2^2$. Then for all $t \geq (U + \sqrt{U^2 + 36\tau^2})/6$, 
\[
\prob\left(\mathrel{\Big\|} \sum_{i = 1}^n\by_i \mathrel{\Big\|_2}  > t\right)\leq 28\exp\left(-\frac{3t^2}{6\tau^2 + 2Ut}\right)
\]
\end{lemma}
Lemma \ref{lemma:Matrix_concentration} below is a generic matrix Chernoff bound due to \cite{tropp2012user}. In the context of random dot product graphs, it allows us to construct the required function $\omega(\cdot)$ in condition A4 of Theorem \ref{thm:AFWZ2020AoS_corr2.1}. See Section \ref{sub:proof_of_ASE_Berry_Esseen_RDPG} for more details. 
\begin{lemma}[Corollary 3.7 in \cite{tropp2012user}]
\label{lemma:Matrix_concentration}
Let $\bZ_1,\ldots,\bZ_n $ be a sequence of symmetric independent random matrices in $\mathbb{R}^{d\times d}$. Assume that there is a function $g:(0,+\infty)\to[0, +\infty]$ and a sequence of deterministic symmetric matrices $\bM_1,\ldots,\bM_n\in\mathbb{R}^{d\times d}$ such that
$\expect e^{\theta\bZ_i}\preceq e^{g(\theta)\bM_i}$ for all $\theta > 0$.
Define the scale parameter 
$\rho = \lambda_{\max}\left(\sum_{i = 1}^n\bM_n\right)$.
Then for all $t\in\mathbb{R}$, 
\[
\prob\left\{\lambda_{\max}\left(\sum_{i = 1}^n\bZ_i\right) \geq t\right\}\leq d\inf_{\theta > 0}\exp\{-\theta t + g(\theta)\rho\}. 
\]
\end{lemma}

We conclude this section with the following Berry-Esseen bound for multivariate nonlinear statistics due to \cite{shao2021berry}, which is useful for us to prove Theorem \ref{thm:ASE_Berry_Esseen} and Theorem \ref{thm:Berry_Esseen_OSE_multivariate}. 
\begin{theorem}[Corollary 2.2 in \cite{shao2021berry}]
\label{thm:Berry_Esseen_Multivariate}
Let $\bxi_1,\ldots,\bxi_n$ be independent random vectors in $\mathbb{R}^d$ such that $\expect\bxi_j = \zero_d$, $j\in [n]$ and $\sum_{j = 1}^n\expect(\bxi_j\bxi_j\transpose) = \eye_d$. Let $\bT = \sum_{j = 1}^n\bxi_j + \bD(\bxi_1,\ldots,\bxi_n)$ be a nonlinear statistic, where $\bD(\cdot)$ is a measurable function from $\mathbb{R}^{n\times d}\to\mathbb{R}^d$. Let $\calO$ be an event and $\Delta$ be a random variable such that $\Delta \geq \|\bD(\bxi_1,\ldots,\bxi_n)\|_2\mathbbm{1}(\calO)$, and suppose $\{\Delta^{(j)}\}_{j = 1}^n$ are random variables such that $\Delta^{(j)}$ is independent of $\bxi_j$, $j \in [n]$. Denote $\gamma: = \sum_{j = 1}^n\expect(\|\xi_j\|_2^3)$ and $\calA$ the collection of all convex measurable sets in $\mathbb{R}^d$. Then
\begin{align*}
\sup_{A\in\calA}|\prob(\bT \in A) - \prob(\bz\in A)|&\lesssim d^{1/2}\gamma + \expect\mathrel{\Big\{}\mathrel{\Big\|}\sum_{j = 1}^n\bxi_j\mathrel{\Big\|_2}\Delta\mathrel{\Big\}} + \sum_{j = 1}^n\expect\{\|\bxi_j\|_2|\Delta - \Delta^{(j)}|\}
 + \prob(\calO^c).
\end{align*}
\end{theorem}

\section{Auxiliary results}\label{app:auxiliary_results}

In this section, we introduce some technical tools that serve as the building blocks for our theory. We first present several useful results that are applied throughout the proofs. 

\begin{result}[Concentration of eigenvalues]
\label{result:eigenvalue_concentration}
Under Assumption \ref{assumption:spectral_norm_concentration}, by Weyl's inequality, with probability at least $1 - c_0n^{-\zeta}$, the $p$ largest eigenvalues of $\bA$ are bounded above by $(1/2)n\rho_n\lambda_d(\bDelta_{n})$, the $q$ smallest eigenvalues of $\bA$ are bounded below by $-(1/2)n\rho_n\lambda_d(\bDelta_{n})$, and the absolute values of the remaining eigenvalues of $\bA$ are bounded by a constant multiple of $(n\rho_n)^{1/2}$. In other words, for sufficiently large $n$, with probability at least $1 - c_0n^{-\zeta}$,
\begin{align}\label{eqn:eigenvalue_concentration}
\begin{aligned}
&\lambda_1(\bA)\geq\ldots\geq\lambda_p(\bA)\geq \frac{1}{2}\lambda_p(\bP) = \frac{1}{2}n\rho_n\lambda_p(\bDelta_{n+})\geq\frac{1}{2}n\rho_n\lambda_d(\bDelta_n),\\
&\max_{p + 1\leq i\leq n - q}|\lambda_i(\bA)|\leq K(n\rho_n)^{1/2},\\
&\lambda_n(\bA)\leq \ldots\leq \lambda_{n - q + 1}(\bA)\leq \frac{1}{2}\lambda_{n - q + 1}(\bP) = -\frac{1}{2}n\rho_n\lambda_q(\bDelta_{n-})\leq -\frac{1}{2}n\rho_n\lambda_d(\bDelta_n).
\end{aligned}
\end{align}
\end{result}

\begin{result}[Concentration of $\bS_\bA$]
\label{result:S_A_concentration}
Suppose Assumptions \ref{assumption:incoherence}-\ref{assumption:spectral_norm_concentration} hold. Then for sufficiently large $n$,
$\|\bS_\bA\|_2\leq 2n\rho_n\lambda_1(\bDelta_{n})$ and $\|\bS_\bA^{-1}\|_2\leq \{2n\rho_n\lambda_d(\bDelta_n)\}^{-1}$with probability at least $1 - c_0n^{-\zeta}$, where $c_0 > 0, \zeta\geq 1$ are absolute constants. This can be implied by the concentration of eigenvalues in Result \ref{result:eigenvalue_concentration} and Assumption \ref{assumption:spectral_norm_concentration}. 
\end{result}

\begin{result}[Eigenvector delocalization]
\label{result:UP_delocalization}
$\bU_\bP$ satisfies that
$\|\bU_\bP\|_{2\to\infty}\leq \|\bX\|_{2\to\infty}/\{n\lambda_d(\bDelta_n)\}^{1/2}$. Consequently, $\lambda_d(\bDelta_n)\leq \|\bX\|_{2\to\infty}^2/d$. To see why these results hold, we first observe that
  \begin{align*}
  \|\bU_{\bP}\|_{2\to\infty}
  & = \|\rho_n^{1/2}[\bX_+,\bX_-]\mathrm{diag}(|\bS_{\bP_+}|,|\bS_{\bP_-}|)^{-1/2}\|_{2\to\infty}\\
  &\leq \rho_n^{1/2}\|\bX\|_{2\to\infty}\max(\|\bS_{\bP_+}^{-1}\|_2, \|\bS_{\bP_-}^{-1}\|_2)^{1/2}
  = \frac{\|\bX\|_{2\to\infty}}{\sqrt{n}\lambda_d(\bDelta_n)^{1/2}}.
  \end{align*}
  Since $\|\bU_\bP\|_{2\to\infty}\geq \sqrt{d/n}$ by the fact that $\bU_\bP\in\mathbb{O}(n, d)$, we obtain $\lambda_d(\bDelta_n)\leq \|\bX\|_{2\to\infty}^2/d$. 
\end{result}

We next present a collection of auxiliary lemmas, the proofs of which are relegated to the Supplementary Material. 
Lemma \ref{lemma:Bernstein_concentration_EW} below essentially states the concentration property of $\be_i\transpose\bE\bV$ for any deterministic matrix $\bV\in\mathbb{R}^{n\times d}$ and can be proved using a matrix Bernstein's inequality \cite{tropp2012user}. 
\begin{lemma}\label{lemma:Bernstein_concentration_EW}
Let $(y_i)_{i = 1}^n$ be independent random variables, $|y_i|\leq 1$ with probability one, and $\max_{i\in [n]}\var(y_i)\leq \sigma^2\rho$ for some constant $\sigma^2 > 0$. Suppose $\bV\in\mathbb{R}^{n\times d}$ is a deterministic matrix. Let $\bv_i = \bV\transpose{}\be_i$, $i\in [n]$. Then there exist constants $C > 0$, such that for any $t \geq 1$, 
\begin{align*}
&\prob\left\{\left\|\sum_{i = 1}^n(y_i - \expect y_i)\bv_i\right\|_2 > 3 t^2\|\bV\|_{2\to\infty} + \sqrt{6}\sigma\rho^{1/2}t\|\bV\|_{\mathrm{F}}\right\}\leq 28e^{-t^2},\\ 
&\prob\left\{\left\|\sum_{i = 1}^n(|y_i - \expect y_i| - \expect|y_i - \expect y_i|)\bv_i\right\|_2 > 3 t^2\|\bV\|_{2\to\infty} + \sqrt{6}\sigma\rho^{1/2}t\|\bV\|_{\mathrm{F}}\right\}\leq 28e^{-t^2}.
\end{align*}
\end{lemma}

\begin{lemma}\label{lemma:Bernstein_concentration_EW_subGaussian}
Let $(y_i)_{i = 1}^n$ be independent random variables such that $\max_{i\in [n]}\|y_i - \expect y_i\|_{\psi_2}\leq \sigma\rho^{1/2}$ for some constant $\sigma > 0$. Suppose $\bV\in\mathbb{R}^{n\times d}$ is a deterministic matrix. Let $\bv_i = \bV\transpose{}\be_i$, $i\in [n]$. Then there exist a constant $C_0 > 0$, such that for any $t \geq 1$, 
\begin{align*}
&\prob\left\{\left\|\sum_{i = 1}^n(y_i - p_i)\bv_i\right\|_2 > C_0\sigma \rho^{1/2}t\|\bV\|_{\mathrm{F}}\right\}\leq 2(d + 1)e^{-t^2},\\ 
&\prob\left\{\left\|\sum_{i = 1}^n(|y_i - p_i| - \expect|y_i - p_i|)\bv_i\right\|_2 > 2C_0\sigma \rho^{1/2}t\|\bV\|_{\mathrm{F}}\right\}\leq 2(d + 1)e^{-t^2}.
\end{align*}
\end{lemma}

One of the difficulties in generalizing the perturbation bounds for a single eigenvector to an eigenvector matrix lies in the control of $\bW^*_\pm\bS_{\bA_\pm} - \bS_{\bP_\pm}\bW_\pm^*$ because the matrix multiplication is not commutative. The following Lemma \ref{lemma:WS_interchange_bound} allows us to tackle this type of technical barrier. 
\begin{lemma}\label{lemma:WS_interchange_bound}
Suppose Assumptions \ref{assumption:incoherence}-\ref{assumption:spectral_norm_concentration} hold. Then there exists a absolute constant $c_0 > 0$, such that for sufficiently large $n$, the following events hold with probability at least $1 - c_0n^{-\zeta} - c_0e^{-t}$ for all $t > 0$:
\begin{align*}
\|\bW^*_{\pm}\bS_{\bA_{\pm}} - \bS_{\bP_{\pm}}\bW^*_{\pm}\|_{2} &\lesssim_\sigma \max\left\{\frac{\kappa(\bDelta_n)}{\lambda_d(\bDelta_n)}, d^{1/2}, t^{1/2}\right\},\\
\|\bW^*_{\pm}|\bS_{\bA_{\pm}}|^{1/2} - |\bS_{\bP_{\pm}}|^{1/2}\bW^*_{\pm}\|_{\mathrm{F}} &\lesssim_\sigma 
\frac{d^{1/2}}{(n\rho_n)^{1/2}\lambda_d(\bDelta_n)^{1/2}}\max\left\{\frac{\kappa(\bDelta_n)}{\lambda_d(\bDelta_n)}, d^{1/2}, t^{1/2}\right\}\\
&\leq\frac{\|\bX\|_{2\to\infty}}{(n\rho_n)^{1/2}\lambda_d(\bDelta_n)}\max\left\{\frac{\kappa(\bDelta_n)}{\lambda_d(\bDelta_n)}, d^{1/2}, t^{1/2}\right\},\\
\|\bW^*_{\pm}|\bS_{\bA_{\pm}}|^{-1/2} - |\bS_{\bP_{\pm}}|^{-1/2}\bW^*_{\pm}\|_{\mathrm{F}} &\lesssim_\sigma \frac{\|\bX\|_{2\to\infty}}{(n\rho_n)^{3/2}\lambda_d(\bDelta_n)^2}\max\left\{\frac{\kappa(\bDelta_n)}{\lambda_d(\bDelta_n)}, d^{1/2}, t^{1/2}\right\}.
\end{align*}
\end{lemma}

With the help of Lemma \ref{lemma:higher_order_remainder} below, we are able to provide a sharp control of several remainder terms. The analyses of these remainders are necessary, as will be seen in Section \ref{sub:proof_sketch_for_theorem_thm:eigenvector_deviation} (see lines \eqref{eqn:ASE_eigenvector_decomposition_II}, \eqref{eqn:ASE_eigenvector_decomposition_III}, and \eqref{eqn:ASE_eigenvector_decomposition_IV}). 
\begin{lemma}\label{lemma:higher_order_remainder}
Suppose Assumptions \ref{assumption:incoherence}-\ref{assumption:spectral_norm_concentration} hold. Let $m\in [n]$ be a fixed row index. Then there exists an absolute constant $c_0 > 0$, such that for all $t\geq 1$, $t\lesssim n\rho_n$, the following events hold with probability at least $1 - c_0n^{-\zeta} - c_0de^{-t}$ for sufficiently large $n$:
\begin{align*}
&\|\bU_{\bP_\pm}\bS_{\bP_\pm}(\bU_{\bP_\pm}\transpose{}\bU_{\bA_\pm}\bS_{\bA_\pm}^{-1} - \bS_{\bP_\pm}^{-1}\bU_{\bP_\pm}\transpose\bU_{\bA_\pm})\|_{2\to\infty} 
\lesssim_\sigma \frac{\|\bU_\bP\|_{2\to\infty}}{n\rho_n\lambda_d(\bDelta_n)}
\max\left\{
t^{1/2},d^{1/2},\frac{1}{\lambda_d(\bDelta_n)}
\right\},\\
&\|\be_m\transpose\bE\bU_{\bP_\pm}(\bW^*_\pm\bS_{\bA_\pm}^{-1} - \bS_{\bP_\pm}^{-1}\bW_\pm^*)\|_{2} 
\lesssim_\sigma \frac{t^{1/2}\|\bU_{\bP}\|_{2\to\infty}}{(n\rho_n)^{3/2}\lambda_d(\bDelta_n)^{2}}\max\left\{
  \frac{\kappa(\bDelta_n)}{\lambda_d(\bDelta_n)},d^{1/2}, t^{1/2}\right\},\\
&\|\bU_{\bP_\mp}\bS_{\bP_\mp}\bU_{\bP_\mp}\transpose\bU_{\bA_\pm}\bS_{\bA_\pm}^{-1}\|_{2\to\infty}
\lesssim_\sigma 
\frac{d^{1/2}\|\bU_\bP\|_{2\to\infty}}{n\rho_n\lambda_d(\bDelta_n)}\max\left\{\frac{1}{\lambda_d(\bDelta_n)},t^{1/2},d^{1/2}\right\}
.   
\end{align*}
Also, for sufficiently large $n$, with probability at least $1 - c_0n^{-\zeta}$, 
\[
\|\bU_{\bP_\pm}(\bU_{\bP_\pm}\transpose{}\bU_{\bA_\pm} - \bW^*_\pm)\|_{2\to\infty} \lesssim \frac{\|\bU_\bP\|_{2\to\infty}}{n\rho_n\lambda_d(\bDelta_n)^{2}}.
\]
\end{lemma}

We conclude this section with the following lemma, which asserts that the two-to-infinity norm of $\bU_\bA$ can be upper bounded by the two-to-infinity norm of $\bU_\bP$ with large probability. It is a direct consequence of Theorem \ref{thm:AFWZ2020AoS_corr2.1}. 
\begin{lemma}\label{lemma:U_A_two_to_infinity_norm}
Suppose Assumptions \ref{assumption:incoherence}-\ref{assumption:spectral_norm_concentration} hold. Then there exists an absolute constant $c_0 > 0$, such that for sufficiently large $n$, with probability at least $1 - c_0n^{-\zeta\wedge\xi}$,
\[
\|\bU_{\bA_\pm}\|_{2\to\infty}\lesssim \{\kappa(\bDelta_n) + \varphi(1)\}\|\bU_\bP\|_{2\to\infty}.
\]
\end{lemma}

\section{Proofs of the Lemmas in Section \ref{app:auxiliary_results}} 
\label{sec:additional_technical_results}

In this section, we provide the proofs of Lemmas \ref{lemma:Bernstein_concentration_EW}, \ref{lemma:Bernstein_concentration_EW_subGaussian}, \ref{lemma:WS_interchange_bound}, and \ref{lemma:higher_order_remainder} in Section \ref{app:auxiliary_results} above. 

\subsection{Proof of Lemma \ref{lemma:Bernstein_concentration_EW}} 
\label{sub:proof_of_lemma_lemma:bernstein_concentration_ew}

The proof is a straightforward application of the vector version of the Bernstein's inequality (Lemma \ref{lemma:Bernstein_inequality_vector}). We set $\by_i = (y_i - \expect y_i)\bv_i$ and $\by_i = (|y_i - \expect y_i| - \expect|y_i - \expect y_i|)\bv_i$, respectively. 
Clearly,
\begin{align*}
\tau^2_1 & = \sum_{i = 1}^n\|\bv_i\|_2^2\var(y_i)\leq \sigma^2\rho\|\bV\|_{\mathrm{F}}^2,\\
\tau^2_2 & = \sum_{i = 1}^n\|\bv_i\|_2^2\var(|y_i - \expect y_i|)\leq \sum_{i = 1}^n\|\bv_i\|_2^2\expect\{(y_i - \expect y_i)^2\}\leq \sigma^2\rho\|\bV\|_{\mathrm{F}}^2
\end{align*}
and we take $U = \|\bV\|_{2\to\infty}$. Note that
\[
\frac{1}{6}(U + \sqrt{U^2 + 36\tau_k^2})\leq \frac{1}{3}U + \tau_k\leq \frac{1}{3}\|\bV\|_{2\to\infty} + \sigma\rho^{1/2}\|\bV\|_{\mathrm{F}}\leq 3t^2\|\bV\|_{2\to\infty} + \sqrt{6}\sigma\rho^{1/2}t\|\bV\|_{\mathrm{F}} 
\]
for any $t \geq 1$ and $k = 1,2$. Then Lemma \ref{lemma:Bernstein_inequality_vector} implies
\begin{align*}
&\prob\left\{\left\|\sum_{i = 1}^n(y_i - p_i)\bv_i\right\|_2 > 3t^2\|\bV\|_{2\to\infty} + \sqrt{6}\sigma\rho^{1/2}t\|\bV\|_{\mathrm{F}}\right\}\\
&\quad\leq 28\exp\left(
-3
\frac
{9t^4\|\bV\|_{2\to\infty}^2 + 6\sigma^2\rho t^2\|\bV\|_{\mathrm{F}}^2 + 6\sqrt{6}\sigma \rho^{1/2}t^3\|\bV\|_{\mathrm{F}}\|\bV\|_{2\to\infty}}
{6\sigma^2\rho\|\bV\|_{\mathrm{F}}^2 + 6t^2\|\bV\|_{2\to\infty}^2 + 2\sqrt{6}\sigma\rho^{1/2}t\|\bV\|_{\mathrm{F}}\|\bV\|_{2\to\infty} }
\right)\\
&\quad\leq 28\exp\left(-3\frac{9t^4\|\bV\|_{2\to\infty}^2 + 6\sigma^2\rho t^2\|\bV\|_{\mathrm{F}}^2 + 6\sqrt{6}\sigma\rho^{1/2}t^3\|\bV\|_{\mathrm{F}}\|\bV\|_{2\to\infty}}{9t^2\|\bV\|_{2\to\infty}^2 + 6\sigma^2\rho\|\bV\|_{\mathrm{F}}^2 + 6\sqrt{6}\sigma\rho^{1/2}t\|\bV\|_{\mathrm{F}}\|\bV\|_{2\to\infty} }\right)
 = 28e^{-3t^2},
\end{align*}
and similarly,
\begin{align*}
&\prob\left\{\left\|\sum_{i = 1}^n(|y_i - p_i| - \expect|y_i - p_i|)\bv_i\right\|_2 > 3t^2\|\bV\|_{2\to\infty} + \sqrt{6}\sigma\rho^{1/2}t\|\bV\|_{\mathrm{F}}\right\}\leq 28e^{-3t^2}. 
\end{align*}


\subsection{Proof of Lemma \ref{lemma:Bernstein_concentration_EW_subGaussian}}
\label{sub:proof_of_lemma_Bernstein_concentration_EW_subGaussian}

We apply a ``symmetric dilation'' trick \cite{10.1214/19-AOS1854,paulsen2002completely} and the matrix Chernoff bound (Lemma \ref{lemma:Matrix_concentration}). 
Define
\[
\bT(\bv_i) = \begin{bmatrix*}
\zero_{d\times d} & \bv_i\\
\bv_i\transpose{} & 0
\end{bmatrix*},\quad
\bZ_i = (y_i - p_i)\bT(\bv_i),\quad i = 1,2,\ldots,n,
\]
and let $\bS_n = \sum_{i = 1}^n\bZ_i$. Clearly, $\|\bS_n\|_2 = \max\{\lambda_{\max}(\bS_n), \lambda_{\max}(-\bS_n)\}$ and $-\bS_n = \sum_{i = 1}^n(-\bZ_i)$. 
Observe that the spectral decomposition of $\bT(\bv_i)$ is given by
\[
\bT(\bv_i)
 = \bQ_i\begin{bmatrix*}
 \|\bv_i\|_2 & \\ & -\|\bv_i\|_2
 \end{bmatrix*}\bQ_i\transpose + 0\times\bQ_{i\perp}\bQ_{i\perp}\transpose,
\]
where 
\[
\bQ_i = \frac{1}{\sqrt{2}}\begin{bmatrix*}
 \frac{\bv_i}{\|\bv_i\|_2} & \frac{\bv_i}{\|\bv_i\|_2}\\
 1 & -1
 \end{bmatrix*}
\]
and $\bQ_{i\perp}\in\mathbb{O}(d + 1, d - 1)$ is the orthogonal complement matrix of $\bQ_i$. 
Then for any $\theta > 0$, we use the above spectral decomposition and Lemma 5.5 in \cite{vershynin2010introduction} to compute the matrix moment generating function of $\bZ_i$:
\begin{align*}
\expect e^{\theta\bZ_i}
& = \expect \exp\{\theta(y_i - \expect y_i)\bT(\bv_i)\}\\
& = \expect\left\{\bQ_i\begin{bmatrix*}
 \exp\{\theta(y_i - \expect y_i)\|\bv_i\|_2\} & \\ & \exp\{-\theta(y_i - \expect y_i)\|\bv_i\|_2\}
 \end{bmatrix*}\bQ_i\transpose + \bQ_{i\perp}\bQ_{i\perp}\transpose\right\}\\
&\preceq \exp\{C_0\sigma^2\rho\theta^2\|\bv_i\|_2^2\}\bQ_i\bQ_i\transpose + \bQ_{i\perp}\bQ_{i\perp}\transpose\\
& = \exp(C_0\sigma^2\rho\theta^2\|\bv_i\|_2^2\bQ_i\bQ_i\transpose) = \exp\{g(\theta)\bM_i\},
\end{align*}
where $C_0 > 0$ is an absolute constant, 
\[
g(\theta) = C_0\sigma^2\rho\theta^2,\quad\mbox{and}\quad \bM_i = \|\bv_i\|_2^2\bQ_i\bQ_i\transpose.
\]
Similarly, $\expect e^{\theta(-\bZ_i)}\preceq \exp\{g(\theta)\bM_i\}$. The corresponding scale parameter can be bounded by
\[
0 < \lambda_{\max}\left(\sum_{i = 1}^m\bM_i\right)\leq \sum_{i = 1}^n\|\bM_i\|_2\leq \sum_{i = 1}^n\|\bv_i\|_2 = \|\bV\|_{\mathrm{F}}^2.
\]
Therefore, by Lemma \ref{lemma:Matrix_concentration}, for any $t > 0$, 
\begin{align*}
\prob\left(\|\bS_n\|_2 > t\right)
&\leq \prob\{\lambda_{\max}(\bS_n) > t\} + \prob\{\lambda_{\max}(-\bS_n) > t\}\\
&\leq 2(d + 1)\exp\left\{\inf_{\theta > 0}\left(-\theta t + C_0\theta^2\sigma^2\rho\|\bV\|_{\mathrm{F}}^2\right)\right\}\\
& = 2(d + 1)\exp\left(-\frac{t^2}{4C_0^2\sigma^2\rho\|\bV\|_{\mathrm{F}}^2}\right).
\end{align*}
Now replacing $t$ by $2C_0\sigma\rho^{1/2}\|\bV\|_{\mathrm{F}}$ and adjust $C_0$ properly leads to the first assertion. The second assertion follows from the first assertion and the fact that
\[
\||y_i - \expect y_i| - \expect|y_i - \expect y_i|\|_{\psi_2}\leq 2\|y_i - \expect y_i\|_{\psi_2}\leq 2\sigma\rho^{1/2}.
\]


\subsection{Proof of Lemma \ref{lemma:WS_interchange_bound}} 
\label{sub:proof_of_lemma_lemma:ws_interchange_bound}

To prove Lemma \ref{lemma:WS_interchange_bound}, we first establish the following concentration bound for $\|\bU_\bP\transpose(\bA - \bP)\bU_\bP\|_2$. 
\begin{lemma}\label{lemma:UtransposeEU_bound}
Suppose Assumptions \ref{assumption:incoherence}-\ref{assumption:spectral_norm_concentration} hold. Let $\bU_1\in\mathbb{O}(n, r)$, $\bU_2\in\mathbb{O}(n, s)$ be two matrices with $r,s\geq 1$, $r,s\leq n$. 
Then there exists a constant $C > 0$ depending on $\sigma$, such that for all $t > 0$, with probability at least $1 - (2 + e)e^{-t}$,
\begin{align*}
\|\bU_1\transpose\bE\bU_2\|_2 \leq C\max(r, s)^{1/2} + Ct^{1/2}.
\end{align*} 
\end{lemma}

\begin{proof}[\bf Proof of Lemma \ref{lemma:UtransposeEU_bound}]
By Assumption \ref{assumption:distribution}, we can write $\bE = \bE_1 + \bE_2$, where $[\bE_1]_{ij}$'s are independent bounded mean-zero random variables with $\var([\bE_1]_{ij})\leq\sigma^2\rho_n$ for some constant $\sigma > 0$ and $[\bE_2]_{ij}$'s are sub-Gaussian random variables whose sub-Gaussian norms are bounded by $\sigma\rho_n^{1/2}$. 
We apply a classical discretization trick to the spectral norm of $\bU_1\transpose{}\bE_k\bU_2$, $k = 1,2$. 
By definition, $\|\bU_1\transpose\bE_k\bU_2\|_2\leq\max_{\|\bv_1\|_2,\|\bv_2\|_2 \leq 1}|\bv_1\transpose{}\bU_1\transpose{}\bE_k\bU_2\bv_2|$. Now let $\calS_{\eps}^{r - 1}$ be an $\eps$-net of the $(r-1)$-dimensional unit sphere $\calS^{r - 1} := \{\bv:\|\bv\|_2 = 1\}$, and similarly define $\calS_\eps^{s - 1}$. Clearly, for any $\bv_1\in \calS^{r - 1}$, $\bv_2\in\calS^{s - 1}$, there exists some $\bw_1(\bv_1)\in\calS_{\eps}^{r - 1}$ and $\bw_2(\bv_2)\in\calS_{\eps}^{s - 1}$, such that $\|\bv_1 - \bw_1(\bv_1)\|_2 < \eps$, $\|\bv_2 - \bw_2(\bv_2)\|_2 < \eps$, and
\begin{align*}
\|\bU_1\transpose\bE_k\bU_2\|_2
& = \max_{\|\bv_1\|_2,\|\bv_2\|_2 \leq 1} |\bv_1\transpose{}\bU_2\transpose{}\bE_k\bU_2\bv_2|\\
& = \max_{\|\bv_1\|_2,\|\bv_2\|_2 \leq 1}|\{\bv_1 - \bw_1(\bv_1) + \bw_1(\bv_1)\}\transpose{}\bU_1\transpose{}\bE_k\bU_2\{\bv_2 - \bw_2(\bv_2) + \bw_2(\bv_2)\}|_2\\
&\leq (\eps^2 + 2\eps)\|\bU_1\transpose\bE_k\bU_2\|_2 + \max_{\bw_1\in\calS_{\eps}^{r - 1},\bw_2\in\calS_{\eps}^{s - 1}}|\bw_1\transpose
\bU_1\transpose\bE_k\bU_2\bw_2|.
\end{align*}
With $\eps = 1/3$, we have
\[
\|\bU_1\transpose\bE_k\bU_2\|_2\leq \frac{9}{2}\max_{\bw_1\in\calS_{1/3}^{r - 1},\bw_2\in\calS_{1/3}^{s - 1}}|\bw_1\transpose
\bU_1\transpose\bE_k\bU_2\bw_2|.
\]
Furthermore, we know that $\calS_{1/3}^{r - 1}$ and $\calS_{1/3}^{s - 1}$ can be selected such that their cardinalities can be upper bounded by
$|\calS_{1/3}^{r - 1}|\leq 18^{r}$ and $|\calS_{1/3}^{s - 1}|\leq 18^{s}$, respectively (see, for example, \cite{pollard1990empirical}).
Now for fixed $\bw_1\in\calS_{1/3}^{r - 1}$ and $\bw_2\in\calS_{1/3}^{s - 1}$, let $\bz_1 = \bU_1\bw_1 = [z_{11},\ldots,z_{1n}]\transpose$ and $\bz_2 = \bU_2\bw_2 = [z_{21},\ldots,z_{2n}]\transpose$. Clearly, $\|\bz_1\|_2,\|\bz_2\|_2\leq 1$, and
\begin{align*}
|\bw_1\transpose\bU_1\transpose\bE_k\bU_2\bw_2|
& = \left|\sum_{i = 1}^n\sum_{j = 1}^n[\bE_k]_{ij}z_{1i}z_{2j}\right|
\leq \left|\sum_{i < j}[\bE_k]_{ij}(z_{1i}z_{2j} + z_{2i}z_{1j}) + \sum_{i = 1}^n[\bE_k]_{ii}z_{1i}z_{2i}\right|
.
\end{align*}
Denote $c_{ij} = z_{1i}z_{2j} + z_{1j}z_{2i}$ if $i\neq j$ and $c_{ii} = z_{1i}z_{2i}$. Note that
\[
\sum_{i\leq j}c_{ij}^2 \leq \sum_{i < j}(2z_{1i}^2z_{2j}^2 + 2z_{2i}^2z_{1j}^2) + \sum_{i = 1}^nz_{1i}^2z_{2i}^2\leq 4\left(\sum_{i = 1}^nz_{1i}^2\right)\left(\sum_{i = 1}^nz_{2i}^2\right) + \sum_{i = 1}^nz_{1i}^2\leq 5.
\]
For $\bE_1$, by Hoeffding's inequality and a union bound over $\bw\in\calS_{1/3}^{d - 1}$, we can pick an absolute constant $C > 0$, such that
\begin{align*}
\prob\left\{
\|\bU_1\transpose\bE_1\bU_2\|_2 > C\max(r, s)^{1/2} + Ct^{1/2}
\right\}\leq 2e^{-t}.
\end{align*}
Appying Proposition 5.10 in \cite{vershynin2010introduction} to $\|\bU_1\transpose\bE_2\bU_2\|_2$ leads to a similar concentration inequality
\begin{align*}
\prob\left\{
\|\bU_1\transpose\bE_2\bU_2\|_2 > C\max(r, s)^{1/2} + Ct^{1/2}
\right\}\leq ee^{-t}.
\end{align*}
 with a potentially different multiplicative constant $C > 0$ depending on $\sigma$. 
The proof is completed by the inequality $\|\bU_1\transpose\bE\bU_2\|_2\leq \|\bU_1\transpose\bE_1\bU_2\|_2 + \|\bU_1\transpose\bE_2\bU_2\|_2$. 
\end{proof}

\begin{proof}[\bf Proof of Lemma \ref{lemma:WS_interchange_bound}]
The proof is based on a modification of Lemma 49 in \cite{JMLR:v18:17-448}. Following the decomposition there with the fact that $\bA\bU_{\bA_\pm} = \bU_{\bA_\pm}\bS_{\bA_\pm}$ and $\bP\bU_{\bP_\pm} = \bU_{\bP_\pm}\bS_{\bP_\pm}$, we have
\begin{align*}
\bW^*_{\pm}\bS_{\bA_\pm} - \bS_{\bP_\pm}\bW^*_{\pm}
& = (\bW^*_{\pm} - \bU_{\bP_{\pm}}\transpose\bU_{\bA_{\pm}})\bS_{\bA_{\pm}} + \bU_{\bP_{\pm}}\transpose(\bA - \bP)(\bU_{\bA_{\pm}} - \bU_{\bP_{\pm}}\bU_{\bP_{\pm}}\transpose \bU_{\bA_{\pm}})\\
&\quad + \bU_{\bP_{\pm}}\transpose(\bA - \bP)\bU_{\bP_{\pm}}\bU_{\bP_{\pm}}\transpose{}\bU_{\bA_{\pm}} + \bS_{\bP_{\pm}}(\bU_{\bP_{\pm}}\transpose\bU_{\bA_{\pm}} - \bW^*_{\pm}).
\end{align*}
By Assumption \ref{assumption:spectral_norm_concentration}, $\|\bA - \bP\|_2\leq K(n\rho_n)^{1/2}$ with probability at least $1 - c_0n^{-\zeta}$ for all sufficiently large $n$ for some absolute constants $K, c_0 > 0$, $\zeta\geq 1$. By Result \ref{result:S_A_concentration}, $\|\bS_\bA\|_2\leq 2n\rho_n\lambda_1(\bDelta_n)$ with probability at least $1 - c_0n^{-\zeta}$ for sufficiently large $n$, where $\zeta\geq 1$.  
By Lemma 6.7 in \cite{cape2017two},
$\|\bW^*_{\pm} - \bU_{\bP_{\pm}}\transpose\bU_{\bA_{\pm}}\|_2
\leq \|\sin\Theta(\bU_{\bA_{\pm}},\bU_{\bP_{\pm}})\|_2^2$.
Then by Davis-Kahan theorem in the form of \cite{cape2017two}, we have
\begin{align*}
\|\bW^*_{\pm} - \bU_{\bP_{\pm}}\transpose\bU_{\bA_{\pm}}\|_2
&\leq 
\|\sin\Theta(\bU_{\bA_{\pm}},\bU_{\bP_{\pm}})\|_2^2\leq \frac{4\|\bA - \bP\|_2^2}{\{n\rho_n\lambda_d(\bDelta_n)\}^2}
\leq \frac{4K^2}{n\rho_n\lambda_d(\bDelta_n)^2}
\end{align*}
when $\|\bA - \bP\|_2\leq K(n\rho_n)^{1/2}$, which occurs with probability at least $1 - c_0n^{-\zeta}$, $\zeta\geq 1$. 
Also, observe that by Lemma 6.7 in \cite{cape2017two} and Davis-Kahan theorem again, we have
\[
\|\bU_{\bA_{\pm}} - \bU_{\bP_{\pm}}\bU_{\bP_{\pm}}\transpose \bU_{\bA_{\pm}}\|_{2} = \|\sin\Theta(\bU_{\bA_{\pm}},\bU_{\bP_{\pm}})\|_2\leq \frac{2\|\bA - \bP\|_2}{n\rho_n\lambda_d(\bDelta_n)}
\leq \frac{2K}{(n\rho_n)^{1/2}\lambda_d(\bDelta_n)}
\] 
provided that $\|\bA - \bP\|_2\leq K(n\rho_n)^{1/2}$, which occurs with probability at least $1 - c_0n^{-\zeta}$, $\zeta\geq 1$. 
Hence, for all sufficiently large $n$, we apply Lemma \ref{lemma:UtransposeEU_bound}
to obtain
\begin{align*}
\|\bW^*_{\pm}\bS_{\bA_{\pm}} - \bS_{\bP_{\pm}}\bW^*_{\pm}\|_2
&\leq \|\bW^*_{\pm} - \bU_{\bP_{\pm}}\transpose{}\bU_{\bA_{\pm}}\|_{\mathrm{2}}(\|\bS_{\bA_{\pm}}\|_2 + \|\bS_{\bP_{\pm}}\|_2)
\\&\quad
 + \|\bE\|_2\|\bU_{\bA_{\pm}} - \bU_{\bP_{\pm}}\bU_{\bP_{\pm}}\transpose\bU_{\bA_{\pm}}\|_{\mathrm{2}}
 + \|\bU_{\bP_{\pm}}\transpose\bE\bU_{\bP_{\pm}}\|_{\mathrm{2}}\\
&\lesssim_\sigma \frac{n\rho_n\lambda_1(\bDelta_n)}{n\rho_n\lambda_d(\bDelta_n)^2} + \frac{(n\rho_n)^{1/2}}{(n\rho_n)^{1/2}\lambda_d(\bDelta_n)} + \max(d^{1/2}, t^{1/2})\\
&\lesssim \max\left\{\frac{\kappa(\bDelta_n)}{\lambda_d(\bDelta_n)}, 
d^{1/2}
, t^{1/2}\right\}
\end{align*}
with probability at least $1 - c_0n^{-\zeta} - c_0e^{-t}$, $\zeta\geq 1$
This completes the proof of the first assertion.

\vspace*{2ex}\noindent
We now turn to the second assertion. For any $k,l\in [p]$, write
\begin{align*}
[\bW^*_+|\bS_{\bA_+}|^{1/2} - |\bS_{\bP_+}|^{1/2}\bW^*_+]_{kl}
& = [\bW^*_+]_{kl}{\lambda}_l(\bA)^{1/2} - \lambda_{k}^{1/2}(\bP)[\bW^*_+]_{kl}\\
& = [\bW^*_+]_{kl}\left\{\frac{{\lambda}_l(\bA) - \lambda_{k}(\bP)}{{\lambda}_l(\bA)^{1/2} + \lambda_{k}(\bP)^{1/2}}\right\}.
\end{align*}
Similarly, for any $k,l\in [q]$, we have
\begin{align*}
[\bW^*_-|\bS_{\bA_-}|^{1/2} - |\bS_{\bP_-}|^{1/2}\bW^*_-]_{kl}
& = [\bW^*_-]_{kl}|{\lambda}_{n - q + l}(\bA)|^{1/2} - |\lambda_{n - q + k}(\bP)|^{1/2}[\bW^*_-]_{kl}\\
& = [\bW^*_-]_{kl}\left\{\frac{{\lambda}_{n - q + k}(\bP) - \lambda_{n - q + l}(\bA)}{|{\lambda}_{n - q + l}(\bA)|^{1/2} + |\lambda_{n - q + k}(\bP)|^{1/2}}\right\}.
\end{align*}
This immediately implies that
\begin{align*}
\|\bW^*_{\pm}|\bS_{\bA_{\pm}}|^{1/2} - |\bS_{\bP_{\pm}}|^{1/2}\bW^*_{\pm}\|_{\mathrm{F}}
&\leq \frac{\sqrt{d}\|\bW^*_{\pm}\bS_{\bA_{\pm}} - \bS_{\bP_{\pm}}\bW^*_{\pm}\|_2}{\{n\rho_n\lambda_d(\bDelta_n)\}^{1/2}}\\
&\leq \frac{\|\bX\|_{2\to\infty}\|\bW^*_{\pm}\bS_{\bA_{\pm}} - \bS_{\bP_{\pm}}\bW^*_{\pm}\|_2}{(n\rho_n)^{1/2}\lambda_d(\bDelta_n)},
\end{align*}
where we have used Result \ref{result:UP_delocalization} that $d^{1/2}\leq \|\bX\|_{2\to\infty}/\lambda_d(\bDelta_n)^{1/2}$. 
Therefore, by the first assertion, for all sufficiently large $n$, 
\[
\|\bW^*_{\pm}|\bS_{\bA_{\pm}}|^{1/2} - |\bS_{\bP_{\pm}}|^{1/2}\bW^*_{\pm}\|_{\mathrm{F}}\lesssim \frac{\|\bX\|_{2\to\infty}}{(n\rho_n)^{1/2}\lambda_d(\bDelta_n)}\max\left\{\frac{\kappa(\bDelta_n)}{\lambda_d(\bDelta_n)}, d^{1/2}, t^{1/2}\right\}
\]
with probability at least $1 - c_0n^{-\zeta} - c_0e^{-t}$ for all $t > 0$, where $\zeta\geq 1$ is given by Assumption \ref{assumption:spectral_norm_concentration}.

\vspace*{2ex}\noindent
The third assertion can be obtained in a similar fashion. By Result \ref{result:S_A_concentration}, for sufficiently large $n$, with probability at least $1 - c_0n^{-\zeta}$,
$\||\bS_{\bA_\pm}|^{-1}\|_2\leq \{2n\rho_n\lambda_d(\bDelta_n)\}^{-1}$. 
For any $k,l\in [p]$, we have
\begin{align*}
[\bW^*_+|\bS_{\bA_{+}}|^{-1/2} - |\bS_{\bP_+}|^{-1/2}\bW^*_+]_{kl}& = [\bW^*_+]_{kl}\{\lambda_l(\bA)^{-1/2} - \lambda_k(\bP)^{-1/2}\}\\
& = \frac{[\bW^*_+]_{kl}\{\lambda_k(\bP) - \lambda_l(\bA)\}}{\lambda_l(\bA)^{1/2}\lambda_k(\bP)^{1/2}\{\lambda_l(\bA)^{1/2} + \lambda_k(\bP)^{1/2}\}}.
\end{align*}
For any $k,l\in [q]$, we have, similarly,
\begin{align*}
&[\bW^*_-|\bS_{\bA_{-}}|^{-1/2} - |\bS_{\bP_-}|^{-1/2}\bW^*_-]_{kl}\\
&\quad = [\bW^*_-]_{kl}\{|\lambda_{n - q + l}(\bA)|^{-1/2} - |\lambda_{n - q + k}(\bP)|^{-1/2}\}\\
&\quad = \frac{[\bW^*_-]_{kl}\{\lambda_{n - q + l}(\bA) - \lambda_{n - q + k}(\bP)\}}{|\lambda_{n - q + l}(\bA)|^{1/2}|\lambda_{n - q + k}(\bP)|^{1/2}\{|\lambda_{n - q + l}(\bA)|^{1/2} + |\lambda_{n - q + k}(\bP)|^{1/2}\}}.
\end{align*}
Therefore, when $\|\bS_\bA^{-1}\|_2\leq 2\{(n\rho_n)\lambda_d(\bDelta_n)\}^{-1}$, by Result \ref{result:UP_delocalization} that $d\leq \|\bX\|_{2\to\infty}^2/\lambda_d(\bDelta_n)$,
\begin{align*}
\|\bW^*_{\pm}|\bS_{\bA_{\pm}}|^{-1/2} - |\bS_{\bP_{\pm}}|^{-1/2}\bW^*_{\pm}\|_{\mathrm{F}}^2
&
 \leq {\|\bS_\bA^{-1}\|_2\|\bS_\bP^{-1}\|_2^2}\|\bW^*_{\pm}\bS_{\bA_{\pm}} - \bS_{\bP_{\pm}}\bW^*_{\pm}\|_{\mathrm{F}}^2
\\&
 \leq {d}{\|\bS_\bA^{-1}\|_2\|\bS_\bP^{-1}\|_2^2}\|\bW^*_{\pm}\bS_{\bA_{\pm}} - \bS_{\bP_{\pm}}\bW^*_{\pm}\|_{2}^2
 \\
&\leq \frac{2\|\bX\|_{2\to\infty}^2\|\bW_{\pm}^*\bS_{\bA_{\pm}}- \bS_{\bP_{\pm}}\bW^*_{\pm}\|_{2}^2}{(n\rho_n)^3\lambda_d(\bDelta_n)^4}.
\end{align*}
The proof of the third assertion is then completed by applying the first assertion.  
\end{proof}


\subsection{Proof of Lemma \ref{lemma:higher_order_remainder}} 
\label{sub:proof_of_lemma_lemma:higher_order_remainder}

  We first analyze the concentration bound for $\|\bS_{\bP_{\pm}}\bU_{\bP_{\pm}}\transpose{}\bU_{\bA_{\pm}} - \bU_{\bP_{\pm}}\transpose{}\bU_{\bA_{\pm}}\bS_{\bA_{\pm}}\|_2$. Note that by definition of eigenvector matrices, $\bP\bU_{\bP_{\pm}} = \bU_{\bP_{\pm}}\bS_{\bP_{\pm}}$ and $\bA\bU_{\bA_{\pm}} = \bU_{\bA_{\pm}}\bS_{\bA_{\pm}}$. 
  Clearly, by Davis-Kahan theorem in the form of \cite{cape2017two},
  \begin{align*}
  \|\bS_{\bP_\pm}\bU_{\bP_\pm}\transpose{}\bU_{\bA_\pm} - \bU_{\bP_\pm}\transpose{}\bU_{\bA_\pm}\bS_{\bA_\pm}\|_2
  &= \|\bU_{\bP_\pm}\transpose{}\bA\bU_{\bA_\pm} - \bU_{\bP_\pm}\transpose\bP\bU_{\bA_\pm}\|_2\\
  &= \|\bU_{\bP_\pm}\transpose{}\bE\bU_{\bP_\pm}\bU_{\bP_\pm}\transpose{}\bU_{\bA_\pm} + \bU_{\bP_\pm}\transpose\bE(\eye - \bU_{\bP_\pm}\bU_{\bP_\pm}\transpose)\bU_{\bA_\pm}\|_2\\
  &\leq \|\bU_{\bP_\pm}\transpose{}\bE\bU_{\bP_\pm}\|_{\mathrm{2}} + \|\bE\|_2\|\bU_{\bA_\pm} - \bU_{\bP_\pm}\bU_{\bP_\pm}\transpose\bU_{\bA_\pm}\|_2\\
  &\leq \|\bU_{\bP_\pm}\transpose{}\bE\bU_{\bP_\pm}\|_{\mathrm{2}} + \|\bE\|_2\|\sin\Theta(\bU_{\bA_\pm},\bU_{\bP_\pm})\|_2\\
  &\leq  \|\bU_{\bP_\pm}\transpose{}\bE\bU_{\bP_\pm}\|_{\mathrm{2}} + \frac{2\|\bE\|_2^2}{n\rho_n\lambda_d(\bDelta_n)}.
  \end{align*}
  For a realization of $\bA$ with $\|\bE\|_2 \lesssim (n\rho_n)^{1/2}$ and $\|\bU_{\bP_\pm}\transpose\bE\bU_{\bP_\pm}\|_{\mathrm{2}}\lesssim d^{1/2} + t^{1/2}$, which occurs with probability at least $1 - c_0n^{-\zeta} - c_0e^{-t}$ by Assumption \ref{assumption:spectral_norm_concentration} and Lemma \ref{lemma:UtransposeEU_bound}, we have,
  \begin{align*}
  \|\bS_{\bP_\pm}\bU_{\bP_\pm}\transpose{}\bU_{\bA_\pm} - \bU_{\bP_\pm}\transpose{}\bU_{\bA_\pm}\bS_{\bA_\pm}\|_2
  \lesssim \max\left\{t^{1/2} , d^{1/2}, \frac{1}{\lambda_d(\bDelta_n)}\right\}.
  \end{align*}
  This event holds with probability at least $1 - c_0n^{-\zeta} - c_0e^{-t}$ for all $t > 0$. 
  Then the first assertion is immediate by Result \ref{result:S_A_concentration}, Result \ref{result:UP_delocalization}, and the observation that
  \begin{align*}
  &\|\bU_{\bP_\pm}\bS_{\bP_\pm}(\bU_{\bP_\pm}\transpose{}\bU_{\bA_\pm}\bS_{\bA_\pm}^{-1} - \bS_{\bP_\pm}^{-1}\bU_{\bP_\pm}\transpose{}\bU_{\bA_\pm})\|_{2\to\infty}\\
  &\quad\leq \|\bU_\bP\|_{2\to\infty}\|\bS_{\bA_\pm}^{-1}\|_2\|\bS_{\bP_\pm}\bU_{\bP_\pm}\transpose{}\bU_{\bA_\pm} - \bU_{\bP_\pm}\transpose{}\bU_{\bA_\pm}\bS_{\bA_\pm}\|_2.
  \end{align*}

  \noindent
  For $\|\bU_{\bP_\pm}(\bU_{\bP_\pm}\transpose{}\bU_{\bA_\pm} - \bW_\pm^*)\|_{2\to\infty}$, note that by Lemma 6.7 in \cite{cape2017two} and Davis-Kahan theorem, we have
  \begin{align*}
  \|\bU_{\bP_\pm}(\bU_{\bP_\pm}\transpose{}\bU_{\bA_\pm} - \bW^*_\pm)\|_{2\to\infty}
  &\leq \|\bU_{\bP}\|_{2\to\infty}\|\bU_{\bP_\pm}\transpose\bU_{\bA_\pm} - \bW^*_\pm\|_2\\
  &\leq \|\bU_{\bP}\|_{2\to\infty}\|\sin\Theta(\bU_{\bA_\pm},\bU_{\bP_\pm})\|_2^2\\
  &\leq \frac{4\|\bE\|_2^2}{(n\rho_n)^2\lambda_d(\bDelta_n)^2}\|\bU_{\bP}\|_{2\to\infty}.
  \end{align*}
  Then the fourth assertion follows from Assumption \ref{assumption:spectral_norm_concentration}. 

  \noindent
  We now focus on $\|\be_m\transpose{}\bE\bU_{\bP_\pm}(\bW^*_\pm\bS_{\bA_\pm}^{-1} - \bS_{\bP_\pm}^{-1}\bW^*_\pm)\|_2$. By Lemma \ref{lemma:Bernstein_concentration_EW}, Lemma \ref{lemma:Bernstein_concentration_EW_subGaussian},
  for all $t\geq1$ and $t\lesssim n\rho_n$, we have
  \begin{align*}
  \|\be_m\transpose{}\bE\bU_{\bP_\pm}\|_{2}
  & = \left\|\sum_{j = 1}^n(A_{mj} - \rho_n\bx_{m}\transpose{}\bx_{j})[\bU_{\bP_\pm}]_{j*}\right\|_2
  \lesssim  t\|\bU_{\bP_\pm}\|_{2\to\infty} + \sigma(\rho_nt)^{1/2}\|\bU_{\bP_\pm}\|_{\mathrm{F}}\\
  &\lesssim_\sigma (n\rho_nt)^{1/2}\|\bU_\bP\|_{2\to\infty}
  \end{align*}
  with probability at least $1 - c_0de^{-t}$. 
  Then by Lemma \ref{lemma:WS_interchange_bound}, 
  \begin{align*}
  \|\be_m\transpose\bE\bU_{\bP_\pm}(\bW^*_\pm\bS_{\bA_\pm}^{-1} - \bS_{\bP_\pm}^{-1}\bW^*_\pm)\|_{2}
  &\leq \|\be_m\transpose{}\bE\bU_{\bP_\pm}\|_{2}\|\bW^*_\pm\bS_{\bA_\pm}^{-1} - \bS_{\bP_\pm}^{-1}\bW^*_\pm\|_2\\
  &\leq \|\be_m\transpose{}\bE\bU_{\bP_\pm}\|_{2}\|\bS_{\bP_\pm}^{-1}\|_2\|\bS_{\bP_\pm}\bW^*_\pm - \bW^*_\pm\bS_{\bA_\pm}\|_2\|\bS_{\bA_\pm}^{-1}\|_2\\
  &\lesssim_\sigma \frac{t^{1/2}\|\bU_{\bP}\|_{2\to\infty}}{(n\rho_n)^{3/2}\lambda_d(\bDelta_n)^{2}}\max\left\{
  \frac{\kappa(\bDelta_n)}{\lambda_d(\bDelta_n)},d^{1/2},t^{1/2}\right\}
  \end{align*}
  with probability at least $1 - c_0n^{-\zeta} - c_0de^{-t}$ for sufficiently large $n$. 

  \noindent
  We finally deal with term $\|\bU_{\bP_\mp}\bS_{\bP_\mp}\bU_{\bP_\mp}\transpose\bU_{\bA_\pm}\bS_{\bA_\pm}^{-1}\|_{2\to\infty}$ by adopting the analysis in Appendix B.1 in \cite{rubin2017statistical}. By construction, 
  $\bU_{\bP_\mp}\transpose\bU_{\bA_\pm}
     = \bU_{\bP_\mp}\transpose\bA\bU_{\bA_\pm}\bS_{\bA_\pm}^{-1} = \bS_{\bP_\mp}^{-1}\bU_{\bP_\mp}\transpose\bP\bU_{\bA_{\pm}}$,
  implying that
  \begin{align*}
  \bU_{\bP_\mp}\transpose\bU_{\bA_\pm}\bS_{\bA_\pm} - \bS_{\bP_\mp}\bU_{\bP_\mp}\transpose\bU_{\bA_\pm}
  & = \bU_{\bP_\mp}\transpose\bE\bU_{\bA_\pm}\\
  & = \bU_{\bP_\mp}\transpose\bE(\eye_n - \bU_{\bP_\pm}\bU_{\bP_\pm}\transpose)\bU_{\bA_\pm} + \bU_{\bP_\mp}\transpose\bE \bU_{\bP_\pm}\bU_{\bP_\pm}\transpose\bU_{\bA_\pm}.
  \end{align*}
  Namely,
  \[
  \bS_{\bP_\mp}\bU_{\bP_\mp}\transpose\bU_{\bA_\pm} = \bU_{\bP_\mp}\transpose\bU_{\bA_\pm}\bS_{\bA_\pm} - \bU_{\bP_\mp}\transpose\bE\bU_{\bA_\pm}
  \]
  and, by Lemma \ref{lemma:UtransposeEU_bound} and Davis-Kahan theorem, with probability at least $1 - c_0n^{-\zeta} - c_0e^{-t}$,
  \begin{align*}
  \|\bU_{\bP_\mp}\transpose\bE\bU_{\bA_\pm}\|_2
  &\leq \|\bE\|_2\|\sin\Theta(\bU_{\bA_\pm}, \bU_{\bP_\pm})\|_2 + \|\bU_{\bP_\mp}\transpose\bE\bU_{\bP_\pm}\|_2\lesssim_\sigma \max\left\{\frac{1}{\lambda_d(\bDelta_n)}, t^{1/2}, d^{1/2}\right\}
  \end{align*}
  for sufficiently large $n$. 
  For any $k\in [q]$ and $l\in [p]$, we have
  \begin{align*}
  [\bU_{\bP_-}\transpose\bU_{\bA_+}\bS_{\bA_+} - \bS_{\bP_-}\bU_{\bP_-}\transpose\bU_{\bA_+}]_{kl}
  & = [\bU_{\bP_-}\transpose\bU_{\bA_+}]_{kl}\{\lambda_l(\bA) - \lambda_{n - q + k}(\bP)\},\\
  [\bU_{\bP_+}\transpose\bU_{\bA_-}\bS_{\bA_-} - \bS_{\bP_+}\bU_{\bP_+}\transpose\bU_{\bA_-}]_{lk}
  & = [\bU_{\bP_+}\transpose\bU_{\bA_-}]_{lk}\{\lambda_{n - q + k}(\bA) - \lambda_{l}(\bP)\}
  \end{align*}
  Note that by the concentration of eigenvalues Result \ref{result:eigenvalue_concentration}
  \begin{align*}
  \min_{k\in[q],l\in[p]}\{\lambda_l(\bA) - \lambda_{n - q + k}(\bP)\}\geq n\rho_n\lambda_d(\bDelta_n),\\
  \min_{k\in[q],l\in[p]}\{\lambda_l(\bP) - \lambda_{n - q + k}(\bA)\}\geq n\rho_n\lambda_d(\bDelta_n)
  \end{align*}
  with probability at least $1 - c_0n^{-\zeta}$ for large $n$. Therefore, by Lemma \ref{lemma:UtransposeEU_bound},
  \begin{align*}
  \|\bU_{\bP_{\mp}}\transpose\bU_{\bA_\pm}\|_{\mathrm{F}}
  &\leq \frac{1}{n\rho_n\lambda_d(\bDelta_n)}\|\bU_{\bP_\mp}\transpose\bU_{\bA_\pm}\bS_{\bA_\pm} - \bS_{\bP_\mp}\bU_{\bP_\mp}\transpose\bU_{\bA_\pm}\|_{\mathrm{F}}\\
  &\leq \frac{\sqrt{d}\|\bE\|_2\|\sin\Theta(\bU_{\bA_\pm}, \bU_{\bP_\pm})\|_2 + \sqrt{d}\|\bU_{\bP_\mp}\transpose\bE \bU_{\bP_\pm}\|_2}{n\rho_n\lambda_d(\bDelta_n)}\\
  &\lesssim_\sigma \frac{d^{1/2}}{n\rho_n\lambda_d(\bDelta_n)}\max\left\{\frac{1}{\lambda_d(\bDelta_n)},t^{1/2},d^{1/2}\right\}
  \end{align*}
  with probability at least $1 - c_0n^{-\zeta} - c_0de^{-t}$ for sufficiently large $n$.
  Therefore,
  \begin{align*}
  \|\bU_{\bP_\mp}\bS_{\bP_\mp}\bU_{\bP_\mp}\transpose\bU_{\bA_\pm}\bS_{\bA_\pm}^{-1}\|_{2\to\infty}
  & = \|\bU_{\bP_\mp}(\bU_{\bP_\mp}\transpose\bU_{\bA_\pm} - \bU_{\bP_\mp}\transpose\bE\bU_{\bA_\pm}\bS_{\bA_\pm}^{-1})\|_{2\to\infty}\\
  &\leq \|\bU_\bP\|_{2\to\infty}\|\bU_{\bP_\mp}\transpose\bU_{\bA_\pm}\|_2 + \|\bU_\bP\|_{2\to\infty}\|\bU_{\bP_\mp}\transpose\bE\bU_{\bA_\pm}\|_2\|\bS_{\bA_\pm}^{-1}\|_2\\
  &\lesssim_\sigma \frac{d^{1/2}\|\bU_\bP\|_{2\to\infty}}{n\rho_n\lambda_d(\bDelta_n)}\max\left\{\frac{1}{\lambda_d(\bDelta_n)},t^{1/2},d^{1/2}\right\}
  \end{align*}
  with probability at least $1 - c_0n^{-\zeta} - c_0de^{-t}$ for sufficiently large $n$.
  The proof is thus completed.




\section{Proofs for Section \ref{sec:entrywise_limit_theorem_for_the_eigenvectors}} 
\label{sec:proofs_main_results}

\subsection{Proof sketch for Theorem \ref{thm:eigenvector_deviation}} 
\label{sub:proof_sketch_for_theorem_thm:eigenvector_deviation}

In this section, we discuss the basic idea of the proof of Theorem \ref{thm:eigenvector_deviation}. We begin with a warm-up matrix decomposition motivated by \cite{10.1214/19-AOS1854} and \cite{cape2019signal}. Denote $\bE = \bA - \bP$. Recall that $\bU_{\bP_\pm} = \bP\bU_{\bP_\pm}\bS_{\bP_\pm}^{-1}$ the definition of $(\bU_{\bP_\pm},\bS_{\bP_\pm})$. This leads to the following observation
\begin{align}\label{eqn:U_A_decomposition_step_I}
\bU_{\bA_\pm} - \bA\bU_{\bP_\pm}\bS_{\bP_\pm}^{-1}\bW^*_{\pm}
& = - \bE\bU_{\bP_\pm}\bS_{\bP_\pm}^{-1}\bW^*_{\pm} + (\bU_{\bA_\pm} - \bU_{\bP_\pm}\bW^*_{\pm})
\end{align}
because $\bU_{\bP_\pm} - \bA\bU_{\bP_\pm}\bS_{\bP_\pm}^{-1} = (\bP  - \bA)\bU_{\bP_\pm}\bS_{\bP_\pm}^{-1}$. 
For the second term on the right-hand side of \eqref{eqn:U_A_decomposition_step_I}, we recall that $\bU_{\bA_+}$ ($\bU_{\bA_-}$, resp.) is the eigenvector matrix of $\bA$ corresponding to the eigenvalues $\lambda_1(\bA),\ldots,\lambda_p(\bA)$ ($\lambda_{n - q + 1}(\bA),\ldots,\lambda_n(\bA)$, resp.). Therefore, 
\begin{align}\label{eqn:U_A_decomposition_step_II}
\bU_{\bA_\pm} = \bA\bU_{\bA_\pm}\bS_{\bA_\pm}^{-1} = \bE\bU_{\bA_\pm}\bS_{\bA_\pm}^{-1} + \bP\bU_{\bA_\pm}\bS_{\bA_\pm}^{-1}.
\end{align}
We first focus on the second term $\bP\bU_{\bA_\pm}\bS_{\bA_\pm}^{-1}$ on the right-hand side of \eqref{eqn:U_A_decomposition_step_II} above. By the spectral decomposition $\bP = \bU_{\bP_+}\bS_{\bP_+}\bU_{\bP_+}\transpose + \bU_{\bP_-}\bS_{\bP_-}\bU_{\bP_-}\transpose$, we can write 
\[
\bP\bU_{\bA_\pm}\bS_{\bA_\pm}^{-1} = \bU_{\bP_\pm}\bS_{\bP_\pm}\bU_{\bP_\pm}\transpose\bU_{\bA_\pm}\bS_{\bA_\pm}^{-1} + \bU_{\bP_\mp}\bS_{\bP_\mp}\bU_{\bP_\mp}\transpose\bU_{\bA_\pm}\bS_{\bA_\pm}^{-1}.
\]
Recall that $\bW^*_\pm$ is the matrix sign of $\bU_{\bP_\pm}\transpose\bU_{\bA_\pm}$, suggesting that $\bW^*_\pm\approx \bU_{\bP_\pm}\transpose\bU_{\bA_\pm}$. 
It is then conceivable that $\bU_{\bP_\pm}\bW^*_\pm \approx \bU_{\bP_\pm}\bU_{\bP_\pm}\transpose\bU_{\bA_\pm} = \bU_{\bP_\pm}\bS_{\bP_\pm}\bS_{\bP_\pm}^{-1}\bU_{\bP_\pm}\transpose\bU_{\bA_\pm}$. This motivates us to write $\bP\bU_{\bA_\pm}\bS_{\bA_\pm}^{-1}$ as
\begin{align}\label{eqn:U_A_decomposition_step_III}
\begin{aligned}
\bP\bU_{\bA_\pm}\bS_{\bA_\pm}^{-1}& = \bU_{\bP_\pm}\bS_{\bP_\pm}\bU_{\bP_\pm}\transpose\bU_{\bA_\pm}\bS_{\bA_\pm}^{-1} + \bU_{\bP_\mp}\bS_{\bP_\mp}\bU_{\bP_\mp}\transpose\bU_{\bA_\pm}\bS_{\bA_\pm}^{-1}\\
& = \bU_{\bP_\pm}\bS_{\bP_\pm}(\bU_{\bP_\pm}\transpose\bU_{\bA_\pm}\bS_{\bA_\pm}^{-1} - \bS_{\bP_\pm}^{-1}\bU_{\bP_\pm}\transpose\bU_{\bA_\pm}) + \bU_{\bP_\pm}\bU_{\bP_\pm}\transpose\bU_{\bA_\pm}\\
&\quad + \bU_{\bP_\mp}\bS_{\bP_\mp}\bU_{\bP_\mp}\transpose\bU_{\bA_\pm}\bS_{\bA_\pm}^{-1}.
\end{aligned}
\end{align}
We next turn our attention to the first term $\bE\bU_{\bA_\pm}\bS_{\bA_\pm}^{-1}$ on the right-hand side of \eqref{eqn:U_A_decomposition_step_II}. Intuitively, this term should be closed to $\bE\bU_{\bP_\pm}\bS_{\bP_\pm}^{-1}\bW^*_\pm$, which leads to the following decomposition
\begin{align}
\label{eqn:U_A_decomposition_step_IV}
\bE\bU_{\bA_\pm}\bS_{\bA_\pm}^{-1}
& = \bE\bU_{\bP_\pm}\bS_{\bP_\pm}^{-1}\bW^*_\pm + (\bE\bU_{\bA_\pm}\bS_{\bA_\pm}^{-1} - \bE\bU_{\bP_\pm} \bS_{\bP_\pm}^{-1}\bW^*_\pm).
\end{align}
Because $\bA$ concentrates around $\bP$ in spectral norm and $\bU_{\bA_\pm},\bU_{\bP_\pm}$ are their eigenvector matrices, the matrix perturbation theory suggests that $\bU_{\bA_\pm}\approx \bU_{\bP_\pm}\bW^*_\pm$. Hence, we can write the second term in \eqref{eqn:U_A_decomposition_step_IV} above as
\begin{align}
\label{eqn:U_A_decomposition_step_V}
\begin{aligned}
\bE\bU_{\bA_\pm}\bS_{\bA_\pm}^{-1} - \bE\bU_{\bP_\pm} \bS_{\bP_\pm}^{-1}\bW^*_\pm
& = \bE(\bU_{\bA_\pm} - \bU_{\bP_\pm}\bW^*_\pm)\bS_{\bA_\pm}^{-1}\\
&\quad + \bE\bU_{\bP_\pm}(\bW^*_\pm\bS_{\bA_\pm}^{-1} - \bS_{\bP_\pm}^{-1}\bW^*_\pm).
\end{aligned}
\end{align}
We now combine equations \eqref{eqn:U_A_decomposition_step_I}, \eqref{eqn:U_A_decomposition_step_II}, \eqref{eqn:U_A_decomposition_step_III}, \eqref{eqn:U_A_decomposition_step_IV}, and \eqref{eqn:U_A_decomposition_step_V} to obtain the following decomposition of $\bU_\bA - \bA\bU_\bP\bS_\bP^{-1}\bW^*$:
\begin{align}
\label{eqn:ASE_eigenvector_decomposition_I}
\bU_{\bA_\pm} - \bA\bU_{\bP_\pm}\bS_{\bP_\pm}^{-1}\bW^*_\pm& = \bE(\bU_{\bA_\pm} - \bU_{\bP_\pm}\bW^*_\pm)\bS_{\bA_\pm}^{-1}\\
\label{eqn:ASE_eigenvector_decomposition_II}
&\quad + \bU_{\bP_\pm}\bS_{\bP_\pm}(\bU_{\bP_\pm}\transpose\bU_{\bA_\pm}\bS_{\bA_\pm}^{-1} - \bS_{\bP_\pm}^{-1}\bU_{\bP_\pm}\transpose\bU_{\bA_\pm})\\
\label{eqn:ASE_eigenvector_decomposition_III}
&\quad + \bU_{\bP_\pm}(\bU_{\bP_\pm}\transpose\bU_{\bA_\pm} - \bW^*_\pm)\\
\label{eqn:ASE_eigenvector_decomposition_IV}
&\quad + \bE\bU_{\bP_\pm}(\bW^*_\pm\bS_{\bA_\pm}^{-1} - \bS_{\bP_\pm}^{-1}\bW^*_\pm)\\
\label{eqn:ASE_eigenvector_decomposition_V}
&\quad + \bU_{\bP_\mp}\bS_{\bP_\mp}\bU_{\bP_\mp}\transpose\bU_{\bA_\pm}\bS_{\bA_\pm}^{-1}.
\end{align}
Among the five terms above, lines \eqref{eqn:ASE_eigenvector_decomposition_II}, \eqref{eqn:ASE_eigenvector_decomposition_III}, \eqref{eqn:ASE_eigenvector_decomposition_IV}, and \eqref{eqn:ASE_eigenvector_decomposition_V} are relatively easy to control using classical matrix perturbation tools and the concentration of $\|\bE\|_2$ due to \cite{lei2015}. The formal concentration bounds of these remainders are given in Lemma \ref{lemma:higher_order_remainder}. 
The challenging part is a delicate analysis of the row-wise behavior of $\bE(\bU_{\bA_\pm} - \bU_{\bP_\pm}\bW^*_\pm)$, which we sketch below. We borrow the decoupling strategy and a ``leave-one-out'' analysis that appeared in \cite{10.1214/19-AOS1854,Bean14563,javanmard2015biasing,lei2019unified,doi:10.1137/17M1122025}. Consider the following collection of auxiliary matrices $\bA^{(1)},\ldots,\bA^{(n)}$. For each row index $m\in [n]$, the matrix $\bA^{(m)} = [A_{ij}]_{n\times n}$ is a function of $\bA$ defined by
\begin{align}\label{eqn:auxiliary_matrix}
{\bA}^{(m)} = \left\{
\begin{array}{ll}
A_{ij},&\quad\text{if }i\neq m\text{ and }j\neq m,\\
\expect A_{ij},&\quad\text{if }i = m\text{ or }j = m.
\end{array}
\right.
\end{align}
Namely, the matrix $\bA^{(m)}$ is constructed by replacing the $m$th row and $m$th column of $\bA$ by their expected values. Now let $\bU_{\bA_\pm}^{(m)}$ be the leading eigenvector matrix of $\bA^{(m)}$ ($\bU_{\bA_+}\in\mathbb{O}(n, p)$ and $\bU_{\bA_-}\in\mathbb{O}(n, q)$) such that $\bU_{\bA_\pm}^{(m)}\bS_{\bA_\pm}^{(m)} = \bA^{(m)}\bU_{\bA_\pm}^{(m)}$, where 
\[
\bS_{\bA_+}^{(m)} = \mathrm{diag}\{\lambda_1(\bA^{(m)}), \ldots,\lambda_d(\bA^{(m)})\}\quad\mbox{and}\quad
\bS_{\bA_-}^{(m)} = \mathrm{diag}\{\lambda_{n - q + 1}(\bA^{(m)}), \ldots,\lambda_n(\bA^{(m)})\}.
\]
Denote $\bH_\pm = \bU_{\bA_\pm}\transpose{}\bU_{\bP_\pm}$ and $\bH^{(m)}_\pm = (\bU_{\bA_\pm}^{(m)})\transpose\bU_{\bP_\pm}$. The smartness of introducing $\bA^{(m)}$ lies in the striking fact that $\be_m\transpose{}\bE$ and $\bA^{(m)}$ are independent. With this in mind, we can focus on the $m$th row of $\bE(\bU_{\bA_\pm} - \bU_{\bP_\pm}\bW^*_\pm)$ by inserting $\bU_\bA^{(m)}$ as follows:
\begin{align}
\label{eqn:crucial_inequality}
\begin{aligned}
\|\be_m\transpose \bE(\bU_{\bA_\pm} - \bU_{\bP_\pm}\bW^*_\pm)\|_2
&\leq \|\be_m\transpose \bE\bU_{\bA_\pm}\{\mathrm{sgn}(\bH_\pm) - \bH_\pm\}\}\|_2\\
&\quad + \|\be_m\transpose{}\bE(\bU_{\bA_\pm}\bH_\pm - \bU_{\bA_\pm}^{(m)}\bH_\pm^{(m)})\|_2
\\
&\quad + \|\be_m\transpose{}\bE(\bU_{\bA_\pm}^{(m)}\bH_\pm^{(m)} - \bU_{\bP_\pm})\|_2.
\end{aligned}
\end{align}
Here, we have used the fact that $\bW^*_\pm = \mathrm{sgn}(\bH_\pm)\transpose\in\mathbb{O}(d)$. Since $\be_m\transpose{}\bE$ and $\bU_{\bA_\pm}^{(m)}\bH_\pm^{(m)} - \bU_{\bP_\pm}$ are independent, we can apply Bernstein's or Hoeffding's inequality to the third term above. The success of this decoupling strategy critically depends on the following sharp concentration bounds on $\|\bU_{\bA_\pm}\bH_\pm - \bU_{\bA_\pm}^{(m)}\bH_\pm^{(m)}\|_2$, $\|\bU_{\bA_\pm}^{(m)}\bH_\pm - \bU_{\bP_\pm}\|_{2\to\infty}$, and $\|\bU_{\bA_\pm}^{(m)}\bH_\pm^{(m)} - \bU_{\bP_\pm}\|_{\mathrm{F}}$.

\begin{lemma}\label{lemma:auxiliary_matrix}
Suppose Assumptions \ref{assumption:incoherence}-\ref{assumption:spectral_norm_concentration} hold. Denote $\bDelta_n = (1/n)\bX\transpose{}\bX$. Let $m\in [n]$ be any fixed row index and $\bA^{(m)}$, $\bU_{\bA_\pm}^{(m)}$, and $\bH_\pm^{(m)}$ be defined as above. Then there exists an absolute constant $c_0 > 0$, such that for sufficiently large $n$, with probability at least $1 - c_0n^{-\zeta\wedge \xi}$,
\begin{align*}
\|{\bU}_{\bA_\pm}^{(m)}{\bH}_\pm^{(m)} - \bU_{\bP_\pm}\|_2 &\lesssim \frac{1}{(n\rho_n)^{1/2}\lambda_d(\bDelta_n)},\\
\|{\bU}_{\bA_\pm}^{(m)}\|_{2\to\infty} &\lesssim \{\kappa(\bDelta_n) + \varphi(1)\}\|\bU_{\bP_\pm}\|_{2\to\infty},\\
\|{\bU}_{\bA_\pm}^{(m)}\mathrm{sgn}({\bH}_\pm^{(m)}) - \bU_{\bP_\pm}\|_{2\to\infty} &\lesssim\left[\kappa(\bDelta_n)\left\{\kappa(\bDelta_n) +  \varphi(1)\right\}\{\gamma + \varphi(\gamma)\} + \kappa(\bDelta_n) + \varphi(1)\right]\|\bU_\bP\|_{2\to\infty},
\end{align*}
and for all $t \geq 1$ and $t\lesssim n\rho_n$, with probability at least $1 - c_0n^{-\zeta\wedge\xi} - c_0de^{-t}$, 
\[
\|{\bU}_{\bA_\pm}^{(m)}{\bH}_\pm^{(m)} - \bU_{\bA_\pm}\bH_\pm\|_2 \lesssim_\sigma \frac{\{\kappa(\bDelta_n) + \varphi(1)\}t^{1/2}}{(n\rho_n)^{1/2}\lambda_d(\bDelta_n) }\|\bU_\bP\|_{2\to\infty}.
\]
\end{lemma}


\subsection{Proof of Lemma \ref{lemma:auxiliary_matrix}} 
\label{sub:proof_of_lemma_lemma:auxiliary_matrix}

The proof of Lemma \ref{lemma:auxiliary_matrix} is slightly involved and is more difficult than Lemma 1 and Lemma 3 in \cite{10.1214/19-AOS1854}. The underlying reason is that we replace the $m$th row and $m$th column in $\bA$ by their expected values in $\bA^{(m)}$, but the construction of $\bA^{(m)}$ in \cite{10.1214/19-AOS1854} is to zero out the $m$th row and column of $\bA$. 
As $\expect\bA^{(m)}$ is the same as $\expect\bA = \bP$, we can borrow the entrywise eigenvector analysis there to $\bU_\bA^{(m)}$. By the construction of ${\bA}^{(m)}$, 
\[
[\bA - {\bA}^{(m)}]_{ij} = \left\{\begin{array}{ll}
0,&\quad\text{if }i\neq m\text{ and }j\neq m,\\
A_{ij} - \expect A_{ij},&\quad\text{if }i = m\text{ or }j = m.
\end{array}\right.
\]
It follows that
\begin{align*}
\|\bA - {\bA}^{(m)}\|_2\leq \|\bA - {\bA}^{(m)}\|_{\mathrm{F}}\leq \left\{2\sum_{j = 1}^n(A_{mj} - \expect A_{mj})^2\right\}^{1/2}\leq \sqrt{2}\|\bE\|_{2\to\infty}\lesssim (n\rho_n)^{1/2}
\end{align*}
with probability at least $1 - c_0n^{-\zeta}$ by Assumption \ref{assumption:spectral_norm_concentration}. 
Denote $\eye_+ = \eye_p$ and $\eye_- = \eye_q$ for convenience. 
Now viewing ${\bA}^{(m)}$ as a perturbed version of $\bP$, we apply Davis-Kahan theorem to obtain
\begin{align*}
\|{\bU}_{\bA_\pm}^{(m)}{\bH}_\pm^{(m)} - \bU_{\bP_\pm}\|_2
& = \|\{{\bU}_{\bA_\pm}^{(m)}({\bU}_{\bA_\pm}^{(m)})\transpose - \eye_\pm\}\bU_{\bP_\pm}\|_2
  = \|\{{\bU}_{\bA_\pm}^{(m)}({\bU}_{\bA_\pm}^{(m)})\transpose - \eye_\pm\}\bU_{\bP_\pm}\bU_{\bP_\pm}\transpose\|_2\\
& = \|\sin\Theta({\bU}_{\bA_\pm}^{(m)}, \bU_{\bP_\pm})\|_2\leq \frac{2\|{\bA}^{(m)} - \bP\|_2}{n\rho_n\lambda_d(\bDelta_n)}\leq \frac{2\|\bA - \bP\|_2 + 2\|{\bA}^{(m)} - \bA\|_2}{n\rho_n\lambda_d(\bDelta_n)}\\
&\lesssim \frac{1}{(n\rho_n)^{1/2}\lambda_d(\bDelta_n)}
\end{align*}
with probability at least $1 - c_0n^{-\zeta}$. This completes the proof of the first assertion.

\vspace*{2ex}\noindent
We now turn the focus to $\|{\bU}_{\bA_\pm}^{(m)}\|_{2\to\infty}$ as well as $\|\bU_{\bA_\pm}^{(m)}\mathrm{sgn}(\bH^{(m)}_\pm) - \bU_{\bP_\pm}\|_{2\to\infty}$. This is the place where we apply Theorem \ref{thm:AFWZ2020AoS_corr2.1} with $\bM = {\bA}^{(m)}$. We set 
\[
\bar{\gamma}_+ = \frac{\max\{3K,\|\bX\|_{2\to\infty}^2\}}{(n\rho_n)^{1/2}\lambda_{\min}(\bDelta_{n\pm})},
\]
where $K > 0$ is the constant selected such that $\|\bE\|_2\leq K(n\rho_n)^{1/2}$ with probability at least $1 - c_0n^{-\zeta}$ according to Assumption \ref{assumption:spectral_norm_concentration} and $c_0 > 0,\zeta\geq 1$ are absolute constants. 
Denote $\Delta_+ = \lambda_p(\bP) = n\rho_n\lambda_p(\bDelta_{n+})$, $\Delta_- = -\lambda_{n - q + 1}(\bP) = n\rho_n\lambda_q(\bDelta_{n-})$, $\kappa_+ = \lambda_1(\bP)/\Delta_+ = \kappa(\bDelta_{n+})$, and $\kappa_- = |\lambda_n(\bP)|/\Delta_- = \kappa(\bDelta_{n-})$. Note that $\kappa_\pm\leq \kappa(\bDelta_n)$. 
We take the function $\omega(\cdot)$ in Theorem \ref{thm:AFWZ2020AoS_corr2.1} to be the same the $\varphi(\cdot)$ given in Assumption \ref{assumption:rowwise_concentration}. 
For condition A1, we see that 
\begin{align*}
\|\bP\|_{2\to\infty} 
&\leq \rho_n\|\bX\|_{2\to\infty}\|\bX\|_2\leq \sqrt{n}\rho_n\|\bX\|_{2\to\infty}^2\\
&\leq\frac{\max\{3K, \|\bX\|_{2\to\infty}^2\}n\rho_n\lambda_{\min}(\bDelta_{n\pm})}{(n\rho_n)^{1/2}\lambda_{\min}(\bDelta_{n\pm})} = \bar{\gamma}_{\pm}\Delta_{\pm}.
\end{align*}
Condition A2 automatically holds because for each fixed $i$, the $i$th row and column of ${\bA}^{(m)}$ are either the $i$th row and column of $\bA$, or their expected values. By the construction of $\bA$, we see that the $i$th row and column of ${\bA}^{(m)}$ are independent of the rest of the random variables in ${\bA}^{(m)}$. For condition A3, since $\gamma_\pm\leq \gamma$ and $\varphi(\cdot)$ is non-decreasing, we see that 
\begin{align*}
32\kappa_{\pm}\max\{\bar{\gamma}_\pm,\omega(\bar{\gamma}_\pm)\} 
&\leq 32\kappa(\bDelta_n)\max\{\gamma,\varphi(\gamma)\} \leq 1
\end{align*}
by Assumption \ref{assumption:spectral_norm_concentration}. Again, by Assumption \ref{assumption:spectral_norm_concentration}, 
\[
\|{\bA}^{(m)} - \bP\|_2\leq \|\bA - {\bA}^{(m)}\|_2 + \|\bE\|_2\leq 3K(n\rho_n)^{1/2}\leq \bar{\gamma}_\pm\Delta_\pm
\]
with probability at least $1 - \delta_0$, where $\delta_0 = c_0n^{-\zeta}$ for constants $c_0 > 0$ and $\zeta\geq 1$. For condition A4, it automatically holds by Assumption \ref{assumption:rowwise_concentration} with probability $1 - \delta_1n^{-1}$ because $\bA^{(m)} - \bP = \bE^{(m)}$, where $\delta_1 = c_0n^{-\xi}$ and $\xi \geq 1$. 
Note that $\varphi(\cdot)$ is non-decreasing, implying that
$\bar{\gamma}_\pm + \varphi(\bar{\gamma}_\pm) \leq \gamma\ + \varphi(\gamma)$.
Also, note that
\[
\frac{\bar{\gamma}_\pm\|\bP\|_{2\to\infty}}{\Delta_\pm}
\lesssim 
\frac{\|\bX\|_{2\to\infty}^2n\rho_n\lambda_1(\bDelta_n)}{(n\rho_n)^{3/2}\lambda_d(\bDelta_n)^2}\|\bU_\bP\|_{2\to\infty}\leq \kappa(\bDelta_n)\|\bU_\bP\|_{2\to\infty}.
\]
Hence, we obtain from Theorem \ref{thm:AFWZ2020AoS_corr2.1} that, with probability at least $1 - c_0n^{-\zeta\wedge\xi}$,
\begin{align*}
&\|{\bU}_{\bA_\pm}^{(m)}\|_{2\to\infty}\lesssim \left\{\kappa(\bDelta_n) +  \varphi(1)\right\}\|\bU_\bP\|_{2\to\infty},\\
&\|{\bU}_{\bA_\pm}^{(m)}\mathrm{sgn}(\bH_\pm^{(m)}) - \bU_{\bP_\pm}\|_{2\to\infty}\lesssim \left[\kappa(\bDelta_n)\left\{\kappa(\bDelta_n) +  \varphi(1)\right\}\{\gamma + \varphi(\gamma)\} + \kappa(\bDelta_n) + \varphi(1)\right]\|\bU_\bP\|_{2\to\infty},
\end{align*}
which are the second and the third assertion.

\vspace*{2ex}\noindent
We then focus on the last assertion regarding $\|\bU_{\bA_\pm}^{(m)}\bH_\pm^{(m)} - \bU_{\bA_\pm}\bH_\pm\|_2$. 
By the concentration of eigenvalues in Result \ref{result:eigenvalue_concentration}, we know that $\lambda_p(\bA)\geq (1/2)n\rho_n\lambda_d(\bDelta_n)$ and $\lambda_{n - q + 1}(\bA)\leq -(1/2)n\rho_n\lambda_d(\bDelta_n)$ with probability at least $1 - c_0n^{-\zeta}$ for sufficiently large $n$. 
By Weyl's inequality, for sufficiently large $n$,
\begin{align*}
\lambda_p({\bA}^{(m)})&\geq \lambda_p(\bA) - \|\bA - {\bA}^{(m)}\|_2\geq \frac{1}{4}n\rho_n\lambda_d(\bDelta_n),\\
\lambda_{n - q + 1}({\bA}^{(m)})&\leq \lambda_{n - q + 1}(\bA) + \|\bA - {\bA}^{(m)}\|_2\leq -\frac{1}{4}n\rho_n\lambda_d(\bDelta_n)
\end{align*}
with probability at least $1 - c_0n^{-\zeta}$,
where we have used the assumption that $n\rho_n\lambda_d(\bDelta_n)^2\to\infty$ and the fact that $\|\bA - \bA^{(m)}\|_2\leq \sqrt{2}K(n\rho_n)^{1/2}\leq (1/4)n\rho_n\lambda_d(\bDelta_n)$ with probability at least $1 - c_0n^{-\zeta}$ for large $n$ by Assumption \ref{assumption:spectral_norm_concentration}. 
On the other hand, applying Weyl's inequality to $\lambda_{p + 1}(\bA)$ and $\lambda_{n - q}(\bA)$ yields
\begin{align*}
\lambda_{p + 1}(\bA)&\leq \lambda_{p + 1}(\bP) + \|\bE\|_2\leq K (n\rho_n)^{1/2}\leq \frac{1}{8}n\rho_n\lambda_d(\bDelta_n),\\
\lambda_{n - q}(\bA)&\geq \lambda_{n - q}(\bP) - \|\bE\|_2\geq -K (n\rho_n)^{1/2}\geq -\frac{1}{8}n\rho_n\lambda_d(\bDelta_n)
\end{align*}
with probability at least $1 - c_0n^{-\zeta}$ for sufficiently large $n$. 
We thus obtain that 
\begin{align*}
\lambda_p({\bA}^{(m)}) - \lambda_{p + 1}(\bA) \geq \frac{1}{8}(n\rho_n)\lambda_d(\bDelta_n)
\quad\mbox{and}\quad\lambda_{n - q}({\bA}) - \lambda_{n - q + 1}(\bA^{(m)}) \geq \frac{1}{8}(n\rho_n)\lambda_d(\bDelta_n)
\end{align*}
with probability at least $1 - c_0n^{-\zeta}$ for large $n$. Hence, by a version of the Davis-Kahan theorem (See Theorem VII.3.4 in \cite{bhatia2013matrix}), 
\begin{align*}
\|\bU_{\bA_+}^{(m)}\bH_+^{(m)} - \bU_{\bA_+}\bH_+\|_2 &= \|\sin\Theta({\bU}_{\bA_+}^{(m)}, \bU_{\bA_+})\|_2
\leq \frac{\|(\bA - \bA^{(m)}){\bU}^{(m)}_{\bA_+}\|_2}{\lambda_{p}({\bA}^{(m)}) - \lambda_{p + 1}(\bA)}\\
&\leq \frac{8\|(\bA - \bA^{(m)}){\bU}^{(m)}_{\bA_+}\|_{\mathrm{F}}}{(n\rho_n)\lambda_d(\bDelta_n)},\\
\|\bU_{\bA_-}^{(m)}\bH_-^{(m)} - \bU_{\bA_-}\bH_-\|_2 &= \|\sin\Theta({\bU}_{\bA_-}^{(m)}, \bU_{\bA_-})\|_2
\leq \frac{\|(\bA - \bA^{(m)}){\bU}^{(m)}_{\bA_-}\|_2}{\lambda_{n - q}({\bA}) - \lambda_{n - q + 1}(\bA^{(m)})}\\
&\leq \frac{8\|(\bA - \bA^{(m)}){\bU}^{(m)}_{\bA_-}\|_{\mathrm{F}}}{(n\rho_n)\lambda_d(\bDelta_n)}
\end{align*}
with probability at least $1 - c_0n^{-\zeta}$ for large $n$. 
We now focus on $\|(\bA - \bA^{(m)}){\bU}^{(m)}_\bA\|_{\mathrm{F}}$. The key idea is that the non-zero entries of $\bA - \bA^{(m)}$ are the centered version of the $m$th row and $m$th column of $\bA$, namely, $\{A_{ij} - \expect A_{ij}:i = m\text{ or }j = m\}$. This is a collection of random variables that are independent of ${\bA}^{(m)}$. Since ${\bU}_{\bA_\pm}^{(m)}$ is the eigenvector matrix of ${\bA}^{(m)}$ corresponding to the eigenvalues in $\bS_{\bA_\pm}$, it follows that $(\bA - \bA^{(m)})$ and ${\bU}^{(m)}_{\bA_\pm}$ are independent. Write
\begin{align*}
\|(\bA - {\bA}^{(m)}){\bU}^{(m)}_{\bA_\pm}\|_{\mathrm{F}}^2
& = \sum_{i \neq m}(A_{im} - \expect A_{im})^2\|[{\bU}^{(m)}_{\bA_\pm}]_{m*}\|_{2}^2
 + \left\|\sum_{j = 1}^n(A_{mj} - \expect A_{mj})[{\bU}^{(m)}_{\bA_\pm}]_{j*}\right\|_2^2\\
&\leq \|\bE\|_{2\to\infty}^2\|{\bU}^{(m)}_{\bA_\pm}\|_{2\to\infty}^2 + \left\|\sum_{j = 1}^n(A_{mj} - \expect A_{mj})[{\bU}^{(m)}_{\bA_\pm}]_{j*}\right\|_2^2.
\end{align*}
By Assumption \ref{assumption:spectral_norm_concentration} and the second assertion, for large $n$, we know that with probability at least $1 - c_0n^{-\zeta\wedge\xi}$, 
\[
\|\bE\|_{2\to\infty}\|{\bU}^{(m)}_{\bA_\pm}\|_{2\to\infty}\leq \|\bE\|_2\|{\bU}^{(m)}_{\bA_\pm}\|_{2\to\infty}\lesssim \{\kappa(\bDelta_n) + \varphi(1)\}(n\rho_n)^{1/2} \|\bU_\bP\|_{2\to\infty}.
\]
For the second part, for any $t \geq 1$ and $t\lesssim n\rho_n$, we consider the following two events:
\begin{align*}
&\calE_1 = \left\{\bA:\mathrel{\Big\|}\sum_{j = 1}^n(A_{mj} - \expect A_{mj})[{\bU}^{(m)}_{\bA_\pm}]_{j*}\mathrel{\Big\|}_2
\leq C_0t\|\bU_{\bA_\pm}^{(m)}\|_{2\to\infty} + C_0\sigma(\rho_n t)^{1/2}\|\bU_{\bA_\pm}^{(m)}\|_{\mathrm{F}}\right\},\\
&\calE_2 = \left\{\bA:\|{\bU}_{\bA_\pm}^{(m)}\|_{2\to\infty}\leq C_0\{\kappa(\bDelta_n) + \varphi(1)\}\|\bU_\bP\|_{2\to\infty}\right\}.
\end{align*}
Here, $C_0 > 0$ is a constant that will be determined later. 
By the independence between $\bA - \bA^{(m)}$ and $\bU_{\bA_\pm}^{(m)}$, Lemma \ref{lemma:Bernstein_concentration_EW}, and Lemma \ref{lemma:Bernstein_concentration_EW_subGaussian}, we can select $C_0$ depending on $\sigma$ such that
\[
\prob(\calE_1) = \sum_{\bA^{(m)}}\prob(\calE_1\mid\bA^{(m)})p(\bA^{(m)})\geq \sum_{\bA^{(m)}}\{1 - c_0de^{-t}\}p(\bA^{(m)}) = 1 - c_0de^{-t}.
\]
Also, by the second assertion, $\prob(\calE_2)\geq 1 - c_0n^{-\zeta\wedge\xi}$ for sufficiently large $n$. Now we consider a realization $\bA\in\calE_1\cap \calE_2$. Then
\begin{align*}
\mathrel{\Big\|}\sum_{j = 1}^n(A_{mj} - \expect A_{mj})[{\bU}^{(m)}_{\bA_\pm}]_{j*}\mathrel{\Big\|_2}
&\leq C_0t\|\bU_{\bA_\pm}^{(m)}\|_{2\to\infty} + C_0\sigma(\rho_n t)^{1/2}\|\bU_{\bA_\pm}^{(m)}\|_{\mathrm{F}}\\
&\leq C_0t\|\bU_{\bA_\pm}^{(m)}\|_{2\to\infty} + C_0\sigma(n\rho_n t)^{1/2}\|\bU_{\bA_\pm}^{(m)}\|_{2\to\infty}\\
&\lesssim_\sigma (n\rho_n t)^{1/2}\|\bU_{\bA_\pm}^{(m)}\|_{2\to\infty}\\
&\lesssim_\sigma (n\rho_n t)^{1/2}\{\kappa(\bDelta_n) + \varphi(1)\}\|\bU_\bP\|_{2\to\infty}.
\end{align*}
Such a realization occurs with probability at least $1 - c_0n^{-\zeta\wedge\xi} - c_0de^{-t}$ for large $n$. 
Hence, we conclude that
\begin{align*}
\|\bU_{\bA_\pm}^{(m)}\bH_\pm^{(m)} - \bU_{\bA_\pm}\bH_\pm\|_2& = \|\sin\Theta({\bU}_{\bA_\pm}^{(m)},\bU_{\bA_\pm})\|_2\\
&\leq \frac{8\|\bE\|_{2\to\infty}\|\bU_{\bA_\pm}^{(m)}\|_{2\to\infty}}{n\rho_n\lambda_d(\bDelta_n)}
 + \frac{8\left\|\sum_{j \neq m}(A_{mj} - \expect A_{mj})[{\bU}^{(m)}_{\bA_\pm}]_{j*}\right\|_2}{n\rho_n\lambda_d(\bDelta_n)}\\
&\lesssim_\sigma \frac{\{\kappa(\bDelta_n) + \varphi(1)\}\{(n\rho_n)^{1/2} + (n\rho_n t)^{1/2}\}}{(n\rho_n)\lambda_d(\bDelta_n)}\|\bU_\bP\|_{2\to\infty}
\\
&\lesssim \frac{\{\kappa(\bDelta_n) + \varphi(1)\}t^{1/2}}{(n\rho_n)^{1/2}\lambda_d(\bDelta_n)}\|\bU_\bP\|_{2\to\infty}
\end{align*}
with probability at least $1 - c_0n^{-\zeta\wedge\xi} - c_0de^{-t}$ for sufficiently large $n$. The proof is thus completed. 


\subsection{Proof of Theorem \ref{thm:eigenvector_deviation}}\label{app:proof_of_eigenvector_deviation}

As discussed in Section \ref{sub:proof_sketch_for_theorem_thm:eigenvector_deviation}, a crucial step in controlling the row-wise perturbation bound of the term $\bU_\bA - \bA\bU_\bP\bS_\bP^{-1}\bW^*$ lies in a sharp control of $\be_m\transpose\bE(\bU_\bA - \bU_\bP\bW^*)$. This result is established in Lemma \ref{lemma:higher_order_eigenvector_deviation} below with the help of the decoupling technique in Section \ref{sub:proof_sketch_for_theorem_thm:eigenvector_deviation} and Lemma \ref{lemma:auxiliary_matrix}. 

\begin{lemma}\label{lemma:higher_order_eigenvector_deviation}
Suppose Assumptions \ref{assumption:incoherence}-\ref{assumption:spectral_norm_concentration} hold. 
Let $\bA^{(m)}$, $\bU_{\bA_\pm}^{(m)}$, and $\bH_\pm^{(m)}$ be defined as in Section \ref{sub:proof_sketch_for_theorem_thm:eigenvector_deviation}. Let $m\in [n]$ be any fixed row index. 
Then there exists an absolute constant $c_0 > 0$, such that for all $t \geq 1$, $t\lesssim n\rho_n$,
\begin{align*}
\|\be_m\transpose{}\bE(\bU_{\bA_\pm} - \bU_{\bP_\pm}\bW^*_\pm)\bS_{\bA_\pm}^{-1}\|_2&\lesssim_\sigma \frac{\|\bU_\bP\|_{2\to\infty}}{n\rho_n \lambda_d(\bDelta_n)}\\
&\quad\times\max\left\{\frac{\{\kappa(\bDelta_n) + \varphi(1)\} t^{1/2}}{\lambda_d(\bDelta_n)} , \frac{\kappa(\bDelta_n) + \varphi(1)}{ \lambda_d(\bDelta_n)^2},
\chi t
\right\}
\end{align*}
holds with probability at least $1 - c_0n^{-\zeta\wedge\xi} - c_0de^{-t}$ for sufficiently large $n$, where
\[
\chi:= 
\varphi(1) + \frac{\|\bX\|_{2\to\infty}^2\vee 1}{\lambda_d(\bDelta_n)}.
\]
\end{lemma}

\begin{proof}[\bf Proof of Lemma \ref{lemma:higher_order_eigenvector_deviation}]
By inequality \eqref{eqn:crucial_inequality} in Section \ref{sub:proof_sketch_for_theorem_thm:eigenvector_deviation}, we immediately obtain
\begin{align*}
\|\be_m\transpose{}\bE(\bU_{\bA_\pm} - \bU_{\bP_\pm}\bW^*_\pm)\bS_{\bA_\pm}^{-1}\|_2
&\leq \|\be_m\transpose{}\bE\bU_{\bA_\pm}\{\mathrm{sgn}(\bH_\pm) - \bH_\pm\}\|_2\|\bS_{\bA_\pm}^{-1}\|_2\\
&\quad + \|\be_m\transpose{}\bE(\bU_{\bA_\pm}\bH_\pm - {\bU}_{\bA_\pm}^{(m)}{\bH}_\pm^{(m)})\|_2\|\bS_{\bA_\pm}^{-1}\|_2\\
&\quad + \|\be_m\transpose{}\bE({\bU}_{\bA_\pm}^{(m)}{\bH}_\pm^{(m)} - \bU_{\bP_\pm})\|_2\|\bS_{\bA_\pm}^{-1}\|_2.
\end{align*}
We first focus on $\|\be_m\transpose\bE\bU_{\bA_\pm}\|_2$. By Assumption \ref{assumption:spectral_norm_concentration} and Lemma 2 in \cite{10.1214/19-AOS1854}, we know that $\|\bH_\pm^{-1}\|_2\leq 2$ with probability at least $1 - c_0n^{-\zeta}$ for sufficiently large $n$. Then by Lemma \ref{lemma:auxiliary_matrix}, Lemma \ref{lemma:Bernstein_concentration_EW}, Lemma \ref{lemma:Bernstein_concentration_EW_subGaussian}, and the fact that $\be_m\transpose\bE$ and $\bU_{\bA_\pm}^{(m)}$ are independent,  we have, for sufficiently large $n$, 
\begin{align*}
\|\be_m\transpose\bE\bU_{\bA_\pm}\|_2
&\leq \|\bH^{-1}_\pm\|_2\{\|\be_m\transpose\bE(\bU_{\bA_\pm}\bH_\pm - \bU_{\bA_\pm}^{(m)}\bH_\pm^{(m)})\|_2 + \|\be_m\transpose\bE\bU_{\bA_\pm}^{(m)}\|_2\}\\
&\lesssim_\sigma \|\bE\|_2\|\bU_{\bA_\pm}\bH_\pm - \bU_{\bA_\pm}^{(m)}\bH_\pm^{(m)}\|_2 + (n\rho_nt)^{1/2}\|\bU_{\bA_\pm}^{(m)}\|_{2\to\infty}\\
&\lesssim_\sigma \frac{\{\kappa(\bDelta_n) + \varphi(1)\}t^{1/2}}{\lambda_d(\bDelta_n)}\|\bU_\bP\|_{2\to\infty} + \{\kappa(\bDelta_n) + \varphi(1)\}(n\rho_nt)^{1/2}\|\bU_\bP\|_{2\to\infty}\\
&\lesssim_\sigma \{\kappa(\bDelta_n) + \varphi(1)\}(n\rho_nt)^{1/2}\|\bU_\bP\|_{2\to\infty}
\end{align*}
with probability at least $1 - c_0n^{-\zeta\wedge\xi} - c_0de^{-t}$, where the last inequality is due to the fact that $\lambda_d(\bDelta_n)^{-1}\lesssim (n\rho_n)^{1/2}$. Letting $t = \{1 + (\xi\wedge\zeta)\}\log n$, we see that 
\[
\|\be_m\transpose\bE\bU_{\bA_\pm}\|_2\lesssim_\sigma \{\kappa(\bDelta_n) + \varphi(1)\}(n\rho_n\log n)^{1/2}\|\bU_{\bP}\|_{2\to\infty}
\]
with probability at least $1 - c_0n^{-\zeta\wedge\xi}$ for large $n$. We now work on the first term, By Assumption \ref{assumption:spectral_norm_concentration}, Result \ref{result:S_A_concentration}, Lemma 6.7 in \cite{cape2017two}, and Davis-Kahan theorem, for large $n$,
\begin{align*}
\|\be_m\transpose{}\bE\bU_{\bA_\pm}\{\mathrm{sgn}(\bH_\pm) - \bH_\pm\}\|_2\|\bS_{\bA_\pm}^{-1}\|_2
&\leq \|\be_m\transpose\bE\bU_{\bA_\pm}\|_2\|\sin\Theta(\bU_{\bA_\pm},\bU_{\bP_\pm})\|_2^2\|\bS_{\bA_\pm}^{-1}\|_2\\
&\leq \|\be_m\transpose\bE\bU_{\bA_\pm}\|_2\frac{4\|\bE\|_2^2}{(n\rho_n)^2\lambda_d(\bDelta_n)^2}\|\bS_{\bA_\pm}^{-1}\|_2\\
&\lesssim_\sigma (n\rho_n)\{\kappa(\bDelta_n) + \varphi(1)\}\|\bU_\bP\|_{2\to\infty}\\
&\quad\times\frac{n\rho_n}{(n\rho_n)^2\lambda_d(\bDelta_n)^2}\times\frac{1}{n\rho_n\lambda_d(\bDelta_n)}\\
& = \frac{\{\kappa(\bDelta_n) + \varphi(1)\}}{n\rho_n\lambda_d(\bDelta_n)^3}\|\bU_\bP\|_{2\to\infty}
\end{align*}
with probability at least $1 - c_0n^{-\zeta\wedge\xi}$. For the second term, for all $t\geq 1$ and $t\lesssim n\rho_n$, by Assumption \ref{assumption:spectral_norm_concentration}, Result \ref{result:S_A_concentration}, and Lemma \ref{lemma:auxiliary_matrix}, for large $n$,
\begin{align*}
\|\be_m\transpose{}\bE(\bU_{\bA_\pm}\bH_\pm - {\bU}_{\bA_\pm}^{(m)}{\bH}_\pm^{(m)})\bS_{\bA_\pm}^{-1}\|_2
&\leq \|\bE\|_{2\to\infty}\|\bU_{\bA_\pm}\bH_\pm - {\bU}_{\bA_\pm}^{(m)}{\bH}_\pm^{(m)}\|_2\|\bS_{\bA_\pm}^{-1}\|_2\\
&\leq \|\bE\|_2\|\bU_{\bA_\pm}\bH_\pm - {\bU}_{\bA_\pm}^{(m)}{\bH}_\pm^{(m)}\|_2\|\bS_\bA^{-1}\|_2\\
&\lesssim (n\rho_n)^{1/2}\times\frac{\{\kappa(\bDelta_n) + \varphi(1)\}t^{1/2}\|\bU_\bP\|_{2\to\infty}}{(n\rho_n)^{1/2}\lambda_d(\bDelta_n)}\times\frac{1}{n\rho_n\lambda_d(\bDelta_n)}\\
& = \frac{\{\kappa(\bDelta_n) + \varphi(1)\}t^{1/2}}{n\rho_n\lambda_d(\bDelta_n)^2}\|\bU_\bP\|_{2\to\infty}
\end{align*}
with probability at least $1 - c_0n^{-\zeta\wedge\xi} - c_0de^{-t}$.

\vspace*{2ex}\noindent
We now focus on the third term. Denote $\bV_\pm^{(m)} = \bU_{\bA_\pm}^{(m)}\bH_\pm^{(m)} - \bU_{\bP_\pm}$. Let $t \geq 1$ and $t\lesssim n\rho_n$. Consider the following events:
\begin{align*}
&\calE_1 = \left\{\bA:\|\be_m\transpose{}\bE\bV^{(m)}_\pm\|_2\leq C_0t\|\bV^{(m)}_\pm\|_{2\to\infty} + C_0\sigma(\rho_nt)^{1/2}\|\bV^{(m)}_\pm\|_2\right\},\\
&\calE_2 = \left\{\bA:\|\bV_\pm^{(m)}\|_2\leq \frac{C_0}{(n\rho_n)^{1/2}\lambda_d(\bDelta_n)},\|\bU_{\bA_\pm}^{(m)}\|_{2\to\infty}\leq C_0\{\kappa(\bDelta_n) + \varphi(1)\}\|\bU_\bP\|_{2\to\infty}\right\},\\
&\calE_3 = \left\{\bA:\|\bU_{\bA_\pm}^{(m)}\mathrm{sgn}(\bH_\pm^{(m)}) - \bU_{\bP_\pm}\|_{2\to\infty}\leq C_0\chi\|\bU_\bP\|_{2\to\infty}\right\},
\end{align*}
where 
\[
\chi:= 
\varphi(1) + \frac{\|\bX\|_{2\to\infty}^2\vee 1}{\lambda_d(\bDelta_n)}
 .
\] 
and $C_0 > 0$ is a constant that will be determined later. 
Note that 
\begin{align*}
\|\bV_\pm^{(m)}\|_{2\to\infty}&\leq \|\bU_{\bA_\pm}^{(m)}\|_{2\to\infty}\|\mathrm{sgn}(\bH_\pm^{(m)}) - \bH_\pm^{(m)}\|_2 + \|\bU_{\bA_\pm}^{(m)}\mathrm{sgn}(\bH_\pm^{(m)}) - \bU_{\bP_\pm}\|_{2\to\infty}\\
&\leq 2\|\bU_{\bA_\pm}^{(m)}\|_{2\to\infty}  + \|\bU_{\bA_\pm}^{(m)}\mathrm{sgn}(\bH_\pm^{(m)}) - \bU_{\bP_\pm}\|_{2\to\infty}.
\end{align*}
By Lemma \ref{lemma:auxiliary_matrix}, we can select the constant $C_0 > 0$, such that $\prob(\calE_2)\geq 1 - c_0n^{-\zeta\wedge\xi}$, $\prob(\calE_3)\geq 1 - c_0n^{-\zeta\wedge\xi}$ for sufficiently large $n$. For event $\calE_1$, we use the conditional distribution and the fact that $\be_m\transpose{}\bE$ is independent of $\bV_\pm^{(m)}$, together with Lemma \ref{lemma:Bernstein_concentration_EW} and Lemma \ref{lemma:Bernstein_concentration_EW_subGaussian}, to obtain
\begin{align*}
\prob(\calE_1)
& = \sum_{\bA^{(m)}}\prob(\calE_1\mid\bA^{(m)})p(\bA^{(m)})\geq \sum_{\bA^{(m)}}(1 - c_0de^{-t})p(\bA^{(m)})
 = 1 - c_0de^{-t}. 
\end{align*}
Then over the event $\calE_1\cap \calE_2\cap\calE_3$, we apply the fact that $\|\bU_\bP\|_{2\to\infty}\geq\sqrt{d/n}$ to obtain
\begin{align*}
\|\be_m\transpose{}\bE\bV^{(m)}_\pm\|_2
& \leq C_0t\|\bV^{(m)}_\pm\|_{2\to\infty} + C_0\sigma(\rho_nt)^{1/2}\|\bV^{(m)}_\pm\|_{\mathrm{F}}\\
& \leq C_0t\|\bU_{\bA_\pm}^{(m)}\|_{2\to\infty} + C_0t\|\bU_{\bA_\pm}^{(m)}\mathrm{sgn}(\bH^{(m)}_\pm) - \bU_{\bP_\pm}\|_{2\to\infty} + C_0\sigma(d\rho_nt)^{1/2}\|\bV^{(m)}_\pm\|_2\\
&\lesssim_\sigma \chi t\|\bU_\bP\|_{2\to\infty} + \frac{t^{1/2}}{\lambda_d(\bDelta_n)}\sqrt{\frac{d}{n}}\\
&\lesssim \chi t\|\bU_\bP\|_{2\to\infty} .
\end{align*}
It follows from Result \ref{result:S_A_concentration} that 
  \begin{align*}
  \|\be_m\transpose{}\bE({\bU}_{\bA_\pm}^{(m)}{\bH}_\pm^{(m)} - \bU_{\bP_\pm})\|_2\|\bS_{\bA_\pm}^{-1}\|_2
  &\lesssim_c 
  \frac{  \chi t\|\bU_\bP\|_{2\to\infty}}{ n\rho_n\lambda_d(\bDelta_n)}
   \end{align*}
with probability at least $1 - c_0n^{-\zeta\wedge\xi} - c_0de^{-t}$ for sufficiently large $n$. 
The proof is completed by combining the above concentration bounds.
\end{proof}

We are now in a position to prove Theorems \ref{thm:eigenvector_deviation}. 
\begin{proof}[\bf Proof of Theorem \ref{thm:eigenvector_deviation}]
The proof follows from Lemmas \ref{lemma:higher_order_eigenvector_deviation} and \ref{lemma:higher_order_remainder}. Following the decomposition of $\bU_\bA - \bA\bU_\bP\bS_\bP^{-1}\bW^*$ in Section \ref{sub:proof_sketch_for_theorem_thm:eigenvector_deviation}, we have
\begin{align*}
\|\be_m\transpose(\bU_{\bA_\pm} - \bA\bU_{\bP_\pm}\bS_{\bP_\pm}^{-1}\bW^*_\pm)\|
& \leq \|\be_m\transpose\bE(\bU_{\bA_\pm} - \bU_{\bP_\pm}\bW^*_\pm)\bS_{\bA_\pm}^{-1}\|_{2\to\infty}\\
&\quad + \|\bU_{\bP_\pm}\bS_{\bP_\pm}(\bU_{\bP_\pm}\transpose\bU_{\bA_\pm}\bS_{\bA_\pm}^{-1} - \bS_{\bP_\pm}^{-1}\bU_{\bP_\pm}\transpose\bU_{\bA_\pm})\|_{2\to\infty}\\
&\quad + \|\bU_{\bP_\pm}(\bU_{\bP_\pm}\transpose\bU_{\bA_\pm} - \bW^*_\pm)\|_{2\to\infty}\\
&\quad + \|\be_m\transpose\bE\bU_{\bP_\pm}(\bW^*_\pm\bS_{\bA_\pm}^{-1} - \bS_{\bP_\pm}^{-1}\bW^*_\pm)\|_2\\
&\quad + \|\bU_{\bP_\mp}\bS_{\bP_\mp}\bU_{\bP_\mp}\transpose\bU_{\bA_\pm}\bS_{\bA_\pm}^{-1}\|_{2\to\infty}.
\end{align*}
By Lemma \ref{lemma:higher_order_eigenvector_deviation}, the first term on the right-hand side above satisfies
\begin{align*}
\|\be_m\transpose{}\bE(\bU_{\bA_\pm} - \bU_{\bP_\pm}\bW^*_\pm)\bS_{\bA_\pm}^{-1}\|_2&\lesssim_\sigma \frac{\|\bU_\bP\|_{2\to\infty}}{n\rho_n \lambda_d(\bDelta_n)}\\
&\quad\times\max\left\{\frac{\{\kappa(\bDelta_n) + \varphi(1)\}t^{1/2}}{\lambda_d(\bDelta_n)}, \frac{\{\kappa(\bDelta_n) + \varphi(1)\}}{\lambda_d(\bDelta_n)^2},
\chi t
\right\}
\end{align*}
with probability at least $1 - c_0n^{-\zeta\wedge\xi} - c_0de^{-t}$ for sufficiently large $n$. We also know from Lemma \ref{lemma:higher_order_remainder} that the following events hold with probability at least $1 - c_0n^{-\zeta\wedge\xi} - c_0de^{-t}$ for $t\geq 1$, $t\lesssim n\rho_n$:
\begin{align*}
&\|\bU_{\bP_\pm}\bS_{\bP_\pm}(\bU_{\bP_\pm}\transpose{}\bU_{\bA_\pm}\bS_{\bA_\pm}^{-1} - \bS_{\bP_\pm}^{-1}\bU_{\bP_\pm}\transpose\bU_{\bA_\pm})\|_{2\to\infty} 
\lesssim \frac{\|\bU_\bP\|_{2\to\infty}}{n\rho_n\lambda_d(\bDelta_n)}
\max\left\{
t^{1/2}, d^{1/2}, \frac{1}{\lambda_d(\bDelta_n)}
\right\},\\
&\|\be_m\transpose\bE\bU_{\bP_\pm}(\bW_\pm^*\bS_{\bA_\pm}^{-1} - \bS_{\bP_\pm}^{-1}\bW^*_\pm)\|_{2} 
\lesssim_\sigma \frac{\|\bU_\bP\|_{2\to\infty}}{n\rho_n\lambda_d^2(\bDelta_n)}\max\left\{
  \frac{\kappa(\bDelta_n)}{\lambda_d(\bDelta_n)}
  ,t^{1/2},d^{1/2}
  \right\},\\
&\|\bU_{\bP_\pm}(\bU_{\bP_\pm}\transpose{}\bU_{\bA_\pm} - \bW^*_\pm)\|_{2\to\infty} 
\lesssim \frac{\|\bU_\bP\|_{2\to\infty}}{n \rho_n\lambda_d(\bDelta_n)^{2}},\\
&\|\bU_{\bP_\mp}\bS_{\bP_\mp}\bU_{\bP_\mp}\transpose\bU_{\bA_\pm}\bS_{\bA_\pm}^{-1}\|_{2\to\infty}
\lesssim 
\frac{d^{1/2}\|\bU_\bP\|_{2\to\infty}}{n\rho_n\lambda_d(\bDelta_n)}\max\left\{\frac{1}{\lambda_d(\bDelta_n)},t^{1/2},d^{1/2}\right\}
\end{align*}
where we have used the fact that $t^{1/2}/(n\rho_n)^{1/2}\lesssim 1$ and $t/(n\rho_n)^{1/2}\lesssim t^{1/2}$. 
Then Lemmas \ref{lemma:higher_order_eigenvector_deviation} and \ref{lemma:higher_order_remainder} immediately imply that
\begin{align*}
&\|\be_m\transpose{}(\bU_{\bA_\pm} - \bA\bU_{\bP_\pm}\bS_{\bP_\pm}^{-1}\bW^*_\pm)\|_2\\
&\quad\lesssim_\sigma
\frac{\|\bU_\bP\|_{2\to\infty}}{n\rho_n \lambda_d(\bDelta_n)}\max\left\{\frac{\{\kappa(\bDelta_n) + \varphi(1)\}t^{1/2}}{\lambda_d(\bDelta_n)}, \frac{\{\kappa(\bDelta_n) + \varphi(1)\}}{\lambda_d(\bDelta_n)^2},
\chi t
\right\}
\\
&\quad\quad + \frac{\|\bU_\bP\|_{2\to\infty}}{n\rho_n\lambda_d(\bDelta_n)}\max\left\{t^{1/2},d^{1/2},\frac{1}{\lambda_d(\bDelta_n)}\right\}
 + \frac{\|\bU_\bP\|_{2\to\infty}}{n\rho_n\lambda_d(\bDelta_n)^{2}}\\
&\quad\quad + \frac{\|\bU_\bP\|_{2\to\infty}}{n \rho_n\lambda_d(\bDelta_n)^{2}}\max\left\{
  \frac{\kappa(\bDelta_n)}{\lambda_d(\bDelta_n)}, t^{1/2}, d^{1/2}\right\}\\
&\quad\quad +  
\frac{d^{1/2}\|\bU_\bP\|_{2\to\infty}}{n\rho_n\lambda_d(\bDelta_n)}\max\left\{\frac{1}{\lambda_d(\bDelta_n)},t^{1/2},d^{1/2}\right\}
\\
&\quad\leq
\frac{\chi\|\bU_\bP\|_{2\to\infty}}{n\rho_n \lambda_d(\bDelta_n)}\max\left\{
  \frac{t^{1/2}}{\lambda_d(\bDelta_n)}, \frac{1}{\lambda_d(\bDelta_n)^2}, t\right\}
\end{align*}
with probability at least $1 - c_0n^{-\zeta\wedge\xi} -c_0de^{-t}$ for sufficiently large $n$. Here, we have used the fact that $d^{1/2}\leq d\leq\|\bX\|_{2\to\infty}^2/\lambda_d(\bDelta_n)$ from Result \ref{result:UP_delocalization} and
\[
\kappa(\bDelta_n) = \frac{\lambda_1(\bDelta_n)}{\lambda_d(\bDelta_n)}\leq \frac{\|\bX\|_{\mathrm{F}}^2}{n\lambda_d(\bDelta_n)}\leq \frac{\|\bX\|_{2\to\infty}^2}{\lambda_d(\bDelta_n)}.
\]
This completes the first assertion. For the entrywise perturbation bound for the scaled eigenvectors, we recall the decomposition \eqref{eqn:keystone_remainder_term}
\begin{align*}
\widetilde\bX_\pm\bW_\pm - \frac{\pm\bA\bX_\pm(\bX_\pm\transpose{}\bX_\pm)^{-1}}{\rho_n^{1/2}}
& = \bU_{\bA_\pm}(\bW_\pm^*|\bS_{\bA_\pm}|^{1/2} - |\bS_{\bP_\pm}|^{1/2}\bW_\pm^*)\transpose\bW_{\bX_\pm}\\
&\quad + (\bU_{\bA_\pm} - \bA\bU_{\bP_\pm}\bS_{\bP_\pm}^{-1}\bW_\pm^*)(\bW_\pm^*)\transpose{}|\bS_{\bP_\pm}|^{1/2}\bW_{\bX_\pm}.
\end{align*}
Then for each fixed row index $m\in [n]$, for all $t\geq 1$ and $t\lesssim n\rho_n$, we apply the first assertion above, Lemma \ref{lemma:U_A_two_to_infinity_norm}, and Lemma \ref{lemma:WS_interchange_bound} to conclude that
\begin{align*}
&\|\be_m\transpose(\widetilde\bX_\pm\bW_\pm - (\pm)\rho_n^{-1/2}\bA\bX_\pm(\bX_\pm\transpose{}\bX_\pm)^{-1})\|_2\\
&\quad\leq \|\bU_{\bA_\pm}\|_{2\to\infty}\|\bW_\pm^*|\bS_{\bA_\pm}|^{1/2} - |\bS_{\bP_\pm}|^{1/2}\bW_\pm^*\|_2
 + \|\be_m\transpose(\bU_{\bA_\pm} - \bA\bU_{\bP_\pm}\bS_{\bP_\pm}^{-1}\bW^*_\pm)\|_2\|\bS_{\bP_\pm}\|_2^{1/2}\\
&\quad\lesssim \frac{\{\kappa(\bDelta_n) + \varphi(1)\}\|\bX\|_{2\to\infty}\|\bU_\bP\|_{2\to\infty}}{(n\rho_n)^{1/2}\lambda_d(\bDelta_n)}
\max\left\{\frac{\kappa(\bDelta_n)}{\lambda_d(\bDelta_n)}, t^{1/2}, d^{1/2}\right\}\\
&\quad\quad + 
\frac{\chi\|\bX\|_{2\to\infty}\|\bU_\bP\|_{2\to\infty}}{(n\rho_n)^{1/2}\lambda_d(\bDelta_n)}\max\left\{\frac{t^{1/2}}{\lambda_d(\bDelta_n)}, \frac{1}{\lambda_d(\bDelta_n)^2}, t\right\}
\\&\quad\lesssim
\frac{\chi\|\bX\|_{2\to\infty}\|\bU_\bP\|_{2\to\infty}}{(n\rho_n)^{1/2}\lambda_d(\bDelta_n)}\max\left\{\frac{(\|\bX\|_{2\to\infty}^2\vee 1)t^{1/2}}{\lambda_d(\bDelta_n)^2}, \frac{\kappa(\bDelta_n)}{\lambda_d(\bDelta_n)^2}, t\right\}
\end{align*}
with probability at least $1 - c_0n^{-\zeta\wedge\xi} - c_0de^{-t}$ for sufficiently large $n$. This completes the proof of the second assertion. The third and fourth assertions regarding the concentrations of 
\[
\|\bU_{\bA_\pm} - \bA\bU_{\bP_\pm}\bS_{\bP_\pm}^{-1}\bW^*_\pm\|_{2\to\infty}
\quad\mbox{and}\quad
\|\widetilde\bX_\pm\bW_\pm - (\pm)\rho_n^{-1/2}\bA\bX_\pm(\bX_\pm\transpose{}\bX_\pm)^{-1}\|_{2\to\infty}
\]
are immediate from the first two assertions and a union bound over $m\in [n]$ because $\zeta\wedge\xi$ is strictly greater than $1$. 
\end{proof}

\subsection{Proof of Theorem \ref{thm:ASE_Berry_Esseen}} 
\label{sub:proof_of_theorem_thm:ase_berry_esseen}

By Theorem \ref{thm:eigenvector_deviation} and decompositions \eqref{eqn:keystone_decomposition}, \eqref{eqn:keystone_decomposition_unscaled_eigenvector} in the manuscript, for each fixed $i\in [n]$, we have
\begin{align*}
\sqrt{n}(\bW\transpose(\widetilde\bx_i)_\pm  - \rho_n^{1/2}(\bx_{i})_\pm)& = \frac{\pm1}{\sqrt{n\rho_n}}\sum_{j = 1}^n(A_{ij} - \expect A_{ij})\bDelta_{n_\pm}^{-1}(\bx_{j})_\pm + \sqrt{n}\bR_{\bX_\pm}\transpose{}\be_i,\\
n\rho_n^{1/2}\bW_{\bX_\pm}\transpose(\bW_\pm^*[\bU_{\bA_\pm}]_{i*}  - [\bU_{\bP_\pm}]_{i*})& = \frac{\pm1}{\sqrt{n\rho_n}}\sum_{j = 1}^n(A_{ij} - \expect A_{ij})\bDelta_{n\pm}^{-3/2}(\bx_{j})_\pm\\
&\quad + n\rho_n^{1/2}\{\bR_{\bU_\pm}(\bW_\pm^*)\transpose\bW_{\bX_\pm}\}\transpose{}\be_i,
\end{align*}
where $\bR_{\bX_\pm} = \widetilde\bX_\pm\bW_\pm - (\pm)\rho_n^{-1/2}\bA\bX_\pm(\bX_\pm\transpose{}\bX_\pm)^{-1}$ and $\bR_{\bU_\pm} = \bU_{\bA_\pm} - \bA\bU_{\bP_\pm}\bS_{\bP_\pm}^{-1}\bW_\pm^*$. Equivalently, we have
\begin{align*}
\sqrt{n}\bSigma_{ni\pm}^{-1/2}(\bW\transpose(\widetilde\bx_i)_\pm - \rho_n^{1/2}(\bx_i)_\pm)
& = \frac{\pm1}{\sqrt{n\rho_n}}\sum_{j = 1}^n({A}_{ij} - \expect A_{ij})\bSigma_{ni\pm}^{-1/2}\bDelta_{n\pm}^{-1}(\bx_{j})_\pm\\
&\quad + \sqrt{n}\bSigma_{ni\pm}^{-1/2}\bR_{\bX_\pm}\transpose{}\be_i,\\
n\rho_n^{1/2}\bGamma_{ni\pm}^{-1/2}\bW_{\bX_\pm}\transpose(\bW_\pm^*[\bU_{\bA_\pm}]_{i*}  - [\bU_{\bP_\pm}]_{i*})& = \frac{1}{\sqrt{n\rho_n}}\sum_{j = 1}^n(A_{ij} - \expect A_{ij})\bGamma_{ni\pm}^{-1/2}\bDelta_{n\pm}^{-3/2}(\bx_{j})_\pm\\
&\quad + n\rho_n^{1/2}\bGamma_{ni\pm}^{-1/2}\{\bR_{\bU_\pm}(\bW_\pm^*)\transpose\bW_{\bX_\pm}\}\transpose{}\be_i.
\end{align*}
To apply Theorem \ref{thm:Berry_Esseen_Multivariate}, we take
\begin{align*}
&\bxi_j = \frac{\pm1}{\sqrt{n\rho_n}}({A}_{ij} - \expect A_{ij})\bSigma_{ni\pm}^{-1/2}\bDelta_{n\pm}^{-1}(\bx_{j})_\pm,\quad
\bD = \sqrt{n}\bSigma_{ni\pm}^{-1/2}\bR_{\bX_\pm}\transpose{}\be_i,\\
&\Delta^{(j)} = \Delta = C
\frac{\chi\|\bSigma_{ni\pm}^{-1/2}\|_2\|\bX\|_{2\to\infty}^2}{(n\rho_n)^{1/2}\lambda_d(\bDelta_{n})^{3/2}}
\max\left\{\frac{(\|\bX\|_{2\to\infty}^2\vee1)(\log n\rho_n)^{1/2}}{\lambda_d(\bDelta_n)^2}, \frac{\kappa(\bDelta_n)}{\lambda_d(\bDelta_n)^2}, {\log n\rho_n} \right\},\\
& \calO = \{\bA: \Delta > \|\bD\|_2\},
\end{align*}
and
\begin{align*}
&\bxi_j' = \frac{\pm}{\sqrt{n\rho_n}}({A}_{ij} - \expect A_{ij})\bGamma_{ni\pm}^{-1/2}\bDelta_{n\pm}^{-3/2}(\bx_{j})_\pm,
\quad
\bD' = \sqrt{n}\bGamma_{ni}^{-1/2}\{\bR_{\bU_\pm}(\bW_\pm^*)\transpose\bW_{\bX_\pm}\}\transpose{}\be_i,\\
&\Delta^{(j)'} = \Delta' = C
\frac{\chi\|\bGamma_{ni\pm}^{-1/2}\|_2\|\bX\|_{2\to\infty}}{(n\rho_n)^{1/2}\lambda_d(\bDelta_n)^{3/2}}\max\left\{\frac{(\log n\rho_n)^{1/2}}{\lambda_d(\bDelta_n)}, \frac{1}{\lambda_d(\bDelta_n)^2}, {\log n\rho_n} \right\},\\
& \calO' = \{\bA: \Delta' > \|\bD'\|_2\}.
\end{align*}
Here $C > 0$ is an absolute constant. In particular, we can select $C > 0$, which may depend on $\sigma$, such that $\prob(\calO^c)\lesssim d/(n\rho_n)$ and  $\prob\{(\calO')^c\}\lesssim d/(n\rho_n)$
for sufficiently large $n$ according to Theorem \ref{thm:eigenvector_deviation}. Note that Assumption \ref{assumption:distribution} implies that $\expect[\bE_{ij}]^2\leq \sigma^2\rho_n$ for all $i,j\in [n]$, so that
\begin{align*}
\bSigma_{ni\pm} &= \bDelta_{n\pm}^{-1}\left\{\frac{1}{n\rho_n}\sum_{j = 1}^n\expect[\bE_{ij}]^2(\bx_j)_\pm(\bx_j)_\pm\transpose\right\}\bDelta_{n\pm}^{-1}\preceq \sigma^2\bDelta_{n\pm}^{-1}\Longrightarrow\|\bSigma_{ni\pm}^{-1/2}\|_2\geq \frac{1}{\sigma}\lambda_1(\bDelta_{n\pm})^{1/2},\\
\bGamma_{ni\pm} &= \bDelta_{n\pm}^{-3/2}\left\{\frac{1}{n\rho_n}\sum_{j = 1}^n\expect[\bE_{ij}]^2(\bx_j)_\pm(\bx_j)_\pm\transpose\right\}\bDelta_{n\pm}^{-3/2}\preceq \sigma^2\bDelta_{n\pm}^{-2}\Longrightarrow\|\bGamma_{ni\pm}^{-1/2}\|_2\geq \frac{1}{\sigma}\lambda_1(\bDelta_{n\pm}).
\end{align*}
Therefore, by Result \ref{result:UP_delocalization}, we have
\begin{align*}
\prob(\calO)&\lesssim \frac{d}{n\rho_n}\leq \frac{\|\bX\|_{2\to\infty}^2}{n\rho_n\lambda_d(\bDelta_n)} \leq \frac{\|\bX\|_{2\to\infty}^2\lambda_1(\bDelta_{n\pm})^{1/2}}{n\rho_n\lambda_d(\bDelta_n)^{3/2}}\lesssim_\sigma \frac{\chi d^{1/2}\|\bSigma_{ni}^{-1/2}\|_2\|\bX\|_{2\to\infty}^2}{n\rho_n\lambda_d(\bDelta_n)^{3/2}}\lesssim_\sigma d^{1/2}\Delta,\\
\prob(\calO')&\lesssim \frac{d}{n\rho_n}\leq \frac{d^{1/2}\|\bX\|_{2\to\infty}}{n\rho_n\lambda_d(\bDelta_n)^{1/2}} \leq \frac{d^{1/2}\|\bX\|_{2\to\infty}\lambda_1(\bDelta_{n\pm})}{n\rho_n\lambda_d(\bDelta_n)^{3/2}}\lesssim_\sigma \frac{\chi d^{1/2}\|\bGamma_{ni}^{-1/2}\|_2\|\bX\|_{2\to\infty}}{n\rho_n\lambda_d(\bDelta_n)^{3/2}}\lesssim_\sigma d^{1/2}\Delta'.
\end{align*}
Note that $|\Delta - \Delta^{(j)}| = 0$, $|\Delta' - \Delta^{(j)'}| = 0$, and $\Delta^{(j)}$'s,  $\Delta^{(j)'}$'s are constant random variables so that $\Delta^{(j)}$ and $\bxi_j$ are independent, and $\Delta^{(j)'}$ and $\bxi_j'$ are independent as well. 
Furthermore, $\expect(\bxi_j) = \expect(\bxi_j') = 0$ and by the definition of $\bSigma_n(\bx_{i})$, $\bGamma_n(\bx_{i})$, 
\begin{align*}
\sum_{j = 1}^n\expect(\bxi_j\bxi_j\transpose{})
& = \frac{1}{n\rho_n}\sum_{j = 1}^n\expect[\bE_{ij}]^2\bSigma_{ni\pm}^{-1/2}\bDelta_{n\pm}^{-1}(\bx_{j})_\pm(\bx_{j})_\pm\transpose{}\bDelta_{n\pm}^{-1}\bSigma_{ni\pm}^{-1/2}\\
& = \bSigma_{ni\pm}^{-1/2}\bDelta_{n\pm}^{-1}
\left\{\frac{1}{n\rho_n}\sum_{j = 1}^n\expect[\bE_{ij}]^2(\bx_{j})_\pm(\bx_{j})_\pm\transpose{}\right\}
\bDelta_{n\pm}^{-1}\bSigma_{ni\pm}^{-1/2}
  = \eye_d,\\
\sum_{j = 1}^n\expect\{(\bxi_j')(\bxi_j')\transpose{}\}
& = \frac{1}{n\rho_n}\sum_{j = 1}^n\expect[\bE]_{ij}^2\bGamma_{ni\pm}^{-1/2}\bDelta_{n\pm}^{-3/2}(\bx_{j})_\pm(\bx_{j})_\pm\transpose{}\bDelta_{n\pm}^{-3/2}\bGamma_{n\pm}^{-1/2}\\
& = \bGamma_{ni\pm}^{-1/2}\bDelta_{n\pm}^{-3/2}
\left\{\frac{1}{n\rho_n}\sum_{j = 1}^n\expect[\bE]_{ij}^2(\bx_{j})_\pm(\bx_{j})_\pm\transpose{}\right\}
\bDelta_{n\pm}^{-3/2}\bGamma_{ni\pm}^{-1/2} = \eye_d. 
\end{align*}
We now proceed to $\sum_{j = 1}^n\expect(\|\bxi_j\|_2^3)$, $\sum_{j = 1}^n\expect(\|\bxi_j'\|_2^3)$, and $\expect(\|\sum_{j = 1}^n\bxi_j\|_2)$, $\expect(\|\sum_{j = 1}^n\bxi_j'\|_2)$. For the first two terms, under Assumption \ref{assumption:distribution} (i), we have
\begin{align*}
\sum_{j = 1}^n\expect(\|\bxi_j\|_2^3)
&= \frac{1}{(n\rho_n)^{3/2}}\sum_{j = 1}^n\expect|[\bE]_{ij}|^3\|\bSigma_{ni\pm}^{-1/2} \bDelta_{n\pm}^{-1}(\bx_{j})_\pm\|_2 (\bx_{j})_\pm\transpose{}\bDelta_{n\pm}^{-1}\bSigma_{ni\pm}^{-1}\bDelta_{n\pm}^{-1}(\bx_{j})_\pm \\
&\leq \frac{\|\bSigma_{n\pm}^{-1/2}\|_2\|\bX\|_{2\to\infty}}{(n\rho_n)^{3/2}\lambda_d(\bDelta_n)}\sum_{j = 1}^n\expect|[\bE]_{ij}|^3(\bx_{j})_\pm\transpose{}\bDelta_{n\pm}^{-1}\bSigma_{ni\pm}^{-1}\bDelta_{n\pm}^{-1}(\bx_{j})_\pm,\\
\sum_{j = 1}^n\expect(\|\bxi_j'\|_2^3)
&= \frac{1}{(n\rho_n)^{3/2}}\sum_{j = 1}^n\expect|[\bE]_{ij}|^3\|\bGamma_{ni\pm}^{-1/2} \bDelta_{n\pm}^{-3/2}(\bx_{j})_\pm\|_2 (\bx_{j})_\pm\transpose{}\bDelta_{n\pm}^{-3/2}\bGamma_{ni\pm}^{-1}\bDelta_{n\pm}^{-3/2}(\bx_{j})_\pm .
\\
&\leq\frac{\|\bGamma_{ni\pm}^{-1/2}\|_2\|\bX\|_{2\to\infty}}{(n\rho_n)^{3/2}\lambda_d(\bDelta_{n})^{3/2}}\sum_{j = 1}^n\expect|[\bE]_{ij}|^3(\bx_{j})_\pm\transpose{}\bDelta_{n\pm}^{-3/2}\bGamma_{ni\pm}^{-1}\bDelta_{n\pm}^{-3/2}(\bx_{j})_\pm.
\end{align*}
For $\expect(\|\sum_{j = 1}^n\bxi_j\|_2)$, we use Jensen's inequality to write
\begin{align*}
\expect\left(\left\|\sum_{j = 1}^n\bxi_j\right\|_2\right)
&\leq \left\{\expect\left(\left\|\sum_{j = 1}^n\bxi_j\right\|_2^2\right)\right\}^{1/2}
  = \left(\sum_{j = 1}^n\expect\|\bxi_j\|_2^2\right)^{1/2}
  = \left[\mathrm{tr}\left\{\sum_{j = 1}^n\expect(\bxi_j\bxi_j\transpose)\right\}\right]^{1/2} = d^{1/2}.
\end{align*}
Similarly, we also have $\expect(\|\sum_{j = 1}^n\bxi_j'\|_2)\leq d^{1/2}$. This immediately implies that
\begin{align*}
\expect\left(\left\|\sum_{j = 1}^n\bxi_j\right\|_2\Delta\right)\leq d^{1/2}\Delta\quad\mbox{and}\quad
\expect\left(\left\|\sum_{j = 1}^n\bxi_j'\right\|_2\Delta'\right)\leq d^{1/2}\Delta'.
\end{align*}
We now apply Theorem \ref{thm:Berry_Esseen_Multivariate} and the aforementioned results to conclude that
\begin{align*}
&\sup_{A\in\calA}\left|\prob\left\{\sqrt{n}\bSigma_{ni\pm}^{-1/2}(\bW\transpose(\widetilde\bx_i)_\pm - \rho_n^{1/2}(\bx_{i})_\pm)\in A\right\} - \prob\left(\bz\in A\right)\right|\\
&\quad\lesssim_\sigma 
\frac{\chi d^{1/2}\|\bSigma_{ni\pm}^{-1/2}\|_2\|\bX\|_{2\to\infty}^2}{(n\rho_n)^{1/2}\lambda_d(\bDelta_{n})^{3/2}}\max\left\{\frac{(\|\bX\|_{2\to\infty}^2\vee 1)(\log n\rho_n)^{1/2}}{\lambda_d(\bDelta_n)^2}, \frac{1}{\lambda_d(\bDelta_n)^2}, {\log n\rho_n} \right\}\\
&\quad\quad + \frac{d^{1/2}\|\bSigma_{n\pm}^{-1/2}\|_2\|\bX\|_{2\to\infty}}{(n\rho_n)^{3/2}\lambda_d(\bDelta_n)}\sum_{j = 1}^n\expect|[\bE]_{ij}|^3\mathrm{tr}\left\{\bSigma_{ni\pm}^{-1/2}\bDelta_{n\pm}^{-1}(\bx_{j})_\pm(\bx_{j})_\pm\transpose{}\bDelta_{n\pm}^{-1}\bSigma_{ni\pm}^{-1/2}\right\}
\end{align*}
and
\begin{align*}
&\sup_{A\in\calA}\left|\prob\left\{n\rho_n^{1/2}\bGamma_{ni\pm}^{-1/2}\bW_{\bX_\pm}\transpose(\bW_\pm^*[\bU_{\bA_\pm}]_{i*} - [\bU_{\bP_\pm}]_{i*})\in A\right\} - \prob\left(\bz\in A\right)\right|\\
&\quad\lesssim_\sigma 
\frac{d^{1/2}\chi\|\bGamma_{ni\pm}^{-1/2}\|_2\|\bX\|_{2\to\infty}}{(n\rho_n)^{1/2}\lambda_d(\bDelta_n)^{3/2}}\max\left\{\frac{(\log n\rho_n)^{1/2}}{\lambda_d(\bDelta_n)}, \frac{1}{\lambda_d(\bDelta_n)^2}, {\log n\rho_n} \right\}\\
&\quad\quad + \frac{d^{1/2}\|\bGamma_{ni\pm}^{-1/2}\|_2\|\bX\|_{2\to\infty}}{(n\rho_n)^{3/2}\lambda_d(\bDelta_{n})^{3/2}}\sum_{j = 1}^n\expect|[\bE]_{ij}|^3\mathrm{tr}\left\{\bGamma_{ni\pm}^{-1/2}\bDelta_{n\pm}^{-3/2}(\bx_{j})_\pm(\bx_{j})_\pm\transpose{}\bDelta_{n\pm}^{-3/2}\bGamma_{ni\pm}^{-1/2}\right\}
\end{align*}
for sufficiently large $n$. This completes the proof. 



\section{Proofs for Section \ref{sub:SNMC}}
\label{sec:proofs_for_section_SNMC}

To prove Theorem \ref{thm:SNMC}, we need to verify Assumptions \ref{assumption:incoherence}-\ref{assumption:spectral_norm_concentration}. The technical tools we applied here are based on Section 3.3 of the Supplementary Material of \cite{10.1214/19-AOS1854}. By the conditions of Theorem \ref{thm:SNMC}, Assumptions \ref{assumption:incoherence} and \ref{assumption:sparsity} hold automatically. For Assumption \ref{assumption:distribution}, we let $[\bE_1]_{ij} = \bx_i\transpose\bx_j(I_{ij} - \rho_n)$ and $[\bE_2]_{ij} = \eps_{ij}I_{ij}/\rho_n$. Clearly, $[\bE]_{ij} = [\bE_1]_{ij} + [\bE_2]_{ij}$ and $[\bE_1]_{ij}$ satisfies Assumption \ref{assumption:distribution} (i). Since 
\[
\|[\bE_2]_{ij}\|_{\psi_2}\leq \frac{1}{\rho_n}\sup_{p\geq 1}\frac{1}{\sqrt{p}}\left(\expect\left|\eps_{ij}I_{ij}\right|^p\right)^{1/p}\leq \frac{1}{\rho_n}\sup_{p\geq 1}\frac{1}{\sqrt{p}}(\expect|\eps_{ij}|^p)^{1/p}\times \sup_{p\geq 1}(\expect|I_{ij}|^p)^{1/p}\lesssim \tau\rho_n,
\]
we see that $[\bE_2]_{ij}$ satisfies Assumption \ref{assumption:distribution} (ii). We now work with Assumptions \ref{assumption:rowwise_concentration} and \ref{assumption:spectral_norm_concentration}. Define
\begin{align*}
\bar{\varphi}(x) &= \left\{
\begin{aligned}
&\frac{4\|\bX\|_{2\to\infty}^2}{\lambda_d(\bDelta_n)}\sqrt{\frac{\log n}{n\rho_n}}\max\left(x, \sqrt{\frac{\log n}{n\rho_n}}\right),&\quad&\mbox{if }x > 0,\\
&0,&\quad&\mbox{if }x = 0,
\end{aligned}
\right.\\
\widetilde{\varphi}(x) &= 
\left\{
\begin{aligned}
&\frac{4\|\bX\|_{2\to\infty}^2}{\lambda_d(\bDelta_n)}\sqrt{\frac{\log n}{n\rho_n}}\frac{\tau\rho_n}{\|\bX\|_{2\to\infty}^2},&\quad&\mbox{if }x > 0,\\
&0,&\quad&\mbox{if }x = 0,
\end{aligned}
\right.
\end{align*}
Let $\varphi(x) = \bar{\varphi}(x) + \widetilde{\varphi}(x)$. Clearly, $\varphi(0) = 0$ and $\varphi(x)/x$ is non-increasing in $(0,+\infty)$. Without loss of generality, we may assume that $\bV\neq \zero_{n\times d}$. 
By Lemma 16 in \cite{10.1214/19-AOS1854}, 
\[
\|\be_i\transpose\bE\bV\|_2\leq n\rho_n\lambda_d(\bDelta_n)\varphi\left(\frac{\|\bV\|_{\mathrm{F}}}{\sqrt{n}\|\bV\|_{2\to\infty}}\right)
\]
with probability at least $1 - c_0n^{-(1 + \xi)}$, where $\xi = 1$ and $c_0 = 5$. To show that the same concentration bound holds for $\|\be_i\transpose\bE^{(m)}\bV\|_2$, we consider $[\bE_1^{(m)}]_{ij}$ and $[\bE_2^{(m)}]_{ij}$ separately. We may assume that $i\neq m$ without loss of generality. Exploiting the proof of the first assertion of Lemma 16 in \cite{10.1214/19-AOS1854}, we see that 
\[
\|\be_i\transpose\bE_1^{(m)}\bV\|_2\leq n\rho_n\lambda_d(\bDelta_n)\bar{\varphi}\left(\frac{\|\bV\|_{\mathrm{F}}}{\sqrt{n}\|\bV\|_{2\to\infty}}\right)
\]
with probability at least $1 - 2n^{-(1 + \xi)}$, where $\xi = 1$. By the proof of the second assertion of Lemma 16 in \cite{10.1214/19-AOS1854}, we have
\begin{align*}
\|\be_i\transpose\bE_2^{(m)}\bV\|_2
&\leq \tau\rho_n^2\|\bV\|_{2\to\infty}\sqrt{\frac{12(n - 1)\log (n - 1)}{\rho_n}}
\leq \tau\rho_n^2\|\bV\|_{2\to\infty}\sqrt{\frac{12n\log n}{\rho_n}}\\
&\leq n\rho_n\lambda_d(\bDelta_n)
\left\{
\frac{4\|\bX\|_{2\to\infty}^2}{\lambda_d(\bDelta_n)}\sqrt{\frac{\log n}{n\rho_n}}\frac{\tau\rho_n}{\|\bX\|_{2\to\infty}^2}
\right\} = n\rho_n\lambda_d(\bDelta_n)\widetilde{\varphi}\left(\frac{\|\bV\|_{\mathrm{F}}}{\sqrt{n}\|\bV\|_{2\to\infty}}\right)
\end{align*}
with probability at least $1 - 4n^{-2}$ for sufficiently large $n$. Therefore, 
\begin{align*}
\|\be_i\transpose\bE^{(m)}\bV\|_2
& \leq \|\be_i\transpose\bE^{(m)}_1\bV\|_2 + \|\be_i\transpose\bE^{(m)}_2\bV\|_2\\
&\leq n\rho_n\lambda_d(\bDelta_n)\left\{\bar{\varphi}\left(\frac{\|\bV\|_{\mathrm{F}}}{\sqrt{n}\|\bV\|_{2\to\infty}}\right) + \widetilde{\varphi}\left(\frac{\|\bV\|_{\mathrm{F}}}{\sqrt{n}\|\bV\|_{2\to\infty}}\right)\right\}\\
& = n\rho_n\lambda_d(\bDelta_n)\varphi\left(\frac{\|\bV\|_{\mathrm{F}}}{\sqrt{n}\|\bV\|_{2\to\infty}}\right)
\end{align*}
with probability at least $1 - 6n^{-(\xi + 1)}$ with $\xi = 1$. 
Hence, Assumption \ref{assumption:rowwise_concentration} holds. 
For Assumption \ref{assumption:spectral_norm_concentration}
, we let 
\[
\gamma = \frac{c_1(\tau\rho_n + \|\bX\|_{2\to\infty}^2)}{(n\rho_n)^{1/2}\lambda_d(\bDelta_n)},
\]
where $c_1 > 0$ is a constant to be determined later. 
By Lemma 14 in \cite{10.1214/19-AOS1854}, we have 
\[
\|\bE\|_2\leq c_2(\|\bX\|_{2\to\infty}^2 + \tau\rho_n)(n\rho_n)^{1/2}
\]
with probability at least $1 - 4n^{-\zeta}$ with $\zeta = 1$, where $c_2 > 1$ is a constant. Then $\|\bA - \expect\bA\|_2\leq K(n\rho_n)^{1/2}$ with with probability at least $1 - 4n^{-\zeta}$ ($\zeta = 1$) if we select $K =  c_2(\|\bX\|_{2\to\infty}^2 + \tau)$. Now set $c_1 = \max\{3K, \|\bX\|_{2\to\infty}^2\}/(\tau\rho_n + \|\bX\|_{2\to\infty}^2)$. 
By the conditions of Theorem \ref{thm:SNMC}, we have
\begin{align*}
&\gamma = \frac{\max(3K, \|\bX\|_{2\to\infty}^2)}{(n\rho_n)^{1/2}\lambda_d(\bDelta_n)}\lesssim \sqrt{\frac{\log n}{n\rho_n\lambda_d(\bDelta_n)^2}}\to 0.
\end{align*}
Note that
\[
c_1 \geq \frac{3K}{\tau\rho_n + \|\bX\|_{2\to\infty}^2} = \frac{3c_2(\|\bX\|_{2\to\infty}^2 + \tau)}{\tau\rho_n + \|\bX\|_{2\to\infty}^2}\geq 3c_2\geq 1.
\]
Then by Lemma 12 in \cite{10.1214/19-AOS1854}, we know that
\begin{align*}
\varphi(\gamma)
&\leq 4\gamma\sqrt{\log n}(1 + \gamma\sqrt{\log n})\\
&\leq 
\frac{4c_1(\tau\rho_n + \|\bX\|_{2\to\infty}^2)}{\lambda_d(\bDelta_n)}\sqrt{\frac{\log n}{n\rho_n}}\left\{1 + 
\frac{c_1(\tau\rho_n + \|\bX\|_{2\to\infty}^2)}{\lambda_d(\bDelta_n)}\sqrt{\frac{\log n}{n\rho_n}}
\right\}\\
&\leq 
\frac{4c_1(\tau\rho_n + \|\bX\|_{2\to\infty}^2)}{\lambda_d(\bDelta_n)}\sqrt{\frac{\log n}{n\rho_n}}\left\{1 + 
\frac{c_1(\tau\rho_n + \|\bX\|_{2\to\infty}^2)\kappa(\bDelta_n)}{\lambda_d(\bDelta_n)}\sqrt{\frac{\log n}{n\rho_n}}
\right\}\\
&\leq \frac{8c_1(\tau\rho_n + \|\bX\|_{2\to\infty}^2)}{\lambda_d(\bDelta_n)}\sqrt{\frac{\log n}{n\rho_n}}
\end{align*}
for sufficiently large $n$ by the condition of Theorem \ref{thm:SNMC}. Note that $\gamma\leq \varphi(\gamma)$. It follows that 
\[
32\kappa(\bDelta_n)\max\{\gamma,\varphi(\gamma)\}\leq \frac{256c_1(\tau\rho_n + \|\bX\|_{2\to\infty}^2)\kappa(\bDelta_n)}{\lambda_d(\bDelta_n)}\sqrt{\frac{\log n}{n\rho_n}}\to 0.
\]
Thus, Assumptions \ref{assumption:incoherence}-\ref{assumption:spectral_norm_concentration} hold, allowing us to apply Theorem \ref{thm:ASE_Berry_Esseen}. Again, by Lemma 12 in \cite{10.1214/19-AOS1854}, we have,
\[
\varphi(1)\lesssim \gamma\sqrt{\log n}\lesssim \frac{\tau\rho_n + \|\bX\|_{2\to\infty}^2}{\lambda_d(\bDelta_n)}\sqrt{\frac{\log n}{n\rho_n}}\lesssim_\tau \frac{\|\bX\|_{2\to\infty}^2}{\lambda_d(\bDelta_n)}. 
\]
Observe that $\kappa(\bDelta_n) = \lambda_1(\bDelta_n)/\lambda_d(\bDelta_n)\leq \|\bDelta_n\|_{\mathrm{F}}/\lambda_d(\bDelta_n)\leq \|\bX\|_{2\to\infty}^2/\lambda_d(\bDelta_n)$. Therefore,
\begin{align*}
\chi 
&= 
\varphi(1) + \frac{\|\bX\|_{2\to\infty}^2\vee 1}{\lambda_d(\bDelta_n)}
\lesssim \frac{\|\bX\|_{2\to\infty}^2\vee 1}{\lambda_d(\bDelta_n)}\leq \frac{\|\bX\|_{2\to\infty}^4\vee 1}{\lambda_d(\bDelta_n)}.
\end{align*}
In addition, we have
\begin{align*}
\expect |[\bE]_{ij}|^3&\lesssim \expect|[\bE_1]_{ij}|^3 + \expect|[\bE_2]_{ij}|^3
\lesssim\|\bX\|_{2\to\infty}^2(\bx_i\transpose\bx_j)^2\rho_n(1 - \rho_n) + \tau^3\rho_n^4\\
&\lesssim_\tau (\|\bX\|_{2\to\infty}^2\vee1)\{\rho_n(1 - \rho_n)(\bx_i\transpose\bx_j)^2 + \tau^2\rho_n^3\},
\end{align*}
which implies that
\begin{align*}
&\frac{1}{n\rho_n}\sum_{j = 1}^n\expect|[\bE]_{ij}|^3\bx_j\transpose\bDelta_n^{-1}\bSigma_{ni}^{-1}\bDelta_{n}^{-1}\bx_j\\
&\quad = \mathrm{tr}\left\{\frac{1}{n\rho_n}\sum_{j = 1}^n\expect|[\bE]_{ij}|^3\bSigma_{ni}^{-1/2}\bDelta_{n}^{-1}\bx_j\bx_j\transpose\bDelta_n^{-1}\bSigma_{ni}^{-1/2}
\right\}\\
&\quad\lesssim_\tau(\|\bX\|_{2\to\infty}^2\vee1) \mathrm{tr}\left\{\frac{1}{n\rho_n}\sum_{j = 1}^n\{\rho_n(1 - \rho_n)(\bx_i\transpose\bx_j)^2 + \tau^2\rho_n^3\}\bSigma_{ni}^{-1/2}\bDelta_{n}^{-1}\bx_j\bx_j\transpose\bDelta_n^{-1}\bSigma_{ni}^{-1/2}
\right\}\\
&\quad = d(\|\bX\|_{2\to\infty}^2\vee1)\leq \frac{\|\bX\|_{2\to\infty}^4\vee1}{\lambda_d(\bDelta_n)}\leq \frac{\|\bX\|_{2\to\infty}^5\vee1}{\lambda_d(\bDelta_n)^{3/2}},
\end{align*}
and similarly,
\[
\frac{1}{n\rho_n}\sum_{j = 1}^n\expect|[\bE]_{ij}|^3\bx_j\transpose\bDelta_n^{-3/2}\bGamma_{ni}^{-1}\bDelta_{n}^{-3/2}\bx_j\lesssim_\tau\frac{\|\bX\|_{2\to\infty}^4\vee1}{\lambda_d(\bDelta_n)}.
\]
We thus conclude from Theorem \ref{thm:ASE_Berry_Esseen} that
\begin{align*}
&\sup_{A\in\calA}\left|\prob\left\{\sqrt{n}\bSigma_{ni}^{-1/2}(\bW\transpose\widetilde\bx_i - \rho_n^{1/2}\bx_{i})\in A\right\} - \prob\left(\bz\in A\right)\right|\\
&\quad \lesssim_\tau \left\{\frac{\|\bX\|_{2\to\infty}^5\vee 1}{\lambda_d(\bDelta_n)}\right\}\frac{d^{1/2}\|\bSigma_{ni}^{-1/2}\|_2\|\bX\|_{2\to\infty}^2}{(n\rho_n)^{1/2}\lambda_d(\bDelta_{n})^{3/2}}
\\
&\quad\quad\times
\max\left\{\frac{(\|\bX\|_{2\to\infty}^2\vee1)(\log n\rho_n)^{1/2}}{\lambda_d(\bDelta_n)^2}, \frac{\kappa(\bDelta_n)}{\lambda_d(\bDelta_n)^2}, {\log n\rho_n}\right\},
\end{align*}
and
\begin{align*}
&\sup_{A\in\calA}\left|\prob\left\{n\rho_n^{1/2}\bGamma_{ni}^{-1/2}\bW_{\bX}\transpose(\bW^*[\bU_{\bA }]_{i*} - [\bU_{\bP}]_{i*})\in A\right\} - \prob\left(\bz\in A\right)\right|\\
&\quad\lesssim_\tau \left\{\frac{\|\bX\|_{2\to\infty}^4\vee 1}{\lambda_d(\bDelta_n)}\right\}\frac{d^{1/2}\|\bGamma_{ni}^{-1/2}\|_2\|\bX\|_{2\to\infty}}{(n\rho_n)^{1/2}\lambda_d(\bDelta_n)^{3/2}}
\max\left\{\frac{(\log n\rho_n)^{1/2}}{\lambda_d(\bDelta_n)}, \frac{1}{\lambda_d(\bDelta_n)^2}, {\log n\rho_n} \right\}.
\end{align*}
The proof is thus completed.

\section{Proofs for Section \ref{sub:RDPG}}
\label{sec:proofs_for_section_RDPG}

\subsection{A sharp concentration inequality}
\label{sub:technical_preparations_RDPG}

The key technical challenge for the application of Theorem \ref{thm:AFWZ2020AoS_corr2.1} lies in finding the function $\varphi(\cdot)$ satisfying condition A4. In the context of a two-block stochastic block model, the authors of \cite{10.1214/19-AOS1854} showed in Lemma 7 there that $\varphi(x)\propto [\max\{1, \log (1/x)\}]^{-1}$. 
Lemma \ref{lemma:Sharpened_concentration_EW} below is a generalization of Lemma 7 in \cite{10.1214/19-AOS1854} to general dimension $d$. Note that it does not follow from the vector Bernstein's inequality (Lemma \ref{lemma:Bernstein_inequality_vector}) but provides a sharper control of the sum of vector-scaled independent centered Bernoulli random variables. 

\begin{lemma}\label{lemma:Sharpened_concentration_EW}
Let $y_i\sim\mathrm{Bernoulli}(p_i)$ independently for all $i = 1,\ldots,n$, and suppose $\bV$ is a deterministic matrix. Let $\bv_i = \bV\transpose{}\be_i$, $i\in [n]$ and $\rho = \max_{i\in [n]}p_i$. Then for any $\alpha > 0$, 
\begin{align*}
\prob\left\{\left\|\sum_{i = 1}^n(y_i - p_i)\bv_i\right\|_2 > \frac{(2 + \alpha)n\rho\|\bV\|_{2\to\infty}}{\Log (\sqrt{n}\|\bV\|_{2\to\infty}/\|\bV\|_{\mathrm{F}})}\right\}\leq 2(d + 1)e^{-\alpha n\rho},
\end{align*}
where $\Log(x):=\max\{1,\log x\}$. 
\end{lemma}
\begin{proof}[\bf Proof of Lemma \ref{lemma:Sharpened_concentration_EW}]
The proof is a non-trivial generalization of Lemma 7 in \cite{10.1214/19-AOS1854}. 
We follow the ``symmetric dilation'' trick \cite{10.1214/19-AOS1854,paulsen2002completely} applied in the proof of Lemma \ref{lemma:Bernstein_concentration_EW_subGaussian} together with a sharp control of the moment generating function of $y_i$, which is motivated by \cite{10.1214/19-AOS1854,lei2019unified}. Without loss of generality, we may assume that $\|\bV\|_{2\to\infty} = 1$, since the event of interest is invariant to rescaling of $\bV$. Let
\[
\bT(\bv_i) = \begin{bmatrix*}
\zero_{d\times d} & \bv_i\\
\bv_i\transpose{} & 0
\end{bmatrix*},\quad
\bZ_i = (y_i - p_i)\bT(\bv_i),\quad i = 1,2,\ldots,n,
\]
and let $\bS_n = \sum_{i = 1}^n\bZ_i$. Clearly, $\|\bS_n\|_2 = \max\{\lambda_{\max}(\bS_n), \lambda_{\max}(-\bS_n)\}$ and $-\bS_n = \sum_{i = 1}^n(-\bZ_i)$. 
Observe that the spectral decomposition of $\bT(\bv_i)$ is given by
\[
\bT(\bv_i)
 = \bQ_i\begin{bmatrix*}
 \|\bv_i\|_2 & \\ & -\|\bv_i\|_2
 \end{bmatrix*}\bQ_i\transpose + 0\times\bQ_{i\perp}\bQ_{i\perp}\transpose,
\]
where 
\[
\bQ_i = \frac{1}{\sqrt{2}}\begin{bmatrix*}
 \frac{\bv_i}{\|\bv_i\|_2} & \frac{\bv_i}{\|\bv_i\|_2}\\
 1 & -1
 \end{bmatrix*}
\]
and $\bQ_{i\perp}\in\mathbb{O}(d + 1, d - 1)$ is the orthogonal complement matrix of $\bQ_i$. 
Then we use the above spectral decomposition to compute the matrix exponentials
\begin{align*}
\expect e^{\theta \bZ_i}& = p_i\exp\{(1 - p_i)\theta \bT(\bv_i)\} + (1 - p_i)\exp\{-p_i\theta\bT(\bv_i)\}\\
& = 
 \bQ_i\begin{bmatrix*}
 p_ie^{(1 - p_i)\theta\|\bv_i\|_2} + (1 - p_i)e^{-p_i\theta\|\bv_i\|_2} & \\
  & p_ie^{-(1 - p_i)\theta\|\bv_i\|_2} + (1 - p_i)e^{p_i\theta\|\bv_i\|_2}
 \end{bmatrix*}
 \bQ_i\transpose{}\\
 &\quad + \bQ_{i\perp}\bQ_{i\perp}\transpose
\end{align*}
and
\begin{align*}
\expect e^{\theta(-\bZ_i)}& = p_i\exp\{-(1 - p_i)\theta \bT(\bv_i)\} + (1 - p_i)\exp\{p_i\theta\bT(\bv_i)\}\\
& = 
\bQ_i
\begin{bmatrix*}
 p_ie^{-(1 - p_i)\theta\|\bv_i\|_2} + (1 - p_i)e^{p_i\theta\|\bv_i\|_2} & \\
  & p_ie^{(1 - p_i)\theta\|\bv_i\|_2} + (1 - p_i)e^{-p_i\theta\|\bv_i\|_2}
 \end{bmatrix*}
 \bQ_i\transpose{}\\
 &\quad + \bQ_{i\perp}\bQ_{i\perp}\transpose.
\end{align*}
Observe the following two basic inequalities: $1 + x\leq e^x$ for $x > -1$ and $e^x\leq 1 + x + e^rx^2/2$ for $|x|\leq r$. We then obtain
\begin{align*}
 &p_ie^{(1 - p_i)\theta\|\bv_i\|_2} + (1 - p_i)e^{-p_i\theta\|\bv_i\|_2}\\
 &\quad = \{(1 - p_i) + p_ie^{\theta\|\bv_i\|_2}\}e^{-p_i\theta\|\bv_i\|_2}
 \leq \exp\{p_i(e^{\theta\|\bv_i\|_2} - 1) -p_i\theta\|\bv_i\|_2\}\\
 &\quad\leq \exp\left\{p_i\theta\|\bv_i\|_2 + p_i\frac{e^{\theta\|\bV\|_{2\to\infty}}}{2}\theta^2\|\bv_i\|_2^2 -p_i\theta\|\bv_i\|_2\right\}
  = \exp\left\{\frac{e^{\theta\|\bV\|_{2\to\infty}}}{2}\theta^2\|\bv_i\|_2^2p_i\right\},\\
 & p_ie^{-(1 - p_i)\theta\|\bv_i\|_2} + (1 - p_i)e^{p_i\theta\|\bv_i\|_2}\\
 &\quad = \{(1 - p_i) + p_ie^{-\theta\|\bv_i\|_2}\}e^{p_i\theta\|\bv_i\|_2}
 \leq \exp\{p_i(e^{-\theta\|\bv_i\|_2} - 1) + p_i\theta\|\bv_i\|_2\}\\
 &\quad\leq \exp\left\{-p_i\theta\|\bv_i\|_2 + p_i\frac{e^{\theta\|\bV\|_{2\to\infty}}}{2}\theta^2\|\bv_i\|_2^2 + p_i\theta\|\bv_i\|_2\right\}
  = \exp\left\{\frac{e^{\theta\|\bV\|_{2\to\infty}}}{2}\theta^2\|\bv_i\|_2^2p_i\right\}
\end{align*}
for any $\theta > 0$. Namely, 
\begin{align*}
\expect e^{\theta \bZ_i}
&\preceq \exp\left\{\frac{e^{\theta\|\bV\|_{2\to\infty}}}{2}\theta^2\|\bv_i\|_2^2p_i\right\}
\bQ_i\bQ_i\transpose
  + \bQ_{i\perp}\bQ_{i\perp}\transpose
 \\&
 = 
 \exp\left(\frac{1}{2}\theta^2e^{\theta}\|\bv_i\|_2^2p_i\right)
 \bQ_i\bQ_i\transpose
  + \bQ_{i\perp}\bQ_{i\perp}\transpose\\
 &= \exp\{g(\theta)\bM_i\},
\end{align*}
where
\[
g(\theta) = \frac{1}{2}\theta^2e^{\theta },\quad \bM_i = (p_i\|\bv_i\|_2^2)\bQ_i\bQ_i\transpose
.
\]
Similarly, we also have $\expect e^{\theta (-\bZ_i)}\preceq \exp\{g(\theta)\bM_i\}$. 
Now we compute the scale parameter
\begin{align*}
\rho = \lambda_{\max}\left(\sum_{i = 1}^n\bM_i\right)\leq \sum_{i = 1}^n\|\bM_i\|_2\leq \rho\|\bV\|_{\mathrm{F}}^2.
\end{align*}
Since $\bM_i\succeq \zero_{d\times d}$, we also see that $\rho > 0$. Now applying Lemma \ref{lemma:Matrix_concentration} yields
\begin{align*}
\prob\left(\|\bS_n\|_2 > t\right)
&\leq \prob\left\{\lambda_{\max}(\bS_n) > t\right\} + \prob\left\{\lambda_{\max}(-\bS_n) > t\right\}\\
&\leq 2(d + 1)\exp\left(-\theta t + \frac{1}{2}\rho\theta^2e^{\theta}\|\bV\|_{\mathrm{F}}^2\right)
\end{align*}
for any $\theta > 0$ and $t \in\mathbb{R}$. 
Set $\theta = \Log(\sqrt{n}/\|\bV\|_{\mathrm{F}})$. Since $\|\bV\|_{\mathrm{F}}\leq \sqrt{n}\|\bV\|_{2\to\infty} = \sqrt{n}$, we see that $\log(\sqrt{n}/\|\bV\|_{\mathrm{F}}) > 0$, and hence, $\theta \leq 1 + \log(\sqrt{n}/\|\bV\|_{\mathrm{F}})$. It follows that
\[
\frac{\rho\theta^2}{2}e^{\theta}\|\bV\|_{\mathrm{F}}^2\leq \frac{\rho\theta^2}{2}e\sqrt{n}\|\bV\|_{\mathrm{F}} = \frac{e\rho n}{2}\frac{\|\bV\|_{\mathrm{F}}}{\sqrt{n}}\left\{\Log\left(\frac{\sqrt{n}}{\|\bV\|_{\mathrm{F}}}\right)\right\}^2\leq \frac{en\rho}{2},
\]
where we have applied the basic inequality $\Log x\leq \sqrt{x}$ for $x\geq 1$. With $t = \{\Log (\sqrt{n}/\|\bV\|_{\mathrm{F}})\}^{-1}(2 + \alpha)n\rho$, we then obtain
\begin{align*}
\prob\left\{\|\bS_n\|_2 > \frac{(2 + \alpha)n\rho}{\Log (\sqrt{n}/\|\bV\|_{\mathrm{F}})}\right\}\leq d\exp\left\{
-(2 + \alpha)n\rho + \frac{en\rho}{2}
\right\}\leq 2(d + 1)e^{-\alpha n\rho}.
\end{align*}
The proof is thus completed. 
\end{proof}

\subsection{Proof of Theorem \ref{thm:ASE_Berry_Esseen_RDPG}}
\label{sub:proof_of_ASE_Berry_Esseen_RDPG}

We first present two useful results for random graph models. 
\begin{result}[Spectral norm concentration for random graphs]
\label{result:spectral_norm_concentration}
If $\bA\sim\mathrm{RDPG}(\rho_n^{1/2}\bX)$ and the conditions of Theorem \ref{thm:ASE_Berry_Esseen_RDPG} hold, then for any $c > 0$, there exists some constant $K_c > 0$ only depending on $c$, such that $\|\bA - \bP\|_2 \leq K_c(n\rho_n)^{1/2}$ with probability at least $1 - n^{-c}$. This follows exactly from Theorem 5.2 in \cite{lei2015}.
\end{result}

\begin{result}[Concentration bound for $\|\bE\|_\infty$ for random graphs]
\label{result:Infinity_norm_concentration}
Suppose $\bA\sim\mathrm{RDPG}(\rho_n^{1/2}\bX)$ and the conditions of Theorem \ref{thm:ASE_Berry_Esseen_RDPG} hold. For any $c > 0$, there exists some constant $K_c > 0$, such that with probability at least $1 - 2n^{-c}$, $\|\bE\|_\infty\leq K_c n\rho_n$. This is a consequence of Bernstein's inequality. To see this, we first observe that $||[\bE]_{ij}| - \expect|[\bE]_{ij}|\leq 1$, $\expect|[\bE]_{ij}| = 2\rho_n\bx_{i}\transpose{}\bx_{j}(1 - \rho_n\bx_{i}\transpose{}\bx_{j})\leq 2\rho_n$, and 
\begin{align*}
&\sum_{j = 1}^n\expect(|[\bE]_{ij}| - \expect|[\bE]_{ij}|)^2
\leq \sum_{j = 1}^n\expect(A_{ij} - \rho_n\bx_{i}\transpose{}\bx_{j})^2\leq n\rho_n.
\end{align*}
Then for any $C > 0$, an application of Bernstein's inequality yields
\begin{align*}
\prob\left\{\sum_{j = 1}^n|[\bE]_{ij}| > (C + 2)n\rho_n\right\}
& \leq \prob\left\{\sum_{j = 1}^n(|[\bE]_{ij}| - \expect|[\bE]_{ij}|) > Cn\rho_n\right\}\leq  2\exp\left(-\frac{3C^2n\rho_n}{6 + 4C}\right).
\end{align*}
The constant $C$ can be selected such that $3C^2n\rho_n/(6 + 4C)\geq (c + 1)\log n$. Now taking $K_c = C + 2$ and applying a union bound over $i\in [n]$ yields that $\|\bE\|_\infty \leq K_cn\rho_n$ with probability at least $1 - 2n^{-c}$. 
\end{result}

To prove Theorem \ref{thm:ASE_Berry_Esseen_RDPG}, we apply Theorem \ref{thm:ASE_Berry_Esseen} by first verifying Assumptions \ref{assumption:incoherence}-\ref{assumption:spectral_norm_concentration}. By the definition of random dot product graphs, Assumption \ref{assumption:incoherence} automatically holds because $\|\bX\|_{2\to\infty}\leq 1$. Assumption \ref{assumption:sparsity} also holds automatically by the conditions of Theorem \ref{thm:ASE_Berry_Esseen_RDPG}. Assumption \ref{assumption:distribution} also holds because one can set $[\bE_2]_{ij} = 0$ and $[\bE_1]_{ij} = A_{ij} - \rho_n\bx_i\transpose\bx_j$. It remains to verify Assumptions \ref{assumption:rowwise_concentration} and \ref{assumption:spectral_norm_concentration}. Let $c \geq 1$ be any fixed constant. By Result \ref{result:spectral_norm_concentration}, there exists a constant $K_c\geq 1$ that depends on $c > 0$, such that $\prob\{\|\bE\|_2\leq K_c(n\rho_n)^{1/2}\}\geq 1 - n^{-c}$. Set $\varphi(x) = (2 + \beta_c)\{\Log(1 / x)\}^{-1}\lambda_d(\bDelta_n)^{-1}$ for a constant $\beta_c > 0$ such that $\beta_cn\rho_n\geq (c + 2)\log n$. Then with
\[
\gamma = \frac{\max\{3K_c, \|\bX\|_{2\to\infty}^2\}}{(n\rho_n)^{1/2}\lambda_d(\bDelta_n)} = \frac{3K_c}{(n\rho_n)^{1/2}\lambda_d(\bDelta_n)},
\]
we immediately see that
\[
32\kappa(\bDelta_n)\{\gamma, \varphi(\gamma)\}\lesssim_c\frac{\kappa(\bDelta_n)}{\lambda_d(\bDelta_n)}\max\left\{
\frac{1}{(n\rho_n)^{1/2}}, \frac{1}{\log(n\rho_n\lambda_d(\bDelta_n)^2)}
\right\}\to 0
\]
by the condition of Theorem \ref{thm:ASE_Berry_Esseen_RDPG}. This shows that Assumption \ref{assumption:spectral_norm_concentration} holds with $\zeta = c\geq 1$ and $c_0 = 1$. It remains to show that Assumption \ref{assumption:rowwise_concentration} holds with the previously selected $\varphi(\cdot)$ function. By Lemma \ref{lemma:Sharpened_concentration_EW}, for any deterministic $\bV\in\mathbb{R}^{n\times d}$, we have
\begin{align*}
&\prob\left\{\|\be_i\transpose\bE\bV\|_2\leq n\rho_n\lambda_d(\bDelta_n)\|\bV\|_{2\to\infty}\varphi\left(\frac{\|\bV\|_{\mathrm{F}}}{\sqrt{n}\|\bV\|_{2\to\infty}}\right)\right\}\\
&\quad = \prob\left\{\|\be_i\transpose\bE\bV\|_2\leq \frac{n\rho_n\lambda_d(\bDelta_n)(2 + \beta_c)\|\bV\|_{2\to\infty}}{\lambda_d(\bDelta_n)\Log({\sqrt{n}\|\bV\|_{2\to\infty}}/{\|\bV\|_{\mathrm{F}}})}\right\}\\
&\quad\geq 1 - 2d\exp(-\beta_cn\rho_n)\geq 1 - c_0n^{-(1 + \xi)},
\end{align*}
where $\xi = c$ and $c_0 = 2$. To show that the same concentration bound also holds for $\|\be_i\transpose\bE^{(m)}\bV\|_2$, we simply observe that $[\bE^{(m)}]_{im}$ can be viewed as a centered Bernoulli random variable whose success probability is zero. Then applying Lemma \ref{lemma:Sharpened_concentration_EW} leads to that
\[
\prob\left\{\|\be_i\transpose\bE^{(m)}\bV\|_2\leq n\rho_n\lambda_d(\bDelta_n)\|\bV\|_{2\to\infty}\varphi\left(\frac{\|\bV\|_{\mathrm{F}}}{\sqrt{n}\|\bV\|_{2\to\infty}}\right)\right\}\geq 1 - c_0n^{-(1 + \xi)},
\]
where $c_0 = 2$ and $\xi = c$. To finish the proof, we observe that $\expect|[\bE]_{ij}|^3\leq \expect[\bE]_{ij}^2 = \rho_n\bx_i\transpose\bx_j(1 - \rho_n\bx_i\transpose\bx_j)$, implying that
\begin{align*}
&\frac{1}{n\rho_n}\sum_{j = 1}^n\expect|[\bE]_{ij}|^3\bx_j\transpose\bDelta_n^{-1}\bSigma_{ni}^{-1}\bDelta_{n}^{-1}\bx_j\leq \mathrm{tr}(\eye_d) = d\leq \frac{\|\bX\|_{2\to\infty}^2}{\lambda_d(\bDelta_n)}\leq \frac{1}{\lambda_d(\bDelta_n)},\\
&\frac{1}{n\rho_n}\sum_{j = 1}^n\expect|[\bE]_{ij}|^3\bx_j\transpose\bDelta_n^{-3/2}\bGamma_{ni}^{-1}\bDelta_{n}^{-3/2}\bx_j\leq \mathrm{tr}(\eye_d) = d\leq \frac{\|\bX\|_{2\to\infty}^2}{\lambda_d(\bDelta_n)}\leq \frac{1}{\lambda_d(\bDelta_n)}.
\end{align*}
Also, observe that $\varphi(1)\lesssim \lambda_d(\bDelta_n)^{-1}$ and $\chi = \varphi(1) + (\|\bX\|_{2\to\infty}^2\vee1)/\lambda_d(\bDelta_n)\lesssim \lambda_d(\bDelta_n)^{-1}$. Then by Theorem \ref{thm:ASE_Berry_Esseen}, we have, for each fixed index $i\in[n]$ and for any sufficiently large $n$,
\begin{align*}
&\begin{aligned}
&\sup_{A\in\calA}\left|\prob\left\{\sqrt{n}\bSigma_n(\bx_{i})^{-1/2}(\bW\transpose\widetilde\bx_i  - \rho_n^{1/2}\bx_{i})\in A\right\} - \prob\left(\bz\in A\right)\right|\\
&\quad\lesssim 
\frac{d^{1/2}\|\bSigma_{n}(\bx_i)^{-1/2}\|_2}{(n\rho_n)^{1/2}\lambda_d(\bDelta_n)^{5/2}}\max\left\{\frac{(\log n\rho_n)^{1/2}}{\lambda_d(\bDelta_n)}, \frac{\kappa(\bDelta_n)}{\lambda_d(\bDelta_n)^2}, {\log n\rho_n} \right\},
\end{aligned}\\
&\begin{aligned}
&\sup_{A\in\calA}\left|\prob\left\{n\rho_n^{1/2}\bGamma_n(\bx_{i})^{-1/2}\bW_\bX\transpose(\bW^*[\bU_\bA]_{i*} - [\bU_\bP]_{i*})\in A\right\} - \prob\left(\bz\in A\right)\right|\\
&\quad\lesssim \frac{d^{1/2}\|\bGamma_n(\bx_i)^{-1/2}\|_2}{(n\rho_n)^{1/2}\lambda_d(\bDelta_n)^{5/2}}\max\left\{\frac{(\log n\rho_n)^{1/2}}{\lambda_d(\bDelta_n)}, \frac{1}{\lambda_d(\bDelta_n)^2}, {\log n\rho_n}\right\},
\end{aligned}
\end{align*}
where $\calA$ is the collection of all convex measurable sets in $\mathbb{R}^d$, and $\bz\sim\mathrm{N}_d(\zero_d, \eye_d)$. The proof is thus completed.

\subsection{Proof of Corollary \ref{corr:Two_to_infinity_norm_eigenvector_bound}} 
\label{sub:proof_of_corollary_corr:two_to_infinity_norm_eigenvector_bound}


From the proof in Section \ref{sub:proof_of_ASE_Berry_Esseen_RDPG}, we see that Assumptions \ref{assumption:incoherence}-\ref{assumption:spectral_norm_concentration} hold with $\xi = \zeta = c$ and $\varphi(x)\propto \{\Log(1/x)\}^{-1}\lambda_d(\bDelta_n)^{-1}$. 
By Theorem \ref{thm:eigenvector_deviation} with $t = (c + 1)\log n$ and a union bound over $m\in [n]$ for sufficiently large $n$, we have
\begin{align*}
\|\bU_\bA - \bA\bU_\bP\bS_\bP^{-1}\bW^*\|_{2\to\infty}
&\lesssim_c \frac{\|\bU_\bP\|_{2\to\infty}}{n \rho_n \lambda_d(\bDelta_n)^2}\max\left\{\frac{(\log n)^{1/2}}{\lambda_d(\bDelta_n)}, \frac{1}{ \lambda_d(\bDelta_n)^2},{\log n} \right\},\\
\left\|\widetilde{\bX}\bW - \frac{\bA\bX(\bX\transpose\bX)^{-1}}{\rho_n^{1/2}}\right\|_{2\to\infty}
&\lesssim_c \frac{\|\bU_\bP\|_{2\to\infty}}{(n\rho_n)^{1/2}\lambda_d(\bDelta_n)^2}
\max\left\{\frac{(\log n)^{1/2}}{\lambda_d(\bDelta_n)}, \frac{\kappa(\bDelta_n)}{ \lambda_d(\bDelta_n)^2},{\log n} \right\}
\end{align*}
with probability at least $1 - c_0n^{-c}$. 
Also, we observe that
\begin{align*}
\|\bU_\bA - \bU_\bP\bW^*\|_{2\to\infty}
&\leq \|\bU_\bA - \bA\bU_\bP\bS_\bP^{-1}\bW^*\|_{2\to\infty} + \|(\bA\bU_\bP\bS_\bP^{-1} - \bU_\bP)\bW^*\|_{2\to\infty}\\
&= \|\bU_\bA - \bA\bU_\bP\bS_\bP^{-1}\bW^*\|_{2\to\infty} + \|\bE\bU_\bP\bS_\bP^{-1}\|_{2\to\infty}\\
&\leq \|\bU_\bA - \bA\bU_\bP\bS_\bP^{-1}\bW^*\|_{2\to\infty} + \frac{\|\bE\bU_\bP\|_{2\to\infty}}{n\rho_n\lambda_d(\bDelta_n)}. 
\end{align*}
By Lemma \ref{lemma:Bernstein_concentration_EW} and a union bound, for any $a > 0$, we have
\begin{align*}
&\prob\left\{\|\bE\bU_\bP\|_{2\to\infty} > 3a\log n\|\bU_\bP\|_{2\to\infty} + (6a\rho_n\log n)^{1/2}\|\bU_\bP\|_{\mathrm{F}}\right\}\\
&\quad\leq \sum_{m = 1}^n\prob\left\{\mathrel{\Big\|} 
\sum_{j = 1}^n(A_{mj} - \rho_n\bx_m\transpose{}\bx_j)(\bU_\bP\transpose{}\be_j)
\mathrel{\Big\|_2} > 3a\log n\|\bU_\bP\|_{2\to\infty} + (6a\rho_n\log n)^{1/2}\|\bU_\bP\|_{\mathrm{F}}\right\}\\
&\quad\leq 28ne^{-3a\log n} = 28n^{-(3a - 1)}.
\end{align*}
Now we can set $a = (c + 1)/3$ to obtain that
\begin{align*}
\frac{\|\bE\bU_\bP\|_{2\to\infty}}{n\rho_n\lambda_d(\bDelta_n)}
&\lesssim_c \frac{(\log n)\|\bU_\bP\|_{2\to\infty}}{n\rho_n\lambda_d(\bDelta_n)} + \frac{(6\rho_n\log n)^{1/2}\|\bU_\bP\|_{\mathrm{F}}}{n\rho_n\lambda_d(\bDelta_n)}\\
&\lesssim \frac{(n\rho_n\log n)^{1/2}\|\bU_\bP\|_{2\to\infty}}{n\rho_n\lambda_d(\bDelta_n)} + \frac{(n\rho_n\log n)^{1/2}\|\bU_\bP\|_{2\to\infty}}{n\rho_n\lambda_d(\bDelta_n)}\\
& = \frac{2(\log n)^{1/2}\|\bU_\bP\|_{2\to\infty}}{(n\rho_n)^{1/2}\lambda_d(\bDelta_n)} 
\end{align*}
with probability at least $1 - c_0n^{-3c}$. Then by the concentration bound for $\|\bU_\bA - \bA\bU_\bP\bS_\bP^{-1}\bW^*\|_{2\to\infty}$, we have
\begin{align*}
\|\bU_\bA - \bU_\bP\bW^*\|_{2\to\infty}
&\leq \|\bU_\bA - \bA\bU_\bP\bS_\bP^{-1}\bW^*\|_{2\to\infty} + \frac{\|\bE\bU_\bP\|_{2\to\infty}}{n\rho_n\lambda_d(\bDelta_n)}\\
&\lesssim_c \|\bU_\bA - \bA\bU_\bP\bS_\bP^{-1}\bW^*\|_{2\to\infty} + \frac{\|\bE\bU_\bP\|_{2\to\infty}}{n\rho_n\lambda_d(\bDelta_n)}
 + \frac{(\log n)^{1/2}\|\bU_\bP\|_{2\to\infty}}{(n\rho_n)^{1/2}\lambda_d(\bDelta_n)}
\end{align*}
with probability at least $1 - c_0n^{-c}$ for sufficiently large $n$. This completes the proof of the concentration bound for the unscaled eigenvectors $\|\bU_\bA - \bU_\bP\bW^*\|_{2\to\infty}$. 

\vspace*{1ex}\noindent
For the fourth assertion, we recall the decomposition \eqref{eqn:keystone_decomposition}
\begin{align*}
\widetilde\bX\bW - \rho_n^{1/2}\bX = \rho_n^{-1/2}(\bA - \bP)\bX(\bX\transpose{}\bX)^{-1} + \{\widetilde\bX\bW - \rho_n^{-1/2}\bA\bX(\bX\transpose{}\bX)^{-1}\}.
\end{align*}
For the first term, we apply Lemma \ref{lemma:Bernstein_concentration_EW} with $t = (c\log n)^{1/2}$ and a union bound over $m\in [n]$ to obtain
\begin{align*}
\|\rho_n^{-1/2}(\bA - \bP)\bX(\bX\transpose{}\bX)^{-1}\|_{2\to\infty}
& = \|\bE\bU_\bP\bS_\bP^{-1/2}\|_{2\to\infty} \leq \max_{m\in [n]}\|\be_m\transpose{}\bE\bU_\bP\|_2\|\bS_\bP^{-1/2}\|_2\\
&
\lesssim_c \frac{(n\rho_n\log n)^{1/2}\|\bU_\bP\|_{2\to\infty}}{(n\rho_n)^{1/2}\lambda_d(\bDelta_n)^{1/2}} = \frac{(\log n)^{1/2}}{\lambda_d(\bDelta_n)^{1/2}}\|\bU_\bP\|_{2\to\infty}
\end{align*}
with probability at least $1 - c_0n^{-2c}$. The proof is thus completed. 


\section{Proofs for Section \ref{sub:entrywise_limit_theorem_for_the_one_step_estimator}} 
\label{sec:proofs_for_section_sec:entrywise_limit_theorem_for_the_one_step_estimator}

\subsection{Outline of the proof of Theorem \ref{thm:asymptotic_normality_OS}} 
\label{sub:outline_of_the_proof_of_theorem_thm:asymptotic_normality_os}

We first present the outline the proof of Theorem \ref{thm:asymptotic_normality_OS}, which is a non-trivial extension of \cite{xie2019efficient} to sparse graphs. Recall that the $i$th row of the one-step refinement
$\widehat\bx_i = \widetilde\bx_i + \rho_n^{1/2}\calI_i(\rho_n^{-1/2}\widetilde\bX)^{-1}\nabla_{\bx_i}\ell_\bA(\rho_n^{-1/2}\widetilde\bX)$. 
Then a simple computation leads to the following decomposition of $\bW\transpose{}\widehat{\bx}_i - \rho_n^{1/2}\bx_i$:
\begin{align}
\bW\transpose{}\widehat{\bx}_i - \rho_n^{1/2}\bx_i
& = \bW\transpose{}\widetilde\bx_i - \rho_n^{1/2}\bx_i + \rho_n^{1/2}\calI_i(\rho_n^{-1/2}\widetilde\bX\bW)^{-1}\nabla_{\bx_i}\ell_\bA(\rho_n^{-1/2}\widetilde\bX\bW)\nonumber\\
\label{eqn:OSE_decomposition_term_I}
& = \bW\transpose{}\widetilde\bx_i - \rho_n^{1/2}\bx_i + \rho_n^{1/2}\calI_i(\bX)^{-1}\left\{\nabla_{\bx_i}\ell_\bA(\rho_n^{-1/2}\widetilde\bX\bW) - \nabla_{\bx_i}\ell_\bA(\bX)\right\}\\
\label{eqn:OSE_decomposition_term_II}
&\quad + \rho_n^{1/2}\left\{\calI_i(\rho_n^{-1/2}\widetilde\bX\bW)^{-1} - \calI_i(\bX)^{-1}\right\}\nabla_{\bx_i}\ell_\bA(\rho_n^{-1/2}\widetilde\bX\bW)\\
\label{eqn:OSE_decomposition_term_III}
&\quad + \rho_n^{1/2}\calI_i(\bX)^{-1}\nabla_{\bx_i}\ell_\bA(\bX).
\end{align}
Since $\|\rho_n^{-1/2}\widetilde\bX\bW - \bX\|_{2\to\infty} = o_{\prob}(1)$ by Corollary \ref{corr:Two_to_infinity_norm_eigenvector_bound}, it is expected that $\calI_i(\rho_n^{-1/2}\widetilde\bX\bW)^{-1}\approx \calI_i(\bX)^{-1}$ by the continuous mapping theorem, and hence, term \eqref{eqn:OSE_decomposition_term_II} should be comparatively small. Term \eqref{eqn:OSE_decomposition_term_III} corresponds to the first term on the right-hand side of \eqref{eqn:OSE_main_decomposition} and is a sum of independent mean-zero random variables. The non-trivial part is the analysis of the term in line \eqref{eqn:OSE_decomposition_term_I}. The intuition is that
\[
\nabla_{\bx_i}\ell_\bA(\rho_n^{-1/2}\widetilde\bX\bW) - \nabla_{\bx_i}\ell_\bA(\bX)\approx - \calI_i(\bX)(\rho_n^{-1/2}\bW\transpose{}\widetilde{\bx}_i - \bx_i)
\]
by a first-order Taylor approximation of $\nabla_{\bx_i}\ell_\bA$. However, making the above approximation precise is technically involved because $\rho_n^{-1/2}\bW\transpose{}\widetilde{\bx}_j\neq \bx_j$ for all $j\neq i$. In \cite{xie2019efficient}, the authors assumed $n\rho_n^5 = \omega((\log n)^2)$ and their proof technique is no longer applicable when $n\rho_n = \Omega(\log n)$. In the present work, we overcome this difficulty by taking advantage of the decoupling strategy developed in Section \ref{sec:entrywise_limit_theorem_for_the_eigenvectors}, together with a delicate second-order Taylor approximation analysis. 

Now for any $\eps > 0$, denote
\[
\calX_n(\eps) := \left\{\bV = [\bv_1,\ldots,\bv_n]\transpose{}\in\mathbb{R}^{n\times d}:\bv_i\transpose{}\bv_j\in [\eps, 1 - \eps]\quad\text{for all }i,j\in [n]\right\}
\]
and for each fixed index $i\in [n]$, define the matrix-valued function $\bH_i:\calX_n(\delta/2)\to \mathbb{R}^{d\times d}$ by
\begin{align}
\label{eqn:Hi_function}
\bH_i(\bV) = \frac{1}{n}\sum_{j = 1}^n\frac{\bv_j\bv_j\transpose{}}{\bv_i\transpose{}\bv_j(1 - \rho_n\bv_i\transpose{}\bv_j)}. 
\end{align}
We continue the decomposition of $\bW\transpose\widehat{\bx}_i - \rho_n^{1/2}\bx_{0i}$ mentioned earlier:
\begin{align*}
\bW\transpose{}\widehat\bx_i - \rho_n^{1/2}\bx_i
& = \bW\transpose{}\widetilde\bx_i - \rho_n^{1/2}\bx_i + \rho_n^{1/2}\calI_i(\bX)^{-1}\left\{\nabla_{\bx_i}\ell_\bA(\rho_n^{-1/2}\widetilde\bX\bW) - \nabla_{\bx_i}\ell_\bA(\bX)\right\}\\
&\quad + \rho_n^{1/2}\left\{\calI_i(\rho_n^{-1/2}\widetilde\bX\bW)^{-1} - \calI_i(\bX)^{-1}\right\}\nabla_{\bx_i}\ell_\bA(\rho_n^{-1/2}\widetilde\bX\bW)\\
&\quad + \rho_n^{1/2}\calI_i(\bX)^{-1}\nabla_{\bx_i}\ell_\bA(\bX)\\
& = \frac{1}{n\rho_n^{1/2}}\sum_{j = 1}^n\frac{[\bE]_{ij}\bG_n(\bx_i)^{-1}\bx_j}{\bx_i\transpose{}\bx_j(1 - \rho_n\bx_i\transpose{}\bx_j)}\\
&\quad + \bG_n(\bx_i)^{-1}\{\bG_n(\bx_i)(\bW\transpose{}\widetilde\bx_i - \rho_n^{1/2}\bx_i) + \br_{i1}\} + \bR_{i2}\br_{i1} + \br_{i3},
\end{align*}
where
\begin{align}
\label{eqn:ri1}
&\br_{i1} = \frac{1}{n\sqrt{\rho_n}}\sum_{j = 1}^n\left\{\frac{(A_{ij} - \widetilde\bx_i\transpose\widetilde\bx_j)(\rho_n^{-1/2}\bW\transpose\widetilde\bx_j)}{\rho_n^{-1}\widetilde\bx_{i}\transpose\widetilde\bx_{j}(1 - \widetilde\bx_i\transpose\widetilde\bx_j)} - \frac{(A_{ij} - \rho_n\bx_{i}\transpose\bx_{j})\bx_{j}}{\bx_{i}\transpose\bx_{j}(1 - \rho_n\bx_{i}\transpose\bx_{j})}\right\},\\
\label{eqn:Ri2}
&\bR_{i2} = \bH_i(\rho_n^{-1/2}\widetilde\bX\bW)^{-1} - \bG_n(\bx_{i})^{-1},\\
&\br_{i3}  = \frac{1}{n\sqrt{\rho_n}}\sum_{j = 1}^n\frac{(A_{ij} - \rho_n\bx_{i}\transpose\bx_{j})}{\bx_{i}\transpose\bx_{j}(1 - \rho_n\bx_{i}\transpose\bx_{j})}\bR_{i2}\bx_{j}.\nonumber
\end{align}
The most challenging part is a sharp concentration bound for $\br_{i1}$. We now sketch the argument for bounding $\br_{i1}$. 
For any constant $\eps\in (0, 1/2)$, define $\calX_2(\eps) = \{(\bu,\bv)\in\mathbb{R}^{d}\times\mathbb{R}^d:\|\bu\|_2,\|\bv\|_2\leq1, \eps\leq \bu\transpose{}\bv\leq 1 - \eps\}$. For each $(\bu,\bv)\in\calX_2(\delta/2)$, define the following functions:
\begin{align*}
\bg(\bu, \bv) & = \frac{\bv}{\bu\transpose{}\bv(1 - \rho_n\bu\transpose{}\bv)},\\ 
\bh_{ij}(\bu, \bv) & = \frac{(\bx_{i}\transpose{}\bx_{j} - \bu\transpose{}\bv)\bv}{\bu\transpose{}\bv(1 - \rho_n\bu\transpose{}\bv)},\\
\bphi_{ij}(\bu, \bv) & = (A_{ij} - \rho_n\bx_i\transpose{}\bx_j)\bg(\bu, \bv) + \rho_n\bh_{ij}(\bu, \bv). 
\end{align*}
Applying a first-order Taylor expansion to $\bg$ and $\bh$ yields
\begin{align*}
\bg(\bu, \bv) - \bg(\bx_{i},\bx_{j}) & = -\frac{(1 - 2\rho_n\bx_{i}\transpose{}\bx_{j})\bx_{j}\bx_{j}\transpose{}}{\{\bx_{i}\transpose{}\bx_{j}(1 - \rho_n\bx_{i}\transpose{}\bx_{j})\}^2}(\bu - \bx_{i})\\
&\quad + \left\{
\frac{\eye_d}{\bx_{i}\transpose{}\bx_{j}(1 - \rho_n\bx_{i}\transpose{}\bx_{j})} - 
\frac{(1 - 2\rho_n\bx_{i}\transpose{}\bx_{j})\bx_{j}\bx_{i}\transpose{}}{\{\bx_{i}\transpose{}\bx_{j}(1 - \rho_n\bx_{i}\transpose{}\bx_{j})\}^2}
\right\}(\bv - \bx_{j})
\\&\quad
 + \br_\bg(\bu, \bv, \bx_i, \bx_j),\\
\bh_{ij}(\bu, \bv) - \bh_{ij}(\bx_{i},\bx_{j}) & = -
\frac{\bx_{j}\bx_{j}\transpose{}(\bu - \bx_{i})}{ \bx_{i}\transpose{}\bx_{j}(1 - \rho_n\bx_{i}\transpose{}\bx_{j}) }
 - \frac{\bx_{j}\bx_{i}\transpose{}(\bv - \bx_{j})}{\bx_{i}\transpose{}\bx_{j}(1 - \rho_n\bx_{i}\transpose{}\bx_{j})}
 + \br_{\bh_{ij}}(\bu, \bv, \bx_i, \bx_j),
\end{align*}
where $\br_\bg, \br_{\bh_{ij}}$ are higher-order remainders of $\bg$ and $\bh_{ij}$. 
Then we can write $\br_{i1}$ as
\begin{align}
\br_{i1}& = \frac{1}{n\rho_n^{1/2}}\sum_{j = 1}^n\left\{ \bphi_{ij}(\rho_n^{-1/2}\bW\transpose{}\widetilde\bx_i, \rho_n^{-1/2}\bW\transpose{}\widetilde\bx_j) - \bphi_{ij}(\bx_i, \bx_j)\right\}\nonumber\\
\label{eqn:ri1_analysis_term_I}
& = - \frac{1}{n\rho_n^{1/2}}\sum_{j = 1}^n\frac{\rho_n\bx_j\bx_j\transpose{}}{\bx_i\transpose{}\bx_j(1 - \rho_n\bx_i\transpose{}\bx_j)}(\rho_n^{-1/2}\bW\transpose{}\widetilde{\bx}_i - \bx_i)\\
\label{eqn:ri1_analysis_term_II}
&\quad - \frac{1}{n\rho_n^{1/2}}\sum_{j = 1}^n\frac{\rho_n \bx_{j}\bx_{i}\transpose}{\bx_{i}\transpose\bx_{j}(1 - \rho_n\bx_{i}\transpose\bx_{j})}
(\rho_n^{-1/2}\bW\transpose{}\widetilde{\bx}_j - \bx_{j})\\
\label{eqn:ri1_analysis_term_III}
& \quad - \frac{1}{n\rho_n^{1/2}}\sum_{j = 1}^n\left[\frac{[\bE]_{ij}(1 - 2\rho_n\bx_{i}\transpose\bx_{j})\bx_{j}\bx_{j}\transpose}{\{\bx_{i}\transpose\bx_{j}(1 - \rho_n\bx_{i}\transpose\bx_{j})\}^2}\right](\rho_n^{-1/2}\bW\transpose\widetilde{\bx}_i - \bx_{i})\\
\label{eqn:ri1_analysis_term_IV}
&\quad + \frac{1}{n\rho_n^{1/2}}\sum_{j = 1}^n
\frac{[\bE]_{ij}\{\bx_{i}\transpose\bx_{j}(1 - \rho_n\bx_{i}\transpose\bx_{j})\eye_d - (1 - 2\rho_n\bx_{i}\transpose\bx_{j})\bx_{j}\bx_{i}\transpose\}}
{\{\bx_{i}\transpose\bx_{j}(1 - \rho_n\bx_{i}\transpose\bx_{j})\}^2}
(\rho_n^{-1/2}\bW\transpose\widetilde{\bx}_j - \bx_{j})\\
\label{eqn:ri1_analysis_term_V}
&\quad + \frac{1}{n\rho_n^{1/2}}\sum_{j = 1}^n[\bE]_{ij}\br_{\bg}
 + \frac{1}{n\rho_n^{1/2}}\sum_{j = 1}^n\rho_n\br_{\bh_{ij}}
,
\end{align}
where we have compressed the notation $\br_{\bg} = \br_\bg
(\rho_n^{-1/2}\bW\transpose{}\widetilde\bx_i, \rho_n^{-1/2}\bW\transpose{}\widetilde\bx_j, \bx_i, \bx_j)$
 and 
 $\br_{\bh_{ij}} = \br_{\bh_{ij}}
 (\rho_n^{-1/2}\bW\transpose{}\widetilde\bx_i, \rho_n^{-1/2}\bW\transpose{}\widetilde\bx_j, \bx_i, \bx_j)$. 
Term \eqref{eqn:ri1_analysis_term_I} is the same as $\bG_n(\bx_i)(\bW\transpose{}\widetilde{\bx}_i - \rho_n^{1/2}\bx_i)$. In what follows, we are going to work on terms \eqref{eqn:ri1_analysis_term_II}, \eqref{eqn:ri1_analysis_term_III}, \eqref{eqn:ri1_analysis_term_IV}, and \eqref{eqn:ri1_analysis_term_V}, respectively, and provide sharp concentration bounds for them. 


\subsection{Some technical preparations} 
\label{sub:some_technical_preparations}

In this section, we make some technical preparations for the proof of Theorem \ref{thm:asymptotic_normality_OS}. The following lemma provides a concentration bound for $\|\widetilde\bX\bW - \rho_n^{-1/2}\bA\bX(\bX\transpose{}\bX)^{-1}\|_{\mathrm{F}}$. 
\begin{lemma}\label{lemma:RX_Frobenius_norm_bound}
Let $\bA\sim\mathrm{RDPG}(\rho_n^{1/2}\bX)$ and assume the conditions of Theorem \ref{thm:ASE_Berry_Esseen_RDPG} hold. Then there exists an absolute constant $c_0 > 0$, such that given any fixed $c > 0$,
for all 
sufficiently large $n$
and 
for all $t > 0$, 
the following event holds with probability at least $1 - c_0n^{-c} - c_0e^{-t}$:
\begin{align*}
\|\widetilde\bX\bW - \rho_n^{1/2}\bA\bX(\bX\transpose{}\bX)\|_{\mathrm{F}}&\lesssim_c 
\frac{1}{(n\rho_n)^{1/2}\lambda_d(\bDelta_n)}\max\left\{t^{1/2}, \frac{\kappa(\bDelta_n)}{\lambda_d(\bDelta_n)}\right\},\\
\|\widetilde\bX\bW - \rho_n^{1/2}\bX\|_{\mathrm{F}}&\lesssim_c \frac{1}{\lambda_d(\bDelta_n)}.
\end{align*}
\end{lemma}

\begin{proof}[\bf Proof of Lemma \ref{lemma:RX_Frobenius_norm_bound}]
We first remark that this lemma does not follow from Theorem \ref{thm:eigenvector_deviation}. Instead, we rely on the following matrix decomposition due to \cite{JMLR:v18:17-448} and \cite{doi:10.1080/10618600.2016.1193505}:
\begin{align*}
\widetilde{\bX} - \bU_\bP\bS_\bP^{1/2}\bW^* 
& = \bE\bU_\bP\bS_\bP^{-1/2}\bW^* - \bU_\bP\bU_\bP\transpose\bE\bU_\bP\bW^*\bS_\bA^{-1/2}\\
&\quad + (\eye - \bU_\bP\bU_\bP\transpose{})\bE(\bU_\bA - \bU_\bP\bW^*)\bS_\bA^{-1/2} + \bU_\bP(\bU_\bP\transpose\bU_\bA - \bW^*)\bS_\bA^{-1/2}\\
&\quad + \bU_\bP(\bW^*\bS_\bA^{1/2} - \bS_\bP^{1/2}\bW^*) - \bE\bU_\bP(\bS_\bP^{-1/2}\bW^* - \bW^* \bS_\bA^{-1/2}).
\end{align*}
Denote $\bE = \bA - \bP$. 
Since $\bU_\bP\bS_\bP^{1/2}\bW^* =  \rho_n^{1/2}\bX\bW\transpose$ and $\bU_\bP\bS_\bP^{-1/2}\bW^* =  \rho_n^{-1/2}\bX(\bX\transpose{}\bX)^{-1}\bW\transpose$, it follows that
\begin{align*}
\widetilde{\bX}\bW - \rho_n^{1/2}\bX
& = \rho_n^{-1/2}\bE\bX(\bX\transpose{}\bX)^{-1} - \bU_\bP\bU_\bP\transpose\bE\bU_\bP\bW^*\bS_\bA^{-1/2}\bW\\
&\quad + (\eye - \bU_\bP\bU_\bP\transpose{})\bE(\bU_\bA - \bU_\bP\bW^*)\bS_\bA^{-1/2}\bW + \bU_\bP(\bU_\bP\transpose\bU_\bA - \bW^*)\bS_\bA^{-1/2}\bW\\
&\quad + \bU_\bP(\bW^*\bS_\bA^{1/2} - \bS_\bP^{1/2}\bW^*)\bW - \bE\bU_\bP(\bS_\bP^{-1/2}\bW^* - \bW^* \bS_\bA^{-1/2})\bW.
\end{align*}
Denote $\bR_\bX = \widetilde\bX\bW - \rho_n^{-1/2}\bA\bX(\bX\transpose{}\bX)^{-1}$. Using Davis-Kahan theorem and the fact that $\bR_\bX = \widetilde\bX\bW - \rho_n^{1/2}\bX - \rho_n^{-1/2}\bE\bX(\bX\transpose{}\bX)^{-1}$, we obtain
\begin{align*}
\|\bR_\bX\|_{\mathrm{F}}
&\leq \sqrt{d}\|\bU_\bP\transpose{}\bE\bU_\bP\|_2\|\bS_\bA^{-1/2}\|_2 + \|\bE\|_2\|\bU_\bA - \bU_\bP\bW^*\|_{\mathrm{F}}\|\bS_\bA^{-1/2}\|_2
\\
&\quad 
 + \sqrt{d}\|\bU_\bP\transpose\bU_\bA - \bW^*\|_2\|\bS_\bA^{-1/2}\|_2+ \|\bW^*\bS_\bA^{1/2} - \bS_\bP^{1/2}\bW^*\|_{\mathrm{F}}\\
&\quad + \|\bE\|_2\|\bW^* \bS_\bA^{-1/2} - \bS_\bP^{-1/2}\bW^*\|_{\mathrm{F}}\\
&\lesssim  \sqrt{d}\|\bU_\bP\transpose\bE\bU_\bP\|_2\|\bS_\bA^{-1/2}\|_2 + \frac{ \sqrt{d}\|\bE\|_2^2}{n\rho_n\lambda_d(\bDelta_n)}\|\bS_\bA^{-1/2}\|_2
 + \frac{\sqrt{d}\|\bE\|_2^2}{(n\rho_n)^2\lambda_d(\bDelta_n)^2}\|\bS_\bA^{-1/2}\|_2
 \\
&\quad 
+ \|\bW^*\bS_\bA^{1/2} - \bS_\bP^{1/2}\bW^*\|_{\mathrm{F}} + \|\bE\|_2\|\bW^* \bS_\bA^{-1/2} - \bS_\bP^{-1/2}\bW^*\|_{\mathrm{F}} .
\end{align*}
By Lemma \ref{lemma:UtransposeEU_bound}, Lemma \ref{lemma:WS_interchange_bound}, Result \ref{result:spectral_norm_concentration}, and Result \ref{result:S_A_concentration}, 
for sufficiently large $n$,
\begin{align*}
&\|\bU_\bP\transpose \bE \bU_\bP\|_{\mathrm{2}}\lesssim_c d^{1/2} + t^{1/2},
\quad\|\bE\|_2\lesssim_c (n\rho_n)^{1/2},\quad\|\bS_\bA^{-1/2}\|_2\lesssim_c \frac{1}{(n\rho_n)^{1/2}\lambda_d(\bDelta_n)^{1/2}}\\
&\|\bW^*\bS_\bA^{1/2} - \bS_\bP^{1/2}\bW^*\|_{\mathrm{F}}\lesssim_c \frac{1}{(n\rho_n)^{1/2}\lambda_d(\bDelta_n)}\max\left\{\frac{\kappa(\bDelta_n)}{\lambda_d(\bDelta_n)}, t^{1/2}\right\},\\
&\|\bW^*\bS_\bA^{-1/2} - \bS_\bP^{-1/2}\bW^*\|_{\mathrm{F}}\lesssim_c \frac{1}{(n\rho_n)^{3/2}\lambda_d(\bDelta_n)^{2}}\max\left\{\frac{\kappa(\bDelta_n)}{\lambda_d(\bDelta_n)}, t^{1/2}\right\}
\end{align*}
with probability at least $1 - c_0n^{-c} - c_0e^{-t}$. Here, we have used the fact that 
\[
d^{1/2}\leq d\leq \lambda_d(\bDelta_n)^{-1}\|\bX\|_{2\to\infty}^2\leq \kappa(\bDelta_n)/\lambda_d(\bDelta_n)
\] 
from Result \ref{result:UP_delocalization}. 
Hence, we conclude that
\begin{align*}
\|\bR_\bX\|_{\mathrm{F}}
&\lesssim_c \frac{d + (dt)^{1/2}}{(n\rho_n)^{1/2}\lambda_d(\bDelta_n)^{1/2}} + \frac{d^{1/2}(n\rho_n)}{(n\rho_n)^{3/2}\lambda_d(\bDelta_n)^{3/2}}
\\
&\quad 
 + \frac{d^{1/2}(n\rho_n)}{(n\rho_n)^{5/2}\lambda_d(\bDelta_n)^{5/2}} + \frac{1}{(n\rho_n)^{1/2}\lambda_d(\bDelta_n)}\max\left\{t^{1/2}, \frac{\kappa(\bDelta_n)}{\lambda_d(\bDelta_n)}\right\}\\
&\quad + \frac{1}{(n\rho_n)\lambda_d(\bDelta_n)^{2}}\max\left\{t^{1/2}, \frac{\kappa(\bDelta_n)}{\lambda_d(\bDelta_n)}\right\}\\
&\lesssim \frac{1}{(n\rho_n)^{1/2}\lambda_d(\bDelta_n)}\max\left\{t^{1/2}, \frac{\kappa(\bDelta_n)}{\lambda_d(\bDelta_n)}\right\}
\end{align*}
with probability at least $1 - c_0n^{-c} - c_0e^{-t}$ for sufficiently large $n$. This completes the proof of the first assertion. The second assertion follows from the fact that
\begin{align*}
\|\widetilde\bX\bW - \rho_n^{1/2}\bX\|_{\mathrm{F}}
&\leq \frac{1}{n\rho_n^{1/2}}\|\bE\|_2\|\bX\|_{\mathrm{F}}\|\bDelta_n^{-1}\|_2 + \|\bR_\bX\|_{\mathrm{F}},
\end{align*}
Result \ref{result:spectral_norm_concentration}, and the assumption that
${\kappa(\bDelta_n)}/\{(n\rho_n)^{1/2}\lambda_d(\bDelta_n)\}\to 0$.
\end{proof}

\begin{lemma}\label{lemma:single_x_i_perturbation_bound}
Suppose $\bA\sim\mathrm{RDPG}(\rho_n^{1/2}\bX)$ and let the conditions of Theorem \ref{thm:ASE_Berry_Esseen_RDPG} hold. Then there exists an absolute constant $c_0 > 0$, such that given any fixed $c > 0$,
for each fixed $i\in [n]$, 
for sufficiently large $n$
and for all $t\geq 1$, $t\lesssim n\rho_n$, 
\begin{align*}
\|\bW\transpose{}\widetilde\bx_i - \rho_n^{1/2}\bx_i\|_2\lesssim_c \frac{\|\bU_\bP\|_{2\to\infty}}{\lambda_d(\bDelta_n)}\max\left\{\frac{t^{1/2}}{\lambda_d(\bDelta_n)^2}, \frac{\kappa(\bDelta_n)}{\lambda_d(\bDelta_n)^2}, t\right\}
\end{align*}
with probability at least $1 - c_0n^{-c} - c_0e^{-t}$. 
\end{lemma}
\begin{proof}[\bf Proof of Lemma \ref{lemma:single_x_i_perturbation_bound}]
By Theorem \ref{thm:eigenvector_deviation} and Lemma \ref{lemma:Bernstein_concentration_EW}, 
for sufficiently large $n$ and
for all $t \geq 1$, $t\lesssim n\rho_n$, 
\begin{align*}
\|\bW\transpose{}\widetilde{\bx}_i - \rho_n^{1/2}\bx_{i}\|_2
&\leq \frac{1}{n\rho_n^{1/2}}\|\be_i\transpose{}\bE\bX\bDelta_n^{-1}\|_2 + \|\be_i\transpose{}(\widetilde\bX\bW - \rho_n^{-1/2}\bA\bX(\bX\transpose{}\bX)^{-1})\|_2
\\
&
\lesssim_c \frac{(n\rho_nt)^{1/2}\|\bX\|_{2\to\infty}\|\bDelta_n^{-1}\|_2}{n\rho_n^{1/2}}\\
&\quad + 
  \frac{\|\bU_\bP\|_{2\to\infty}}{(n\rho_n)^{1/2}\lambda_d(\bDelta_n)^2}\max\left\{\frac{t^{1/2}}{\lambda_d(\bDelta_n)^2}, \frac{\kappa(\bDelta_n)}{\lambda_d(\bDelta_n)^2}, t\right\}
  \\
&\leq \frac{t^{1/2}}{\sqrt{n}\lambda_d(\bDelta_n)} + \frac{\|\bU_\bP\|_{2\to\infty}}{ \lambda_d(\bDelta_n) }\max\left\{\frac{t^{1/2}}{ \lambda_d(\bDelta_n)^2}, \frac{\kappa(\bDelta_n)}{ \lambda_d(\bDelta_n)^2}, t\right\}\\
&\lesssim \frac{\|\bU_\bP\|_{2\to\infty}}{\lambda_d(\bDelta_n)}\max\left\{\frac{t^{1/2}}{ \lambda_d(\bDelta_n)^2}, \frac{\kappa(\bDelta_n)}{ \lambda_d(\bDelta_n)^2}, t\right\}
\end{align*}
with probability at least $1 - c_0n^{-c} - c_0e^{-t}$. Note that the tail probability $c_0de^{-t}$ can be replaced by $c_0e^{-t}$ when $A_{ij}$'s are Bernoulli random variables because the vector Bernstein's inequality (Lemma \ref{lemma:Bernstein_concentration_EW}) is dimension free. The proof is thus completed.
\end{proof}

We finally present the following lemma that characterizes the Taylor expansion behavior of the functions $\bg$ and $\bh_{ij}$ defined in Section \ref{sub:outline_of_the_proof_of_theorem_thm:asymptotic_normality_os}. 
\begin{lemma}\label{lemma:gh_Taylor_expansion}
Let $\calX_2(\eps)$, $\bg(\bu,\bv)$, $\bh_{ij}(\bu,\bv)$, $\br_\bg(\bu,\bv, \bx_i, \bx_j)$, and $\br_{\bh_{ij}}(\bu,\bv, \bx_i, \bx_j)$ be defined as in Section \ref{sub:outline_of_the_proof_of_theorem_thm:asymptotic_normality_os}. Suppose $(\bx_i, \bx_j)\in\calX(\delta)$. Then:
\begin{itemize}
  \item[(a)] For all $(\bu, \bv)\in\calX_2(\delta/2)$, 
  \begin{align*}
  \left\|\bg(\bu, \bv) - \bg(\bx_{i},\bx_{j}) \right\|_2&\lesssim \frac{1}{\delta^4}(\|\bu - \bx_{i}\|_2 + \|\bv - \bx_{j}\|_2),\\
  \left\|\bh_{ij}(\bu, \bv) - \bh_{ij}(\bx_{i},\bx_{j}) \right\|_2&\lesssim \frac{1}{\delta^4}(\|\bu - \bx_{i}\|_2 + \|\bv - \bx_{j}\|_2).
  \end{align*}

  \item[(b)]  For all $(\bu, \bv)\in\calX_2(\delta/2)$,
\begin{align*}
 \|\br_\bg(\bu,\bv, \bx_i, \bx_j)\|_2&\lesssim \frac{d^{1/2}}{\delta^6}(\|\bu - \bx_{i}\|_2^2 + \|\bv - \bx_{j}\|_2^2),\\
 \|\br_{\bh_{ij}}(\bu,\bv, \bx_i, \bx_j)\|_2&\lesssim \frac{d^{1/2}}{\delta^6}(\|\bu - \bx_{i}\|_2^2 + \|\bv - \bx_{j}\|_2^2).
\end{align*}
\end{itemize}
\end{lemma}

\begin{proof}[\bf Proof of Lemma \ref{lemma:gh_Taylor_expansion}]
For each $k\in [d]$, denote
\[
g_k(\bu, \bv) = \frac{\be_k\transpose{}\bv}{\bu\transpose{}\bv(1 - \rho_n\bu\transpose{}\bv)},\quad 
h_{ijk}(\bu, \bv) = \frac{(\bx_{i}\transpose{}\bx_{j} - \bu\transpose{}\bv)\be_k\transpose{}\bv}{\bu\transpose{}\bv(1 - \rho_n\bu\transpose{}\bv)}.
\]
A simple algebra shows that the gradients of $g_k$ and $h_{ijk}$ are
\begin{align*}
&\frac{\partial g_k}{\partial\bu\transpose{}}(\bu, \bv)  = \frac{(1 - 2\rho_n\bu\transpose{}\bv)\be_k\transpose\bv\bv\transpose{}}{\{\bu\transpose{}\bv(1 - \rho_n\bu\transpose{}\bv)\}^2},\quad
\frac{\partial g_k}{\partial\bv\transpose{}}(\bu, \bv)  = \frac{\eye_d}{\bu\transpose{}\bv(1 - \rho_n\bu\transpose{}\bv)} + 
\frac{(1 - 2\rho_n\bu\transpose{}\bv)\be_k\transpose{}\bv\bu\transpose{}}{\{\bu\transpose{}\bv(1 - \rho_n\bu\transpose{}\bv)\}^2},\\
&\frac{\partial h_{ijk}}{\partial\bu\transpose{}}(\bu, \bv)  = 
- \frac{\be_k\transpose{}\bv\bv\transpose{}}{\bu\transpose{}\bv(1 - \rho_n\bu\transpose{}\bv)}
-
\frac{(1 - 2\rho_n\bu\transpose{}\bv)(\bx_{i}\transpose{}\bx_{j} - \bu\transpose{}\bv)(\be_k\transpose{}\bv)\bv\transpose{}}{\{\bu\transpose{}\bv(1 - \rho_n\bu\transpose{}\bv)\}^2},
\\
&
\frac{\partial h_{ijk}}{\partial\bv\transpose{}}(\bu, \bv)  = \frac{
(\bx_{i}\transpose{}\bx_{j} - \bu\transpose{}\bv)\be_k\transpose{}
}{\bu\transpose{}\bv(1 - \rho_n\bu\transpose{}\bv)}
 - \frac{\be_k\transpose{}\bv\bu\transpose{}}{\bu\transpose{}\bv(1 - \rho_n\bu\transpose{}\bv)}
 -  
\frac{(1 - 2\rho_n\bu\transpose{}\bv)(\bx_{i}\transpose{}\bx_{j} - \bu\transpose{}\bv)(\be_k\transpose{}\bv)\bu\transpose{}}{\{\bu\transpose{}\bv(1 - \rho_n\bu\transpose{}\bv)\}^2}.
\end{align*}
Clearly,
\begin{align*}
&\sup_{(\bu, \bv)\in\calX_2(\delta/2)}\max\left\{\left\|\frac{\partial \bg}{\partial\bu\transpose{}}(\bu, \bv)\right\|_2 + \left\|\frac{\partial \bg}{\partial\bv\transpose{}}(\bu, \bv)\right\|_2, \left\|\frac{\partial \bh_{ij}}{\partial\bu\transpose{}}(\bu, \bv)\right\|_2 + \left\|\frac{\partial \bh_{ij}}{\partial\bv\transpose{}}(\bu, \bv)\right\|_2\right\}\lesssim \frac{1}{\delta^4}.
\end{align*}
Then assertion (a) then follows directly from the mean-value inequality for vector-valued functions. To prove assertion (b), we need to first compute the Hessian of $g_k$:
\begin{align*}
\frac{\partial^2g_k}{\partial\bu\partial\bu\transpose{}}(\bu, \bv)
& = \left[\frac{2\rho_n(\be_k\transpose{}\bv)}{\{\bu\transpose{}\bv(1 - \rho_n\bu\transpose{}\bv)\}^2}
   + \frac{2(1 - 2\rho_n\bu\transpose{}\bv)^2(\be_k\transpose{}\bv)}{\{\bu\transpose{}\bv(1 - \rho_n\bu\transpose{}\bv)\}^3}\right]\bv\bv\transpose,\\
\frac{\partial^2g_k}{\partial\bv\partial\bu\transpose{}}(\bu, \bv)
& = \left[\frac{2\rho_n(\be_k\transpose{}\bv)}{\{\bu\transpose{}\bv(1 - \rho_n\bu\transpose{}\bv)\}^2}
   + \frac{2(1 - 2\rho_n\bu\transpose{}\bv)^2(\be_k\transpose{}\bv)}{\{\bu\transpose{}\bv(1 - \rho_n\bu\transpose{}\bv)\}^3}\right]\bu\bv\transpose
   \\
&\quad - \frac{(1 - 2\rho_n\bu\transpose{}\bv)}{\{\bu\transpose{}\bv(1 - \rho_n\bu\transpose{}\bv\}^2}(\be_k\bv\transpose{} + \be_k\transpose{}\bv\eye_d),\\
\frac{\partial^2g_k}{\partial\bv\partial\bv\transpose{}}(\bu, \bv)
& = \left[\frac{2\rho_n(\be_k\transpose{}\bv)}{\{\bu\transpose{}\bv(1 - \rho_n\bu\transpose{}\bv)\}^2}
   + \frac{2(1 - 2\rho_n\bu\transpose{}\bv)^2(\be_k\transpose{}\bv)}{\{\bu\transpose{}\bv(1 - \rho_n\bu\transpose{}\bv)\}^3}\right]\bu\bu\transpose
   \\
&\quad - \frac{(1 - 2\rho_n\bu\transpose{}\bv)}{\{\bu\transpose{}\bv(1 - \rho_n\bu\transpose{}\bv)\}^2}(\be_k\bu\transpose{} + \bu\be_k\transpose{}).
\end{align*}
Since $(\bu,\bv)\in\calX_2(\delta/2)$ and $\rho_n\in (0, 1]$, we see that
\begin{align*}
\sup_{(\bu,\bv)\in\calX_2(\delta/2)}\max\left\{\left\|\frac{\partial^2g_k}{\partial\bu\partial\bu\transpose{}}(\bu, \bv)\right\|_2, \left\|\frac{\partial^2g_k}{\partial\bv\partial\bu\transpose{}}(\bu, \bv)\right\|_2, \left\|\frac{\partial^2g_k}{\partial\bv\partial\bv\transpose{}}(\bu, \bv)\right\|_2\right\}
&\lesssim \frac{1}{\delta^6}.
\end{align*}
By the mean-value inequality, for any $(\bu_1,\bv_1),(\bu_2,\bv_2)\in\calX_2(\delta/2)$, 
\begin{align*}
\left\|
\begin{bmatrix*}
\frac{\partial g_k}{\partial\bu}(\bu_1, \bv_1)\\
\frac{\partial g_k}{\partial\bv}(\bu_1, \bv_1)
\end{bmatrix*}
- 
\begin{bmatrix*}
\frac{\partial g_k}{\partial\bu}(\bu_2, \bv_2)\\
\frac{\partial g_k}{\partial\bv}(\bu_2, \bv_2)
\end{bmatrix*}
\right\|_2\lesssim \frac{1}{\delta^6}\left\|
\begin{bmatrix*}
\bu_1 - \bu_2\\
\bv_1 - \bv_2
\end{bmatrix*}
\right\|_2.
\end{align*}
Namely, the gradient of $g_k$ is Lipschitz continuous over $\calX_2(\delta/2)$ with a Lipschitz constant upper bounded by an absolute constant factor of $1/\delta^6$. By Taylor's theorem, for any $(\bu, \bv)\in\calX_2(\delta/2)$, 
\begin{align*}
|\be_k\transpose{}\br_\bg(\bu, \bv, \bx_i, \bx_j)|
\lesssim \frac{\|\bu - \bx_{i}\|_2^2 + \|\bv - \bx_{j}\|_2^2}{\delta^6}, 
\end{align*}
and hence, 
\begin{align*}
 \|\br_\bg(\bu, \bv, \bx_i, \bx_j)\|_2 = \left(\sum_{k = 1}^d|\be_k\transpose{}\br_\bg(\bu, \bv, \bx_i, \bx_j)|
^2\right)^{1/2}\lesssim \frac{ \sqrt{d}}{\delta^6}(\|\bu - \bx_{i}\|_2^2 + \|\bv - \bx_{j}\|_2^2). 
\end{align*}
The Hessian of $h_{ijk}$ can be computed similarly:
\begin{align*}
\frac{\partial^2h_{ijk}}{\partial\bu\partial\bu\transpose{}}(\bu, \bv)
& = \left[\frac{2(1 - 2\rho_n\bu\transpose{}\bv) + 2\rho_n(\bx_{i}\transpose{}\bx_{j} - \bu\transpose\bv)}{\{\bu\transpose{}\bv(1 - \rho_n\bu\transpose{}\bv)\}^2}
   + \frac{2(1 - 2\rho_n\bu\transpose{}\bv)^2(\bx_{i}\transpose{}\bx_{j} - \bu\transpose{}\bv)}{\{\bu\transpose{}\bv(1 - \rho_n\bu\transpose{}\bv)\}^3}\right]\\
   &\quad\times(\be_k\transpose{}\bv)\bv\bv\transpose,\\
\frac{\partial^2h_{ijk}}{\partial\bv\partial\bu\transpose{}}(\bu, \bv)
& = \left[\frac{2(1 - 2\rho_n\bu\transpose{}\bv) + 2\rho_n(\bx_{i}\transpose{}\bx_{j} - \bu\transpose\bv)}{\{\bu\transpose{}\bv(1 - \rho_n\bu\transpose{}\bv)\}^2}
   + \frac{2(1 - 2\rho_n\bu\transpose{}\bv)^2(\bx_{i}\transpose{}\bx_{j} - \bu\transpose{}\bv)}{\{\bu\transpose{}\bv(1 - \rho_n\bu\transpose{}\bv)\}^3}\right]\\
   &\quad\times(\be_k\transpose{}\bv)\bu\bv\transpose
 - \left[
\frac{1}{\bu\transpose{}\bv(1 - \rho_n\bu\transpose{}\bv)} + 
\frac{(1 - 2\rho_n\bu\transpose{}\bv)(\bx_{i}\transpose{}\bx_{j} - \bu\transpose{}\bv)}{\{\bu\transpose{}\bv(1 - \rho_n\bu\transpose{}\bv\}^2}\right]\\
&\quad\times (\be_k\bv\transpose{} + \be_k\transpose{}\bv\eye_d),\\
\frac{\partial^2h_{ijk}}{\partial\bv\partial\bv\transpose{}}(\bu, \bv)
& = -\left[
\frac{1}{\bu\transpose{}\bv(1 - \rho_n\bu\transpose{}\bv)} + 
\frac{(1 - 2\rho_n\bu\transpose{}\bv)(\bx_{i}\transpose{}\bx_{j} - \bu\transpose{}\bv)}{\{\bu\transpose{}\bv(1 - \rho_n\bu\transpose{}\bv\}^2}\right]
   (\bu\be_k\transpose{}+ \be_k\bu\transpose{})
   \\
  &\quad
 + \left[\frac{2(1 - 2\rho_n\bu\transpose{}\bv) + 2\rho_n(\bx_{i}\transpose{}\bx_{j} - \bu\transpose\bv)}{\{\bu\transpose{}\bv(1 - \rho_n\bu\transpose{}\bv)\}^2}
   + \frac{2(1 - 2\rho_n\bu\transpose{}\bv)^2(\bx_{i}\transpose{}\bx_{j} - \bu\transpose{}\bv)}{\{\bu\transpose{}\bv(1 - \rho_n\bu\transpose{}\bv)\}^3}\right]\\
   &\quad\times (\be_k\transpose{}\bv)\bu\bu\transpose{}.
\end{align*}
This implies that
\begin{align*}
\sup_{(\bu,\bv)\in\calX_2(\delta/2)}\max\left\{\left\|\frac{\partial^2h_{ijk}}{\partial\bu\partial\bu\transpose{}}(\bu, \bv)\right\|_2,
\left\|\frac{\partial^2h_{ijk}}{\partial\bv\partial\bu\transpose{}}(\bu, \bv)\right\|_2,
\left\|\frac{\partial^2h_{ijk}}{\partial\bv\partial\bv\transpose{}}(\bu, \bv)\right\|_2
\right\}
&\lesssim \frac{1}{\delta^6}.
\end{align*}
An identical argument shows that
\begin{align*}
\|\br_{\bh_{ij}}(\bu, \bv, \bx_i, \bx_j)\|_2
\lesssim \frac{ \sqrt{d}}{\delta^6}(\|\bu - \bx_{i}\|_2^2 + \|\bv - \bx_{j}\|_2^2).
\end{align*}
\end{proof}


\subsection{Concentration bound for \eqref{eqn:ri1_analysis_term_II}} 
\label{sub:concentration_bound_for_eqn:ri1_analysis_term_ii}

\begin{lemma}\label{lemma:R12k_analysis}
Let $\bA\sim\mathrm{RDPG}(\rho_n^{1/2}\bX)$ with and assume the conditions of Theorem \ref{thm:asymptotic_normality_OS} hold. Then there exists an absolute constant $c_0 > 0$, such that given any fixed $c > 0$, for each fixed row index $i\in [n]$, for all
$t\geq 1$, $t\lesssim n\rho_n$, and for sufficiently large $n$, with probability at least $1 - c_0n^{-c} - c_0e^{-t}$, 
\begin{align*}
&\left\|\frac{1}{n\sqrt{\rho_n}}\sum_{j = 1}^n\frac{\rho_n\bx_{j}\bx_{i}\transpose}{\bx_{i}\transpose\bx_{j}(1 - \rho_n\bx_{i}\transpose\bx_{j})}(\rho_n^{-1/2}\bW\transpose\widetilde\bx_j - \bx_{j})\right\|
\lesssim_c \frac{1}{ n\rho_n^{1/2}\delta^2\lambda_d(\bDelta_n) }\max\left\{\frac{\kappa(\bDelta_n)}{\lambda_d(\bDelta_n)}, t^{1/2}\right\}.
\end{align*}
\begin{proof}
[\bf Proof of Lemma \ref{lemma:R12k_analysis}]
Denote $\bR_\bX = \widetilde\bX\bW - \rho_n^{-1/2}\bA\bX(\bX\transpose{}\bX)^{-1}$. By the decomposition \eqref{eqn:keystone_decomposition}, for any $j\in[n]$, we have
\[
\bW\transpose\widetilde\bx_j - \rho_n^{1/2}\bx_{j} = \frac{1}{n\sqrt{\rho_n}}\sum_{a = 1}^n(A_{ja} - \rho_n\bx_{j}\transpose\bx_{a}) \bDelta_n^{-1}\bx_{a}  + \bR_\bX\transpose{}\be_j. 
\]
It follows that
\begin{align*}
&\frac{1}{n\sqrt{\rho_n}}\sum_{j = 1}^n\frac{\rho_n\bx_{j}\bx_{i}\transpose{}}{\bx_{i}\transpose\bx_{j}(1 - \rho_n\bx_{i}\transpose\bx_{j})}(\rho_n^{-1/2}\bW\transpose\widetilde\bx_j - \bx_{j})\\
&\quad = \frac{1}{n^2\sqrt{\rho_n}} \sum_{j = 1}^n\sum_{a = 1}^n\frac{\bx_{j}\bx_{i}\transpose{}\bDelta_n^{-1}\bx_{a}}{\bx_{i}\transpose\bx_{j}(1 - \rho_n\bx_{i}\transpose\bx_{j})}(A_{ja} - \rho_n\bx_{j}\transpose\bx_{a})
 + \frac{1}{n}\sum_{j = 1}^n \frac{\bx_{j}\bx_{i}\transpose\bR_\bX\transpose{}\be_j}{\bx_{i}\transpose\bx_{j}(1 - \rho_n\bx_{i}\transpose\bx_{j})}\\
&\quad = \frac{1}{n\sqrt{\rho_n}} \left\{\sum_{j \leq a}\bz_{ija} + \sum_{j > a}\bz_{ija}\right\} + \frac{1}{n}\sum_{j = 1}^n \frac{\bx_{j}\bx_{i}\transpose\bR_\bX\transpose{}\be_j}{\bx_{i}\transpose\bx_{j}(1 - \rho_n\bx_{i}\transpose\bx_{j})},
\end{align*}
where 
\[
\bz_{ija} = \frac{(A_{ja} - \rho_n\bx_{j}\transpose\bx_{a})\bx_{j}\bx_{i}\transpose{}\bDelta_n^{-1}\bx_{a}}{n\bx_{i}\transpose\bx_{j}(1 - \rho_n\bx_{i}\transpose\bx_{j})}.
\]
By Lemma \ref{lemma:Bernstein_concentration_EW}, with $t \geq 1$ and $t\lesssim n\rho_n$, we see that
\begin{align*}
\frac{1}{n\sqrt{\rho_n}}\left\|\sum_{j \leq a}\bz_{ija}\right\|_2 + \frac{1}{n\sqrt{\rho_n}}\left\|\sum_{j > a}\bz_{ija}\right\|_2
&\lesssim_c \left\{\frac{t + (n^2\rho_nt)^{1/2}}{n\sqrt{\rho_n}}\right\}\max_{j,a\in [n]}\left\|\frac{\bx_{j}\bx_{i}\transpose{}\bDelta_n^{-1}\bx_{a}}{n\bx_{i}\transpose\bx_{j}(1 - \rho_n\bx_{i}\transpose\bx_{j})}\right\|_2\\
&\lesssim \frac{t^{1/2}}{n\delta^2\lambda_d(\bDelta_n)}
\end{align*}
with probability at least $1 - c_0e^{-t}$. 
In addition, by Lemma \ref{lemma:RX_Frobenius_norm_bound}, 
for sufficiently large $n$,
\[
\| \bR_\bX\|_{\mathrm{F}} \lesssim 
 \frac{1}{(n\rho_n)^{1/2}\lambda_d(\bDelta_n) }\max\left\{\frac{\kappa(\bDelta_n)}{\lambda_d(\bDelta_n)}, t^{1/2}\right\}
\]
with probability at least $1 - c_0n^{-c} - c_0e^{-t}$ for all $t > 0$. By Cauchy-Schwarz inequality, we have
\begin{align*}
 \left\|\frac{1}{n}\sum_{j = 1}^n 
\frac{\bx_{j}\bx_{i}\transpose{}\bR_\bX\transpose{}\be_j}{\bx_{i}\transpose\bx_{j}(1 - \rho_n\bx_{i}\transpose\bx_{j})}
\right\|_2
&\leq  \frac{1}{n}\sum_{j = 1}^n\left\|\frac{\bx_{j}\bx_{i}\transpose{}}{\bx_{i}\transpose\bx_{j}(1 - \rho_n\bx_{i}\transpose\bx_{j})}\right\|_2\|\bR_\bX\transpose{}\be_j\|_2\\
&\leq \frac{1}{n} \left[\sum_{j = 1}^n\left\|\frac{\bx_{j}\bx_{i}\transpose{}}{\bx_{i}\transpose\bx_{j}(1 - \rho_n\bx_{i}\transpose\bx_{j})}\right\|_2^2\right]^{1/2}\|\bR_\bX\|_{\mathrm{F}}
\leq \frac{\|\bR_\bX\|_{\mathrm{F}}}{\sqrt{n}\delta^2}\\
&\lesssim_c \frac{1}{ n\rho_n^{1/2}\delta^2\lambda_d(\bDelta_n)}\max\left\{\frac{\kappa(\bDelta_n)}{\lambda_d(\bDelta_n)}, t^{1/2}\right\},
\end{align*}
with probability at least $1 - c_0n^{-c} - c_0e^{-t}$ for all $t > 0$ whenever $n$ is sufficiently large. 
Therefore, we conclude that
\begin{align*}
& \left\|\frac{1}{n\sqrt{\rho_n}}\sum_{j = 1}^n\frac{\rho_n\bx_{j}\bx_{i}\transpose}{\bx_{i}\transpose\bx_{j}(1 - \rho_n\bx_{i}\transpose\bx_{j})}(\rho_n^{-1/2}\bW\transpose\widetilde\bx_j - \bx_{j})\right\|
\\&\quad
\lesssim_c  \frac{t^{1/2}}{n\delta^2\lambda_d(\bDelta_n)} + \frac{1}{ n\rho_n^{1/2}\delta^2\lambda_d(\bDelta_n) }\max\left\{\frac{\kappa(\bDelta_n)}{\lambda_d(\bDelta_n)}, t^{1/2}\right\}\\
&\quad\lesssim 
\frac{1}{ n\rho_n^{1/2}\delta^2\lambda_d(\bDelta_n) }\max\left\{\frac{\kappa(\bDelta_n)}{\lambda_d(\bDelta_n)}, t^{1/2}\right\}
\end{align*}
with probability at least $1 - c_0n^{-c} - c_0e^{-t}$ for all $t \geq1$, $t\lesssim n\rho_n$, provided that $n$ is sufficiently large. 
The proof is thus completed. 
\end{proof}
\end{lemma}


\subsection{Concentration bound for \eqref{eqn:ri1_analysis_term_III}} 
\label{sub:concentration_bound_for_eqn:ri1_analysis_term_iii}

\begin{lemma}\label{lemma:R13k_analysis}
Let $\bA\sim\mathrm{RDPG}(\rho_n^{1/2}\bX)$ with and assume the conditions of Theorem \ref{thm:asymptotic_normality_OS} hold. 
Suppose $\{\bB_{nij}:i,j\in[n]\}$ is a collection of deterministic $d\times d$ matrices with $\sup_{i,j\in[n]}\|\bB_{nij}\|_{\mathrm{F}} \leq \delta^{-4}$.
Then given any fixed $c > 0$, for each fixed row index $i\in [n]$, for all
$t > 0$, $t\lesssim n\rho_n$, and sufficiently large $n$,
\begin{align*}
& \mathrel{\Big\|}\frac{1}{n\sqrt{\rho_n}}\sum_{j = 1}^n[\bE]_{ij}\bB_{nij}(\rho_n^{-1/2}\bW\transpose\widetilde\bx_i - \bx_{i})\mathrel{\Big\|_2}
\\
&\quad
\lesssim_c \frac{\|\bU_\bP\|_{2\to\infty}}{(n\rho_n)^{1/2}\delta^4\lambda_d(\bDelta_n)}\max\left\{\frac{t}{ \lambda_d(\bDelta_n)^2}, \frac{\kappa(\bDelta_n)t^{1/2}}{ \lambda_d(\bDelta_n)^2}, {t^{3/2}} \right\}
\end{align*}
with probability at least $1 - c_0n^{-c} - c_0e^{-t}$, where $c_0 > 0$ is an absolute constant.  
\begin{proof}[\bf Proof]
First observe that by definition of the matrix norm, 
\begin{align*}
&
\left\|\frac{1}{n\sqrt{\rho_n}}\sum_{j = 1}^n[\bE]_{ij}\bB_{nij}(\rho_n^{-1/2}\bW\transpose\widetilde\bx_i - \bx_{i})\right\|_2
\\&\quad
\leq \frac{1}{\sqrt{\rho_n}}\|\bW\transpose\widetilde\bx_i - \rho_n^{1/2}\bx_{i}\|_2\left\|\frac{1}{n\sqrt{\rho_n}}\sum_{j = 1}^n\bB_{nij}(A_{ij} - \rho_n\bx_{i}\transpose\bx_{j})\right\|_2
\\&\quad
\leq \frac{1}{\sqrt{\rho_n}}\|\bW\transpose\widetilde\bx_i - \rho_n^{1/2}\bx_{i}\|_2\left\|\frac{1}{n\sqrt{\rho_n}}\sum_{j = 1}^n\vect(\bB_{nij})(A_{ij} - \rho_n\bx_{i}\transpose\bx_{j})\right\|_2.
\end{align*}
By Lemma \ref{lemma:Bernstein_concentration_EW}, with for all $t\lesssim n\rho_n$ and $t \geq 1$, we have
\begin{align*}
\left\|\frac{1}{n\sqrt{\rho_n}}\sum_{j = 1}^n\vect(\bB_{nij})(A_{ij} - \rho_n\bx_{i}\transpose\bx_{j})\right\|_2
&\lesssim \frac{(n\rho_nt)^{1/2}}{n\rho_n^{1/2}}\max_{j\in [n]}\|\bB_{nij}\|_{\mathrm{F}}\leq\frac{t^{1/2}}{\sqrt{n}\delta^4}
\end{align*}
with probability at least $1 - c_0e^{-t}$. 
Hence, by Lemma \ref{lemma:single_x_i_perturbation_bound} for sufficiently large $n$,
\begin{align*}
& \mathrel{\Big\|}\frac{1}{n\sqrt{\rho_n}}\sum_{j = 1}^n[\bE]_{ij}\bB_{nij}(\rho_n^{-1/2}\bW\transpose\widetilde\bx_i - \bx_{i})\mathrel{\Big\|}\\
&\quad\leq \frac{\|\bW\transpose\widetilde\bx_i - \rho_n^{1/2}\bx_{i}\|_{2}}{\sqrt{\rho_n}} \mathrel{\Big\|}\frac{1}{n\sqrt{\rho_n}}\sum_{j = 1}^n\vect(\bB_{nij})(A_{ij} - \rho_n\bx_{i}\transpose\bx_{j})\mathrel{\Big\|}_2\\
&\quad\lesssim_c \frac{\|\bU_\bP\|_{2\to\infty}}{(n\rho_n)^{1/2}\delta^4\lambda_d(\bDelta_n)}\max\left\{\frac{t}{ \lambda_d(\bDelta_n)^2}, \frac{\kappa(\bDelta_n)t^{1/2}}{ \lambda_d(\bDelta_n)^2}, {t^{3/2}} \right\}
\end{align*}
with probability at least $1 - c_0n^{-c} - c_0e^{-t}$. The proof is thus completed. 
\end{proof}
\end{lemma}


\subsection{Concentration bound for \eqref{eqn:ri1_analysis_term_IV}} 
\label{sub:concentration_bound_for_eqn:ri1_analysis_term_iv}

\begin{lemma}\label{Lemma:R14k_analysis}
Let $\bA\sim\mathrm{RDPG}(\rho_n^{1/2}\bX)$ with and assume the conditions of Theorem \ref{thm:asymptotic_normality_OS} hold. Suppose $\{\bB_{nij}:i,j\in[n]\}$ is a collection of deterministic $d\times d$ matrices such that $\sup_{i,j\in[n]}\|\bB_{nij}\|_{2} \leq \delta^{-4}$. Then given any fixed $c > 0$, for each fixed index $i\in [n]$, 
for all $t\geq 1$, $t\lesssim n\rho_n$ and sufficiently large $n$,
\begin{align*}
&\mathrel{\Big\|}\frac{1}{n\sqrt{\rho_n}}\sum_{j = 1}^n[\bE]_{ij}\bB_{nij}(\rho_n^{-1/2}\bW\transpose\widetilde\bx_j - \bx_{j})\mathrel{\Big\|_2}
\\&\quad
\lesssim_c \frac{\|\bU_\bP\|_{2\to\infty}}{(n\rho_n)^{1/2}\delta^4\lambda_d(\bDelta_n) }
\max\left\{\frac{\kappa(\bDelta_n)^{1/2}t^{1/2}}{\lambda_d(\bDelta_n)}, \frac{\kappa(\bDelta_n)}{\lambda_d(\bDelta_n)^2}, t\right\}
\end{align*}
with probability at least $1 - c_0n^{-c} - c_0e^{-t}$ for sufficiently large $n$, where $c_0 > 0$ is an absolute constant 
\end{lemma}

\begin{proof}[\bf Proof of Lemma \ref{Lemma:R14k_analysis}]
Let $\bA^{(m)}$, $\bU_\bA^{(m)}$, and $\bH^{(m)}$, $m = 1,\ldots,n$ be the auxiliary matrices defined in Section \ref{sec:entrywise_limit_theorem_for_the_eigenvectors} of the manuscript. Now we fix the row index $i\in [n]$. Observe that
\begin{align*}
\begin{aligned}
\widetilde{\bX} \bW - \rho_n^{1/2}\bX
& = \bU_\bA\bS_\bA^{1/2}(\bW^*)\transpose\bW_\bX - \bU_\bP\bS_\bP^{1/2}\bW_\bX
  = (\bU_\bA\bS_\bA^{1/2}\mathrm{sgn}(\bH) - \bU_\bP\bS_\bP^{1/2})\bW_\bX
\\
& = \bU_\bA\{\bS_\bA^{1/2}\mathrm{sgn}(\bH) - \mathrm{sgn}(\bH)\bS_\bP^{1/2}\}\bW_\bX
 + \bU_\bA\{\mathrm{sgn}(\bH) - \bH\}\bS_\bP^{1/2}\bW_\bX\\
&\quad + (\bU_\bA\bH - \bU_\bA^{(i)}\bH^{(i)})\bS_\bP^{1/2}\bW_\bX + (\bU_\bA^{(i)}\bH^{(i)} - \bU_\bP)\bS_\bP^{1/2}\bW_\bX.
\end{aligned}
\end{align*}
This immediately leads to the following decomposition of the quantity of interest:
\begin{align}
&\left\{\frac{1}{n\rho_n}\sum_{j = 1}^n[\bE]_{ij}
\bB_{nij}(\bW\transpose{}\widetilde\bx_j - \rho_n^{1/2}\bx_{j})\right\}\transpose{}\nonumber\\
\label{eqn:R14k_analysis_term_I}
&\quad = \frac{1}{n\rho_n}\sum_{j = 1}^n(A_{ij} - \rho_n\bx_{i}\transpose{}\bx_{j})\be_j\transpose{}\bU_\bA\{\bS_\bA^{1/2}\mathrm{sgn}(\bH) - \mathrm{sgn}(\bH)\bS_\bP^{1/2}\}\bW_\bX\bB_{nij}\transpose{}\\
\label{eqn:R14k_analysis_term_II}
&\quad\quad + \frac{1}{n\rho_n}\sum_{j = 1}^n(A_{ij} - \rho_n\bx_{i}\transpose{}\bx_{j})\be_j\transpose{}\bU_\bA\{\mathrm{sgn}(\bH) - \bH\}\bS_\bP^{1/2}\bW_\bX\bB_{nij}\transpose{}
\\
\label{eqn:R14k_analysis_term_III}
&\quad\quad + \frac{1}{n\rho_n}\sum_{j = 1}^n(A_{ij} - \rho_n\bx_{i}\transpose{}\bx_{j})\be_j\transpose{}(\bU_\bA\bH - \bU_\bA^{(i)}\bH^{(i)})\bS_\bP^{1/2}\bW_\bX\bB_{nij}\transpose{}
\\
\label{eqn:R14k_analysis_term_IV}
&\quad\quad + \frac{1}{n\rho_n}\sum_{j = 1}^n(A_{ij} - \rho_n\bx_{i}\transpose{}\bx_{j})\be_j\transpose{}(\bU_\bA^{(i)}\bH^{(i)} - \bU_\bP)\bS_\bP^{1/2}\bW_\bX\bB_{nij}\transpose{}.
\end{align}
For term \eqref{eqn:R14k_analysis_term_I}, for all $t > 0$, we apply Result \ref{result:Infinity_norm_concentration}, Lemma \ref{lemma:WS_interchange_bound}, and Lemma \ref{lemma:U_A_two_to_infinity_norm} to 
obtain that for sufficiently large $n$,
\begin{align*}
&\left\|\frac{1}{n\rho_n}\sum_{j = 1}^n(A_{ij} - \rho_n\bx_{i}\transpose{}\bx_{j})\be_j\transpose{}\bU_\bA\{\bS_\bA^{1/2}\mathrm{sgn}(\bH) - \mathrm{sgn}(\bH)\bS_\bP^{1/2}\}\bW_\bX\bB_{nij}\transpose{}\right\|_2\\
&\quad\leq \frac{1}{n\rho_n}\sum_{j = 1}^n|[\bE]_{ij}|\|\bU_\bA\|_{2\to\infty}\|\bW^*\bS_\bA^{1/2} - \bS_\bP^{1/2}\bW^*\|_2\|\max_{j\in [n]}\|\bB_{nij}\|_2\\
&\quad = \frac{1}{n\rho_n}\|\bE\|_{\infty}\|\bU_\bA\|_{2\to\infty}\|\bW^*\bS_\bA^{1/2} - \bS_\bP^{1/2}\bW^*\|_2\|\max_{j\in [n]}\|\bB_{nij}\|_2\\
&\quad\lesssim_c \frac{1}{n\rho_n\delta^4}(n\rho_n)\frac{\|\bU_\bP\|_{2\to\infty}}{ \lambda_d(\bDelta_n) }\frac{1}{(n\rho_n)^{1/2}\lambda_d(\bDelta_n)}\max\left\{\frac{\kappa(\bDelta_n)}{\lambda_d(\bDelta_n)}, {t^{1/2}}\right\}\\
&\quad = \frac{\|\bU_\bP\|_{2\to\infty}}{(n\rho_n)^{1/2}\delta^4\lambda_d(\bDelta_n)^{2}}\max\left\{\frac{\kappa(\bDelta_n)}{\lambda_d(\bDelta_n)}, {t^{1/2}}\right\}
\end{align*}
with probability at least $1 - c_0n^{-c} - c_0e^{-t}$. 
For term \eqref{eqn:R14k_analysis_term_II}, we invoke Lemma \ref{lemma:WS_interchange_bound}, Lemma 6.7 in \cite{cape2017two}, and the Davis-Kahan theorem to obtain that for sufficiently large $n$,
\begin{align*}
&\left\|\frac{1}{n\rho_n}\sum_{j = 1}^n(A_{ij} - \rho_n\bx_{i}\transpose{}\bx_{j})\be_j\transpose{}\bU_\bA\{\mathrm{sgn}(\bH) - \bH\}\bS_\bP^{1/2}\bW_\bX\bB_{nij}\transpose{}\right\|_2\\
&\quad \leq \frac{1}{n\rho_n}\sum_{j = 1}^n|[\bE]_{ij}|\|\bU_\bA\|_{2\to\infty}\|\bW^* - \bU_\bP\transpose\bU_\bA\|_2\|\bS_\bP\|_2^{1/2}\max_{j\in [n]}\|\bB_{nij}\|_2\\
&\quad \leq \frac{1}{n\rho_n}\|\bE\|_\infty \|\bU_\bA\|_{2\to\infty}\frac{4\|\bE\|_2^2}{\lambda_d(\bP)^2}\|\bS_\bP\|_2^{1/2}\max_{j\in [n]}\|\bB_{nij}\|_2\\
&\quad \lesssim_c \frac{1}{n\rho_n}(n\rho_n)\frac{\|\bU_\bP\|_{2\to\infty}}{ \lambda_d(\bDelta_n) }\frac{(n\rho_n)}{(n\rho_n)^2\lambda_d(\bDelta_n)^2}\{(n\rho_n)\lambda_1(\bDelta_n)\}^{1/2}\frac{1}{\delta^4}\\
&\quad\lesssim \frac{\|\bU_\bP\|_{2\to\infty}}{(n\rho_n)^{1/2}\delta^4\lambda_d(\bDelta_n)^3}
\end{align*}
with probability at least $1 - 4n^{-c}$. 
We now turn the focus to the more complicated terms \eqref{eqn:R14k_analysis_term_III} and \eqref{eqn:R14k_analysis_term_IV}. Denote $\bOmega^{(i)} = [\bomega^{(i)}_1,\ldots,\bomega^{(i)}_n]\transpose{}$, where $[\bomega^{(i)}_j]\transpose{} = \be_j\transpose{}(\bU_\bA^{(i)}\bH^{(i)} - \bU_\bA\bH)\bS_\bP^{1/2}\bW_\bX\bB_{nij}\transpose{}$, $j = 1,\ldots,n$. Then for the term \eqref{eqn:R14k_analysis_term_III}, we invoke Result \ref{result:spectral_norm_concentration} and Lemma \ref{lemma:auxiliary_matrix} to obtain that for sufficiently large $n$,
for all $t\geq 1$, $t\lesssim n\rho_n$,
\begin{align*}
&\left\|\frac{1}{n\rho_n}\sum_{j = 1}^n(A_{ij} - \rho_n\bx_{i}\transpose{}\bx_{j})\be_j\transpose{}(\bU_\bA\bH - \bU_\bA^{(i)}\bH^{(i)})\bS_\bP^{1/2}\bW_\bX\bB_{nij}\transpose{}\right\|_2\\
&\quad = \frac{1}{n\rho_n}\left\|\sum_{j = 1}^n[\bE]_{ij}\bomega_j^{(i)}\right\|_2 = \frac{1}{n\rho_n}\|\be_i\transpose{}\bE\bOmega^{(i)}\|_2\\
&\quad\leq \frac{1}{n\rho_n}\|\bE\|_{2\to\infty}\|\bOmega^{(i)}\|_{\mathrm{F}}
 = \frac{1}{n\rho_n}\|\bE\|_{2\to\infty}\left(\sum_{j = 1}^n\|\bomega_j^{(n)}\|_2^2\right)^{1/2}\\
&\quad\leq \frac{1}{n\rho_n}\|\bE\|_{2\to\infty}\left(\sum_{j = 1}^n\|\be_j(\bU_\bA^{(i)}\bH^{(i)} - \bU_\bA\bH)\|_2^2\|\bS_\bP\|_2\max_{j\in [n]}\|\bB_{nij}\|_2^2\right)^{1/2}\\
&\quad\leq \frac{1}{n\rho_n}\|\bE\|_{2}\left(d\|\bU_\bA^{(i)}\bH^{(i)} - \bU_\bA\bH\|_2^2\|\bS_\bP\|_2\max_{j\in [n]}\|\bB_{nij}\|_2^2\right)^{1/2}\\
&\quad\lesssim_c \frac{d^{1/2}}{(n\rho_n)^{1/2}}\{n\rho_n\lambda_1(\bDelta_n)\}^{1/2}\frac{1}{\delta^4}\frac{t^{1/2}\|\bU_\bP\|_{2\to\infty}}{(n\rho_n)^{1/2}\lambda_d(\bDelta_n)^2}\\
&\quad \leq \frac{\kappa(\bDelta_n)^{1/2}\|\bU_\bP\|_{2\to\infty}t^{1/2}}{(n\rho_n)^{1/2}\delta^4\lambda_d(\bDelta_n)^2}
\end{align*}
with probability at least $1 - c_0n^{-c} - c_0e^{-t}$. We finally turn our attention to the term \eqref{eqn:R14k_analysis_term_IV}. Denote $\bTheta^{(i)} = [\btheta_1^{(i)},\ldots,\btheta_n^{(i)}]\transpose{}$, where $(\btheta_j^{(i)})\transpose{} = \be_j\transpose{}(\bU_\bA^{(i)}\bH^{(i)} - \bU_\bP)\bS_\bP^{1/2}\bW_\bX\bB_{nj}\transpose{}$, $j = 1,\ldots,n$. Let $t\geq 1$ and $t\lesssim n\rho_n$. We take advantage of the fact that $\bTheta^{(i)}$ and $(A_{ij} - \rho_{n}\bx_{i}\transpose{}\bx_{j})_{j = 1}^n$ are independent and consider the following events:
\begin{align*}
&\calE_1 = \left\{\bA:\left\|\sum_{j = 1}^n(A_{ij} - \rho_n\bx_{i}\transpose{}\bx_{j})\btheta_j^{(i)}\right\|_2\leq 
t\|\bTheta^{(i)}\|_{2\to\infty} + (2t\rho_n)^{1/2}\|\bTheta^{(i)}\|_{\mathrm{F}}
\right\},\\
&\calE_2 = \left\{\bA:\|\bU_\bA^{(i)}\mathrm{sgn}(\bH^{(i)}) - \bU_\bP\|_{2\to\infty}\leq \frac{C_c\|\bU_\bP\|_{2\to\infty}}{ \lambda_d(\bDelta_n) }, \|\bU_\bA^{(i)}\|_{2\to\infty}\leq \frac{C_c\|\bU_\bP\|_{2\to\infty}}{ \lambda_d(\bDelta_n) }\right\},\\
&\calE_3 = \left\{\bA:\|\bU_\bA^{(i)}\bH^{(i)} - \bU_\bP\|_{2}\leq \frac{C_c}{(n\rho_n)^{1/2}\lambda_d(\bDelta_n)}\right\}.
\end{align*}
Here, $C_c > 0$ is a constant only depending on $c$ that will be determined later. 
By Lemma \ref{lemma:Bernstein_concentration_EW}, 
\[
\prob(\calE_1) = \sum_{\bA^{(m)}}\prob(\calE_1\mid\bA^{(m)})p(\bA^{(m)})\geq (1 - 28e^{-t})\sum_{\bA^{(m)}}p(\bA^{(m)}) = 1 - 28e^{-t}.
\]
By Lemma \ref{lemma:auxiliary_matrix}, 
for sufficiently large $n$,
$\prob(\calE_2)\geq 1 - 6n^{-c}$ and $\prob(\calE_3)\geq 1 - 3n^{-c}$. Hence, over the event $\calE_1\cap\calE_2\cap\calE_3$, which has probability at least $1 - c_0n^{-c} - c_0e^{-t}$, we have
\begin{align*}
&\left\|\frac{1}{n\rho_n}\sum_{j = 1}^n(A_{ij} - \rho_n\bx_{i}\transpose{}\bx_{j})\be_j\transpose{}(\bU_\bA^{(i)}\bH^{(i)} - \bU_\bP)\bS_\bP^{1/2}\bW_\bX\bB_{nij}\transpose{}\right\|_2\\
&\quad\leq \frac{3t}{n\rho_n}\|\bTheta^{(i)}\|_{2\to\infty} + \frac{(6\rho_nt)^{1/2}}{n\rho_n}\left(\sum_{j = 1}^n\|\btheta_j^{(i)}\|_2^2\right)^{1/2}\\
&\quad = \frac{3t}{n\rho_n}\max_{j\in [n]}\|\be_j\transpose{}(\bU_\bA^{(i)}\bH^{(i)} - \bU_\bP)\bS_\bP^{1/2}\bW_\bX\bB_{nij}\transpose{}\|_2\\
&\quad\quad + \frac{(6\rho_nt)^{1/2}}{n\rho_n}\left(\sum_{j = 1}^n\|\be_j\transpose{}(\bU_\bA^{(i)}\bH^{(i)} - \bU_\bP)\bS_\bP^{1/2}\bW_\bX\bB_{nij}\transpose{}\|_2^2\right)^{1/2}\\
&\quad\leq \frac{3t}{n\rho_n}(2\|\bU_\bA^{(i)}\|_{2\to\infty} + \|\bU_\bA^{(i)}\mathrm{sgn}(\bH^{(i)}) - \bU_\bP\|_{2\to\infty})\|\bS_\bP\|_2^{1/2}\max_{j\in [n]}\|\bB_{nij}\|_2\\
&\quad\quad + \frac{(6\rho_nt)^{1/2}}{n\rho_n}\left(\|\bU_\bA^{(i)}\bH^{(i)} - \bU_\bP\|_{\mathrm{F}}^2\|\bS_\bP\|_2^{1/2}\max_{j\in [n]}\|\bB_{nij}\transpose{}\|_2^2\right)^{1/2}\\
&\quad\lesssim_c \frac{t}{n\rho_n}\frac{\|\bU_\bP\|_{2\to\infty}}{ \lambda_d(\bDelta_n) }\{n\rho_n\lambda_1(\bDelta_n)\}^{1/2}\frac{1}{\delta^4} + \frac{(d\rho_nt)^{1/2}}{n\rho_n}\frac{\{n\rho_n\lambda_1(\bDelta_n)\}^{1/2}}{(n\rho_n)^{1/2}\lambda_d(\bDelta_n)}\frac{1}{\delta^4}\\
&\quad = \frac{\|\bU_\bP\|_{2\to\infty} t }{(n\rho_n)^{1/2}\delta^4\lambda_d(\bDelta_n)} + \frac{t^{1/2}} {(n\rho_n)^{1/2}\delta^4\lambda_d(\bDelta_n)}\sqrt{\frac{d}{n}}\lesssim \frac{\|\bU_\bP\|_{2\to\infty} t }{(n\rho_n)^{1/2}\delta^4\lambda_d(\bDelta_n)}.
\end{align*}
Combining the aforementioned concentration bounds for \eqref{eqn:R14k_analysis_term_I}, \eqref{eqn:R14k_analysis_term_II}, \eqref{eqn:R14k_analysis_term_III}, and \eqref{eqn:R14k_analysis_term_IV} completes the proof.  
\end{proof}


\subsection{Concentration bound for \eqref{eqn:ri1_analysis_term_V}} 
\label{sub:concentration_bound_for_eqn:ri1_analysis_term_v}

We now focus on the concentration bound for term \eqref{eqn:ri1_analysis_term_V} by taking advantage of the auxiliary matrices $\bU_\bA^{(i)},\bH^{(i)}$ defined in Section \ref{sec:entrywise_limit_theorem_for_the_eigenvectors}. Observe that term \eqref{eqn:ri1_analysis_term_V} consists of two terms: 
\[
\frac{1}{n\rho_n^{1/2}}\sum_{j = 1}^n[\bE]_{ij}\br_{\bg}
\quad\text{and}\quad
\frac{1}{n\rho_n^{1/2}}\sum_{j = 1}^n\rho_n\br_{\bh_{ij}}.
\]
The second term is relatively easy to analyze, whereas the first term is more involved. Recall that we assume
\[
\frac{\kappa(\bDelta_n)}{\lambda_d(\bDelta_n)(n\rho_n)^{1/2}}\to 0,\quad
\frac{1}{n\rho_n\lambda_d(\bDelta_n)^{3/2}}\max\left\{\frac{(\log n)^{1/2}}{\lambda_d(\bDelta_n)^3}, \frac{\kappa(\bDelta_n)}{\lambda_d(\bDelta_n)^3}, \frac{\log n}{\lambda_d(\bDelta_n)}\right\}\to 0.
\]
By Corollary \ref{corr:Two_to_infinity_norm_eigenvector_bound}, given any fixed $c > 0$, 
for sufficiently large $n$,
\begin{align*}
\|\rho_n^{-1/2}\widetilde\bX\bW - \bX\|_{2\to\infty}
&\lesssim_c \frac{\|\bU_\bP\|_{2\to\infty}}{n^{1/2}\rho_n\lambda_d(\bDelta_n)}\max\left\{\frac{(\log n)^{1/2}}{\lambda_d(\bDelta_n)^3}, \frac{\kappa(\bDelta_n)}{\lambda_d(\bDelta_n)^3}, \frac{\log n}{\lambda_d(\bDelta_n)}\right\}\\
&\quad + \frac{(\log n)^{1/2}}{\rho_n^{1/2}\lambda_d(\bDelta_n)^{1/2}}\|\bU_\bP\|_{2\to\infty}\\
&\leq \frac{1}{n \rho_n\lambda_d(\bDelta_n)^{3/2}}\max\left\{\frac{(\log n)^{1/2}}{\lambda_d(\bDelta_n)^3}, \frac{\kappa(\bDelta_n)}{\lambda_d(\bDelta_n)^3}, \frac{\log n}{\lambda_d(\bDelta_n)}\right\}\\
&\quad + \frac{(\log n)^{1/2}}{(n\rho_n)^{1/2}\lambda_d(\bDelta_n) } 
\end{align*}
with probability at least $1 - c_0n^{-c}$. By assumption, 
\[
\frac{1}{n \rho_n\lambda_d(\bDelta_n)^{3/2}}\max\left\{\frac{(\log n)^{1/2}}{\lambda_d(\bDelta_n)^3}, \frac{\kappa(\bDelta_n)}{\lambda_d(\bDelta_n)^3}, \frac{\log n}{\lambda_d(\bDelta_n)}\right\}\quad\Longrightarrow
\quad
\frac{(\log n)^{1/2}}{(n\rho_n)^{1/2}\lambda_d(\bDelta_n) }\to 0.
\]
Therefore, 
for sufficiently large $n$,
\[
(\rho_n^{-1/2}\bW\transpose{}\widetilde\bx_i, \rho_n^{-1/2}\bW\transpose{}\widetilde\bx_j)\in\calX_2(\delta/2)\quad\text{for all }i,j\in [n]
\] 
with probability at least $1 - c_0n^{-c}$. Then
we can apply Lemma \ref{lemma:gh_Taylor_expansion} to further obtain
\begin{align*}
&\left\|\frac{1}{n\rho_n^{1/2}}\sum_{j = 1}^n[\bE]_{ij}\br_{\bg}\right\|_2 + \left\|\frac{1}{n\rho_n^{1/2}}\sum_{j = 1}^n\rho_n\br_{\bh_{ij}}\right\|_2\\
&\quad\lesssim \frac{d^{1/2}}{n\rho_n^{1/2}\delta^6}\|\bE\|_{\infty}\|\rho_n^{-1/2}\bW\transpose{}\widetilde\bx_i -  \bx_i\|_2^2
 + \frac{d^{1/2}}{n\rho_n^{3/2}\delta^6}\sum_{j = 1}^n|[\bE]_{ij}|\|\bW\transpose{}\widetilde\bx_j - \rho_n^{1/2}\bx_j\|_2^2\\
&\quad\quad + \frac{d^{1/2}}{n\rho_n^{1/2}\delta^6}\sum_{j = 1}^n\|\bW\transpose{}\widetilde\bx_i - \rho_n^{1/2}\bx_i\|_2^2
 + \frac{d^{1/2}}{n\rho_n^{1/2}\delta^6}\sum_{j = 1}^n\|\bW\transpose{}\widetilde\bx_j - \rho_n^{1/2}\bx_j\|_2^2\\
&\quad= \frac{d^{1/2}(\|\bE\|_{\infty} + n\rho_n)}{n\rho_n^{3/2}\delta^6}\| \bW\transpose{}\widetilde\bx_i - \rho_n^{1/2}\bx_i\|_2^2
 + \frac{d^{1/2}}{n\rho_n^{3/2}\delta^6}\sum_{j = 1}^n|[\bE]_{ij}|\|\bW\transpose{}\widetilde\bx_j - \rho_n^{1/2}\bx_j\|_2^2\\
&\quad\quad + 
 \frac{d^{1/2}}{n\rho_n^{1/2}\delta^6}\|\widetilde\bX\bW - \rho_n^{1/2}\bX\|_{\mathrm{F}}^2
\end{align*}
with probability at least $1 - c_0n^{-c}$ for large $n$. By Result \ref{result:Infinity_norm_concentration}, Lemma \ref{lemma:RX_Frobenius_norm_bound}, and Lemma \ref{lemma:single_x_i_perturbation_bound}, 
for sufficiently large $n$,
\begin{align*}
&\|\bE\|_\infty\lesssim_c n\rho_n\quad\text{with probability at least }1 - c_0n^{-c},\\
&\|\widetilde\bX\bW - \rho_n^{1/2}\bX\|_{\mathrm{F}}\lesssim_c \frac{1}{\lambda_d(\bDelta_n)}\quad\text{with probability at least }1 - c_0n^{-c},\\
&\|\bW\transpose\widetilde\bx_i - \rho_n^{1/2}\bx_i\|_2\lesssim_c \frac{\|\bU_\bP\|_{2\to\infty}}{\lambda_d(\bDelta_n)}\max\left\{ \frac{t^{1/2}}{\lambda_d(\bDelta_n)^2}, \frac{\kappa(\bDelta_n)}{\lambda_d(\bDelta_n)^2}, t \right\}\\
&\quad\text{with probability at least }1 - c_0n^{-c} - c_0e^{-t}.
\end{align*}
It suffices to provide a concentration bound for
\[
 \frac{1}{n\rho_n^{3/2}}\sum_{j = 1}^n|[\bE]_{ij}|\|\bW\transpose{}\widetilde\bx_j - \rho_n^{1/2}\bx_j\|_2^2.
\]
\begin{lemma}\label{Lemma:Taylor_remainder_second_order_analysis}
Let $\bA\sim\mathrm{RDPG}(\rho_n^{1/2}\bX)$ and assume the conditions of Theorem \ref{thm:asymptotic_normality_OS} hold. Then given any fixed $c > 0$, for each fixed index $i\in [n]$, for all
$t\geq 1$, $t\lesssim n\rho_n$, and sufficiently large $n$,
\begin{align*}
&\frac{1}{n\rho_n^{3/2}}\sum_{j = 1}^n|[\bE]_{ij}|\|\bW\transpose\widetilde\bx_j - \rho_n^{1/2}\bx_{j}\|_2^2
\lesssim_c \frac{\|\bU_\bP\|_{2\to\infty}^2}{\rho_n^{1/2} \lambda_d(\bDelta_n)^{2}}
\max\left\{t, \frac{1}{\lambda_d(\bDelta_n)^2}\right\},\\
&\frac{1}{n\rho_n}\sum_{j = 1}^n|[\bE]_{ij}|\|\bW\transpose\widetilde\bx_j - \rho_n^{1/2}\bx_{j}\|_2
\\&\quad
\lesssim_c 
\frac{1}{\sqrt{n}\lambda_d(\bDelta_n)} + \frac{\|\bU_\bP\|_{2\to\infty}}{(n\rho_n)^{1/2}\lambda_d(\bDelta_n)}\max\left\{t, \frac{t^{1/2}}{\lambda_d(\bDelta_n)}, \frac{\kappa(\bDelta_n)}{\lambda_d(\bDelta_n)^2}\right\}
\end{align*}
with probability at least $1 - c_0n^{-c} - c_0e^{-t}$ for sufficiently large $n$, where $c_0 > 0$ is an absolute constant.
\end{lemma}

\begin{proof}[\bf Proof of Lemma \ref{Lemma:Taylor_remainder_second_order_analysis}]
The proof is quite similar to that of Lemma \ref{Lemma:R14k_analysis} modulus some slight modifications. Following the decomposition
\begin{align*}
\widetilde{\bX} \bW - \rho_n^{1/2}\bX
& = \bU_\bA\{\bS_\bA^{1/2}\mathrm{sgn}(\bH) - \mathrm{sgn}(\bH)\bS_\bP^{1/2}\}\bW_\bX
 + \bU_\bA\{\mathrm{sgn}(\bH) - \bH\}\bS_\bP^{1/2}\bW_\bX\\
&\quad + (\bU_\bA\bH - \bU_\bA^{(i)}\bH^{(i)})\bS_\bP^{1/2}\bW_\bX + (\bU_\bA^{(i)}\bH^{(i)} - \bU_\bP)\bS_\bP^{1/2}\bW_\bX
,
\end{align*}
we obtain from the Cauchy-Schwarz inequality that $(a + b + c + d)^2\leq 4a^2 + 4b^2 + 4c^2 + 4d^2$ and the triangle inequality that
\begin{align}
&\frac{1}{n\rho_n^{3/2}}\sum_{j = 1}^n|[\bE]_{ij}|\|\bW\transpose\widetilde\bx_j - \rho_n^{1/2}\bx_{j}\|_2^2\nonumber\\
&\quad
\leq \frac{1}{n\rho_n^{3/2}}\sum_{j = 1}^n(|[\bE]_{ij}| - \expect |[\bE]_{ij}|)\|\bW\transpose\widetilde\bx_j - \rho_n^{1/2}\bx_{j}\|_2^2 + \frac{2}{n\rho_n^{1/2}}\|\widetilde{\bX}\bW - \rho_n^{1/2}\bX\|_{\mathrm{F}}^2\nonumber\\
\label{eqn:Taylor_remainder_second_order_term_I}
&\quad\leq \frac{4}{n\rho_n^{3/2}}\sum_{j = 1}^n(|[\bE]_{ij}| - \expect |[\bE]_{ij}|)\|\be_j\transpose{}\bU_\bA\{\bS_\bA^{1/2}\mathrm{sgn}(\bH) - \mathrm{sgn}(\bH)\bS_\bP^{1/2}\}\bW_\bX\|_2^2\\
\label{eqn:Taylor_remainder_second_order_term_II}
&\quad\quad + \frac{4}{n\rho_n^{3/2}}\sum_{j = 1}^n(|[\bE]_{ij}| - \expect |[\bE]_{ij}|)\|\be_j\transpose{}\bU_\bA\{\mathrm{sgn}(\bH) - \bH\}\bS_\bP^{1/2}\bW_\bX\|_2^2\\
\label{eqn:Taylor_remainder_second_order_term_III}
&\quad\quad + \frac{4}{n\rho_n^{3/2}}\sum_{j = 1}^n(|[\bE]_{ij}| - \expect |[\bE]_{ij}|)\|\be_j\transpose{}(\bU_\bA\bH - \bU_\bA^{(i)}\bH^{(i)})\bS_\bP^{1/2}\bW_\bX\|_2^2\\
\label{eqn:Taylor_remainder_second_order_term_IV}
&\quad\quad + \frac{4}{n\rho_n^{3/2}}\sum_{j = 1}^n(|[\bE]_{ij}| - \expect |[\bE]_{ij}|)\|\be_j\transpose{}(\bU_\bA^{(i)}\bH^{(i)} - \bU_\bP)\bS_\bP^{1/2}\bW_\bX\|_2^2\\
\label{eqn:Taylor_remainder_second_order_term_V}
&\quad\quad + \frac{2}{n\rho_n^{1/2}}\|\widetilde{\bX}\bW - \rho_n^{1/2}\bX\|_{\mathrm{F}}^2
\end{align}
and
\begin{align}
&\frac{1}{n\rho_n}\sum_{j = 1}^n|[\bE]_{ij}|\|\bW\transpose\widetilde\bx_j - \rho_n^{1/2}\bx_{j}\|_2\nonumber\\
&\quad
\leq \frac{1}{n\rho_n}\sum_{j = 1}^n(|[\bE]_{ij}| - \expect |[\bE]_{ij}|)\|\bW\transpose\widetilde\bx_j - \rho_n^{1/2}\bx_{j}\|_2 + \frac{2}{n}\sum_{j = 1}^n\|\be_j\transpose(\widetilde{\bX}\bW - \rho_n^{1/2}\bX)\|_2\nonumber\\
\label{eqn:Taylor_remainder_first_order_term_I}
&\quad\leq \frac{1}{n\rho_n}\sum_{j = 1}^n(|[\bE]_{ij}| - \expect |[\bE]_{ij}|)\|\be_j\transpose{}\bU_\bA\{\bS_\bA^{1/2}\mathrm{sgn}(\bH) - \mathrm{sgn}(\bH)\bS_\bP^{1/2}\}\bW_\bX\|_2\\
\label{eqn:Taylor_remainder_first_order_term_II}
&\quad\quad + \frac{1}{n\rho_n}\sum_{j = 1}^n(|[\bE]_{ij}| - \expect |[\bE]_{ij}|)\|\be_j\transpose{}\bU_\bA\{\mathrm{sgn}(\bH) - \bH\}\bS_\bP^{1/2}\bW_\bX\|_2\\
\label{eqn:Taylor_remainder_first_order_term_III}
&\quad\quad + \frac{1}{n\rho_n}\sum_{j = 1}^n(|[\bE]_{ij}| - \expect |[\bE]_{ij}|)\|\be_j\transpose{}(\bU_\bA\bH - \bU_\bA^{(i)}\bH^{(i)})\bS_\bP^{1/2}\bW_\bX\|_2\\
\label{eqn:Taylor_remainder_first_order_term_IV}
&\quad\quad + \frac{1}{n\rho_n}\sum_{j = 1}^n(|[\bE]_{ij}| - \expect |[\bE]_{ij}|)\|\be_j\transpose{}(\bU_\bA^{(i)}\bH^{(i)} - \bU_\bP)\bS_\bP^{1/2}\bW_\bX\|_2\\
\label{eqn:Taylor_remainder_first_order_term_V}
&\quad\quad + \frac{2}{\sqrt{n}}\|\widetilde{\bX}\bW - \rho_n^{1/2}\bX\|_{\mathrm{F}}.
\end{align}
For terms \eqref{eqn:Taylor_remainder_second_order_term_V} and \eqref{eqn:Taylor_remainder_first_order_term_V}, we know from Lemma \ref{lemma:RX_Frobenius_norm_bound} that 
\begin{align*}
\frac{1}{n\rho_n^{1/2}}\|\widetilde\bX\bW - \rho_n^{1/2}\bX\|_{\mathrm{F}}^2
\lesssim_c \frac{1}{n\rho_n^{1/2}\lambda_d(\bDelta_n)^2},\quad
\frac{1}{\sqrt{n}}\|\widetilde\bX\bW - \rho_n^{1/2}\bX\|_{\mathrm{F}}
\lesssim_c \frac{1}{\sqrt{n}\lambda_d(\bDelta_n)}
\end{align*}
with probability at least $1 - c_0n^{-c}$ for sufficiently large $n$. 
For terms \eqref{eqn:Taylor_remainder_second_order_term_I} and \eqref{eqn:Taylor_remainder_first_order_term_I}, by Result \ref{result:Infinity_norm_concentration}, Lemma \ref{lemma:WS_interchange_bound}, and Lemma \ref{lemma:U_A_two_to_infinity_norm}, 
for all $t\geq 1$, $t\lesssim n\rho_n$, 
\begin{align*}
&\frac{4}{n\rho_n^{3/2}}\sum_{j = 1}^n(|[\bE]_{ij}| - \expect |[\bE]_{ij}|)\|\be_j\transpose{}\bU_\bA\{\bS_\bA^{1/2}\mathrm{sgn}(\bH) - \mathrm{sgn}(\bH)\bS_\bP^{1/2}\}\bW_\bX\|_2^2\\
&\quad\leq \frac{4}{n\rho_n^{3/2}}\sum_{j = 1}^n|[\bE]_{ij}|\|\bU_\bA\|_{2\to\infty}^2\|\bW^*\bS_\bA^{1/2} - \bS_\bP^{1/2}\bW^*\|_2^2\\
&\quad = \frac{4}{n\rho_n^{3/2}}\|\bE\|_{\infty}\|\bU_\bA\|_{2\to\infty}^2\|\bW^*\bS_\bA^{1/2} - \bS_\bP^{1/2}\bW^*\|_2^2\\
&\quad\lesssim_c \frac{1}{n\rho_n^{3/2}}n\rho_n\left\{\frac{\|\bU_\bP\|_{2\to\infty}}{\lambda_d(\bDelta_n) }\right\}^2
\left[\frac{1}{(n\rho_n)^{1/2}\lambda_d(\bDelta_n)}\max\left\{\frac{\kappa(\bDelta_n)}{\lambda_d(\bDelta_n)}, t^{1/2}\right\}\right]^2\\
&\quad = \frac{\|\bU_\bP\|_{2\to\infty}^2}{n \rho_n^{3/2}\lambda_d(\bDelta_n)^4}\max\left\{\frac{\kappa(\bDelta_n)^2}{\lambda_d(\bDelta_n)^2}, t\right\}
= \frac{\|\bU_\bP\|_{2\to\infty}^2}{\rho_n^{1/2}\lambda_d(\bDelta_n)^2}\max\left\{\frac{\kappa(\bDelta_n)^2}{n\rho_n\lambda_d(\bDelta_n)^4}, \frac{t}{n\rho_n\lambda_d(\bDelta_n)^2}\right\}\\
&\quad\lesssim \frac{\|\bU_\bP\|_{2\to\infty}^2}{\rho_n^{1/2}\lambda_d(\bDelta_n)^2}\max\left\{\frac{1}{\lambda_d(\bDelta_n)^2}, t\right\}
\end{align*}
and
\begin{align*}
&\frac{1}{n\rho_n}\sum_{j = 1}^n(|[\bE]_{ij}| - \expect |[\bE]_{ij}|)\|\be_j\transpose{}\bU_\bA\{\bS_\bA^{1/2}\mathrm{sgn}(\bH) - \mathrm{sgn}(\bH)\bS_\bP^{1/2}\}\bW_\bX\|_2\\
&\quad\leq \frac{1}{n\rho_n}\sum_{j = 1}^n|[\bE]_{ij}|\|\bU_\bA\|_{2\to\infty}\|\bW^*\bS_\bA^{1/2} - \bS_\bP^{1/2}\bW^*\|_2\\
&\quad = \frac{1}{n\rho_n}\|\bE\|_{\infty}\|\bU_\bA\|_{2\to\infty}\|\bW^*\bS_\bA^{1/2} - \bS_\bP^{1/2}\bW^*\|_2\\
&\quad\lesssim_c \frac{1}{n\rho_n}n\rho_n\left\{\frac{\|\bU_\bP\|_{2\to\infty}}{\lambda_d(\bDelta_n) }\right\}
\frac{1}{(n\rho_n)^{1/2}\lambda_d(\bDelta_n)}\max\left\{\frac{\kappa(\bDelta_n)}{\lambda_d(\bDelta_n)}, t^{1/2}\right\}\\
&\quad = \frac{\|\bU_\bP\|_{2\to\infty}}{(n \rho_n)^{1/2}\lambda_d(\bDelta_n)}\max\left\{\frac{\kappa(\bDelta_n)}{\lambda_d(\bDelta_n)^2}, \frac{t^{1/2}}{\lambda_d(\bDelta_n)}\right\}
\end{align*}
with probability at least $1 - c_0n^{-c} - c_0e^{-t}$. 
For terms \eqref{eqn:Taylor_remainder_second_order_term_II} and \eqref{eqn:Taylor_remainder_first_order_term_II}, by Lemma \ref{lemma:U_A_two_to_infinity_norm}, Lemma 6.7 in \cite{cape2017two}, and the Davis-Kahan theorem, 
\begin{align*}
&\frac{4}{n\rho_n^{3/2}}\sum_{j = 1}^n(|[\bE]_{ij}| - \expect |[\bE]_{ij}|)\|\be_j\transpose{}\bU_\bA\{\mathrm{sgn}(\bH) - \bH\}\bS_\bP^{1/2}\bW_\bX\|_2^2\\
&\quad\leq \frac{4}{n\rho_n^{3/2}}\sum_{j = 1}^n |[\bE]_{ij}| \| \bU_\bA\|_{2\to\infty}^2\|\bW^* - \bU_\bP\transpose{}\bU_\bA\|_2^2\|\bS_\bP\|_2
\\
&\quad
\leq \frac{4}{n\rho_n^{3/2}}\|\bE\|_\infty \| \bU_\bA\|_{2\to\infty}^2\|\sin\Theta(\bU_\bA,\bU_\bP)\|_2^4\|\bS_\bP\|_2\\
&\quad\leq \frac{4}{n\rho_n^{3/2}}\|\bE\|_\infty \| \bU_\bA\|_{2\to\infty}^2\frac{\|\bE\|_2^4}{\lambda_d(\bP)^4}\|\bS_\bP\|_2\\
&\quad\lesssim_c \frac{n\rho_n}{n\rho_n^{3/2}}\left\{\frac{\|\bU_\bP\|_{2\to\infty}}{\lambda_d(\bDelta_n) }\right\}^2\frac{(n\rho_n)^2}{(n\rho_n)^4\lambda_d(\bDelta_n)^4}(n\rho_n)\lambda_1(\bDelta_n)
 = \frac{\kappa(\bDelta_n)\|\bU_\bP\|_{2\to\infty}^2}{n \rho_n^{3/2}\lambda_d(\bDelta_n)^5}\\
&\quad = \frac{\|\bU_\bP\|_{2\to\infty}^2}{\rho_n^{1/2}\lambda_d(\bDelta_n)^2}\left\{\frac{\kappa(\bDelta_n)}{n\rho_n\lambda_d(\bDelta_n)^3}\right\}\leq \frac{\|\bU_\bP\|_{2\to\infty}^2}{\rho_n^{1/2}\lambda_d(\bDelta_n)^2}\left\{\frac{1}{\lambda_d(\bDelta_n)}\right\}. 
\end{align*}
and
\begin{align*}
&\frac{1}{n\rho_n}\sum_{j = 1}^n(|[\bE]_{ij}| - \expect |[\bE]_{ij}|)\|\be_j\transpose{}\bU_\bA\{\mathrm{sgn}(\bH) - \bH\}\bS_\bP^{1/2}\bW_\bX\|_2\\
&\quad\leq \frac{1}{n\rho_n }\sum_{j = 1}^n |[\bE]_{ij}| \| \bU_\bA\|_{2\to\infty}\|\bW^* - \bU_\bP\transpose{}\bU_\bA\|_2\|\bS_\bP\|_2^{1/2}
\\
&\quad
\leq \frac{1}{n\rho_n }\|\bE\|_\infty \| \bU_\bA\|_{2\to\infty} \|\sin\Theta(\bU_\bA,\bU_\bP)\|_2^2\|\bS_\bP\|_2^{1/2}\\
&\quad\leq \frac{1}{n\rho_n }\|\bE\|_\infty \| \bU_\bA\|_{2\to\infty} \frac{\|\bE\|_2^2}{\lambda_d(\bP)^2}\|\bS_\bP\|_2^{1/2}\\
&\quad\lesssim_c \frac{n\rho_n}{n\rho_n}\left\{\frac{\|\bU_\bP\|_{2\to\infty}}{\lambda_d(\bDelta_n) }\right\}\frac{(n\rho_n)}{(n\rho_n)^2\lambda_d(\bDelta_n)^2}(n\rho_n)^{1/2}\lambda_1(\bDelta_n)^{1/2}
\\&\quad
 = \frac{ \|\bU_\bP\|_{2\to\infty}}{(n \rho_n)^{1/2}\lambda_d(\bDelta_n)}\left\{\frac{1}{\lambda_d(\bDelta_n)^2}\right\}. 
\end{align*}
with probability at least $1 - c_0n^{-c}$ for sufficiently large $n$. 
For terms \eqref{eqn:Taylor_remainder_second_order_term_III} and \eqref{eqn:Taylor_remainder_first_order_term_III}, we invoke Lemma \ref{lemma:auxiliary_matrix} to obtain that
\begin{align*}
&\frac{4}{n\rho_n^{3/2}}\sum_{j = 1}^n(|[\bE]_{ij}| - \expect |[\bE]_{ij}|)\|\be_j\transpose{}(\bU_\bA\bH - \bU_\bA^{(i)}\bH^{(i)})\bS_\bP^{1/2}\bW_\bX\|_2^2\\
&\quad\lesssim_c \frac{1}{n\rho_n^{3/2}}\sum_{j = 1}^n\|\be_j\transpose{}(\bU_\bA\bH - \bU_\bA^{(i)}\bH^{(i)})\|_2^2\|\bS_\bP\|_2
 = \frac{d}{n\rho_n^{3/2}}\|\bU_\bA\bH - \bU_\bA^{(i)}\bH^{(i)}\|_2^2\|\bS_\bP\|_2\\
&\quad\lesssim \frac{d}{n\rho_n^{3/2}}\frac{t\|\bU_\bP\|_{2\to\infty}^2}{n\rho_n\lambda_d(\bDelta_n)^4}n\rho_n\lambda_1(\bDelta_n)\leq \frac{\kappa(\bDelta_n)\|\bU_\bP\|_{2\to\infty}^2t}{n\rho_n^{3/2}\lambda_d(\bDelta_n)^4}\\
&\quad = \frac{\|\bU_\bP\|_{2\to\infty}^2}{\rho_n^{1/2}\lambda_d(\bDelta_n)^2}\left\{\frac{\kappa(\bDelta_n)t}{n\rho_n\lambda_d(\bDelta_n)^2}\right\}\lesssim \frac{\|\bU_\bP\|_{2\to\infty}^2t}{\rho_n^{1/2}\lambda_d(\bDelta_n)^2}
\end{align*}
and
\begin{align*}
&\frac{1}{n\rho_n}\sum_{j = 1}^n(|[\bE]_{ij}| - \expect |[\bE]_{ij}|)\|\be_j\transpose{}(\bU_\bA\bH - \bU_\bA^{(i)}\bH^{(i)})\bS_\bP^{1/2}\bW_\bX\|_2\\
&\quad\lesssim_c \frac{1}{n\rho_n}\sum_{j = 1}^n\|\be_j\transpose{}(\bU_\bA\bH - \bU_\bA^{(i)}\bH^{(i)})\|_2\|\bS_\bP\|_2^{1/2}
 = \frac{d^{1/2}}{n\rho_n }\|\bU_\bA\bH - \bU_\bA^{(i)}\bH^{(i)}\|_2\|\bS_\bP\|_2^{1/2}\\
&\quad\lesssim \frac{d^{1/2}}{n\rho_n}\frac{t^{1/2}\|\bU_\bP\|_{2\to\infty}}{(n\rho_n)^{1/2}\lambda_d(\bDelta_n)^2}(n\rho_n)^{1/2}\lambda_1(\bDelta_n)^{1/2}\leq \frac{\|\bU_\bP\|_{2\to\infty}\kappa(\bDelta_n)^{1/2}t^{1/2}}{n\rho_n\lambda_d(\bDelta_n)^2}\\
&\quad = \frac{\|\bU_\bP\|_{2\to\infty}}{(n\rho_n)^{1/2}\lambda_d(\bDelta_n)^2}\left\{\frac{\kappa(\bDelta_n)^{1/2}t^{1/2}}{(n\rho_n)^{1/2}}\right\} \lesssim \frac{\|\bU_\bP\|_{2\to\infty}}{(n\rho_n)^{1/2}\lambda_d(\bDelta_n)^2}\kappa(\bDelta_n)^{1/2}
\end{align*}
with probability at least $1 - c_0n^{-c} - c_0e^{-t}$ for sufficiently large $n$. Finally, for terms \eqref{eqn:Taylor_remainder_second_order_term_IV} and \eqref{eqn:Taylor_remainder_first_order_term_IV}, we denote $\bTheta^{(i)} = [\btheta_1^{(i)},\ldots,\btheta_n^{(i)}]\transpose{}$, where $(\btheta_j^{(i)})\transpose{} = \be_j\transpose (\bU_\bA^{(i)} \bH^{(i)} - \bU_\bP)\bS_\bP^{1/2}\bW_\bX$, $j = 1,\ldots,n$ and consider the following events that are similar to those in the proof of Lemma \ref{Lemma:R14k_analysis}:
\begin{align*}
&\calE_1  = \left\{\bA:\sum_{j = 1}^n(|[\bE]_{ij}| - \expect |[\bE]_{ij}|)\|\btheta_j^{(i)}\|_2^2\leq t\max_{j\in [n]}\|\btheta_{j}^{(i)}\|_2^2 + (2\rho_nt)^{1/2}\mathrel{\Big(}\sum_{j = 1}^n\|\btheta_j^{(i)}\|_2^4\mathrel{\Big)}^{1/2}\right\},\\
&\calE_1'  = \left\{\bA:\sum_{j = 1}^n(|[\bE]_{ij}| - \expect |[\bE]_{ij}|)\|\btheta_j^{(i)}\|_2\leq t\max_{j\in [n]}\|\btheta_{j}^{(i)}\|_2 + (2\rho_nt)^{1/2}\|\bTheta_j^{(i)}\|_{\mathrm{F}}\right\},\\
&\calE_2 = \left\{\bA:\|\bU_\bA^{(i)}\mathrm{sgn}(\bH^{(i)}) - \bU_\bP\|_{2\to\infty}\leq \frac{C_c\|\bU_\bP\|_{2\to\infty}}{\lambda_d(\bDelta_n)}, \|\bU_\bA^{(i)}\|_{2\to\infty}\leq \frac{C_c\|\bU_\bP\|_{2\to\infty}}{\lambda_d(\bDelta_n)}\right\},\\
&\calE_3 = \left\{\bA:\|\bU_\bA^{(i)}\bH^{(i)} - \bU_\bP\|_{2}\leq \frac{C_c}{(n\rho_n)^{1/2}\lambda_d(\bDelta_n)}\right\},
\end{align*}
where $C_c$ is a constant only depending on $c$ and will be selcted later. 
By Lemma \ref{lemma:Bernstein_concentration_EW} and the independence between $\bTheta^{(i)}$ and $([\bE]_{ij})_{j = 1}^n$, we have $\prob(\calE_1)\geq 1 - 28e^{-t}$ and $\prob(\calE_1')\geq 1 - 28e^{-t}$. By Lemma \ref{lemma:auxiliary_matrix}, for sufficiently large $n$,
$\prob(\calE_2) \geq 1 - 6n^{-c}$ and $\prob(\calE_3)\geq 1 - 3n^{-c}$. Hence, over the event $\calE_1\cap \calE_2\cap \calE_3$, we have
\begin{align*}
&\frac{4}{n\rho_n^{3/2}}\sum_{j = 1}^n(|[\bE]_{ij}| - \expect |[\bE]_{ij}|)\|\be_j\transpose{}(\bU_\bA^{(i)}\bH^{(i)} - \bU_\bP)\bS_\bP^{1/2}\bW_\bX\|_2^2\\
&\quad\leq \frac{4t}{n\rho_n^{3/2}}\max_{j\in [n]}\|\be_j\transpose(\bU_\bA^{(i)}\bH^{(i)} - \bU_\bP)\|_2^2\|\bS_\bP\|_2
\\&\quad\quad
 + \frac{4(2\rho_nt)^{1/2}}{n\rho_n^{3/2}}\left\{\sum_{j = 1}^n\|\be_j\transpose(\bU_\bA^{(i)}\bH^{(i)} - \bU_\bP)\|_2^4\|\bS_\bP\|_2^2\right\}^{1/2}\\
&\quad\leq \frac{4t}{n\rho_n^{3/2}}\|\bU_\bA^{(i)}\bH^{(i)} - \bU_\bP\|_{2\to\infty}^2\|\bS_\bP\|_2\\
&\quad\quad + \frac{4(2\rho_nt)^{1/2}}{n\rho_n^{3/2}}\|\bS_\bP\|_2\|\bU_\bA^{(i)}\bH^{(i)} - \bU_\bP\|_{2\to\infty}\|\bU_\bA^{(i)}\bH^{(i)} - \bU_\bP\|_{\mathrm{F}}\\
&\quad\leq \frac{4t}{n\rho_n^{3/2}}\|\bU_\bA^{(i)}\bH^{(i)} - \bU_\bP\|_{2\to\infty}^2\|\bS_\bP\|_2\\
&\quad\quad + \frac{4(2d\rho_nt)^{1/2}}{n\rho_n^{3/2}}\|\bS_\bP\|_2\|\bU_\bA^{(i)}\bH^{(i)} - \bU_\bP\|_{2\to\infty}\|\bU_\bA^{(i)}\bH^{(i)} - \bU_\bP\|_{2}\\
&\quad\leq \frac{4t}{n\rho_n^{3/2}}(2\|\bU_\bA^{(i)}\|_{2\to\infty} + \|\bU_\bA^{(i)}\mathrm{sgn}(\bH^{(i)}) - \bU_\bP\|_{2\to\infty})^2\|\bS_\bP\|_2\\
&\quad\quad + \frac{4(2d\rho_nt)^{1/2}}{n\rho_n^{3/2}}\|\bS_\bP\|_2\|\bU_\bA^{(i)}\bH^{(i)} - \bU_\bP\|_{2\to\infty}\|\bU_\bA^{(i)}\bH^{(i)} - \bU_\bP\|_{2}\\
&\quad\lesssim_c 
\frac{t}{n\rho_n^{3/2}}\frac{\|\bU_\bP\|_{2\to\infty}^2}{\lambda_d(\bDelta_n)^2}n\rho_n\lambda_1(\bDelta_n) + \frac{(dt)^{1/2}n\rho_n\lambda_1(\bDelta_n)}{n\rho_n}\frac{\|\bU_\bP\|_{2\to\infty}}{\lambda_d(\bDelta_n)}\frac{1}{(n\rho_n)^{1/2}\lambda_d(\bDelta_n)}\\
&\quad = \frac{t\|\bU_\bP\|_{2\to\infty}^2}{\rho_n^{1/2}\lambda_d(\bDelta_n)^2} + \frac{t^{1/2}\|\bU_\bP\|_{2\to\infty}}{\rho_n^{1/2}\lambda_d(\bDelta_n)^2}\sqrt{\frac{d}{n}}\lesssim \frac{t\|\bU_\bP\|_{2\to\infty}^2}{\rho_n^{1/2}\lambda_d(\bDelta_n)^2}.
\end{align*}
Similarly, over the event $\calE_1'\cap \calE_2\cap \calE_3$, we have
\begin{align*}
&\frac{1}{n\rho_n}\sum_{j = 1}^n(|[\bE]_{ij}| - \expect |[\bE]_{ij}|)\|\be_j\transpose{}(\bU_\bA^{(i)}\bH^{(i)} - \bU_\bP)\bS_\bP^{1/2}\bW_\bX\|_2\\
&\quad\leq \frac{t}{n\rho_n }\max_{j\in [n]}\|\be_j\transpose(\bU_\bA^{(i)}\bH^{(i)} - \bU_\bP)\|_2\|\bS_\bP\|_2^{1/2}
\\&\quad\quad
 + \frac{(2\rho_nt)^{1/2}}{n\rho_n}\left\{\sum_{j = 1}^n\|\be_j\transpose(\bU_\bA^{(i)}\bH^{(i)} - \bU_\bP)\|_2^2\|\bS_\bP\|_2\right\}^{1/2}\\
&\quad\leq \frac{t}{n\rho_n }\|\bU_\bA^{(i)}\bH^{(i)} - \bU_\bP\|_{2\to\infty}\|\bS_\bP\|_2^{1/2}
 + \frac{ (2\rho_nt)^{1/2}}{n\rho_n }\|\bS_\bP\|_2^{1/2} \|\bU_\bA^{(i)}\bH^{(i)} - \bU_\bP\|_{\mathrm{F}}\\
&\quad\leq \frac{t}{n\rho_n }\|\bU_\bA^{(i)}\bH^{(i)} - \bU_\bP\|_{2\to\infty}\|\bS_\bP\|_2^{1/2}
 + \frac{ (2d\rho_nt)^{1/2}}{n\rho_n }\|\bS_\bP\|_2^{1/2}\|\bU_\bA^{(i)}\bH^{(i)} - \bU_\bP\|_{2}\\
&\quad\leq \frac{t}{n\rho_n }(2\|\bU_\bA^{(i)}\|_{2\to\infty} + \|\bU_\bA^{(i)}\mathrm{sgn}(\bH^{(i)}) - \bU_\bP\|_{2\to\infty})\|\bS_\bP\|_2^{1/2}
\\
&\quad\quad
 + \frac{(2d\rho_nt)^{1/2}}{n\rho_n}\|\bS_\bP\|_2^{1/2}\|\bU_\bA^{(i)}\bH^{(i)} - \bU_\bP\|_{2}\\
&\quad\lesssim_c 
\frac{t}{n\rho_n}\frac{\|\bU_\bP\|_{2\to\infty}}{\lambda_d(\bDelta_n)}\{n\rho_n\lambda_1(\bDelta_n)\}^{1/2}
 + \frac{(d\rho_nt)^{1/2}(n\rho_n)^{1/2}\lambda_1(\bDelta_n)^{1/2}}{n\rho_n}\frac{1}{(n\rho_n)^{1/2}\lambda_d(\bDelta_n)}\\
&\quad = \frac{t\|\bU_\bP\|_{2\to\infty} }{(n\rho_n)^{1/2}\lambda_d(\bDelta_n)} + \frac{t^{1/2}}{(n\rho_n)^{1/2}\lambda_d(\bDelta_n)}\sqrt{\frac{d}{n}}\lesssim \frac{t\|\bU_\bP\|_{2\to\infty}}{(n\rho_n)^{1/2}\lambda_d(\bDelta_n)}.
\end{align*}
The events $\calE_1\cap\calE_2\cap\calE_3$ and  $\calE_1'\cap\calE_2\cap\calE_3$ both occur with probability at least $1 - c_0n^{-c} - c_0e^{-t}$. The proof is completed by combining the concentration bounds above. 
\end{proof}

We now combine the aforementioned analysis to obtain the concentration bound for \eqref{eqn:ri1_analysis_term_V}. 
\begin{lemma}\label{Lemma:Taylor_remainder_error_bound}
Let $\bA\sim\mathrm{RDPG}(\rho_n^{1/2}\bX)$ and assume the conditions of Theorem \ref{thm:asymptotic_normality_OS} hold. Then given any fixed $c > 0$, for each fixed index $i\in [n]$, 
 for all $t\geq 1$, $t\lesssim n\rho_n$, and sufficiently large $n$,
\begin{align*}
&\left\|\frac{1}{n\rho_n^{1/2}}\sum_{j = 1}^n[\bE]_{ij}\br_{\bg} + \frac{1}{n\rho_n^{1/2}}\rho_n\br_{\bh_{ij}}\right\|_2
\lesssim_c \frac{d^{1/2}\|\bU_\bP\|_{2\to\infty}^2}{\rho_n^{1/2}\delta^6\lambda_d(\bDelta_n)^2}
\max\left\{\frac{t}{\lambda_d(\bDelta_n)^4}, \frac{\kappa(\bDelta_n)^2}{\lambda_d(\bDelta_n)^4}, t^2\right\}
\end{align*}
with probability at least $1 - c_0n^{-c} - c_0e^{-t}$ for sufficiently large $n$, where $c_0 > 0$ is an absolute constant. 
\end{lemma}


\subsection{Concentration bound for \eqref{eqn:ri1}} 
\label{sub:concentration_bound_for_eqn:ri1}

We are now in a position to obtain a concentration bound for term \eqref{eqn:ri1} by collecting the results in Sections \ref{sub:concentration_bound_for_eqn:ri1_analysis_term_ii}, \ref{sub:concentration_bound_for_eqn:ri1_analysis_term_iii}, \ref{sub:concentration_bound_for_eqn:ri1_analysis_term_iv}, and \ref{sub:concentration_bound_for_eqn:ri1_analysis_term_v}. 

\begin{lemma}\label{lemma:Taylor_expansion_phi}
Let $\bA\sim\mathrm{RDPG}(\rho_n^{1/2}\bX)$ and assume the conditions of Theorem \ref{thm:asymptotic_normality_OS} hold. 
Let $\bphi_{ij}(\bu, \bv)$, $\bg$, $\bh_{ij}$ be defined as in Section \ref{sub:outline_of_the_proof_of_theorem_thm:asymptotic_normality_os} and $\br_{i1}$ be defined as in \eqref{eqn:ri1}. 
Then given any fixed $c > 0$, for each fixed row index $i\in [n]$, 
 for all $t\geq 1$, $t\lesssim \log n$, and sufficiently large $n$,
\begin{align*}
&\|\br_{i1}\|_2
\lesssim_c\frac{\|\bU_\bP\|_{2\to\infty}}{\delta^4\lambda_d(\bDelta_n)}\max\left\{
\frac{t^{1/2}}{\lambda_d(\bDelta_n)^2}, \frac{\kappa(\bDelta_n)}{\lambda_d(\bDelta_n)^2}, t
\right\},\\
&\left\|\br_{i1} + \frac{1}{n\rho_n^{1/2}}\sum_{j = 1}^n
\frac{\rho_n\bx_{j}\bx_{j}\transpose{}(\rho_n^{-1/2}\bW\transpose{}\widetilde{\bx}_i - \bx_{i})}{\bx_{i}\transpose{}\bx_{j}(1 - \rho_n\bx_{i}\transpose{}\bx_{j})}\right\|_2
\\&\quad 
\lesssim_c \frac{d^{1/2}\|\bU_\bP\|_{2\to\infty}}{(n\rho_n)^{1/2}\delta^6\lambda_d(\bDelta_n)^{5/2}}
\max\left\{\frac{t}{\lambda_d(\bDelta_n)^4}, \frac{\kappa(\bDelta_n)^2}{\lambda_d(\bDelta_n)^4}, t^2\right\}
\end{align*}
with probability at least $1 - c_0n^{-c} - c_0e^{-t}$ for sufficiently large $n$, where $c_0 > 0$ is an absolute constants. 
\end{lemma}

\begin{proof}[\bf Proof of Lemma \ref{lemma:Taylor_expansion_phi}]
For convenience, we denote $\widetilde{\by}_i = \rho_n^{-1/2}\bW\transpose\widetilde{\bx}_i$, $i \in [n]$. We first show that 
for sufficiently large $n$,
$(\widetilde\by_i, \widetilde\by_j)\in\calX_2(\delta/2)$ for all $(i, j)\in [n]\times [n]$ with large probability. 
By Corollary \ref{corr:Two_to_infinity_norm_eigenvector_bound}, for all $c > 0$, 
\begin{align*}
\max_{i\in [n]}\|\widetilde\by_i - \bx_i\|_2
& = \rho_n^{-1/2}\|\widetilde\bX\bW - \rho_n^{1/2}\bX\|_{2\to\infty}\\
&\lesssim_c \frac{1}{n \rho_n\lambda_d(\bDelta_n)^{3/2}}\max\left\{\frac{(\log n)^{1/2}}{\lambda_d(\bDelta_n)^3}, \frac{\kappa(\bDelta_n)}{\lambda_d(\bDelta_n)^3}, \frac{\log n}{\lambda_d(\bDelta_n)}\right\}
 + \frac{(\log n)^{1/2}}{(n\rho_n)^{1/2}\lambda_d(\bDelta_n) }.
\end{align*}
 with probability at least $1 - c_0n^{-c}$ for sufficiently large $n$. 
The upper bound on the preceeding display converges to $0$ as $n\to\infty$ by our assumption. Therefore, for sufficiently large $n$,
with probability at least $1 - c_0n^{-c}$,
\[
\max_{i\in [n]}\|\widetilde\by_i - \bx_i\|_2\leq \frac{\delta}{6}.
\]
Then for any $i,j\in [n]$, with probability at least $1 - c_0n^{-c}$,
\begin{align*}
\widetilde\by_i\transpose{}\widetilde\by_j
& = (\widetilde\by_i - \bx_{i} + \bx_{i})\transpose{}(\widetilde\by_j - \bx_{j} + \bx_{j})
  \leq \bx_{i}\transpose{}\bx_{j} + \|\widetilde\by_i - \bx_{i}\|_2^2 + 2\|\widetilde\by_i - \bx_{i}\|_2\\
&\leq \bx_{i}\transpose{}\bx_{j} + 3\|\widetilde\by_i - \bx_{i}\|_2\leq 1 - \delta + 3(\delta/6) = 1 - \delta/2,\\
\widetilde\by_i\transpose{}\widetilde\by_j
&\geq  \bx_{i}\transpose{}\bx_{j} - 3\|\widetilde\by_i - \bx_{i}\|_2\geq \delta - 3(\delta/6) = \delta/2. 
\end{align*}
This further implies that $\|\widetilde\by_i\|_2\leq 1$ for all $i\in [n]$, and hence, $(\widetilde\by_i, \widetilde\by_j)\in\calX_2(\delta/2)$ for all $(i, j)\in [n]\times [n]$ with probability at least $1 - c_0n^{-c}$ for sufficiently large $n$. 

\vspace*{1ex}\noindent 
We now proceed to the first assertion. 
Invoking assertion (a) of Lemma \ref{lemma:gh_Taylor_expansion}, we have
\begin{align*}
\left\| \bphi_{ij}(\widetilde\by_i, \widetilde\by_j) - \bphi_{ij}(\bx_{i}, \bx_{j})\right\|_2\lesssim (|[\bE]_{ij}| + \rho_n)\|\rho_n^{-1/2}\widetilde{\bX}\bW - \bX\|_{2\to\infty}
\end{align*}
with probability at least $1 - c_0n^{-c}$ for sufficiently large $n$. 
It follows from Result \ref{result:Infinity_norm_concentration}, Lemma \ref{lemma:single_x_i_perturbation_bound}, and Lemma \ref{Lemma:Taylor_remainder_second_order_analysis} that, for any $t\geq 1$ and $t\lesssim \log n$,
\begin{align*}
&\frac{1}{n\rho_n^{1/2}}\left\|\sum_{j = 1}^n\left\{\bphi_{ij}(\widetilde\by_i, \widetilde\by_j) - \bphi_{ij}(\bx_{i}, \bx_{j})\right\}\right\|_2\\
&\quad\lesssim \frac{1}{\delta^4}\left(\|\bE\|_\infty + n\rho_n\right)\|\bW\transpose\widetilde{\bx}_i - \rho_n^{1/2}\bx_i\|_2 + \frac{1}{n\rho_n\delta^4}\sum_{j = 1}^n|[\bE]_{ij}|\|\bW\transpose\widetilde{\bx}_j - \rho_n^{1/2}\bx_j\|_2\\
&\quad\quad + \frac{1}{n\delta^4}\sum_{j = 1}^n\|\bW\transpose\widetilde{\bx}_j - \rho_n^{1/2}\bx_j\|_2\\
&\quad\lesssim_c \frac{n\rho_n}{n\rho_n\delta^4}\frac{\|\bU_\bP\|_{2\to\infty}}{\lambda_d(\bDelta_n)}\max\left\{
\frac{t^{1/2}}{\lambda_d(\bDelta_n)^2}, \frac{\kappa(\bDelta_n)}{\lambda_d(\bDelta_n)^2}, t
\right\} + \frac{1}{\sqrt{n}\delta^4\lambda_d(\bDelta_n)}\\
&\quad\quad + \frac{\|\bU_\bP\|_{2\to\infty}}{(n\rho_n)^{1/2}\lambda_d(\bDelta_n)}\max\left\{t, \frac{t^{1/2}}{\lambda_d(\bDelta_n)}, \frac{\kappa(\bDelta_n)}{\lambda_d(\bDelta_n)^2}\right\} + \frac{1}{\sqrt{n}\delta^4\lambda_d(\bDelta_n)}\\
&\quad\lesssim_c \frac{\|\bU_\bP\|_{2\to\infty}}{\delta^4\lambda_d(\bDelta_n)}\max\left\{
\frac{t^{1/2}}{\lambda_d(\bDelta_n)^2}, \frac{\kappa(\bDelta_n)}{\lambda_d(\bDelta_n)^2}, t
\right\} + \frac{1}{\sqrt{n}\delta^4\lambda_d(\bDelta_n)}\\
&\quad\quad + \frac{\|\bU_\bP\|_{2\to\infty}}{\lambda_d(\bDelta_n)}\max\left\{\frac{t}{(n\rho_n)^{1/2}}, \frac{t^{1/2}}{(n\rho_n)^{1/2}\lambda_d(\bDelta_n)}, \frac{\kappa(\bDelta_n)}{(n\rho_n)^{1/2}\lambda_d(\bDelta_n)^2}\right\}\\
&\quad\lesssim \frac{\|\bU_\bP\|_{2\to\infty}}{\delta^4\lambda_d(\bDelta_n)}\max\left\{
\frac{t^{1/2}}{\lambda_d(\bDelta_n)^2}, \frac{\kappa(\bDelta_n)}{\lambda_d(\bDelta_n)^2}, t
\right\}
\end{align*}
with probability at least $1 - c_0n^{-c} - c_0e^{-t}$ for sufficiently large $n$. 
This completes the proof of the first assertion. 
For the second assertion, we recall that
\begin{align*}
\left\|\br_{i1} + \frac{1}{n\rho_n^{1/2}}\sum_{j = 1}^n\frac{\rho_n\bx_j\bx_j\transpose{}}{\bx_i\transpose{}\bx_j(1 - \rho_n\bx_i\transpose{}\bx_j)}(\rho_n^{-1/2}\bW\transpose{}\widetilde\bx_i - \bx_i)\right\|_2
\end{align*}
can be decomposed into the four terms \eqref{eqn:ri1_analysis_term_II}, \eqref{eqn:ri1_analysis_term_III}, \eqref{eqn:ri1_analysis_term_IV}, and \eqref{eqn:ri1_analysis_term_V}. 
More specifically, 
\begin{align*}
&\br_{i1} + \frac{1}{n\rho_n^{1/2}}\sum_{j = 1}^n\frac{\rho_n\bx_j\bx_j\transpose{}}{\bx_i\transpose{}\bx_j(1 - \rho_n\bx_i\transpose{}\bx_j)}(\rho_n^{-1/2}\bW\transpose{}\widetilde{\bx}_i - \bx_i)\\
&\quad = - \frac{1}{n\rho_n^{1/2}}\sum_{j = 1}^n\frac{\rho_n \bx_{j}\bx_{i}\transpose}{\bx_{i}\transpose\bx_{j}(1 - \rho_n\bx_{i}\transpose\bx_{j})}
(\rho_n^{-1/2}\bW\transpose{}\widetilde{\bx}_j - \bx_{j})\\
&\quad  \quad - \frac{1}{n\rho_n^{1/2}}\sum_{j = 1}^n\left[\frac{[\bE]_{ij}(1 - 2\rho_n\bx_{i}\transpose\bx_{j})\bx_{j}\bx_{j}\transpose}{\{\bx_{i}\transpose\bx_{j}(1 - \rho_n\bx_{i}\transpose\bx_{j})\}^2}\right](\rho_n^{-1/2}\bW\transpose\widetilde{\bx}_i - \bx_{i})\\
&\quad \quad + \frac{1}{n\rho_n^{1/2}}\sum_{j = 1}^n
\frac{[\bE]_{ij}\{\bx_{i}\transpose\bx_{j}(1 - \rho_n\bx_{i}\transpose\bx_{j})\eye_d - (1 - 2\rho_n\bx_{i}\transpose\bx_{j})\bx_{j}\bx_{i}\transpose\}}
{\{\bx_{i}\transpose\bx_{j}(1 - \rho_n\bx_{i}\transpose\bx_{j})\}^2}
(\rho_n^{-1/2}\bW\transpose\widetilde{\bx}_j - \bx_{j})\\
&\quad \quad + \frac{1}{n\rho_n^{1/2}}\sum_{j = 1}^n[\bE]_{ij}\br_{\bg}
 + \frac{1}{n\rho_n^{1/2}}\sum_{j = 1}^n\rho_n\br_{\bh_{ij}}
.
\end{align*}
By Lemma \ref{lemma:R12k_analysis}, 
for all $t\geq 1$, $t\lesssim n\rho_n$, for sufficiently large $n$, with probability at least $1 - c_0n^{-c} - c_0e^{-t}$, 
\begin{align*}
\left\|
\frac{1}{n\rho_n^{1/2}}\sum_{j = 1}^n
\frac{\rho_n \bx_{j}\bx_{i}\transpose(\rho_n^{-1/2}\bW\transpose\widetilde{\bx}_j - \bx_{j})}{\bx_{i}\transpose\bx_{j}(1 - \rho_n\bx_{i}\transpose\bx_{j})}
\right\|_2\lesssim_c 
\frac{1}{n\rho_n^{1/2}\delta^2\lambda_d(\bDelta_n)}\max\left\{\frac{\kappa(\bDelta_n)}{\lambda_d(\bDelta_n)}, t^{1/2}\right\}. 
\end{align*}
We next apply Lemma \ref{lemma:R13k_analysis} with 
\[
\bB_{nij} = \frac{(1 - 2\rho_n\bx_{i}\transpose{}\bx_{j})\bx_{j}\bx_{j}\transpose{}}{\{\bx_{i}\transpose{}\bx_{j}(1 - \rho_n\bx_{i}\transpose{}\bx_{j})\}^2}
\]
to obtain that 
for sufficiently large $n$,
for all $t\geq 1$, $t\lesssim n\rho_n$, with probability at least $1 - c_0n^{-c} - c_0e^{-t}$, 
\begin{align*}
&\left\|
\frac{1}{n\rho_n^{1/2}}\sum_{j = 1}^n[\bE]_{ij}
\frac{(1 - 2\rho_n\bx_{i}\transpose{}\bx_{j})\bx_{j}\bx_{j}\transpose(\rho_n^{-1/2}\bW\transpose\widetilde{\bx}_j - \bx_{j})}{\{\bx_{i}\transpose\bx_{j}(1 - \rho_n\bx_{i}\transpose\bx_{j})\}^2}
\right\|_2\\
&\quad \lesssim_c \frac{\|\bU_\bP\|_{2\to\infty}}{(n\rho_n)^{1/2}\delta^4\lambda_d(\bDelta_n)}
\max\left\{\frac{t}{\lambda_d(\bDelta_n)^2}, \frac{\kappa(\bDelta_n)t^{1/2}}{\lambda_d(\bDelta_n)^2}, t^{3/2}\right\}.
\end{align*}
In addition, by Lemma \ref{Lemma:R14k_analysis}, with
\[
\bB_{nij} =  \frac{\{\bx_{i}\transpose\bx_{j}(1 - \rho_n\bx_{i}\transpose\bx_{j})\eye_d - (1 - 2\rho_n\bx_{i}\transpose\bx_{j})\bx_{j}\bx_{i}\transpose\}}
{\{\bx_{i}\transpose\bx_{j}(1 - \rho_n\bx_{i}\transpose\bx_{j})\}^2},
\]
for all $t\geq 1$, $t\lesssim n\rho_n$, 
\begin{align*}
&\left\|
\frac{1}{n\rho_n^{1/2}}\sum_{j = 1}^n
\frac{[\bE]_{ij}\{\bx_{i}\transpose\bx_{j}(1 - \rho_n\bx_{i}\transpose\bx_{j})\eye_d - (1 - 2\rho_n\bx_{i}\transpose\bx_{j})\bx_{j}\bx_{i}\transpose\}}
{\{\bx_{i}\transpose\bx_{j}(1 - \rho_n\bx_{i}\transpose\bx_{j})\}^2}
(\rho_n^{-1/2}\bW\transpose\widetilde{\bx}_j - \bx_{j})
\right\|_2
\\&\quad
\lesssim_c \frac{\|\bU_\bP\|_{2\to\infty}}{(n\rho_n)^{1/2}\delta^4\lambda_d(\bDelta_n)}
\max\left\{\frac{\kappa(\bDelta_n)^{1/2}t^{1/2}}{\lambda_d(\bDelta_n)}, \frac{\kappa(\bDelta_n)}{\lambda_d(\bDelta_n)^2}, t\right\}
\end{align*}
with probability at least $1 - c_0n^{-c} - c_0e^{-t}$ for sufficiently large $n$.
Finally, by Lemma \ref{Lemma:Taylor_remainder_error_bound},
for all $t\geq 1$, $t\lesssim n\rho_n$,
\begin{align*}
&\left\|\frac{1}{n\rho_n^{1/2}}\sum_{j = 1}^n[\bE]_{ij}\br_{\bg} + \frac{1}{n\rho_n^{1/2}}\rho_n\br_{\bh_{ij}}\right\|_2
\lesssim_c \frac{d^{1/2}\|\bU_\bP\|_{2\to\infty}}{(n\rho_n)^{1/2}\delta^6\lambda_d(\bDelta_n)^{5/2}}
\max\left\{\frac{t}{\lambda_d(\bDelta_n)^4}, \frac{\kappa(\bDelta_n)^2}{\lambda_d(\bDelta_n)^4}, t^2\right\}
\end{align*}
with probability at least $1 - c_0n^{-c} - c_0e^{-t}$ for sufficiently large $n$. The proof is then completed by combining the aforementioned concentration bounds. 
\end{proof}


\subsection{Concentration bound for \eqref{eqn:Ri2}} 
\label{sub:concentration_bound_for_eqn:ri2}

In this section, we work on a concentration bound for term \eqref{eqn:Ri2}. Since $\widetilde\bX\bW$ is close to $\rho_n^{1/2}\bX$ in the stringent two-to-infinity norm distance by Theorem \ref{thm:eigenvector_deviation}, it is expected that term \eqref{eqn:Ri2} is asymptotically negligible by the continuous mapping theorem. A formal description of this result requires some work. To begin with, we first oberve the following fact that guarantees that $\rho_n^{-1/2}\widetilde\bX\bW$ is close to $\bX$ in the two-to-infinity norm distance. 
\begin{result}\label{result:two_to_infinity_norm_consistency_X}
By Corollary \ref{corr:Two_to_infinity_norm_eigenvector_bound}, given any fixed $c > 0$,
\begin{align*}
\|\rho_n^{-1/2}\widetilde\bX\bW - \bX\|_{2\to\infty}
&\lesssim_c \frac{1}{n \rho_n\lambda_d(\bDelta_n)^{3/2}}\max\left\{\frac{(\log n)^{1/2}}{\lambda_d(\bDelta_n)^3}, \frac{\kappa(\bDelta_n)}{\lambda_d(\bDelta_n)^3}, \frac{\log n}{\lambda_d(\bDelta_n)}\right\}\\
&\quad + \frac{(\log n)^{1/2}}{(n\rho_n)^{1/2}\lambda_d(\bDelta_n) } 
\end{align*}
with probability at least $1 - c_0n^{-c}$ for sufficiently large $n$. By assumption, 
\[
\frac{1}{n \rho_n\lambda_d(\bDelta_n)^{3/2}}\max\left\{\frac{(\log n)^{1/2}}{\lambda_d(\bDelta_n)^3}, \frac{\kappa(\bDelta_n)}{\lambda_d(\bDelta_n)^3}, \frac{\log n}{\lambda_d(\bDelta_n)}\right\}\quad\Longrightarrow
\quad
\frac{(\log n)^{1/2}}{(n\rho_n)^{1/2}\lambda_d(\bDelta_n) }\to 0.
\]
Therefore, for sufficiently large $n$,
\[
(\rho_n^{-1/2}\bW\transpose{}\widetilde\bx_i, \rho_n^{-1/2}\bW\transpose{}\widetilde\bx_j)\in\calX_2(\delta/2)\quad\text{for all }i,j\in [n]
\] 
with probability at least $1 - c_0n^{-c}$. 
\end{result}
\begin{lemma}\label{lemma:Lipschitz_continuity_Hmatrix}
Let $\bA\sim\mathrm{RDPG}(\rho_n^{1/2}\bX)$  and assume the conditions of Theorem \ref{thm:asymptotic_normality_OS} hold. 
Then given any fixed $c > 0$, for each fixed row index $i\in [n]$, 
for all $t\geq 1$, $t\lesssim n\rho_n$, and sufficiently large $n$,
\begin{align*}
\left\|\bH_i(\rho_n^{-1/2}\widetilde\bX\bW) - \bH_i(\bX)\right\|_2\lesssim_c \frac{1}{(n\rho_n)^{1/2}\delta^4\lambda_d(\bDelta_n)^{3/2}}\max\left\{\frac{t^{1/2}}{\lambda_d(\bDelta_n)^2}, \frac{\kappa(\bDelta_n)}{\lambda_d(\bDelta_n)^2}, t\right\}
\end{align*}
with probability at least $1 - c_0n^{-c} - c_0e^{-t}$ for sufficiently large $n$. 
\end{lemma}

\begin{proof}[\bf Proof of Lemma \ref{lemma:Lipschitz_continuity_Hmatrix}]
For any $\bu, \bv$ with $\|\bu\|_2,\|\bv\|_2\leq 1$ and $\delta/2\leq \bu\transpose{}\bv\leq 1 - \delta/2$, define
\[
\bH(\bu, \bv) = \frac{\bv\bv\transpose{}}{\bu\transpose{}\bv(1 - \rho_n\bu\transpose{}\bv)}. 
\]
By the matrix differential calculus (see, e.g., \cite{MAGNUS1985474}), we can compute
\begin{align*}
\frac{\partial}{\partial\bu\transpose{}}\vect\{\bH(\bu, \bv)\} & = \frac{(1 - 2\rho_n\bu\transpose{}\bv)(\bv\otimes\bv)\bu\transpose{}}{\{\bu\transpose{}\bv(1 - \rho_n\bu\transpose{}\bv)\}^2},\\
\frac{\partial}{\partial\bv\transpose{}}\vect\{\bH(\bu, \bv)\} & = \frac{(1 - 2\rho_n\bu\transpose{}\bv)(\bv\otimes\bv)\bv\transpose{}}{\{\bu\transpose{}\bv(1 - \rho_n\bu\transpose{}\bv)\}^2} + \frac{\eye_d\otimes \bv + \bv\otimes\eye_d}{\bu\transpose{}\bv(1 - \rho_n\bu\transpose{}\bv)}.
\end{align*}
Clearly, 
\begin{align*}
\sup_{(\bu, \bv)\in\calX_2(\delta/2)}\left\{\left\|\frac{\partial}{\partial\bu\transpose{}}\vect\{\bH(\bu, \bv)\}\right\|_2 + \left\|\frac{\partial}{\partial\bu\transpose{}}\vect\{\bH(\bu, \bv)\}\right\|_2\right\}\leq \frac{40}{\delta^4}.
\end{align*}
Then by the mean-value inequality, for any $\bu,\bv$ with $\|\bu\|_2,\|\bv\|_2\leq 1$, $\delta/2\leq \bu\transpose{}\bv\leq 1 - \delta/2$, 
\[
\|\bH(\bu, \bv) - \bH(\bx_{i}, \bx_{j})\|_2\leq \frac{40}{\delta^4}(\|\bu - \bx_{i}\|_2 + \|\bv - \bx_{j}\|_2)
\]
Denote $\widetilde\by_i = \rho_n^{-1/2}\bW\transpose{}\widetilde\bx_i$, $i\in [n]$. 
We then apply Result \ref{result:two_to_infinity_norm_consistency_X}, Lemma \ref{lemma:RX_Frobenius_norm_bound}, and Lemma \ref{lemma:single_x_i_perturbation_bound} to obtain that
 for all $t\geq 1$, $t\lesssim n\rho_n$,
\begin{align*}
&\|\bH_i(\rho_n^{-1/2}\widetilde\bX\bW) - \bH_i(\bX)\|_2\\
&\quad\leq \frac{1}{n}\sum_{j = 1}^n\|\bH(\widetilde\by_i,\widetilde\by_j) - \bH(\bx_{i}, \bx_{j})\|_2
\lesssim \frac{1}{n\delta^4}\sum_{j = 1}^n(\|\widetilde{\by}_i - \bx_{i}\|_2 + \|\widetilde\by_j - \bx_{j}\|_2)\\
&\quad\leq \frac{1}{\rho_n^{1/2}\delta^4}\|\bW\transpose{}\widetilde{\bx}_i - \rho_n^{1/2}\bx_{i}\|_2 + \frac{1}{n\rho_n^{1/2}}\sum_{j = 1}^n\|\bW\transpose{}\widetilde{\bx}_j - \rho_n^{1/2}\bx_{j}\|_2\\
&\quad\leq \frac{1}{\rho_n^{1/2}\delta^4}\|\bW\transpose{}\widetilde{\bx}_i - \rho_n^{1/2}\bx_{i}\|_2 + \frac{1}{(n\rho_n)^{1/2}\delta^4}\|\widetilde\bX\bW - \rho_n^{1/2}\bX \|_{\mathrm{F}}\\
&\quad\lesssim_c \frac{\|\bU_\bP\|_{2\to\infty}}{\rho_n^{1/2}\delta^4\lambda_d(\bDelta_n) }\max\left\{\frac{t^{1/2}}{\lambda_d(\bDelta_n)^2}, \frac{\kappa(\bDelta_n)}{\lambda_d(\bDelta_n)^2}, t\right\} + \frac{1}{(n\rho_n)^{1/2}\delta^4\lambda_d(\bDelta_n)}\\
&\quad\lesssim \frac{1}{(n\rho_n)^{1/2}\delta^4\lambda_d(\bDelta_n)^{3/2}}\max\left\{\frac{t^{1/2}}{\lambda_d(\bDelta_n)^2}, \frac{\kappa(\bDelta_n)}{\lambda_d(\bDelta_n)^2}, t\right\}
\end{align*}
with probability at least $1 - c_0n^{-c} - c_0e^{-t}$ for sufficiently large $n$. The proof is thus completed. 
\end{proof}


\subsection{Proofs of Theorems \ref{thm:asymptotic_normality_OS} and \ref{thm:Berry_Esseen_OSE_multivariate}} 
\label{sub:proofs_of_theorems_thm:asymptotic_normality_os_and_thm:berry_esseen_ose_functional}

\begin{proof}[\bf Proof of Theorem \ref{thm:asymptotic_normality_OS}]
We first recall the following decomposition of $\bW\transpose\widetilde{\bx}_i - \rho_n^{1/2}\bx_i$ in Section \ref{sub:outline_of_the_proof_of_theorem_thm:asymptotic_normality_os}:
\begin{align*}
\bG_n(\bx_i)^{1/2}(\bW\transpose{}\widetilde\bx_i - \rho_n^{1/2}\bx_i)
& = \frac{1}{n\rho_n^{1/2}}\sum_{j = 1}^n\frac{[\bE]_{ij}\bG_n(\bx_i)^{-1/2}\bx_j}{\bx_i\transpose{}\bx_j(1 - \rho_n\bx_i\transpose{}\bx_j)} + \widehat{\br}_i,
\end{align*}
where
\begin{align*}
&\widehat{\br}_i = \bG_n(\bx_i)^{-1/2}\{\bG_n(\bx_i)(\bW\transpose{}\widetilde\bx_i - \rho_n^{1/2}\bx_i) + \br_{i1}\}
 + \bG_n(\bx_i)^{1/2}\bR_{i2}\br_{i1} + \bG_n(\bx_i)^{1/2}\br_{i3},\\
&\br_{i1} = \frac{1}{n\sqrt{\rho_n}}\sum_{j = 1}^n\left\{\frac{(A_{ij} - \widetilde\bx_i\transpose\widetilde\bx_j)(\rho_n^{-1/2}\bW\transpose\widetilde\bx_j)}{\rho_n^{-1}\widetilde\bx_{i}\transpose\widetilde\bx_{j}(1 - \widetilde\bx_i\transpose\widetilde\bx_j)} - \frac{(A_{ij} - \rho_n\bx_{i}\transpose\bx_{j})\bx_{j}}{\bx_{i}\transpose\bx_{j}(1 - \rho_n\bx_{i}\transpose\bx_{j})}\right\},\\
&\bR_{i2} = \bH_i(\rho_n^{-1/2}\widetilde\bX\bW)^{-1} - \bG_n(\bx_{i})^{-1},
\end{align*}
$\bH_i(\cdot)$ is the function defined in \eqref{eqn:Hi_function}, and
\begin{align*}
\br_{i3}  = \frac{1}{n\sqrt{\rho_n}}\sum_{j = 1}^n\frac{(A_{ij} - \rho_n\bx_{i}\transpose\bx_{j})}{\bx_{i}\transpose\bx_{j}(1 - \rho_n\bx_{i}\transpose\bx_{j})}\bR_{i2}\bx_{j}.
\end{align*}
By Lemma \ref{lemma:Taylor_expansion_phi}, 
for all $t\geq 1$, $t\lesssim \log n$,
\begin{align*}
\|\br_{i1}\|_2&\lesssim_c \frac{\|\bU_\bP\|_{2\to\infty}}{\delta^4\lambda_d(\bDelta_n)}\max\left\{
\frac{t^{1/2}}{\lambda_d(\bDelta_n)^2}, \frac{\kappa(\bDelta_n)}{\lambda_d(\bDelta_n)^2}, t
\right\}\\
&\leq \frac{1}{\sqrt{n}\delta^4\lambda_d(\bDelta_n)^{3/2}}\max\left\{
\frac{t^{1/2}}{\lambda_d(\bDelta_n)^2}, \frac{\kappa(\bDelta_n)}{\lambda_d(\bDelta_n)^2}, t
\right\}
\end{align*}
with probability at least $1 - c_0n^{-c} - c_0e^{-t}$ for sufficiently large $n$, and
\begin{align*}
\left\|
\br_{i1} + \bG_n(\bx_i)(\bW\transpose\widetilde{\bx}_i - \rho_n^{1/2}\bx_i)
\right\|_2
&\lesssim_c \frac{d^{1/2}\|\bU_\bP\|_{2\to\infty}}{(n\rho_n)^{1/2}\delta^6\lambda_d(\bDelta_n)^{5/2}}\max\left\{\frac{t}{\lambda_d(\bDelta_n)^4}, \frac{\kappa(\bDelta_n)^2}{\lambda_d(\bDelta_n)^4}, t^2\right\}\\
&\leq \frac{d^{1/2}}{n\rho_n^{1/2}\delta^6\lambda_d(\bDelta_n)^3}\max\left\{\frac{t}{\lambda_d(\bDelta_n)^4}, \frac{\kappa(\bDelta_n)^2}{\lambda_d(\bDelta_n)^4}, t^2\right\}
\end{align*}
with probability at least $1 - c_0n^{-c} - c_0e^{-t}$. We next focus on $\bR_{i2}$. Observe that
\[
\bG_n(\bx_i) = \frac{1}{n}\sum_{j = 1}^n\frac{\bx_j\bx_j\transpose}{\bx_i\transpose\bx_j(1 - \rho_n\bx_i\transpose\bx_j)} = \bH_i(\bX)\succeq \frac{1}{n}\sum_{j = 1}^n\bx_j\bx_j\transpose = \bDelta_n.
\]
By definition of $\bH_i(\cdot)$ and Result \ref{result:two_to_infinity_norm_consistency_X}, 
\begin{align*}
\bH_i(\rho_n^{-1/2}\widetilde\bX\bW)& = \bW\transpose\left\{\frac{1}{n}\sum_{j = 1}^n\frac{\widetilde{\bx}_j\widetilde{\bx}_j\transpose}{\widetilde\bx_i\transpose\widetilde\bx_j(1 - \widetilde\bx_i\transpose{}\widetilde\bx_j)}\right\}\bW\\
& = \bW\transpose\left[\frac{1}{n}\sum_{j = 1}^n\frac{\widetilde{\bx}_j\widetilde{\bx}_j\transpose}{\rho_n(\rho_n^{-1/2}\widetilde\bx_i\transpose\rho_n^{-1/2}\widetilde\bx_j)(1 - \rho_n\rho_n^{-1/2}\widetilde\bx_i\transpose{}\rho_n^{-1/2}\widetilde\bx_j)}\right]\bW\\
&\gtrsim\frac{1}{\rho_n}\bW\transpose
\left(\frac{1}{n}\sum_{j = 1}^n{\widetilde{\bx}_j\widetilde{\bx}_j\transpose}\right)
\bW = \frac{1}{n\rho_n}\bW\transpose(\widetilde\bX\widetilde\bX\transpose)\bW
\end{align*}
with probability at least $1 - c_0n^{-c}$ for sufficiently large $n$.
Namely,
\begin{align*}
\lambda_d\left\{\bH_i(\rho_n^{-1/2}\widetilde\bX\bW)\right\}
&\gtrsim \frac{1}{n\rho_n}\lambda_d(\widetilde\bX\transpose\widetilde\bX) = \frac{1}{n\rho_n}\lambda_d(\bA),
\end{align*}
and hence, by Result \ref{result:S_A_concentration},
\[
\|\bH_i(\rho_n^{-1/2}\widetilde\bX\bW)^{-1}\|_2\lesssim n\rho_n\|\bS_\bA^{-1}\|_2\lesssim \frac{1}{\lambda_d(\bDelta_n)}
\]
with probability at least $1 - c_0n^{-c}$ for sufficiently large $n$. Also, by Lemma \ref{lemma:Lipschitz_continuity_Hmatrix}, 
for all $t\geq 1$, $t\lesssim n\rho_n$,
\begin{align*}
\|\bH_i(\rho_n^{-1/2}\widetilde\bX\bW) - \bH_i(\bX)\|_2
&\lesssim_c \frac{1}{(n\rho_n)^{1/2}\delta^4\lambda_d(\bDelta_n)^{3/2}}\max\left\{\frac{t^{1/2}}{\lambda_d(\bDelta_n)^2}, \frac{\kappa(\bDelta_n)}{\lambda_d(\bDelta_n)^2}, t\right\}
\end{align*}
with probability at least $1 - c_0n^{-c} - c_0e^{-t}$ for sufficiently large $n$. It follows that
\begin{align*}
 \|\bG_n(\bx_{i})^{1/2}\bR_{i2}\|_{2}
& =  \|\bG_n(\bx_{i})^{-1/2}\{\bH_i(\rho_n^{-1/2}\widetilde\bX\bW) - \bG_n(\bx_{i})\}\bH_i(\rho_n^{-1/2}\widetilde\bX\bW)^{-1}\|_2\\
&\leq  \|\bG_n(\bx_{i})^{-1/2}\|_2\|\bH_i(\rho_n^{-1/2}\widetilde\bX\bW) - \bG_n(\bx_{i})\|_{\mathrm{2}}\|\bH_i(\rho_n^{-1/2}\widetilde\bX\bW)^{-1}\|_2\\
&\lesssim \frac{1}{\lambda_d(\bDelta_n)^{3/2}} \|\bH_i(\rho_n^{-1/2}\widetilde\bX\bW) - \bG_n(\bx_{i})\|_2\\
&\lesssim_c \frac{1}{(n\rho_n)^{1/2}\delta^4\lambda_d(\bDelta_n)^3}\max\left\{\frac{t^{1/2}}{\lambda_d(\bDelta_n)^2}, \frac{\kappa(\bDelta_n)}{\lambda_d(\bDelta_n)^2}, t\right\}
\end{align*}
with probability at least $1 - c_0n^{-c} - c_0e^{-t}$ for sufficiently large $n$. 
We then move forward to the analysis of $\br_{i3}$.  Write
\begin{align*}
 \|\bG_n(\bx_{i})^{1/2}\br_{i3}\|_{2}
&\leq \|\bG_n(\bx_{i})^{1/2}\bR_{i2}\|_{2}\frac{1}{n\sqrt{\rho_n}}\left\|\sum_{j = 1}^n\frac{(A_{ij} - \rho_n\bx_{i}\transpose\bx_{j})\bx_{j}}{\bx_{i}\transpose\bx_{j}(1 - \rho_n\bx_{i}\transpose\bx_{j})}\right\|_2.
\end{align*}
By Lemma \ref{lemma:Bernstein_concentration_EW}, for all $t\geq 1$ and $t\lesssim n\rho_n$, we have
\begin{align*}
 \left\|\sum_{j = 1}^n\frac{(A_{ij} - \rho_n\bx_{i}\transpose\bx_{j})\bx_{j}}{\bx_{i}\transpose\bx_{j}(1 - \rho_n\bx_{i}\transpose\bx_{j})}\right\|_2
 &\leq 3t\max_{j\in [n]}\left\|\frac{ \bx_{j}}{\bx_{i}\transpose\bx_{j}(1 - \rho_n\bx_{i}\transpose\bx_{j})}\right\|_2\\
 &\quad + (6\rho_nt)^{1/2}\left\{\sum_{j = 1}^n\left\|\frac{ \bx_{j}}{\bx_{i}\transpose\bx_{j}(1 - \rho_n\bx_{i}\transpose\bx_{j})}\right\|_2^2\right\}^{1/2}\\
 &\lesssim (n\rho_nt)^{1/2}\max_{j\in [n]}\left\|\frac{ \bx_{j}}{\bx_{i}\transpose\bx_{j}(1 - \rho_n\bx_{i}\transpose\bx_{j})}\right\|_2
 \lesssim \frac{(n\rho_nt)^{1/2}}{\delta^2}
\end{align*}
with probability at least $1 - c_0e^{-t}$. It follows immediately that
\begin{align*}
\|\bG_n(\bx_{i})^{1/2}\br_{i3}\|_{2}
&\lesssim_c \frac{1}{(n\rho_n)^{1/2}\delta^4\lambda_d(\bDelta_n)^3}\max\left\{\frac{t^{1/2}}{\lambda_d(\bDelta_n)^2}, \frac{\kappa(\bDelta_n)}{\lambda_d(\bDelta_n)^2}, t\right\}\frac{(n\rho_nt)^{1/2}}{n\rho_n^{1/2}\delta^2}\\
& = \frac{1}{n\rho_n^{1/2}\delta^6\lambda_d(\bDelta_n)^3}\max\left\{\frac{t}{\lambda_d(\bDelta_n)^2}, \frac{\kappa(\bDelta_n)t^{1/2}}{\lambda_d(\bDelta_n)^2}, t^{3/2}\right\}
\end{align*}
with probability at least $1 - c_0n^{-c} - c_0e^{-t}$ for all sufficiently large $n$.

\vspace*{2ex}\noindent
Summarizing the above large probability bounds, 
for all $t\geq 1$, $t\lesssim \log n$,
\begin{align*}
\|\widehat\br_i\|_2 & \leq \|\bG_n(\bx_{i})^{-1/2}\|_2\|\bG_n(\bx_{i})(\bW\transpose\widetilde \bx_i - \rho_n^{1/2}\bx_{i}) + \br_{i1}\|_2\\
&\quad + \|\bG_n(\bx_{i})^{1/2}\bR_{i2}\|_2\|\br_{i1}\|_2 + \|\bG_n(\bx_{i})^{1/2}\br_{i3}\|_2
\\&
\lesssim_c \frac{d^{1/2}}{n\rho_n^{1/2}\delta^6\lambda_d(\bDelta_n)^{7/2}}\max\left\{\frac{t}{\lambda_d(\bDelta_n)^4}, \frac{\kappa(\bDelta_n)^2}{\lambda_d(\bDelta_n)^4}, t^2\right\}\\
&\quad + \frac{1}{(n\rho_n)^{1/2}\delta^4\lambda_d(\bDelta_n)^3}\max\left\{\frac{t^{1/2}}{\lambda_d(\bDelta_n)^2}, \frac{\kappa(\bDelta_n)}{\lambda_d(\bDelta_n)^2}, t\right\}\\
&\quad\quad\times\frac{1}{\sqrt{n}\delta^4\lambda_d(\bDelta_n)^{3/2}}\max\left\{
\frac{t^{1/2}}{\lambda_d(\bDelta_n)^2}, \frac{\kappa(\bDelta_n)}{\lambda_d(\bDelta_n)^2}, t
\right\}\\
&\quad + \frac{1}{n\rho_n^{1/2}\delta^6\lambda_d(\bDelta_n)^3}\max\left\{\frac{t}{\lambda_d(\bDelta_n)^2}, \frac{\kappa(\bDelta_n)t^{1/2}}{\lambda_d(\bDelta_n)^2}, t^{3/2}\right\}
\\& 
\leq\frac{d^{1/2}}{n\rho_n^{1/2}\delta^6\lambda_d(\bDelta_n)^{7/2}}\max\left\{\frac{t}{\lambda_d(\bDelta_n)^4}, \frac{\kappa(\bDelta_n)^2}{\lambda_d(\bDelta_n)^4}, t^2\right\}\\
&\quad + \frac{1}{n\rho_n^{1/2}\delta^8\lambda_d(\bDelta_n)^{9/2}}\max\left\{
\frac{t}{\lambda_d(\bDelta_n)^4}, \frac{\kappa(\bDelta_n)^2}{\lambda_d(\bDelta_n)^4}, t^2
\right\}\\
&\quad + \frac{1}{n\rho_n^{1/2}\delta^6\lambda_d(\bDelta_n)^3}\max\left\{\frac{t}{\lambda_d(\bDelta_n)^2}, \frac{\kappa(\bDelta_n)t^{1/2}}{\lambda_d(\bDelta_n)^2}, t^{3/2}\right\}\\
&\lesssim \frac{1}{n\rho_n^{1/2}\delta^8\lambda_d(\bDelta_n)^{9/2}}\max\left\{
\frac{t}{\lambda_d(\bDelta_n)^4}, \frac{\kappa(\bDelta_n)^2}{\lambda_d(\bDelta_n)^4}, t^2
\right\}
\end{align*}
 with probability at least $1 - c_0n^{-c} - c_0e^{-t}$ for sufficiently large $n$. The proof is thus completed. 
\end{proof}

\begin{proof}[\bf Proof of Theorem \ref{thm:Berry_Esseen_OSE_multivariate}]
We apply Theorem \ref{thm:Berry_Esseen_Multivariate} to obtain the desired Berry-Esseen bound. Let
\begin{align*}
&\bxi_j = \frac{(A_{ij} - \rho_n\bx_{i}\transpose{}\bx_{j})\bG_n(\bx_{i})^{-1/2}\bx_{j}}{\sqrt{n\rho_n}\bx_{i}\transpose{}\bx_{j}(1 - \rho_n\bx_{i}\transpose{}\bx_{j})},
\quad
\bD = \sqrt{n}\widehat{\br}_i,\\
&\Delta^{(j)} = \Delta = \frac{C}{n\rho_n^{1/2}\delta^8\lambda_d(\bDelta_n)^{9/2}}\max\left\{
\frac{\log n\rho_n}{\lambda_d(\bDelta_n)^4}, \frac{\kappa(\bDelta_n)^2}{\lambda_d(\bDelta_n)^4}, (\log n\rho_n)^2
\right\},
\\
&\calO = \{\bA: \Delta > \|\bD\|_2\},
\end{align*}
where $C > 0$ is some absolute constant. 
By definition of $\bxi_j$ and $\bG_n(\bx_{i})$, $\expect(\bxi_j) = 0$ and $\bSigma_n(\bx_{i})$, 
\begin{align*}
\sum_{j = 1}^n\expect_0(\bxi_j\bxi_j\transpose{})
& = \frac{1}{n}\sum_{j = 1}^n\frac{\bG_n(\bx_{i})^{-1/2}\bx_{j}\bx_{j}\transpose{}\bG_n(\bx_{i})^{-1/2}}{\bx_{i}\transpose\bx_{j}(1 - \rho_n\bx_{i}\transpose{}\bx_{j})}= \eye_d. 
\end{align*}
We now proceed to $\sum_{j = 1}^n\expect(\|\bxi_j\|_2^3)$ and $\expect(\|\sum_{j = 1}^n\bxi_j\|_2)$. For the first term, we have
\begin{align*}
&\sum_{j = 1}^n\expect(\|\bxi_j\|_2^3)
= \frac{1}{(n\rho_n)^{3/2}}\sum_{j = 1}^n\expect|A_{ij} - \rho_n\bx_{i}\transpose{}\bx_{j}|^3\left\|\frac{\bG_n(\bx_{i})^{-1/2}\bx_{j}}{\bx_{i}\transpose{}\bx_{j}(1 - \rho_n\bx_{i}\transpose{}\bx_{j})}\right\|_2^3\\
&\quad\leq \frac{\|\bG_n(\bx_{i})^{-1/2}\|_2}{(n\rho_n)^{3/2}\delta^2}\sum_{j = 1}^n\expect\{(A_{ij} - \rho_n\bx_{i}\transpose{}\bx_{j})^2\}\left\|\frac{\bG_n(\bx_{i})^{-1/2}\bx_{j}}{\bx_{i}\transpose{}\bx_{j}(1 - \rho_n\bx_{i}\transpose{}\bx_{j})}\right\|_2^2\\
&\quad\leq \frac{1}{(n\rho_n)^{3/2}\delta^2\lambda_d(\bDelta_n)^{1/2}}\sum_{j = 1}^n\rho_n\bx_{i}\transpose{}\bx_{j}(1 - \rho_n\bx_{i}\transpose{}\bx_{j})
\mathrm{tr}\left[\frac{\bG_n(\bx_{i})^{-1/2}\bx_{j}\bx_{j}\transpose{}\bG_n(\bx_{i})^{-1/2}}{\{\bx_{i}\transpose{}\bx_{j}(1 - \rho_n\bx_{i}\transpose{}\bx_{j})\}^2}\right]\\
&\quad = \frac{1}{(n\rho_n)^{1/2}\delta^2\lambda_d(\bDelta_n)^{1/2}}\mathrm{tr}\left[\bG_n(\bx_{i})^{-1/2}\left\{\frac{1}{n}\sum_{j = 1}^n\frac{\bx_{j}\bx_{j}\transpose{}}{\bx_{i}\transpose{}\bx_{j}(1 - \rho_n\bx_{i}\transpose{}\bx_{j})}\right\}
\bG_n(\bx_{i})^{-1/2}
\right]\\
&\quad = \frac{1}{(n\rho_n)^{1/2}\delta^2\lambda_d(\bDelta_n)^{1/2}}\mathrm{tr}(\eye_d) = \frac{d}{(n\rho_n)^{1/2}\delta^2\lambda_d(\bDelta_n)^{1/2}}.
\end{align*}
For $\expect(\|\sum_{j = 1}^n\bxi_j\|_2)$, we use Jensen's inequality to write
\begin{align*}
\expect\left(\left\|\sum_{j = 1}^n\bxi_j\right\|_2\right)
&\leq \left\{\expect\left(\left\|\sum_{j = 1}^n\bxi_j\right\|_2^2\right)\right\}^{1/2}
  = \left(\sum_{j = 1}^n\expect\|\bxi_j\|_2^2\right)^{1/2}
  = \left\{\mathrm{tr}\left(\sum_{j = 1}^n\expect\bxi_j\bxi_j\transpose\right)\right\}^{1/2}
  = d^{1/2}.
\end{align*}
This immediately implies that
\begin{align*}
\expect\left(\left\|\sum_{j = 1}^n\bxi_j\right\|_2\Delta\right)
&\lesssim \frac{d^{1/2}}{n\rho_n^{1/2}\delta^8\lambda_d(\bDelta_n)^{9/2}}\max\left\{
\frac{\log n\rho_n}{\lambda_d(\bDelta_n)^4}, \frac{\kappa(\bDelta_n)^2}{\lambda_d(\bDelta_n)^4}, (\log n\rho_n)^2
\right\}
\end{align*}
Finally, for $\prob(\calO^c)$, the concentration bound in Theorem \ref{thm:asymptotic_normality_OS} implies that $\prob(\calO^c)\lesssim (n\rho_n)^{-1/2}$ for sufficiently large $n$. 
We hence conclude from Theorem \ref{thm:Berry_Esseen_Multivariate} that
\begin{align*}
&\sup_{A\in\calA}\left|\prob\left\{\sqrt{n}\bG_n(\bx_{i})^{-1/2}(\bW_n\transpose\widehat\bx_i - \rho_n^{1/2}\bx_{i})\in A\right\} - \prob\left(\bz\in A\right)\right|\\
&\quad\lesssim d^{1/2}\gamma + \expect\left(\left\|\sum_{j = 1}^n\bxi_j\right\|_2\Delta\right)
 + \prob(\calO^c)\\
 &\quad\lesssim \frac{d^{1/2}}{n\rho_n^{1/2}\delta^8\lambda_d(\bDelta_n)^{9/2}}\max\left\{
\frac{\log n\rho_n}{\lambda_d(\bDelta_n)^4}, \frac{\kappa(\bDelta_n)^2}{\lambda_d(\bDelta_n)^4}, (\log n\rho_n)^2
\right\}.
\end{align*}
The proof is thus completed.
\end{proof}

\section{Proofs for Section \ref{sub:subsequent_graph_inference}}
\label{sec:proof_subsequent_graph_inference}

\subsection{Proof of Theorem \ref{thm:MMSBM}}
\label{sub:proof_theorem_MMSBM}

We follow the notations and definitions in Sections \ref{sec:preliminaries} and \ref{sec:entrywise_limit_theorem_for_the_eigenvectors}. 

\vspace*{1ex}\noindent
$\blacksquare$ \textbf{Row-wise concentration bound for the membership profile matrix estimate.}
First note that $\sigma_1(\bTheta)\leq \sqrt{n}$ because
\[
\sigma_1(\bTheta)\leq \|\bTheta\|_{\mathrm{F}} \leq \sqrt{n}\|\bTheta\|_{2\to\infty}\leq \sqrt{n}\|\bTheta\|_\infty = \sqrt{n}.
\]
Also, we have $\bTheta\transpose\bTheta \preceq c_1^2n\eye_d$ by the condition of Theorem \ref{thm:MMSBM}. Therefore,
\[
\lambda_d(\bDelta_n) = \frac{1}{n}\lambda_d\{(\bX^*)\transpose\bTheta\transpose\bTheta(\bX^*)\}\geq c_1^2\lambda_d\{(\bX^*)\transpose(\bX^*)\}\geq c_1^4.
\]
Namely, $\lambda_d(\bDelta_n)$ is bounded away from $0$. By Corollary \ref{corr:Two_to_infinity_norm_eigenvector_bound}, there exists constants $K_1, c_1 > 0$, such that
\[
\|\bU_\bA - \bU_\bP\bW^*\|_{2\to\infty}\leq \frac{K_1(\log n)^{1/2}}{n\rho_n^{1/2}}\quad\mbox{with probability at least }1 - c_1n^{-2}.
\]
By Lemma \ref{lemma:U_A_two_to_infinity_norm}, Result \ref{result:spectral_norm_concentration}, and Davis-Kahan theorem, 
\begin{align*}
\|\bU_\bA\bU_\bA\transpose - \bU_\bP\bU_\bP\transpose\|_{2\to\infty}
&\leq \|\bU_\bA - \bU_\bP\bW^*\|_{2\to\infty}\|\bU_\bA\transpose\|_2 + \|\bU_\bP\|_{2\to\infty}\|\bU_\bA - \bU_\bP\bW^*\|_{2}\\
&\leq \|\bU_\bA - \bU_\bP\bW^*\|_{2\to\infty}\|\bU_\bA\transpose\|_2 + \|\bU_\bP\|_{2\to\infty}\frac{4\|\bE\|_2}{n\rho_n\lambda_d(\bDelta_n)}\\
&\leq \frac{K_2(\log n)^{1/2}}{n\rho_n^{1/2}}\quad\mbox{with probability at least }1 - c_0n^{-2}
\end{align*}
for constants $K_2, c_0 > 0$ for sufficiently large $n$. By Lemma 2.1 in \cite{doi:10.1080/01621459.2020.1751645}, $\bU_\bP = \bTheta\bV_\bP$, where $\bV_\bP\in\mathbb{R}^{d\times d}$ is the submatrix of $\bU_\bP$ corresponding to the pure node indices $\{i_1,\ldots,i_d\}$. By Lemma II.3 in \cite{doi:10.1080/01621459.2020.1751645}, we have $\sigma_1(\bV_\bP)\leq (c_1^2n)^{-1/2}$ and $\sigma_d(\bV_\bP)\geq n^{-1/2}$. Since 
\begin{align*}
&\prob\left\{
\|\bU_\bA\bU_\bA\transpose - \bU_\bP\bU_\bP\transpose\|_{2\to\infty}
\leq \frac{K_2(\log n)^{1/2}}{n\rho_n^{1/2}}
\right\}\geq 1 - c_0n^{-2},\\
&\sigma_d(\bU_\bP\bV_\bP\transpose) = \sigma_d(\bV_\bP) \geq n^{-1/2},\quad
\|(\bU_\bP\bV_\bP)\transpose\|_{2\to\infty} = \|\bV_\bP\transpose\|_{2\to\infty}\leq \sigma_1(\bV_\bP)\leq \frac{1}{\sqrt{n}},\\
&\frac{K_2(\log n)^{1/2}}{n\rho_n^{1/2}} = o(n^{-1/2})\leq \frac{1}{\sqrt{n}}\min\left(\frac{1}{2\sqrt{d - 1}}, \frac{1}{4}\right)\left(1 + 80\frac{n^{-1}}{n^{-1}}\right)^{-1}\\
&\quad\quad\quad\quad\quad\quad\leq \sigma_d(\bU_\bP\bV_\bP\transpose)\min\left(\frac{1}{2\sqrt{d - 1}}, \frac{1}{4}\right)\left(1 + 80\frac{\|(\bU_\bP\bV_\bP)\transpose\|_{2\to\infty}^2}{\sigma_d(\bU_\bP\bV_\bP\transpose)^2}\right)^{-1},
\end{align*}
then by Theorem 3 in \cite{6656801}, there exists a permutation matrix $\bPi_n\in\{0,1\}^{d\times d}$ such that
\[
\max_{k\in [d]}\|(\bU_\bA\bV_\bA\transpose - \bU_\bP\bV_\bP\transpose\bPi_n)\be_k\|_2\leq \frac{K_3(\log n)^{1/2}}{n\rho_n^{1/2}}\quad\mbox{with probability at least }1 - c_0n^{-2}
\]
for constants $K_3, c_0 > 0$ for sufficiently large $n$. It follows that
\begin{align*}
\|\bV_\bA - \bPi_n\transpose\bV_\bP\bU_\bP\transpose\bU_\bA\|_{2\to\infty}
&= \max_{k\in [d]}\|\be_k\transpose(\bV_\bA\bU_\bA\transpose - \bPi_n\transpose\bV_\bP\bU_\bP\transpose)\bU_\bA\|_2\\
&\leq \max_{k\in [d]}\|(\bU_\bA\bV_\bA\transpose - \bU_\bP\bV_\bP\transpose\bPi_n)\be_k\|_2\\
&\leq \frac{K_3(\log n)^{1/2}}{n\rho_n^{1/2}}\quad\mbox{with probability at least }1 - c_0n^{-2}
\end{align*}
for sufficiently large $n$. By Lemma 6.7 in \cite{cape2017two}, Result \ref{result:eigenvalue_concentration}, and Davis-Kahan theorem, we further have
\begin{align*}
\|\bV_\bA - \bPi_n\transpose\bV_\bP\bW^*\|_2
&\leq \|\bV_\bA - \bPi_n\transpose\bV_\bP\bU_\bP\transpose\bU_\bA\|_2 + \|\bPi_n\transpose\bV_\bP(\bU_\bP\transpose\bU_\bA - \bW^*)\|_2\\
&\leq \|\bV_\bA - \bPi_n\transpose\bV_\bP\bU_\bP\transpose\bU_\bA\|_2 + \|\bV_\bP\|_{2}\|\sin\Theta(\bU_\bA, \bU_\bP)\|_2^2\\
&\leq \sqrt{d}\|\bV_\bA - \bPi_n\transpose\bV_\bP\bU_\bP\transpose\bU_\bA\|_{2\to\infty} + \frac{4\|\bE\|_2^2}{c_1\sqrt{n}(n\rho_n)^2\lambda_d(\bDelta_n)^2}\\
&\leq \frac{K_4(\log n)^{1/2}}{n\rho_n^{1/2}}\quad\mbox{with probability at least }1 - c_0n^{-2}
\end{align*}
for sufficiently large $n$, where $K_4, c_0 > 0$ are constants. By Weyl's inequality, we have $\sigma_d(\bV_\bA)\geq (1/2)\sigma_d(\bV_\bP)\geq (1/2)n^{-1/2}$ with probability at least $1 - c_0n^{-2}$ for sufficiently large $n$ since $\log n = o(n\rho_n)$. Namely, 
\[
\|\bV_\bA^{-1}\|_2\leq 2\sqrt{n}\quad\mbox{with probability at least }1 - c_0n^{-2}
\]
for sufficiently large $n$. Hence, we have
\begin{align*}
\|\widehat{\bTheta} - \bTheta\bPi_n\|_{2\to\infty}
& = \max_{i\in [n]}\|\be_i\transpose(\bU_\bA\bV_\bA^{-1} - \bU_\bP\bV_\bP^{-1}\bPi_n)\|_2\\
&\leq \max_{i\in [n]}\|\be_i\transpose(\bU_\bA - \bU_\bP\bW^*)\|_2\|\bV_\bA^{-1}\|_2 + \|\bU_\bP\|_{2\to\infty}\|\bW^* - \bU_\bP\transpose\bU_\bA\|_2\|\bV_\bA^{-1}\|_2\\
&\quad + \max_{i\in [n]}\|\be_i\transpose\bU_\bP\bU_\bP\transpose\bU_\bA(\bPi_n\transpose\bV_\bP\bU_\bP\transpose\bU_\bA)^{-1}(\bPi_n\transpose\bV_\bP\bU_\bP\transpose\bU_\bA - \bV_\bA)\bV_\bA^{-1}\|_2\\
&\leq \|\bU_\bA - \bU_\bP\bW^*\|_{2\to\infty}\|\bV_\bA^{-1}\|_2 + \|\bU_\bP\|_{2\to\infty}\|\sin\Theta(\bU_\bA, \bU_\bP)\|_2^2\|\bV_\bA^{-1}\|_2\\
&\quad + \max_{i\in [n]}\|\be_i\transpose\bU_\bP\bV_\bP^{-1}\bPi_n(\bPi_n\transpose\bV_\bP\bU_\bP\transpose\bU_\bA - \bV_\bA)\bV_\bA^{-1}\|_2\\
&\leq \|\bU_\bA - \bU_\bP\bW^*\|_{2\to\infty}\|\bV_\bA^{-1}\|_2 + \|\bU_\bP\|_{2\to\infty}\frac{4\|\bE\|_2^2}{(n\rho_n)^2\lambda_d(\bDelta_n)^2}
\|\bV_\bA^{-1}\|_2
\\
&\quad + \|\bTheta\bPi_n\|_\infty\|\bPi_n\transpose\bV_\bP\bU_\bP\transpose\bU_\bA - \bV_\bA\|_{2\to\infty}\|\bV_\bA^{-1}\|_{2}\\
&\leq \frac{K_1(\log n)^{1/2}}{n\rho_n^{1/2}}\times 2\sqrt{n} + \frac{1}{\sqrt{n}\lambda_d(\bDelta_n)^{1/2}}\times \frac{K_5}{(n\rho_n)\lambda_d(\bDelta_n)^2}\times 2\sqrt{n}\\
&\quad + \frac{K_4(\log n)^{1/2}}{n\rho_n^{1/2}}\times 2\sqrt{n}\\
&\leq K\sqrt{\frac{\log n}{n\rho_n}}\quad\mbox{with probability at least }1 - c_0n^{-2}
\end{align*}
for sufficiently large $n$ for constants $K, K_5, c_0 > 0$. This completes the proof for the two-to-infinity norm error bounds on the membership profile matrix estimation. 

\vspace*{2ex}\noindent
$\blacksquare$ \textbf{Asympototic normality of the estimators for the pure nodes.}
Let $j_k\in [n]$ be the row index such that $\bPi_n\transpose\btheta_{j_k} = \be_k$. For each $k\in [d]$, define
\begin{align*}
\calI_k &= \{i\in [n]:\be_i\transpose\bTheta = \be_k\},\quad
\calJ_k = \left\{i\in [n]:\|\be_i\transpose\widehat{\bTheta} - \be_k\transpose\|_2\leq \eta\right\}.
\end{align*}
Let $\pi_n\in\calS_d$ be the permutation such that $\bPi_n\be_k = \be_{\pi_n(k)}$. We claim that $\calI_{\pi_n(k)} = \calJ_k$ with probability at least $1 - c_0n^{-2}$ for sufficiently large $n$. 
For any $i\in\calI_{\pi_n(k)}$, we know that $\be_i\transpose\bTheta = \be_{\pi_n(k)} = \bPi_n\be_k$. 
Therefore,
\begin{align*}
\|\be_i\transpose\widehat{\bTheta} - \be_k\transpose\|_2& = \|\be_i\transpose\widehat{\bTheta}\bPi_n\transpose - \be_k\transpose\bPi_n\transpose\|_2 = \|\be_i\transpose\widehat{\bTheta}\bPi_n\transpose - \be_i\transpose\bTheta\|_2\\
& = \|\be_i\transpose(\widehat{\bTheta} - \bTheta\bPi_n)\|_2\leq \|\widehat{\bTheta} - \bTheta\bPi_n\|_{2\to\infty}\\
&\leq K\sqrt{\frac{\log n}{n\rho_n}}\leq \eta
\quad\mbox{with probability at least }1 - c_0n^{-2}
\end{align*}
for sufficiently large $n$.
This shows that $\calI_{\pi_n(k)}\subset\calI_j$ with probability at least $1 - c_0n^{-2}$ for sufficiently large $n$. Conversely, for any $j\in \calJ_k$, we have
\begin{align*}
\|\btheta_{j} - \be_{\pi_n(k)}\|_2 
& = \|\be_j\transpose\bTheta - \be_k\transpose\bPi_n\transpose\|_2
  = \|\be_j\transpose\bTheta\bPi_n - \be_k\|_2
\leq \|\be_j\transpose(\bTheta\bPi_n - \widehat{\bTheta})\|_2 + \|\be_j\transpose\widehat{\bTheta} - \be_k\transpose\|_2\\
&\leq \|\widehat{\bTheta} - \bTheta\bPi_n\|_{2\to\infty} + \eta\leq K\sqrt{\frac{\log n}{n\rho_n}} + \frac{c_2}{2}\quad\mbox{with probability at least }1 - c_0n^{-2}
\end{align*}
for sufficiently large $n$. Then for sufficiently large $n$, we have 
\[
\|\btheta_{j} - \be_{\pi_n(k)}\|_2 \leq \frac{3c_2}{4} < \min_{i\in [n],\btheta_i\neq \be_{\pi_n(k)}}\|\btheta_i - \be_{\pi_n(k)}\|_2,
\]
with probability at least $1 - c_0n^{-2}$, implying that for sufficiently large $n$, $\btheta_j = \be_{\pi_n(k)}$ by the condition of the theorem. Therefore, $\calJ_k\subset\calI_{\pi_n(k)}$ with probability at least $1 - c_0n^{-2}$ for sufficiently large $n$, implying that
\[
\iota_k = \min_{j\in \calJ_k}j = \min_{i\in \calI_{\pi_n(k)}}i = i_{\pi_n(k)}\quad\mbox{with probability at least }1 - c_0n^{-2}
\]
for sufficiently large $n$. Now for any $k\in [d]$, define
\begin{align*}
\widetilde{\bt}_{nk} = \sqrt{n}\bSigma_n(\bx_k^*)^{-1/2}(\bW\transpose\widetilde{\bx}_{i_k} - \rho_n^{1/2}\bx_k^*)\quad\mbox{and}\quad
\widehat{\bt}_{nk} = \sqrt{n}\bG_n(\bx_k^*)^{1/2}(\bW\transpose\widehat{\bx}_{i_k} - \rho_n^{1/2}\bx_k^*).
\end{align*}
By Theorem \ref{thm:ASE_Berry_Esseen_RDPG}, we know that for any convex measurable $A\subset\mathbb{R}^d$, 
\[
\prob\left(\widetilde{\bt}_{nk} \in A\right)\to \prob(\bz\in A)\quad\mbox{and}\quad
\prob\left(\widehat{\bt}_{nk} \in A\right)\to \prob(\bz\in A),
\]
where $\bz\sim\mathrm{N}_d(\zero_d, \eye_d)$. This implies that
\begin{align*}
\max_{\pi\in\calS_d}\prob\left(\widetilde{\bt}_{n\pi(k)} \in A\right)\to \prob(\bz\in A),\quad
\min_{\pi\in\calS_d}\prob\left(\widetilde{\bt}_{n\pi(k)} \in A\right)\to \prob(\bz\in A),\\
\max_{\pi\in\calS_d}\prob\left(\widehat{\bt}_{n\pi(k)} \in A\right)\to \prob(\bz\in A),\quad
\min_{\pi\in\calS_d}\prob\left(\widehat{\bt}_{n\pi(k)} \in A\right)\to \prob(\bz\in A).
\end{align*}
Hence, we have
\begin{align*}
&\prob\left\{\sqrt{n}\bSigma_n(\bx_{\pi_n(k)}^*)^{-1/2}(\bW\transpose\widetilde{\bx}_{\iota_k} - \rho_n^{1/2}\bx_{\pi_n(k)}^*)\in A\right\}\\
&\quad \leq \prob\left\{\sqrt{n}\bSigma_n(\bx_{\pi_n(k)}^*)^{-1/2}(\bW\transpose\widetilde{\bx}_{\iota_k} - \rho_n^{1/2}\bx_{\pi_n(k)}^*)\in A, \iota_k = i_{\pi_n(k)}\right\} + \prob(\iota_k \neq i_{\pi_n(k)})\\
&\quad \leq \prob\left\{\sqrt{n}\bSigma_n(\bx_{\pi_n(k)}^*)^{-1/2}(\bW\transpose\widetilde{\bx}_{i_{\pi_n(k)}} - \rho_n^{1/2}\bx_{\pi_n(k)}^*)\in A\right\} + c_0n^{-2}\\
&\quad\leq \max_{\pi\in\calS_d}\prob(\widetilde{\bt}_{n\pi(k)} \in A) + c_0n^{-2}\to \prob(\bz\in A),\\
&\prob\left\{\sqrt{n}\bSigma_n(\bx_{\pi_n(k)}^*)^{-1/2}(\bW\transpose\widetilde{\bx}_{\iota_k} - \rho_n^{1/2}\bx_{\pi_n(k)}^*)\in A\right\}\\
&\quad \geq \prob\left\{\sqrt{n}\bSigma_n(\bx_{\pi_n(k)}^*)^{-1/2}(\bW\transpose\widetilde{\bx}_{\iota_k} - \rho_n^{1/2}\bx_{\pi_n(k)}^*)\in A, \iota_k = i_{\pi_n(k)}\right\}\\
&\quad = \prob\left(\widetilde{\bt}_{n\pi_n(k)}\in A\right) + \prob\left(\iota_k = i_{\pi_n(k)}\right) - \prob\left[
\left\{\widetilde{\bt}_{n\pi_n(k)}\in A\right\} \cup \left\{\iota_k = i_{\pi_n(k)}\right\}
\right]\\
&\quad \geq \prob\left(\widetilde{\bt}_{n\pi_n(k)}\in A\right) + 1 - c_0n^{-2} - 1\\
&\quad \geq \min_{\pi\in\calS_d}\prob(\widetilde{\bt}_{n\pi(k)} \in A) - c_0n^{-2}\to \prob(\bz\in A).
\end{align*}
This implies that $\widetilde{\bt}_{n\pi_n(k)}\to\mathrm{N}_d(\zero_d, \eye_d)$. The same reasoning also implies that $\widehat{\bt}_{n\pi_n(k)}\to\mathrm{N}_d(\zero_d, \eye_d)$, and the proof is thus completed.

\subsection{Proof of Theorem \ref{thm:ASE_OSE_testing}}
\label{sub:proof_of_ASE_OSE_testing}

By the proof of Theorem \ref{thm:ASE_Berry_Esseen_RDPG}, Theorem \ref{thm:eigenvector_deviation}, and Theorem \ref{thm:asymptotic_normality_OS}, under the condition that $d$ is fixed and $\lambda_d(\bDelta_n)$ is bounded away from $0$, we have
\begin{align*}
(\bW\transpose\widetilde\bx_i - \rho_n^{1/2}\bx_{0i}) &= \frac{1}{n\rho_n^{1/2}}\bDelta_n^{-1}\sum_{a = 1}^n(A_{ia} - \rho_n\bx_i\transpose\bx_a)\bx_a + o_{\prob}(n^{-1/2}),\\
(\bW\transpose\widehat\bx_i - \rho_n^{1/2}\bx_{0i}) &= \frac{1}{n\rho_n^{1/2}}\sum_{a = 1}^n\frac{(A_{ia} - \rho_n\bx_i\transpose\bx_a)\bx_a}{\bx_i\transpose\bx_a(1 - \rho_n\bx_i\transpose\bx_a)} + o_{\prob}(n^{-1/2}).
\end{align*}
Since $|A_{ij} - \rho_n\bx_i\transpose\bx_j| \leq 1$ and $n\rho_n^{1/2} = \omega(n^{1/2})$, we have,
\begin{align*}
\sqrt{n}\left\{\bW\transpose(\widetilde\bx_i - \widetilde\bx_j) - \rho_n^{1/2}(\bx_i - \bx_j)\right\}
& = \frac{1}{(n\rho_n)^{1/2}}\bDelta_n^{-1}\sum_{a = 1}^n(A_{ia} - \rho_n\bx_i\transpose\bx_a)\bx_a\\
&\quad - \frac{1}{(n\rho_n)^{1/2}}\bDelta_n^{-1}\sum_{b\neq i}^n(A_{jb} - \rho_n\bx_j\transpose\bx_b)\bx_b + o_{\prob}(1),\\
\sqrt{n}\left\{\bW\transpose(\widehat\bx_i - \widehat\bx_j) - \rho_n^{1/2}(\bx_i - \bx_j)\right\}
& = \frac{1}{(n\rho_n)^{1/2}}\sum_{a = 1}^n\frac{(A_{ia} - \rho_n\bx_i\transpose\bx_a)\bx_a}{\bx_i\transpose\bx_a(1 - \rho_n\bx_i\transpose\bx_a)}\\
&\quad - \frac{1}{(n\rho_n)^{1/2}} \sum_{b\neq i}^n\frac{(A_{jb} - \rho_n\bx_j\transpose\bx_b)\bx_b}{\bx_j\transpose\bx_b(1 - \rho_n\bx_j\transpose\bx_b)} + o_{\prob}(1).
\end{align*}
Note that 
\begin{align*}
&\frac{1}{(n\rho_n)^{1/2}}\bDelta_n^{-1}\sum_{a = 1}^n(A_{ia} - \rho_n\bx_i\transpose\bx_a)\bx_a - \frac{1}{(n\rho_n)^{1/2}}\bDelta_n^{-1}\sum_{b\neq i}^n(A_{jb} - \rho_n\bx_j\transpose\bx_b)\bx_b\quad\mbox{and}\\
&\frac{1}{(n\rho_n)^{1/2}}\sum_{a = 1}^n\frac{(A_{ia} - \rho_n\bx_i\transpose\bx_a)\bx_a}{\bx_i\transpose\bx_a(1 - \rho_n\bx_i\transpose\bx_a)} - \frac{1}{(n\rho_n)^{1/2}} \sum_{b\neq i}^n\frac{(A_{jb} - \rho_n\bx_j\transpose\bx_b)\bx_b}{\bx_j\transpose\bx_b(1 - \rho_n\bx_j\transpose\bx_b)}
\end{align*}
are sums of mean-zero independent random vectors. In addition, observe that $\bSigma_n(\bx_i)$,  $\bSigma_n(\bx_i)^{-1}$,  $\bG_n(\bx_i)$, and  $\bG_n(\bx_i)^{-1}$ are all $O(1)$ and $\Omega(1)$, and that
\begin{align*}
&\var\left\{
\frac{1}{(n\rho_n)^{1/2}}\bDelta_n^{-1}\sum_{a = 1}^n(A_{ia} - \rho_n\bx_i\transpose\bx_a)\bx_a - \frac{1}{(n\rho_n)^{1/2}}\bDelta_n^{-1}\sum_{b\neq i}^n(A_{jb} - \rho_n\bx_j\transpose\bx_b)\bx_b
\right\}
\\
&\quad
 = \bSigma_n(\bx_i) + \bSigma_n(\bx_j) + o(1),\\
&\var\left\{\frac{1}{(n\rho_n)^{1/2}}\sum_{a = 1}^n\frac{(A_{ia} - \rho_n\bx_i\transpose\bx_a)\bx_a}{\bx_i\transpose\bx_a(1 - \rho_n\bx_i\transpose\bx_a)} - \frac{1}{(n\rho_n)^{1/2}} \sum_{b\neq i}^n\frac{(A_{jb} - \rho_n\bx_j\transpose\bx_b)\bx_b}{\bx_j\transpose\bx_b(1 - \rho_n\bx_j\transpose\bx_b)}
\right\}
\\
&\quad
 = \bG_n(\bx_i)^{-1} + \bG_n(\bx_j)^{-1} + o(1),\\
&\sum_{a = 1}^n\expect\left\|\bDelta_n^{-1}\frac{(A_{ia} - \rho_n\bx_i\transpose\bx_a)\bx_a}{(n\rho_n)^{1/2}}\right\|_2^3 + \sum_{b\neq i}^n\expect\left\|\bDelta_n^{-1}\frac{(A_{jb} - \rho_n\bx_j\transpose\bx_b)\bx_b}{(n\rho_n)^{1/2}}\right\|_2^3\\
&\quad\leq\frac{ \|\bDelta_n^{-1}\|_2^3}{(n\rho_n)^{1/2}}\left\{\frac{1}{n}\sum_{a = 1}^n\bx_i\transpose\bx_a(1 - \rho_n\bx_i\transpose\bx_a)\|\bx_a\|_2^3 + \frac{1}{n}\sum_{b\neq i}^n\bx_j\transpose\bx_b(1 - \rho_n\bx_j\transpose\bx_b)\|\bx_b\|_2^3\right\}\\
&\quad\leq \frac{\|\bDelta_n^{-1}\|_2^3}{(n\rho_n)^{1/2}}\to 0,\\
&\sum_{a = 1}^n\expect\left\|\frac{(A_{ia} - \rho_n\bx_i\transpose\bx_a)\bx_a}{\bx_i\transpose\bx_a(1 - \rho_n\bx_i\transpose\bx_a)(n\rho_n)^{1/2}}\right\|_2^3 + \sum_{b\neq i}^n\expect\left\| \frac{(A_{jb} - \rho_n\bx_j\transpose\bx_b)\bx_b}{\bx_i\transpose\bx_a(1 - \rho_n\bx_i\transpose\bx_a)(n\rho_n)^{1/2}}\right\|_2^3\\
&\quad\leq \frac{1}{(n\rho_n)^{1/2}}\left\{\frac{1}{n}\sum_{a = 1}^n\frac{\|\bx_a\|_2^3}{\{\bx_i\transpose\bx_a(1 - \rho_n\bx_i\transpose\bx_a)\}^2} + \frac{1}{n}\sum_{b\neq i}^n\frac{\|\bx_b\|_2^3}{\{\bx_j\transpose\bx_b(1 - \rho_n\bx_j\transpose\bx_b)\}^2}\right\}\\
&\quad\leq \frac{2}{(n\rho_n)^{1/2}\delta^4}\to 0.
\end{align*}
Therefore, by Lyapunov's central limit theorem (see, for example, Theorem 7.1.2 in \cite{chung2001course}), 
\begin{align*}
\sqrt{n}\{\bSigma_n(\bx_i) + \bSigma_n(\bx_j)\}^{-1/2}\left\{\bW\transpose(\widetilde\bx_i - \widetilde\bx_j) - \rho_n^{1/2}(\bx_i - \bx_j)\right\}\overset{\calL}{\to}\mathrm{N}(\zero_d, \eye_d),\\
\sqrt{n}\{\bG_n(\bx_i)^{-1} + \bG_n(\bx_j)^{-1}\}^{-1/2}\left\{\bW\transpose(\widehat\bx_i - \widehat\bx_j) - \rho_n^{1/2}(\bx_i - \bx_j)\right\}\overset{\calL}{\to}\mathrm{N}(\zero_d, \eye_d).
\end{align*}
We next show that
\begin{align*}
&\rho_n^{-1}\bW\transpose\widetilde{\bDelta}_n\bW = \bDelta_n + o_{\prob}(1),
\quad\bW\transpose\widetilde{\bSigma}_n(\widetilde\bx_i)\bW = \bSigma_n(\bx_i) + o_{\prob}(1),
\quad\bW\transpose\widetilde{\bG}_n(\widetilde\bx_i)\bW = \bG_n(\bx_i) + o_{\prob}(1).
\end{align*}
For the first equation, by Lemma \ref{lemma:RX_Frobenius_norm_bound}, we have
\begin{align*}
\|\rho_n^{-1}\bW\transpose\widetilde{\bDelta}_n\bW - \bDelta_n\|_2
&\leq \frac{1}{n\rho_n}\|\widetilde{\bX}\bW - \rho_n^{1/2}\bX\|_{\mathrm{F}}\|\widetilde{\bX}\bW\|_2 + \frac{1}{n\rho_n}\|\rho_n^{1/2}\bX\|_2\|\widetilde{\bX}\bW - \rho_n^{1/2}\bX\|_2\\
& = \frac{1}{n\rho_n}O_{\prob}(1)\times O_{\prob}((n\rho_n)^{1/2}) = O_{\prob}\left(\frac{1}{(n\rho_n)^{1/2}}\right).
\end{align*}
For the second equation, we denote
\begin{align*}
\widetilde{\bD}_i &= \frac{1}{\rho_n}\mathrm{diag}\left\{\widetilde{\bx}_i\transpose\widetilde{\bx}_1(1 - \widetilde{\bx}_i\transpose\widetilde{\bx}_1), \ldots, \widetilde{\bx}_i\transpose\widetilde{\bx}_n(1 - \widetilde{\bx}_i\transpose\widetilde{\bx}_n)\right\},\\
\bD_i &= \mathrm{diag}\left\{\bx_i\transpose\bx_1(1 - \rho_n\bx_i\transpose\bx_1), \ldots, \bx_i\transpose\bx_n(1 - \rho_n\bx_i\transpose\bx_n)\right\}.
\end{align*}
By Result \ref{result:two_to_infinity_norm_consistency_X} and Corollary \ref{corr:Two_to_infinity_norm_eigenvector_bound}, we have
\begin{align*}
\max_{i,j\in [n]}|\widetilde{\bx}_i\transpose\widetilde{\bx}_j - \rho_n\bx_i\transpose\bx_j|
&\leq \left(\|\widetilde{\bX}\bW - \rho_n^{1/2}\bX\|_{2\to\infty} + 2\max_{i,j\in [n]}\|\rho_n^{1/2}\bX\|_{2\to\infty}\right)\|\widetilde{\bX}\bW - \rho_n^{1/2}\bX\|_{2\to\infty}
 = O_{\prob}\left(\rho_n\sqrt{\frac{\log n}{n\rho_n}}\right)
\end{align*}
and
\begin{align*}
\max_{i,j\in [n]}|(\widetilde{\bx}_i\transpose\widetilde{\bx}_j)^2 - (\rho_n\bx_i\transpose\bx_j)^2|
&\leq \max_{i,j\in [n]}|\widetilde{\bx}_i\transpose\widetilde{\bx}_j - \rho_n\bx_i\transpose\bx_j|\left(|\widetilde{\bx}_i\transpose\widetilde{\bx}_j - \rho_n\bx_i\transpose\bx_j| + 2\rho_n\bx_i\transpose\bx_j\right)
 = O_{\prob}\left(\rho_n^2\sqrt{\frac{\log n}{n\rho_n}}\right).
\end{align*}
It follows that
\begin{align*}
\|\widetilde{\bD}_i - \bD_i\|_2
&\leq \rho_n^{-1}\max_{i,j\in [n]}|\widetilde{\bx}_i\transpose\widetilde{\bx}_j - \rho_n\bx_i\transpose\bx_j| + \rho_n^{-1}\max_{i,j\in [n]}|(\widetilde{\bx}_i\transpose\widetilde{\bx}_j)^2 - (\rho_n\bx_i\transpose\bx_j)^2|
= O_{\prob}\left(\sqrt{\frac{\log n}{n\rho_n}}\right).
\end{align*}
Therefore, by Lemma \ref{lemma:RX_Frobenius_norm_bound},
\begin{align*}
&\left\|\bW\transpose\frac{1}{n\rho_n^2}\sum_{j = 1}^n\widetilde{\bx}_i\transpose\widetilde{\bx}_j(1 - \widetilde{\bx}_i\transpose\widetilde{\bx}_j)\widetilde{\bx}_j\widetilde{\bx}_j\transpose\bW - \frac{1}{n}\sum_{j = 1}^n\bx_i\transpose\bx_j(1 - \rho_n\bx_i\transpose\bx_j)\bx_j\bx_j\transpose\right\|_2\\
&\quad = \frac{1}{n}\left\|(\rho_n^{-1/2}\widetilde{\bX}\bW)\transpose\widetilde{\bD}_i(\rho_n^{-1/2}\widetilde{\bX}\bW) - \bX\transpose\bD_i\bX\right\|_2\\
&\quad\leq \frac{1}{n}\left\|\rho_n^{-1/2}\widetilde{\bX}\bW - \bX\right\|_2\|\widetilde{\bD}_i\|_2\|\rho_n^{-1/2}\widetilde{\bX}\bW\|_2 + \frac{1}{n}\|\bX\|_2\|\widetilde{\bD}_i - \bD_i\|_2\|\rho_n^{-1/2}\bX\bW\|_2\\
&\quad\quad + \frac{1}{n}\|\bX\|_2\|\bD_i\|_2\|\rho_n^{-1/2}\widetilde{\bX}\bW - \bX\|_2\\
&\quad = \frac{1}{n}\times O_{\prob}(\rho_n^{-1/2})\times O_{\prob}(1)\times O_{\prob}(n^{1/2}) + \frac{1}{n}\times O(n^{1/2})\times O_{\prob}\left(\sqrt{\frac{\log n}{n\rho_n}}\right)\times O_{\prob}(n^{1/2})\\
&\quad\quad + \frac{1}{n}\times O(n^{1/2})\times O(1)\times O_{\prob}(\rho_n^{-1/2}) = o_{\prob}(1).
\end{align*}
Hence, for the second equation, we have
\begin{align*}
\bW\transpose\widetilde{\bSigma}_n(\widetilde\bx_i)\bW
& = (\rho_n^{-1}\bW\transpose\widetilde{\bDelta}_n\bW)^{-1} \bW\transpose\left\{\frac{1}{n\rho_n^2}\sum_{j = 1}^n\widetilde{\bx}_i\transpose\widetilde{\bx}_j(1 - \widetilde{\bx}_i\transpose\widetilde{\bx}_j)\widetilde{\bx}_j\widetilde{\bx}_j\transpose\bW\right\}(\rho_n^{-1}\bW\transpose\widetilde{\bDelta}_n\bW)^{-1}\\
& = \{\bDelta_n + o_{\prob}(1)\}^{-1}
\left\{\frac{1}{n}\sum_{j = 1}^n\bx_i\transpose\bx_j(1 - \rho_n\bx_i\transpose\bx_j)\bx_j\bx_j\transpose + o_{\prob}(1)\right\}
\{\bDelta_n + o_{\prob}(1)\}^{-1}\\
& = \bDelta_n^{-1} 
\left\{\frac{1}{n}\sum_{j = 1}^n\bx_i\transpose\bx_j(1 - \rho_n\bx_i\transpose\bx_j)\bx_j\bx_j\transpose\right\}\bDelta_n^{-1} + o_{\prob}(1)\\
& = \bSigma_n(\bx_i) + o_{\prob}(1).
\end{align*}
For the third equation, it follows directly from Lemma \ref{lemma:Lipschitz_continuity_Hmatrix} with $t = \log n$. 
Hence, we conclude that 
\begin{align*}
&\bW\transpose\widetilde{\bSigma}_{ij}\bW = \bW\transpose\widetilde{\bSigma}_n(\widetilde{\bx}_i)\bW + \bW\transpose\widetilde{\bSigma}_n(\widetilde{\bx}_j)\bW = \bSigma_n(\bx_i) + \bSigma_n(\bx_j) + o_{\prob}(1),\\
&\bW\transpose\widetilde{\bG}_{ij}\bW = \bW\transpose\widetilde{\bG}_n(\widetilde{\bx}_i)^{-1}\bW + \bW\transpose\widetilde{\bG}_n(\widetilde{\bx}_j)^{-1}\bW = \bG_n(\bx_i)^{-1} + \bG_n(\bx_j)^{-1} + o_{\prob}(1).
\end{align*}
By Slutsky's lemma, under the null distribution $H_0:\bx_i = \bx_j$, we have
\begin{align*}
T_{ij}^{(\mathrm{ASE})}& = n(\widetilde{\bx}_i - \widetilde{\bx}_j)\transpose\widetilde{\bSigma}_{ij}(\widetilde{\bx}_i - \widetilde{\bx}_j)\\
& = n\{\bW\transpose(\widetilde{\bx}_i - \widetilde{\bx}_j)\}\transpose\bW\transpose\widetilde{\bSigma}_{ij}\bW\{\bW\transpose(\widetilde{\bx}_i - \widetilde{\bx}_j)\}\\
& = n\{\bW\transpose(\widetilde{\bx}_i - \widetilde{\bx}_j)\}\transpose[\{\bSigma_n(\bx_i) + \bSigma_n(\bx_j)\}^{-1} + o_{\prob}(1)]\{\bW\transpose(\widetilde{\bx}_i - \widetilde{\bx}_j)\}\\
&\overset{\calL}{\to}\chi^2_d,\\
T_{ij}^{(\mathrm{OSE})}& = 
n(\widehat{\bx}_i - \widehat{\bx}_j)\transpose\widetilde{\bG}_{ij}(\widehat{\bx}_i - \widehat{\bx}_j)\\
& = n\{\bW\transpose(\widehat{\bx}_i - \widehat{\bx}_j)\}\transpose\bW\transpose\widetilde{\bG}_{ij}\bW\{\bW\transpose(\widehat{\bx}_i - \widehat{\bx}_j)\}\\
& = n\{\bW\transpose(\widehat{\bx}_i - \widehat{\bx}_j)\}\transpose[\{\bG_n(\bx_i)^{-1} + \bG_n(\bx_j)^{-1}\}^{-1} + o_{\prob}(1)]
\{\bW\transpose(\widehat{\bx}_i - \widehat{\bx}_j)\}\\
&\overset{\calL}{\to}\chi^2_d.
\end{align*}
We now consider the distributions of $T_{ij}^{(\mathrm{ASE})}$ and $T_{ij}^{(\mathrm{OSE})}$ under the alternative $H_A:\bx_i\neq \bx_j$ but $(n\rho_n)^{1/2}(\bx_i - \bx_j) \to \bmu\neq\zero_d$. Under the condition that $\bSigma_n(\bx_i) \to \bSigma_i$ and $\bG_n(\bx_i)\to \bG_i$, we have,
\begin{align*}
\sqrt{n}\{\bSigma_n(\bx_i) + \bSigma_n(\bx_j)\}^{-1/2}\bW\transpose(\widetilde{\bx}_i - \widetilde{\bx}_j)
&\overset{\calL}{\to}\mathrm{N}\left(\bmu\transpose(\bSigma_i + \bSigma_j)^{-1}\bmu, \eye_d\right),\\
\sqrt{n}\{\bG_n(\bx_i)^{-1} + \bG_n(\bx_j)^{-1}\}^{-1/2}\bW\transpose(\widetilde{\bx}_i - \widetilde{\bx}_j)
&\overset{\calL}{\to}\mathrm{N}\left(\bmu\transpose(\bG_i + \bG_j)^{-1}\bmu, \eye_d\right).
\end{align*}
Since $\bW\transpose(\widetilde{\bx}_i - \widetilde{\bx}_j) = O_{\prob}(n^{-1/2})$ and $\bW\transpose(\widehat{\bx}_i - \widehat{\bx}_j) = O_{\prob}(n^{-1/2})$, 
it follows that under $H_A:\bx_i\neq \bx_j$ but $(n\rho_n)^{1/2}(\bx_i - \bx_j)\to \bmu\neq\zero_d$,
\begin{align*}
T_{ij}^{(\mathrm{ASE})}& = n(\widetilde{\bx}_i - \widetilde{\bx}_j)\transpose\widetilde{\bSigma}_{ij}(\widetilde{\bx}_i - \widetilde{\bx}_j)\\
& = n\{\bW\transpose(\widetilde{\bx}_i - \widetilde{\bx}_j)\}\transpose\bW\transpose\widetilde{\bSigma}_{ij}\bW\{\bW\transpose(\widetilde{\bx}_i - \widetilde{\bx}_j)\}\\
& = n\{\bW\transpose(\widetilde{\bx}_i - \widetilde{\bx}_j)\}\transpose[\{\bSigma_n(\bx_i) + \bSigma_n(\bx_j)\}^{-1} + o_{\prob}(1)]\{\bW\transpose(\widetilde{\bx}_i - \widetilde{\bx}_j)\}\\
& = \left\|\sqrt{n}\{\bSigma_n(\bx_i) + \bSigma_n(\bx_j)\}^{-1/2}\bW\transpose(\widetilde{\bx}_i - \widetilde{\bx}_j)\right\|_2^2 + o_{\prob}(n\|\bW\transpose(\widetilde{\bx}_i - \widetilde{\bx}_j)\|_2^2)\\
&\overset{\calL}{\to}\chi^2_d(\bmu\transpose(\bSigma_i + \bSigma_j)^{-1}\bmu),\\
T_{ij}^{(\mathrm{OSE})}& = 
n(\widehat{\bx}_i - \widehat{\bx}_j)\transpose\widetilde{\bG}_{ij}(\widehat{\bx}_i - \widehat{\bx}_j)\\
& = n\{\bW\transpose(\widehat{\bx}_i - \widehat{\bx}_j)\}\transpose\bW\transpose\widetilde{\bG}_{ij}\bW\{\bW\transpose(\widehat{\bx}_i - \widehat{\bx}_j)\}\\
& = n\{\bW\transpose(\widehat{\bx}_i - \widehat{\bx}_j)\}\transpose[\{\bG_n(\bx_i)^{-1} + \bG_n(\bx_j)^{-1}\}^{-1} + o_{\prob}(1)]
\{\bW\transpose(\widehat{\bx}_i - \widehat{\bx}_j)\}\\
& = \left\|\sqrt{n}\{\bG_n(\bx_i)^{-1} + \bG_n(\bx_j)^{-1}\}^{-1/2}\bW\transpose(\widehat{\bx}_i - \widehat{\bx}_j)\right\|_2^2 + o_{\prob}(n\|\bW\transpose(\widehat{\bx}_i - \widehat{\bx}_j)\|_2^2)\\
&\overset{\calL}{\to}\chi^2_d(\bmu\transpose(\bG_i^{-1} + \bG_j^{-1})^{-1}\bmu).
\end{align*}
The proof is thus completed.


\section{The successive projection algorithm}
\label{app:SPA_algorithm}

This section provides the detailed successive projection algorithm proposed in \cite{6656801}, which finds the row indices corresponding to the pure nodes based on the noisy observed adjacency matrix $\bA$. It is used to construct the estimator $\widehat{\bTheta}$ for the membership profile matrix $\bTheta$ in a mixed membership stochastic block model in Section \ref{sub:subsequent_graph_inference} of the manuscript. 
\begin{algorithm}[htbp] 
  \renewcommand{\algorithmicrequire}{\textbf{Input:}}
  \renewcommand{\algorithmicensure}{\textbf{Output:} }
  \caption{Successive projection algorithm (SPA)} 
  \label{alg:SPA} 
  \begin{algorithmic}[1] 
    \State{\textbf{Input:}
      Data matrix $\bA = [A_{ij}]_{n\times n}$, rank $d$}
    \State{Compute the leading eigenvectors of $\bA$: $\bA\bU_\bA = \bU_\bA\bS_\bA$, where $\bU_\bA\in\mathbb{O}(n, d)$, $\bS_\bA = \mathrm{diag}(\widehat{\lambda}_1,\ldots,\widehat{\lambda}_d)$, and  $|\widehat\lambda_1|\geq|\widehat\lambda_2|\geq\ldots\geq|\widehat\lambda_n|$. 
        } 

    \State{Let $\bR = \bU_\bA\bU_\bA\transpose$, $J = \varnothing$, and $k = 1$.}

    \State{While $\bR\neq \zero_{n\times n}$ and $j\leq d$

  Set $j^* \longleftarrow \argmax_{j\in [n]}\|\bR\be_j\|_2^2$. If there are ties, pick $j^*$ to be the smallest index. 

  Set $\bu_j \longleftarrow \bR\be_{j^*}$. 

  Set $\bR \longleftarrow (\eye_n - \|\bu_j\|_2^{-2}\bu_j\bu_j\transpose)\bR$.

  Let $J\longleftarrow J\cup\{j^*\}$.

  Set $j \longleftarrow j + 1$.

    \noindent
    End While}
    \State{\textbf{Output: } Set of indices $J$. }
  \end{algorithmic}
\end{algorithm}

\section{ Additional simulation examples}
\label{sec:additional_simulation_examples}

\subsection{ Symmetric noisy matrix completion}
\label{sub:SNMC_numerical}

We first consider a synthetic example for the symmetric noisy matrix completion problem and illustrate Theorem \ref{thm:SNMC}. The setup here is similar to the two-block stochastic block model in Section \ref{sub:a_motivating_example}. We set $n = 5000$, $a = 0.9$, $b = 0.05$, $\alpha = 5$, $n\rho_n = \log n$, $\lambda_1 = n\rho_n(a + b)/2$, $\lambda_2 = n\rho_n(a - b)/2$, $\bu_1 = n^{-1/2}[1,\ldots,1]\transpose$, $\bu_2 = n^{-1/2}[1,\ldots,1,-1,\ldots,-1]\transpose$ (the first $n/2$ entries of $\bu_2$ are $n^{-1/2}$ and the remaining entries are $-n^{-1/2}$), and $\bX = \sqrt{n}[\bu_1, \bu_2]$. Let $\eps_{ij}\sim\mathrm{N}(0, \rho_n^4)$ and $I_{ij}\sim\mathrm{Bernoulli}(\rho_n)$ independently for $i,j\in [n]$, $i\leq j$, and let $\eps_{ij} = \eps_{ji}$, $I_{ij} = I_{ji}$ if $i > j$. The noisy observed matrix $\bA = [A_{ij}]_{n\times n}$ is generated by taking $A_{ij} = (\rho_n\bx_i\transpose\bx_j + \eps_{ij})I_{ij}/\rho_n$, where $\bx_i$ is the $i$th row of $\bX$, $i\in [n]$. We follow the notations in Section \ref{sub:a_motivating_example} by letting $\widehat{\bu}_2$ be the unscaled eigenvector (i.e., $\|\widehat{\bu}_2\|_2 = 1$) of $\bA$ corresponding to $\lambda_2(\bA)$, $\widehat{\bv}_2 = \lambda_2(\bA)^{1/2}\widehat{\bu}_2$ be the scaled eigenvector, $\bv_2 = \lambda_2^{1/2}\bu_2$, $\widehat{u}_{i2},\widehat{v}_{i2}$, $u_{i2}$, and $v_{i2}$ be the $i$th coordinate of $\widehat{\bu}_2$, $\widehat{\bv}_2$, $\bu_2$, and $\bv_2$, respectively. Then by Theorem \ref{thm:ASE_Berry_Esseen}, 
\[
\sqrt{n}(\mathrm{sgn}(\widehat{\bu}_2\transpose\bu_2)\widehat{v}_{i2} - v_{i2})\overset{\calL}{\to}\mathrm{N}\left(0, \frac{a^2 + b^2}{a - b}\right),\quad
n\rho_n^{1/2}(\mathrm{sgn}(\widehat{\bu}_2\transpose\bu_2)\widehat{u}_{i2} - u_{i2})\overset{\calL}{\to}\mathrm{N}\left(0, \frac{2(a^2 + b^2)}{2(a - b)^2}\right).
\]

We next generate $3000$ independent Monte Carlo replicates of $\bA$ from $\mathrm{SNMC}(\rho_n\bX\bX\transpose, \rho_n, \rho_n^2)$. For each realization of $\bA$, we compute the unscaled eigenvector $\widehat{\bu}_2$, the scaled eigenvector $\widehat{\bv}_2$, and plot the histograms of $\sqrt{n}(\mathrm{sgn}(\widehat{\bu}_2\transpose\bu_2)\widetilde{v}_{12} - v_{12})$ and $n\rho_n^{1/2}(\mathrm{sgn}(\widehat{\bu}_2\transpose\bu_2)\widetilde{v}_{12} - v_{12})$ in Figure \ref{fig:Eigenvector_SNMC_normality}, together with their respective asymptotic normal densities. It is clear that the empirical distributions of the Monte Carlo samples of $\sqrt{n}(\mathrm{sgn}(\widehat{\bu}_2\transpose\bu_2)\widetilde{v}_{12} - v_{12})$ and $n\rho_n^{1/2}(\mathrm{sgn}(\widehat{\bu}_2\transpose\bu_2)\widetilde{v}_{12} - v_{12})$ can be well approximated by their respective asymptotic normal distributions.
 This numerical observation is in agreement with the asymptotic normality established in Theorem \ref{thm:SNMC}. 
\begin{figure}[t]
\includegraphics[width = 15cm]{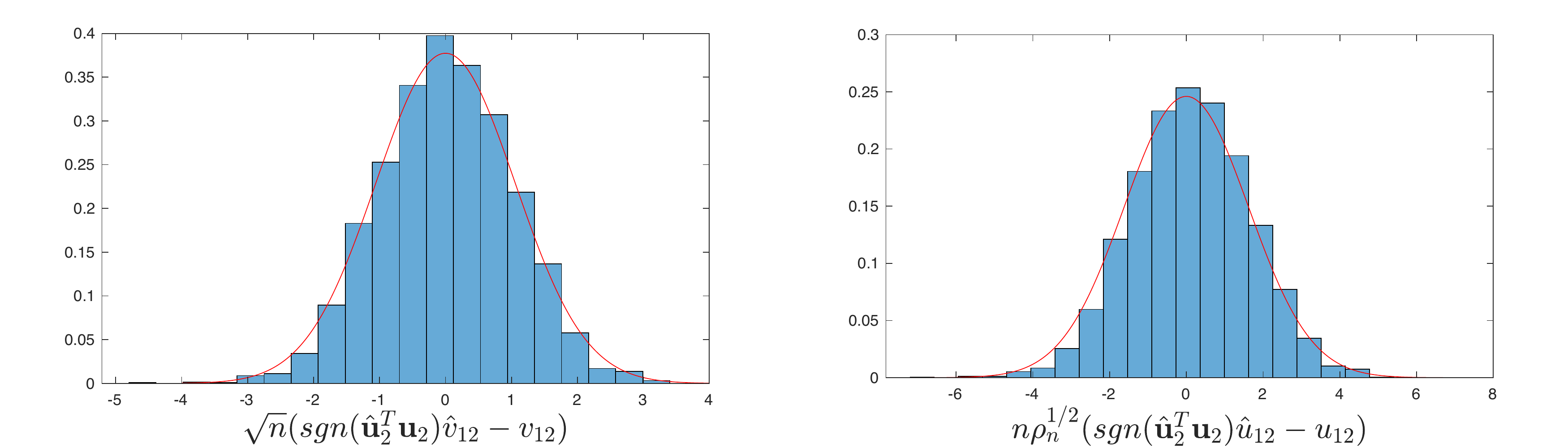}
\caption{Numerical results of Section \ref{sub:SNMC_numerical}. The left and the right panels are the histograms of $\sqrt{n}(\mathrm{sgn}(\widehat{\bu}_2\transpose\bu_2)\widetilde{v}_{12} - v_{12})$ and $n\rho_n^{1/2}(\mathrm{sgn}(\widehat{\bu}_2\transpose\bu_2)\widetilde{v}_{12} - v_{12})$ over the $3000$ Monte Carlo replicates with the asymptotic normal densities highlighted in the red curves, respectively. }
\label{fig:Eigenvector_SNMC_normality}
\end{figure}

\subsection{ Eigenvector-based inference in random graphs}
\label{sub:subsequent_inference_numerical}

We now consider the numerical experiments for the two subsequent random graph inference tasks in Section \ref{sub:subsequent_graph_inference}. Consider a mixed membership stochastic block model specified as follows. The block probability matrix $\bB$ is a $2\times 2$ symmetric matrix with the diagonals being $0.9$ and the off-diagonals being $0.1$. The corresponding two pure nodes are $\bx_1^* = [0.7071, 0.6325]\transpose$ and $\bx_2^* = [0.7071, -0.6325]\transpose$. We set the number of vertices to be $n = 4500$ with $n_0 = 900$ pure nodes in each community and set the sparsity factor to be $\rho_n = 5(\log n)^{3/2}/n$. The membership profile matrix $\bTheta$ has the form
\[
\bTheta = \begin{bmatrix*}
\one_{n_0} & \zero_{n_0}\\
\zero_{n_0} & \one_{n_0}\\
\bt & \one_{n - 2n_0} - \bt
\end{bmatrix*},
\]
where $\bt\in \mathbb{R}^{n - 2n_0}$ is the vector whose entries are equidistant points over $[0.2, 0.8]$ and $\one_{n_0}\in \mathbb{R}^{n_0}$ is the vector of all ones. Equivalently, the mixed membership stochastic block model can be written as $\mathrm{RPDG}(\rho_n^{1/2}\bX)$ with $\bX = \bTheta\bX^*$, where $\bX^* = [\bx_1^*, \bx_2^*]\transpose$. Let $\bx_i$ be the $i$th row of $\bX$ for each $i\in [n]$. 

We draw $2000$ independent Monte Carlo replicates of $\bA$ from $\mathrm{MMSBM}(\bTheta, \bB, \rho_n)$ specified above and investigate the performance of the two inference tasks in Section \ref{sub:subsequent_graph_inference}: The estimation of the pure nodes and the hypothesis testing of the equality of latent positions. For the first task, given a realization $\bA\sim\mathrm{MMSBM}(\bTheta, \bB, \rho_n)$, we 
first compute the adjacency spectral embedding $\widetilde{\bX}$ of $\bA$ into $\mathbb{R}^2$ and then apply Algorithm \ref{alg:SPA} to obtain the estimated pure node indices $J$. Next, we compute $\iota_k$ using formula \eqref{eqn:pure_node_index} for $k = 1,2$, with the tuning parameter $\eta$ being $0.1$. For each $k = 1,2$, we then
compute the two estimators given by $\widetilde{\bx}_{\iota_k}$ (the estimator based on the adjacency spectral embedding) and $\widehat{\bx}_{\iota_k}$ (the estimator based on the one-step refinement). For the second task, we let $i = 1$ and consider testing the null hypothesis $H_0:\bx_i = \bx_j$ against a collection of alternative hypotheses $H_A:\bx_i\neq \bx_j$ for $j\in \{1901,2101,\ldots,3701\}$ using the two test statistics $T_{ij}^{(\mathrm{ASE})}$ and $T_{ij}^{(\mathrm{OSE})}$ defined in Section \ref{sub:subsequent_graph_inference}. We set the significance level to be $0.05$. 

For the estimation of the pure nodes, for each $k = 1,2$, we compute the empirical mean-squared errors (MSE) for estimating $\bx_k^*$ using $\widetilde{\bx}_{\iota_k}$ and $\widehat{\bx}_{\iota_k}$ across the $2000$ repeated experiments. We also compute the corresponding sample covariance matrices. These numerical results are summarized in Table \ref{table:subsequent_inference_I}. Clearly, Table \ref{table:subsequent_inference_I} suggests that the estimators $\widehat{\bx}_{\iota_1},\widehat{\bx}_{\iota_2}$ based on the one-step refinement have smaller mean-squared errors and smaller variances compared to the estimators $\widetilde{\bx}_{\iota_1},\widetilde{\bx}_{\iota_2}$ based on the adjacency spectral embedding. This phenomenon validates Theorem \ref{thm:MMSBM} empirically. 
\begin{table}[htbp!]
  \caption{Numerical results for Section \ref{sub:subsequent_inference_numerical}: Summary statistic for estimating the pure nodes. }
  \centering{%
  \begin{tabular}{c | c c c c }
    \hline
    Pure node & MSE for $\widetilde{\bx}_{\iota_k}$ & MSE for $\widehat{\bx}_{\iota_k}$ & Sample covariance for $\widetilde{\bx}_{\iota_k}$ & Sample covariance for $\widehat{\bx}_{\iota_k}$\\
    \hline
    $\bx_1^*$  & $7.4\times 10^{-4}$ & $5.4\times 10^{-4}$ & $\frac{1}{n}\begin{bmatrix*}1.00 & 0.90\\0.90 & 2.25\end{bmatrix*}$ & $\frac{1}{n}\begin{bmatrix*}0.97 & 0.75\\0.75 & 1.34\end{bmatrix*}$ \\
    \hline
    $\bx_2^*$ & $7.8\times 10^{-4}$ & $5.4\times 10^{-4}$ & $\frac{1}{n}\begin{bmatrix*}0.99 & -0.96\\-0.96 & 2.36\end{bmatrix*}$ & $\frac{1}{n}\begin{bmatrix*}0.95 & -0.75\\-0.75 & 1.36\end{bmatrix*}$\\
    \hline
  \end{tabular}%
  }
  \label{table:subsequent_inference_I}
\end{table}%
For the hypothesis testing of the equality of the latent positions, we compare the empirical powers of the two testing procedures across the $2000$ repeated experiments. Below, Table \ref{table:subsequent_inference_II} tabulates the empirical powers of tests (i.e., the numbers of successful rejections divided by $2000$) based on $T_{ij}^{(\mathrm{ASE})}$ and $T_{ij}^{(\mathrm{OSE})}$ as functions of the distance $\|\bx_i - \bx_j\|_2$ when $j$ varies in $\{1901, 2101,\ldots, 3701\}$. It is clear from Table \ref{table:subsequent_inference_II} that the test statistic $T_{ij}^{(\mathrm{OSE})}$ is more powerful than the test statistic $T_{ij}^{(\mathrm{ASE})}$, which validates Theorem \ref{thm:ASE_OSE_testing} empirically.  
\begin{table}[htbp!]
  \caption{Numerical results for Section \ref{sub:subsequent_inference_numerical}: Power comparison for testing the equality of latent positions. }
  \centering{%
  \begin{tabular}{c | c c c c  c c c c  c c }
    \hline
    $\|\bx_i - \bx_j\|_2$ &     0.31  &  0.37  &  0.42  &  0.48  &  0.53  &  0.59  &  0.65  &  0.70  &  0.76  &  0.82  \\
    \hline
    Power of $T_{ij}^{(\mathrm{ASE})}$ &  0.33 & 0.42 & 0.52 & 0.61 & 0.70 & 0.76 & 0.85 & 0.89 & 0.94 &0.95 \\
    \hline
    Power of $T_{ij}^{(\mathrm{OSE})}$ & 0.35 & 0.45 & 0.56 & 0.67 & 0.75 & 0.80 & 0.88 & 0.92 & 0.95 & 0.97 \\
    \hline
  \end{tabular}%
  }
  \label{table:subsequent_inference_II}
\end{table}%

\end{appendix}

\bibliographystyle{acm} 
\bibliography{reference1, reference2}       

\end{document}